\documentclass[11pt,oneside]{amsart}
\usepackage{amsmath,amsfonts,amsgen,amstext,amsbsy,amsopn,amsthm, amssymb}
\usepackage{multirow, verbatim, cancel, mathabx}
\usepackage{bbold}
\usepackage{enumitem}
\usepackage[T1]{fontenc}
\usepackage[colorlinks=true,hyperindex=true]{hyperref}
\usepackage{tikz}
\usetikzlibrary{calc}
\usepackage{verbatim}
\usetikzlibrary{arrows,chains,matrix,positioning,scopes}
\usepackage{mathrsfs}
\usepackage{color,soul}
\usepackage{xcolor}
\usepackage{tabulary}
\usepackage{pgfplots}
\usepackage{mathrsfs}
\usepackage{bbm}
\usepackage{dsfont}
\usepackage{bm}
\usepackage{mathtools}
\usepackage{float}

\DeclarePairedDelimiter\floor{\lfloor}{\rfloor}
\usepackage{arydshln}
\usepackage{color,soul}
\usepackage{subcaption}
 \usepackage{soul,xcolor}
\usepackage{wrapfig}
\usepackage{lscape}
\usepackage{rotating}
\usepackage{epstopdf}
\usepackage{float}
\usepackage{flafter}

\usepackage{array}
\newcolumntype{R}[1]{>{\raggedleft\arraybackslash}p{#1}}

\usepackage{tabu} 
\usepackage{booktabs}

\usepackage{mathtools}

\usepackage[mathscr]{euscript}

\DeclareSymbolFont{rsfs}{U}{rsfs}{m}{n}
\DeclareSymbolFontAlphabet{\mathscrsfs}{rsfs}

\makeatletter

\usepackage{tkz-fct} 
\usepackage{xcolor}

\usepackage{tkz-fct} 
\usepackage{pgfplots}
\usepackage{setspace}
\usetikzlibrary{matrix}
\usepackage{tikz}

\numberwithin{equation}{section}

\newcounter{smallarabics}

\newcounter{smallroman}
\newenvironment{romanenumerate}
{\begin{list}{{\normalfont\textrm{(\roman{smallroman})}}}
  {\usecounter{smallroman}\setlength{\itemindent}{0cm}
   \setlength{\leftmargin}{5ex}\setlength{\labelwidth}{4ex}
   \setlength{\topsep}{0.75\parsep}\setlength{\partopsep}{0ex}
   \setlength{\itemsep}{0ex}}}
{\end{list}}

\newcommand{\ben}{\begin{romanenumerate}}  
\newcommand{\een}{\end{romanenumerate}}  

\newtheorem{theoreme}{theorem }[section]
\newtheorem{theorem}[theoreme]{Theorem}
\newtheorem{proposition}[theoreme]{Proposition}
\newtheorem{Lemma}[theoreme]{Lemma}

\newtheorem{corollary}[theoreme]{Corollary}
\newtheorem{conjecture}[theoreme]{Conjecture}
\newtheorem{remark}{Remark}[section]
\newtheorem{example}[theoreme]{Example}

\newcolumntype{L}{>{\centering\arraybackslash}m{3cm}}

\newcommand\nn\nonumber
\renewcommand\leq\varleq
\renewcommand\geq\vargeq

 \newcommand{\R}{\mathbb{R}}
 \newcommand{\N}{\mathbb{N}}
\newcommand{\Z}{\mathbb{Z}} \newcommand{\C}{\mathbb{C}}
 \newcommand{\F}{\mathcal{F}}
 \newcommand{\G}{\mathcal{G}}  
\newcommand{\E}{\mathcal{E}} \renewcommand{\H}{\mathcal{H}}

\newcommand{\grad}{\nabla}

\renewcommand{\i}{\mathrm{i}}

\renewcommand{\F} {\mathcal{F}}

\renewcommand{\epsilon}{\varepsilon}

\newcommand\blue[1]{\textcolor{blue}{#1}}

\usepackage{tikz-cd}

\usepackage{xcolor}
\usepackage{dcolumn}
\usepackage{tabu}

\newcolumntype{A}{D{.}{.}{2.3}}

\usepackage{tikz,tkz-tab}
\usepackage{pgfplots}
\pgfplotsset{compat=1.11}
\usetikzlibrary{positioning}
\makeatletter
      \def\@setcopyright{}
      \def\serieslogo@{}
      \makeatother
\begin{document}

\author{Sylvain Gol\'enia and Marc-Adrien Mandich}
   \address{Univ. Bordeaux, CNRS, Bordeaux INP, IMB, UMR 5251,  F-33400, Talence, France}
   \email{sylvain.golenia@math.u-bordeaux.fr}
      \address{Independent researcher, Jersey City, 07305, NJ, USA}
	\email{marcadrien.mandich@gmail.com}
   
   % title

   \title[LAP for discrete Schr\"odinger operator]{Thresholds and more bands of A.C.\ spectrum for the discrete Schr{\"o}dinger operator with a more general long range condition}

   % abstract (optional)
   \begin{abstract}
   We continue the investigation of the existence of absolutely continuous (a.c.) spectrum for the discrete Schr\"odinger operator $\Delta+V$ on $\ell^2(\Z^d)$, in dimensions $d\geq 2$, for potentials $V$ satisfying the long range condition $n_i(V-\tau_i ^{\kappa}V)(n) = O(\ln^{-q}(|n|))$ for some $q>2$, $\kappa \in \N$, and all $1 \leq i \leq d$, as $|n| \to \infty$. $\tau_i ^{\kappa} V$ is the potential shifted by $\kappa$ units on the $i^{\text{th}}$ coordinate. The difference between this article and \cite{GM2} is that here \textit{finite} linear combinations of conjugate operators are constructed leading to more bands of a.c.\ spectrum being observed. The methodology is backed primarily by graphical evidence because the linear combinations are built by numerically implementing a polynomial interpolation. On the other hand an infinitely countable set of thresholds, whose exact definition is given later, is rigorously identified. Our overall conjecture, at least in dimension 2, is that the spectrum of $\Delta+V$ is void of singular continuous spectrum, and consecutive thresholds are endpoints of a band of a.c. spectrum.
   \end{abstract}

   % AMS subject classifications (used in AMS journals)
%
%\renewcommand{\subjclassname}{\textup{2010} Mathematics Subject Classification} 
\subjclass[2010]{39A70, 81Q10, 47B25, 47A10.}

   % AMS keywords (used in AMS journals)
   \keywords{discrete Schr\"{o}dinger operator, long range potential, limiting absorption principle, Mourre theory, Chebyshev polynomials, polynomial interpolation, threshold}
 
   %\date{\today}

\maketitle
\hypersetup{linkbordercolor=black}
\hypersetup{linkcolor=blue}
\hypersetup{citecolor=blue}
\hypersetup{urlcolor=blue}
\tableofcontents

\section{Introduction}

The discrete Schr\"odinger operators on the lattice $\Z^d$ have a long history in modeling quantum phenomena in media with discrete postions such as crystals, or more general media by means of discretisation. This article deals with a specific modeling aspect in the spectral theory of these operators and is a direct sequel to \cite{GM2}. Let $\mathscr{H} := \ell^2(\Z^d)$. The discrete Laplacian on $\Z^d$, which models the kinetic energy of a quantum particle, is
\begin{align}
\label{def:std}
&\Delta = \Delta[d] :=  \Delta_1+...+\Delta_d, \quad \text{where } \Delta_i :=  (S_i + S^*_i)/2.
\end{align}
Here $S_i = S_i^1$ and $S^*_i = S^{-1}_i$ are the shifts to the right and left respectively on the $i^{th}$ coordinate. So $(S_i ^{\pm 1} u)(n) = u(n_1,\ldots,n_i \mp 1,\ldots,n_d)$ for $u \in \mathscr{H}$, $n=(n_1,\ldots, n_d) \in \Z^d$. Set $|n|^2 = n_1^2 +...+n_d^2$. Let $\sigma(\cdot)$ denote the spectrum of an operator. A Fourier transformation shows that the spectra of $\Delta_i$ and $\Delta$ are purely absolutely continuous (a.c.), $\sigma(\Delta_i) \equiv [-1,1]$ and $\sigma(\Delta) = [-d,d]$. 

Let $V$ model a discrete electric potential and act pointwise, i.e.\ $(Vu)(n) = V(n) u(n)$, for $u \in \mathscr{H}$. We always assume $V$ is real-valued and goes to zero at infinity. Thus the essential spectrum of $\Delta+V$ equals $\sigma(\Delta)$. Let $\N$ and $\N^*$ be the positive integers, including and excluding zero respectively. Fix $\kappa \in \N^*$. The shifted potential by $\pm \kappa$ units is defined by 
%$\kappa$ units on the $i^{th}$ coordinate is again a multiplication operator defined by
\begin{equation*}
(\tau_i ^{\pm \kappa} V) u(n) := V(n_1,\ldots, n_i \mp \kappa,\ldots n_d)u(n), \quad \forall  1 \leq i  \leq d.
\end{equation*} 
As in \cite{GM2}, we are interested in potentials $V$ satisfying a non-radial condition of the form 
\begin{equation}
\label{generalLR condition}
n_i (V - \tau_i ^{\kappa} V)(n) = O(g(n)), \ \text{as} \ |n| \to \infty, \quad \forall  1 \leq i  \leq d,
\end{equation}
 where $g(n)$ is a radial function which goes to zero at infinity at an appropriate rate, e.g.\ $g(n) =  \ln ^{-q}(|n|+1)$, $q>2$. %Although the optimality of $g(n)$ is an interesting question in itself let us stick to $g(n) = \ln^{-q} (|n|)$, $q>2$, for definiteness \blue{on garde ?}. This condition arises naturally in a wider framework of applied Mourre theory/commutator methods.
 We refer to \cite{GM2} for some examples of Schr\"odinger operators that satisfy \eqref{generalLR condition}. Also, one may generalize \eqref{generalLR condition} by shifting $V$ by a different amount in each direction (i.e.\ $\kappa_i$ instead of $\kappa$ in \eqref{generalLR condition}). We do not study this question here and refer instead to \cite{GM2} for numerical examples. We also want to mention that the class of $V$ given by \eqref{generalLR condition} is quite close to the $V(n) = O(1/|n|)$ class, for which the absence of singular continuous (s.c.) spectrum is proved in dimension 1, see \cite{Li1} and \cite{Ki}, but remains an open problem in higher dimensions. In our opinion the transition of spectral components at the $V(n) = O(1/|n|)$ level is still not very well understood, especially in dimensions $\geq 2$, but see e.g.\ \cite{Li1} and \cite{Li2} for $d=1$. This article may be viewed as a contribution in this direction.

For a closed interval $I \subset \R$ let $I_{\pm}:= \{ z \in \C_{\pm}: \mathrm{Re} (z) \in I\}$, $\C_\pm:=\{z\in \C, \pm\mathrm{Im}(z)>0\}$. The limiting absorption principle (LAP) is a statement about the extension of the holomorphic maps 
%These 3 paths yield different perspectives.
\begin{equation}
\label{LaP_generic}
I_{\pm} \ni z \mapsto (\Delta+V-z)^{-1} %\in \mathscr{B}(\mathcal{K}, \mathcal{K}^*),
\end{equation}
to $I$. % for some appropriate Banach space $\mathcal{K} \subset \mathscr{H}$. Here $\mathcal{K}^*$ is the antidual of $\mathcal{K}$, when $\mathscr{H}=\mathscr{H}^*$ by the Riesz isomorphism ; $\mathscr{B}(\mathcal{K}, \mathcal{K}^*)$ are the bounded operators from $\mathcal{K}$ to $\mathcal{K}^*$. 
The LAP on an interval $I$ implies amongst other things the absence of s.c.\ spectrum for $\Delta+V$ on that set. This article aims for such type of results. In Mourre theory, which has its origins in \cite{Mo1} and \cite{Mo2}, and is extensively refined in \cite{ABG}, the strategy to obtain a LAP \eqref{LaP_generic} on an interval $I \subset \sigma(\Delta+V)$ depends roughly on the ability to prove two key estimates. The first estimate is a \textit{strict Mourre estimate} for $\Delta$ with respect to some self-adjoint conjugate operator $\mathbb{A}$ on this interval, that is to say, $\exists \gamma >0$ such that 
\begin{equation}
\label{mourreEstimate123}
1_I(\Delta) [\Delta, \i \mathbb{A}] _{\circ} 1_{I} (\Delta) \geq \gamma 1_{I} (\Delta),
\end{equation}
where $1_I(\Delta)$ is the spectral projection of $\Delta$ on $I$, and $[ \cdot , \i \mathbb{A} ] _{\circ}$ initially defined on the compactly supported sequences is the extension of the commutator between two operators to a bounded operator on $\mathscr{H}$ (this definition suffices for this article). The second estimate is one involving $V$, and according to a later version of the theory, is such as
\begin{equation}
\label{compactcomm}
 \ln^p(1+|n|) \cdot [V, \i \mathbb{A}]_{\circ}  \cdot \ln^p(1+|n|) \quad is \ a \ compact \ operator \ on \ \mathscr{H} \ for \ some \ p >1.
\end{equation}
To specify our choice of $\mathbb{A}$ we need the position operators $(N_i u)(n) := n_i u(n)$. 
%, defined on the domain $\mathrm{Dom}[N_i] = \big \{ u \in \mathscr{H} : \sum_{n \in \Z^d} |n_i u(n)|^2 < \infty \big \}$.
%\begin{equation}
%\label{notation_position}
%(N_i u)(n) := n_i u(n), \quad \mathrm{Dom}[N_i] = \Big \{ u \in \mathscr{H} = \ell^2(\Z^d) : \sum_{n \in \Z^d} |n_i u(n)|^2 < \infty \Big \}.
%\end{equation}
To handle condition \eqref{generalLR condition} we consider a (\textit{finite}) linear combination of conjugate operators of the form
\begin{equation}
\label{LINEAR_combinationA}
\mathbb{A} = \sum _{j \geq 1} \rho_{j\kappa} \cdot A_{j \kappa}, \quad \rho_{j\kappa} \in \R, \quad A_{j \kappa} := \sum_{1 \leq i \leq d} A_i (j,\kappa),
\end{equation}
where each $A_i (j,\kappa)$, initially defined on compactly supported sequences, is the closure in $\mathscr{H}$ of :
\small
\begin{equation}
\label{generatorDilations_0111}
A_{i}(j,\kappa) := \frac{1}{2\i} \bigg [  \frac{j \kappa}{2}(S_i ^{j \kappa} + S_i^{-j \kappa}) + (S_i ^{j \kappa} - S_i^{-j \kappa}) N_i \bigg ]  = \frac{1}{4\i} \bigg[  (S_i ^{j \kappa} - S_i^{-j \kappa})N_i + N_i (S_i^{j \kappa} - S_i^{-j \kappa}) \bigg].
\end{equation} 
\normalsize
Each $A_{j\kappa}$ is self-adjoint in $\mathscr{H}$ by an adaptation of the case $(j,\kappa) = (1,1)$, and so $\mathbb{A}$ is self-adjoint, at least whenever it is a finite sum. The reason choice \eqref{LINEAR_combinationA} is relevant is that 
$$[V, \i A_{j\kappa} ]_{\circ} = \sum_{1 \leq i \leq d}  (4\i)^{-1} \left( (V - \tau_i ^{j \kappa} V) S_i ^{j \kappa} - (V - \tau_i ^{-j\kappa}V)S_i ^{-j\kappa} \right) N_i + \text{hermitian conjugate},
$$
and so \eqref{generalLR condition} implies \eqref{compactcomm}, again, at least when $\mathbb{A}$ is a finite sum and $g(n) = \ln^{-q}(1+|n|)$, $q>2$. The frequencies of the $A_{i}(j,\kappa)$ are in sync with the long range frequency decay of $V$. But the coefficients $\rho_{j \kappa}$ need to be chosen so that \eqref{mourreEstimate123} holds. This is a challenge. Categorize energies into two sets : $\boldsymbol{\mu}_{\kappa}(\Delta)$ and $\boldsymbol{\Theta}_{\kappa}(\Delta)$. $\boldsymbol{\mu}_{\kappa}(\Delta)$ are energies $E \in \sigma(\Delta)$ for which there is a self-adjoint linear combination (finite or infinite) of the form \eqref{LINEAR_combinationA}, an interval $I \ni E$ and $\gamma > 0$ such that the Mourre estimate \eqref{mourreEstimate123} holds. $\boldsymbol{\Theta}_{\kappa}(\Delta)$ are energies $E \in \sigma(\Delta)$ for which there is no self-adoint linear combination (finite or infinite) of the form \eqref{LINEAR_combinationA}, no interval $I \ni E$ and no $\gamma >0$ such that \eqref{mourreEstimate123} holds. By definition $\sigma(\Delta)$ is a disjoint union of $\boldsymbol{\mu}_{\kappa}(\Delta)$ and $\boldsymbol{\Theta}_{\kappa}(\Delta)$. From Mourre theory $\boldsymbol{\mu}_{\kappa}(\Delta)$ is an open set and so $\boldsymbol{\Theta}_{\kappa}(\Delta)$ is closed. In this article, including title and abstract, we refer to energies in $\boldsymbol{\Theta}_{\kappa}(\Delta)$ as \textit{thresholds}. \textit{This definition depends on the modeling assumption of $\mathbb{A}$} (see end of introduction for a comment). To be clear, the family $\mathbb{A}$ given by \eqref{LINEAR_combinationA} is a modeling assumption within the larger modeling framework of discrete Schr\"odinger operators given by \eqref{def:std}. Theorem \ref{lapy305} below highlights the usefulness of the sets $\boldsymbol{\mu}_{\kappa}(\Delta)$. Let $\sigma_p(\Delta+V)$ be the point spectrum of $\Delta+V$. Let $\langle \mathbb{A} \rangle := \sqrt{1+\mathbb{A}^* \mathbb{A} }$. 
%They imply LAPs and absence of s.c.\ spectrum. 
\begin{theorem}
\label{lapy305} 
Let $q > 2$, $\kappa \in \N^*$ be such that $\limsup \left( |V(n)|, | n_i (V - \tau_i^{\kappa} V)(n) | \right) = O \left(  \ln ^{-q} (|n|) \right )$, as $|n| \to \infty$ and $\forall 1 \leq i \leq d$.
%Let $w^q (n) := \ln ^{-q} (|n|)$.  
%\begin{equation}
%\max \left(V(n), n_j (V - \tau_j^{\kappa} V)(n) \right) = O \left(w ^{q}(n) \right ),\quad as \ |n| \to \infty, \quad j=1,\ldots,d.
%\limsup \left( |V(n)|, | n_i (V - \tau_i^{\kappa} V)(n) | \right) = O \left(  \ln ^{-q} (|n|) \right ),\quad as \ |n| \to \infty, \quad i=1,\ldots,d.
%\label{A20}
%\end{equation}
Let $E \in \boldsymbol{\mu}_{\kappa}(\Delta) \setminus \sigma_p(\Delta+V)$. Let $\mathbb{A} = \sum_j \rho_{j \kappa} A_{j \kappa}$ be a finite sum such that \eqref{mourreEstimate123} holds in a neighorhood of $E$. Then there is an open interval $I$, $I \ni E$, such that
\begin{enumerate}
\item $\sigma_p(\Delta+V) \cap I$ is at most finite (including multiplicity),
\item$\forall p >1/2$ the map $I_{\pm} \ni z \mapsto (\Delta+V-z)^{-1} \in \mathscr{B}\left( \mathcal{K}, \mathcal{K}^* \right)$ extends to a uniformly bounded map on $I$, with $\mathcal{K} = L^2_{1/2,p}(\mathbb{A}) = \left \{ \psi \in \mathscr{H} : \| \langle \mathbb{A} \rangle ^{1/2}  \ln ^{p} (\langle \mathbb{A} \rangle) \psi \| < \infty \right \}$,
\item The singular continuous spectrum of $\Delta+V$ is void in $I$. %$\sigma_{\mathrm{sc}} (\Delta+V) \cap I = \emptyset$. 
\end{enumerate}
\end{theorem}
This theorem can be refined, see \cite{GM2} and references therein. Theorem \ref{lapy305} as such follows directly from \cite{GM1} and \cite{GM2}. The main technique underlying our approach are commutator methods. In this context, operator regularity is a necessary and important topic. According to the standard literature, the regularity $\Delta, V \in C^{1,1}(\mathbb{A})$, or adequate variations thereof, are required. As far as we are concerned, it is clear that $\Delta$, $V$ belongs to $C^1(A_{j\kappa})$ for $1 \leq j \leq N < \infty$, and this implies $\Delta$, $V \in C^1(\sum_{1 \leq j \leq N} \rho_{j\kappa} A_{j\kappa})$. Although the $C^1(\mathbb{A})$ compliance falls short of the required regularity, since $C^1(\mathbb{A}) \subset C^{1,1}(\mathbb{A})$, we refer to the section on regularity in \cite{GM2} and let the reader fill in the details. We do not expect regularity to bring complications at least for $\mathbb{A}$'s = \eqref{LINEAR_combinationA} consisting of finite sums. 

Properties of eigenfunctions of $\Delta+V$ can also be analyzed thanks to the Mourre estimate \eqref{mourreEstimate123}, see \cite{FH}. For $\kappa=1$, an eigenvalue of $\Delta+V$ belonging to $\boldsymbol{\mu}_{\kappa}(\Delta)$ with $\mathbb{A} = \sum_{i=1}^d A_i (1,1)$ is such that the corresponding eigenfunction decays sub-exponentially ; in dimension 1 it decays exponentially at a rate depending on the distance to the nearest threshold, see \cite{Ma}. As far as we know, the sub-exponential decay of eigenfunctions at energy $\in \boldsymbol{\mu}_{\kappa}(\Delta)$ is an open problem for $\kappa \geq 2$, any $d \geq 1$ ; the exponential decay to nearest threshold is unknown for $\kappa \geq 1$, any $d \geq 2$. The reason we make this observation is because here we find many more ``thresholds'', and we wonder if the prospect of thresholds is part of the reason adapting Froese and Herbst's method to the discrete Schr\"odinger operators is met with difficulty, see \cite{Ma}. \\

\noindent \textbf{(P) Problem of article :} determine for as many energies $E \in \sigma(\Delta)$ if $E \in \boldsymbol{\mu}_{\kappa}(\Delta)$ or $E \in \boldsymbol{\Theta}_{\kappa}(\Delta)$. \\

In \cite{BSa} it is proved that $\boldsymbol{\mu}_{\kappa=1}(\Delta) \supset [-d,d] \setminus \{-d+2l : l=0,...,d\}$ in any dimension $d$ and this is done choosing $\mathbb{A} = A_1 = \sum_{i=1}^d A_i (1,1)$. Actually, equality holds and this is easy to prove (see Lemma \ref{lemSUMcosINTRO_00} below). In \cite{GM2} we fully solved problem \textbf{(P)} in dimension 1, $\forall \kappa \in \N^*$ (see Lemma \ref{lemSUMcosINTRO_00} below), but in higher dimensions we obtained incomplete results for $\boldsymbol{\mu}_{\kappa}(\Delta)$ and $\boldsymbol{\Theta}_{\kappa}(\Delta)$ for $\kappa \geq 2$, and there we chose $\mathbb{A} = A_{\kappa} = \sum_{i=1}^d A_i (1,\kappa)$. Note that this corresponds to \eqref{LINEAR_combinationA} with $\rho_{j\kappa} = 1$ if $j=1$ and $\rho_{j \kappa} = 0$ if $j \geq 2$. Table \ref{tab:table1011} displays the intervals already determined (numerically) to belong to $\boldsymbol{\mu}_{\kappa}(\Delta)$ for $2 \leq \kappa \leq 8$ (cf. \cite[Tables 11 and 13] {GM2}). In this article we continue to determine $\boldsymbol{\mu}_{\kappa}(\Delta)$ and $\boldsymbol{\Theta}_{\kappa}(\Delta)$ for $d, \kappa \geq 2$.

%Regarding \eqref{generalLR condition}, when $\kappa =1$ and $g(n) = \ln ^{-q}(|n|+1)$, $q>2$, then it was proved in \cite{GM1} that a LAP holds locally for $\Delta+V$ on $[-d,d] \setminus \{-d+2l : l=0,...,d\}$, and this was done by choosing $\mathbb{A} = A_1 = \sum_{i=1}^d A_i (1,1)$, see \eqref{LINEAR_combinationA} and \eqref{generatorDilations_0111} below. larger values of $\kappa$, and there we chose $\mathbb{A} = A_{\kappa} = \sum_{i=1}^d A_i (1,\kappa)$. Note that this corresponds to \eqref{LINEAR_combinationA} with $\rho_{j\kappa} = 1$ if $j=1$ and $\rho_{j \kappa} = 0$ otherwise. Table \ref{tab:table1011} displays the intervals where it has already been determined numerically that a LAP holds locally for $\Delta+V$ and $2 \leq \kappa \leq 8$ (cf. \cite[Tables 11 and 15] {GM2}). In this article we continue treating potentials satisfying \eqref{generalLR condition} for larger values of $\kappa$, but instead consider more involved but \textit{finite} linear combinations of the form \eqref{LINEAR_combinationA}.

The overall high level strategy we adopt is perhaps best summarized in 3 steps :
\begin{enumerate}
\item Fix a dimension $d$ and a value of $\kappa \geq 2$. (we really only treat $d=2$ ; $d=3$ very briefly). 
\item Determine as many threshold energies in $\boldsymbol{\Theta}_{\kappa}(\Delta)$ as possible. To do this we use a simple idea (but which is generally very complicated to actually execute and solve in full generality). This idea yields threshold energies $\E=x_1+... + x_d$ and their decomposition into coordinate-wise energies $\vec{x} = (x_1, ..., x_d)$. These play a key role in the next step.  
\item Pick 2 consecutive threshold energies $\E_{i_1}$ and $\E_{i_2}$ determined in the previous step (consecutive means that there aren't any other thresholds between $\E_{i_1}$ and $\E_{i_2}$), and try to construct a conjugate operator $\mathbb{A}$ of the form \eqref{LINEAR_combinationA} such that for \textit{every} $E \in (\E_{i_1}, \E_{i_2})$, there is an interval $I \ni E$ and $\gamma>0$ such that \eqref{mourreEstimate123} holds. \textit{The $\mathbb{A}$ is the same for every $E \in (\E_{i_1}, \E_{i_2})$}. To determine the coefficients $\rho_{j \kappa}$, we perform polynomial interpolation. % based on constraints using the endpoints $\E_{i_1}$ and $\E_{i_2}$ as well as the passive variables.
\end{enumerate}  
For the most part, we give rigorous proofs for the existence of the thresholds in step (2), whereas for step (3), we don't know how to actually carry out the polynomial interpolation theoretically and so we implement it numerically, determine the $\rho_{j \kappa}$ numerically, and plot a functional representation of \eqref{mourreEstimate123} to convince ourselves that strict positivity is in fact obtained. 
%Moreover, it is important to mention that we have only tested this high level approach for $\Delta$ in dimension 2. 

%Graphical evidence will be presented to suggest that choosing appropriately the $\rho_{j\kappa}$'s in \eqref{LINEAR_combinationA} yields a Mourre estimate \eqref{mourreEstimate123} on some intervals $I \subset \sigma(\Delta)$. 

%Although we believe the coefficients $\rho_{j\kappa}$ can be determined appropriately through polynomial interpolation, we don't know how to prove rigorously that the estimate \eqref{mourreEstimate123} follows. Our belief is based on several examples where the $\rho_{j \kappa}$ are determined numerical implementation of the polynomial interpolation, and then showing that a functional representation of \eqref{mourreEstimate123} yields a strictly positive graph. 

We now describe the idea to get thresholds. Let $U_{\kappa}$ be the Chebyshev polynomials of the second kind of order $\kappa$. As 
$[\Delta, \i A_{j\kappa}]_{\circ} = \sum_{i=1} ^d (1-\Delta_i^2) U_{j\kappa - 1} (\Delta_i)$ and the $\Delta_i$ are self-adjoint commuting operators we may apply functional calculus. To this commutator associate the polynomial %$g_{j\kappa} : \vec{x} \ni  \mapsto \R$, $\vec{x} = (x_1, ..., x_d)$,
\begin{equation}
\label{def:gE_intro}
[-1,1]^d \ni \vec{x} \mapsto g_{j\kappa} (\vec{x}) :=  \sum_{i=1} ^d (1-x_i^2) U_{j\kappa - 1} (x_i) \in \R, \quad \vec{x} = (x_1, ..., x_d).
\end{equation}
If the linear combination of conjugate operators is $\mathbb{A} = \sum_{j \geq1} \rho_{j\kappa} \cdot A_{j\kappa}$, set $G_{\kappa} : [-1,1]^d \mapsto \R$,
\begin{equation}
\label{def:GGE_intro}
G_{\kappa} (\vec{x}) :=  \sum_{j \geq 1} \rho_{j\kappa} \cdot g_{j\kappa} (\vec{x}) .
\end{equation}
%In terms of notation $j\kappa$ and $\kappa$ are just superscripts in $g_E ^{j\kappa}$ and $G_E ^{\kappa}$ respectively. 
$G_{\kappa}$ is a functional representation of $[\Delta, \i \mathbb{A}]_{\circ}$. Consider the constant energy $E \in \sigma(\Delta)$ surface 
\begin{equation}
\label{constE_MV}
S_E := \left \{ \vec{x} \in [-1,1]^d : E = x_1+...+x_d \right \}.
\end{equation} 
By functional calculus and continuity of the function $G_{\kappa}$ we have $E \in \boldsymbol{\mu}_{\kappa} (\Delta)$ iff $ \left.G_{\kappa} \right|_{S_E} >0$. \\
\begin{comment}
\begin{equation}
\label{constE_MV_intro}
S_E := \left \{ (x_1, ..., x_d) \in [-1,1]^d : E = x_1+...+x_d \right \}.
\end{equation} 
\end{comment}

\noindent \textit{Definition of $\boldsymbol{\Theta}_{0,\kappa}(\Delta)$}. $E \in \boldsymbol{\Theta}_{0,\kappa}(\Delta)$ iff $\exists \ \vec{x} := (x_1, ..., x_d) \in S_E$ such that $g_{j \kappa} (\vec{x}) = 0$, $\forall j \in \N^*$. 

\noindent If $\vec{x}$ is such a solution, then for any choice of coefficients $\rho_{ j \kappa} \in \R$, \eqref{def:GGE_intro} $= G_{\kappa} (\vec{x})  =0$.

%\begin{equation}
%\label{special relationship_intro_0}
%G_{\alpha \kappa} (\vec{x}) = \sum_{j \geq 1} \rho_{ j  \alpha \kappa} \cdot g_{  j \alpha \kappa} (\vec{x}) = 0.
%\end{equation}
%Clearly $\boldsymbol{\Theta}_{0,\kappa}(\Delta) \subset \boldsymbol{\Theta}_{0,\alpha\kappa}(\Delta) \subset \boldsymbol{\Theta}_{\alpha\kappa}(\Delta)$, $\forall \alpha \in \N^*$. 
%because if $E, x_1, ..., x_d$ are such a solution, then $G_{\kappa} (x_1,...,x_d) = \sum_{j \geq 1} \rho_{ j \kappa} g_{ j \kappa} (x_1,...,x_d) = 0$ for any choice of coefficients $\rho_{ j \kappa}$.

\noindent \textit{Definition of $\boldsymbol{\Theta}_{m,\kappa}(\Delta)$, $m \in \N^*$}. $E \in \boldsymbol{\Theta}_{m,\kappa}(\Delta)$ iff there are $(\vec{x}_q)_{q=0} ^{m} := (x_{q,1}, ...,x_{q,d})_{q=0}^m \subset S_E$, and $(\omega_q)_{q=0}^{m-1} \subset \R$, $\omega_q \leq 0$ (crucial), $\forall \ 0 \leq q \leq m-1$, such that 
\begin{equation}
\label{special relationship_intro}
g_{j \kappa} (\vec{x}_m) = \sum_{q=0} ^{m-1} \omega_q \cdot g_{j \kappa} (\vec{x}_q), \quad \forall j \in \N^*.
\end{equation} 
%Denote $\vec{x}_q := (x_{q,1}, ..., x_{q,d})$. 
If the $\vec{x}_q$ are such a solution, then for any choice of coefficients $\rho_{ j \kappa} \in \R$,
\begin{equation}
\label{special G relationship_intro}
G_{\kappa} (\vec{x}_m) = \sum _{j\geq 1} \rho_{j \kappa} \cdot g_{j  \kappa} (\vec{x}_m) = \sum _{q=0} ^{m-1} \omega_q \sum _{j\geq 1} \rho_{j \kappa} \cdot g_{j \kappa} (\vec{x}_q) = \sum _{q=0} ^{m-1} \omega_q \cdot G_{\kappa} (\vec{x}_q).
\end{equation}
%If $E$ were to belong to $\boldsymbol{\mu}_{\alpha \kappa}(\Delta)$, the lhs.\ of \eqref{special G relationship} would be a strictly positive number, whereas the rhs.\ of \eqref{special G relationship} would be a strictly negative number : an absurdity. Hence the  special linear relationship \eqref{special relationship} implies $E \in \boldsymbol{\Theta}_{\alpha \kappa}(\Delta)$, for all $\alpha \in \N^*$.
If $\boldsymbol{\Theta}_{m,\kappa}(\Delta) \cap \boldsymbol{\mu}_{\kappa}(\Delta)$ was non-empty, then the lhs of \eqref{special G relationship_intro} would be strictly positive whereas the rhs of \eqref{special G relationship_intro} would be non-positive. An absurdity. Thus :
\begin{Lemma}
Fix $d \geq 1$, $\kappa \geq 1$. Then $\boldsymbol{\Theta}_{m,\kappa}(\Delta) \subset \boldsymbol{\Theta}_{m,\alpha\kappa}(\Delta) \subset \boldsymbol{\Theta}_{\alpha\kappa}(\Delta)$, $\forall m \in \N$, and $\forall \alpha \in \N^*$.
\end{Lemma}
Simply because we haven't found any counterexamples, we actually conjecture :

\begin{conjecture}
\label{conjecture000}
Fix $d=2$, $\kappa \geq 1$. $\cup_{m \geq 0} \boldsymbol{\Theta}_{m,\kappa}(\Delta) = \boldsymbol{\Theta}_{\kappa}(\Delta)$.
\end{conjecture}

We don't understand $d=3$ nearly as well to submit a Conjecture like \ref{conjecture000} for it. It turns out it is very easy to find threshold energies in $\boldsymbol{\Theta}_{0,\kappa}(\Delta)$. We prove :
\begin{Lemma}
\label{lemSUMcosINTRO} 
$\forall \ d,\kappa \in \N^*$, $\boldsymbol{\theta}_{0,\kappa} (\Delta) := \left\{\sum_{1 \leq q \leq d} \cos( j_q \pi / \kappa) : (j_1,...,j_d) \in \{0,...,\kappa \}^d \right\} \subset \boldsymbol{\Theta}_{0,\kappa} (\Delta)$.
\end{Lemma}
\begin{remark} 
This lemma supports the conjectures in relation to the band endpoints in Table \ref{tab:table1011}.
\end{remark}
We prove equality in Lemma \ref{lemSUMcosINTRO} for $\kappa \in \{2,3,4,6\}$ in dimension 2 (Lemma \ref{lem_2346}). We conjecture:
\begin{conjecture}
\label{conjecture0001}
The inclusion in Lemma \ref{lemSUMcosINTRO} is not strict, i.e.\ equality holds.
\end{conjecture}
Thresholds $\boldsymbol{\theta}_{0,\kappa} (\Delta)$ in Lemma \ref{lemSUMcosINTRO} were already found in \cite{GM2}. Here we prove $\boldsymbol{\Theta}_{m,\kappa}(\Delta) \neq \emptyset$, $\forall m \geq 0$, $\forall d,\kappa \geq 2$. Thus, there are infinitely many thresholds for $d,\kappa \geq 2$. This is a remarkable difference with the case of the dimension 1, or the case of $\kappa=1$ in any dimension (see \cite{GM2}) :

\begin{Lemma}
\label{lemSUMcosINTRO_00} 
Let $(d, \kappa) \in \N^* \times \{1\} \cup \{1\} \times \N^*$. Then $\boldsymbol{\theta}_{0,\kappa} (\Delta) = \boldsymbol{\Theta}_{0,\kappa} (\Delta) = \boldsymbol{\Theta}_{\kappa} (\Delta)$.
\end{Lemma}

Now we clarify the interpolation setup of step (3) above. Suppose $\E_{i_1}$ and $\E_{i_2}$ are consecutive thresholds, with $\E_{i_1} \in \boldsymbol{\Theta}_{m_1,\kappa}(\Delta)$ and $\E_{i_2} \in \boldsymbol{\Theta}_{m_2, \kappa}(\Delta)$ for some $m_1, m_2 \in \N$. Suppose the coordinate-wise energies are $(\vec{x}_q)_{q=0}^{m_1} \subset S_{\E_{i_1}}$ and $(\vec{y}_r)_{r=0}^{m_2} \subset S_{\E_{i_2}}$. Recall we have the assumption that the conjugate operator $\mathbb{A}$ is the \textit{same} $\forall E \in (\E_{i_1},\E_{i_2})$. Thus, while we want $G_{\kappa} > 0$ on $(\E_{i_1},\E_{i_2})$, a continuity argument implies that $G_{\kappa}$ is at best non-negative at the endpoints $\E_{i_1}$ and $\E_{i_2}$, due to \eqref{special G relationship_intro}. Also by continuity, $G_{\kappa}(\vec{x})$ must be a local minimum whenever  $G_{\kappa}(\vec{x}) = 0$ and $\vec{x}$ is an interior point of $S_{\E_{i_1}}$ or $S_{\E_{i_2}}$. Let $\mathrm{int}(\cdot)$ be the interior of a set. Thus we require :
\begin{equation}
\begin{cases}
\label{interpol_intro}
& G_{\kappa}(\vec{x}_q) = 0, \quad \forall 0 \leq q \leq m_1, \quad and \quad \grad G_{\kappa} (\vec{x}_q) = 0, \quad for \ \vec{x}_q \in \mathrm{int}(S_{\E_{i_1}}) \quad [\text{left}],\\
& G_{\kappa}(\vec{y}_r) = 0, \quad \forall 0 \leq r \leq m_2, \quad and \quad \grad G_{\kappa} (\vec{y}_r) = 0, \quad for \ \vec{y}_r \in \mathrm{int}(S_{\E_{i_2}})  \quad [\text{right}].
\end{cases}
\end{equation}
By \eqref{special G relationship_intro} the conditions $G_{\kappa}(\vec{x}_{m_1}) = 0$ and $G_{\kappa}(\vec{y}_{m_2}) = 0$ are redundant. Constraints \eqref{interpol_intro} set up a system of linear equations to be solved for the coefficients $\rho_{j \kappa}$, i.e.\ we have polynomial interpolation. But in order for the computer to numerically solve the linear system, we need to assume a certain set of multiples of $\kappa$ : $\Sigma := \{ j_1 \kappa, j_2 \kappa, j_3 \kappa, ..., j_{\ell} \kappa \}$. We choose $\Sigma$ essentially by trial and error but prioritize lower order polynomials to keep things as simple as possible. In other words, we loop over sets $\Sigma$ until we find an appropriate $\mathbb{A}$. Of course, if our assumption that the $\mathbb{A}$ is the same for all $E \in (\E_{i_1},\E_{i_2})$ is valid, then constraints \eqref{interpol_intro} are necessary but not necessarily sufficient in order to find an appropriate $\mathbb{A}$, see section \ref{howchooseindices} for two illustrations of an inappropriate $\mathbb{A}$. Furthermore, for the sake of argument, suppose that \eqref{interpol_intro} gives rise to $N$ linearly independent equations, then it is natural to consider a finite sum $\mathbb{A}$ consisting of $N+1$ terms, so that the linear system is exactly specified (up to a constant multiple). Although we have many examples where this works, we have an example where it does not ; instead we considered an $\mathbb{A}$ with $>N+1$ terms and this led to an appropriate $\mathbb{A}$, see section \ref{KAPPA8more}. In that case linear system \eqref{interpol_intro} was underspecified as such.

Problem \textbf{(P)} is harder as $d,\kappa$ increase. So we focus mostly on the dimension 2. Some of those results will carry over to higher dimensions. As for $\kappa$ we mostly limit the numerical illustrations and evidence to a handful of values. \textit{We always restrict our analysis to positive energies}, because $\boldsymbol{\mu}_{\kappa}(\Delta) = - \boldsymbol{\mu}_{\kappa}(\Delta)$, by Lemma \ref{Lemma_symmetryDelta}, but see the subtle observations after Lemmas \ref{Lemma_symmetryDelta_88} and \ref{Lemma_symmetryTT}.

\textit{Until otherwise specified, we now focus exclusively on the \underline{dimension 2}}. For $\kappa \geq 2$, let
\begin{equation}
\label{setsJJ}
%J_3 (\kappa) := \left( 1+\cos(\frac{2 \pi}{ \kappa}), 2\cos(\frac{\pi}{ \kappa}) \right), \quad 
J_2 = J_2 (\kappa) := \left( 2 \cos(\pi / \kappa), 1+\cos(\pi / \kappa) \right), \quad J_1 = J_1(\kappa) := \left(1+\cos(\pi / \kappa), 2\right).
\end{equation}
%\blue{le probleme est que cet ensemble depend d'un operateur conjugue. alors qu'il est interpr\'et\'e ici comme ensemble o\`u l'on peut avoir un LAP} 
By Lemma \ref{lemSUMcosINTRO}, $\inf J_2, \sup J_2 = \inf J_1, \sup J_1 \in \boldsymbol{\theta}_{0,\kappa}(\Delta)$. In \cite{GM2} we proved $J_1 \subset \boldsymbol{\mu}_{\kappa}(\Delta)$, $\forall \kappa$. As for $J_2$ it was identified (numerically for the most part) as a gap between 2 bands of a.c.\ spectrum, but this was based on \eqref{LINEAR_combinationA} with $j=1$ only, see Table \ref{tab:table1011}. Let $\E_0 = \E_0 (\kappa) := \sup J_2$. We prove : 

\begin{theorem} 
\label{thm_decreasing energy general}
Fix $\kappa \geq 2$. There is a strictly decreasing sequence of energies $\{ \E_n \}_{n=1} ^{\infty} = \{ \E_n (\kappa) \}_{n=1} ^{\infty}$, which depends on $\kappa$, such that $\{ \E_n \} \subset J_2 \cap \boldsymbol{\Theta}_{\kappa}(\Delta)$ and $\E_n \searrow \inf J_2 = 2 \cos(\pi / \kappa)$. Also, $\E_{2n-1}$ and $\E_{2n} \in \boldsymbol{\Theta}_{n, \kappa}(\Delta)$, $\forall n \geq 1$.
%Fix $\kappa \geq 2$. There is a strictly decreasing sequence of energies $\{ \E_n \}_{n=1} ^{\infty} = \{ \E_n (\kappa) \}_{n=1} ^{\infty}$, which depends on $\kappa$, such that $\{ \E_n \} \subset J_2 \cap \boldsymbol{\Theta}_{\kappa}(\Delta)$ and $\E_n \searrow \inf J_2 = 2 \cos(\pi / \kappa)$. Also, $\E_{2n-1}$ and $\E_{2n} \in \mathfrak{T}_{n, \kappa}$, $\forall n \geq 1$.
\end{theorem}

\begin{proposition}
\label{propk2_2/n+2}
Fix $\kappa = 2$. The sequence in Theorem \ref{thm_decreasing energy general} is simply $\E_n = 2/(n+2)$, $n \in \N^*$.
\end{proposition}

For $\kappa \geq 3$ the $\E_n$ are complicated numbers (see Proposition \ref{lem1_sequence3344} and the discussion preceding it). Table \ref{table with endpoints3and4} gives the first values of $\{ \E_n \}$ for $\kappa = 3,4$. After graphing some numerical solutions for $\E_{2n}$, $1 \leq n \leq 4800$, for $\kappa \in \{3,4,5,6,8\}$ (see also \cite{GM4}), we propose a conjecture :

\begin{table}[H]
\small
  \begin{center}
    \begin{tabular}{c|c|c|c|c|c|c|c} % <-- Alignments: 1st column left, 2nd middle and 3rd right, with vertical lines in between
    $\kappa$ & $\E_6$ & $\E_5$ & $\E_4$ & $\E_3$ & $\E_2$ & $\E_1$ & $\E_0$   \\ %[0.1em]
      \hline
%2 & 1/4 & 2/7 & 1/3 & 2/5 & 1/2 & 2/3 & 1 \\ %[0.2em]    
3 & $\simeq 1.093$ & $\simeq 1.112$ & \eqref{solution_k3_5th_band} $\simeq 1.137$ & \eqref{E_L_k3_3} $\simeq 1.173$ & $\frac{9+\sqrt{33}}{12} \simeq 1.228$  & $\frac{5+3\sqrt{2}}{7} \simeq 1.320$ & 3/2  \\
4 &  $\simeq 1.466$ & $\simeq 1.476$ & $\simeq 1.491$ & $\simeq 1.512$ & $\simeq 1.545$ & 8/5 & $1+\frac{1}{\sqrt{2}} \simeq 1.707$ 
    \end{tabular}
  \end{center}
    \caption{First few values of $\{\E_n\}$ in Theorem \ref{thm_decreasing energy general}. $\Delta$ in dimension $2$. $\kappa = 3,4$.}
        \label{table with endpoints3and4}
\end{table}
\normalsize

\begin{conjecture}
\label{conjecture12}
Let $\{ \E_n \}$ be the sequence in Theorem \ref{thm_decreasing energy general}. $\E_n - \inf J_2 = c(\kappa)/n^2 + o(1/n^2)$, $\forall \kappa \geq 3$, where $c(\kappa)$ means a constant depending on $\kappa$.
\end{conjecture}

In section \ref{section_gen1} we state two Theorems and a Conjecture generalizing Theorem \ref{thm_decreasing energy general}. The next Theorem is the only mathematically rigorous proof we have of a Mourre estimate.
\begin{theorem}
\label{OnlyMourreband}
Fix $\kappa = 2$. Then $(2/3,1) = (\E_1, \E_0) \subset \boldsymbol{\mu}_{\kappa}(\Delta)$. Specifically, the Mourre estimate \eqref{mourreEstimate123} can be obtained with the conjugate operator $\mathbb{A} = A_{2} + \frac{9}{14} A_4$ for all energies $E \in (\E_1, \E_0)$.
\end{theorem}

\begin{comment}
\begin{conjecture}
\label{conjecture11}
Let $\{ \E_n \}$ be the sequence in Theorem \ref{thm_decreasing energy general}. $\{\E_n\} = J_2 \cap \boldsymbol{\Theta}_{\kappa}(\Delta)$, $\forall \kappa \geq 2$.
\end{conjecture}
\end{comment}

Of the gaps identified in \cite{GM2}, we believe $J_2$ is the simplest to understand its structure : 

\begin{conjecture}
\label{conjecture22}
Fix $\kappa \geq 2$. Let $\{ \E_n \}$ be the sequence in Theorem \ref{thm_decreasing energy general}. For each interval $( \mathcal{E}_n , \mathcal{E}_{n-1} )$, $n \geq 1$, $\exists$ a conjugate operator $\mathbb{A}(n) = \sum_{q=1} ^{N(n)} \rho_{j_q \kappa} (n) A_{j_q \kappa}$, $A_{j_q \kappa} = \sum_{1 \leq i \leq 2} A_i (j_q,\kappa)$, such that the Mourre estimate \eqref{mourreEstimate123} holds with $\mathbb{A}(n)$, $\forall E \in ( \mathcal{E}_n , \mathcal{E}_{n-1} )$. $\mathbb{A}(n)$ is typically not unique. It can be chosen so that $N(n) = 2n$. In particular, $\{\E_n\} = J_2 \cap \boldsymbol{\Theta}_{\kappa}(\Delta)$, $\forall \kappa \geq 2$.
\end{conjecture}

This Conjecture is based on graphical evidence for $\kappa =2,3$, see sections \ref{appli_scheme_k2} and \ref{appli_scheme_k3}, and see \cite{GM4} for more evidence. $\kappa=2$ is the only value of $\kappa$ for which the closure of $J_2 \cup J_1$ equals $\sigma(\Delta) \cap [0,2]$. Thus, if Conjecture \ref{conjecture22} is true, problem \textbf{(P)} is fully solved in the case of $\kappa=2$ (in dimension 2). But for $\kappa \geq 3$, Theorem \ref{thm_decreasing energy general} and Conjecture \ref{conjecture22}, together with the already existing results recorded in Table \ref{tab:table1011}, do not paint a complete picture. For example, the above discussion does not address the situation on the interval $(0,\beta) \simeq (0,0.542477)$, $\beta = \frac{1}{2} (\frac{1}{2}( 5 - \sqrt{7}) )^{1/2}$, for $\kappa=3$, or on $(0,\beta') \simeq (0,1.026054)$, $\beta' = (\frac{3}{2} - \frac{1}{\sqrt{5}})^{1/2}$, for $\kappa=4$. We make some progress in that direction, but things are getting even more complicated. In addition to \eqref{setsJJ}, for $\kappa \geq 3$ set 
$$J_3 = J_3 (\kappa) := \left( 1+\cos(2 \pi / \kappa), 2\cos(\pi / \kappa) \right).$$
$J_3$, $J_2$, $J_1$ are adjacent intervals. Based on evidence for $\kappa=3,4$ (section \ref{sec_Conjecture J3}) we conjecture :

\begin{conjecture}
\label{conjecture_k344_2d}
For any $\kappa \geq 3$, $J_3(\kappa) \subset \boldsymbol{\mu}_{\kappa}(\Delta)$.
\end{conjecture}

Conjecture \ref{conjecture_k344_2d} is not the highlight of this article, but it is a head-scratching curiosity, and perhaps quite a narrow question to investigate. We prove the existence of thresholds below $J_3$ :

\begin{theorem} 
\label{thm_decreasing energy general_k3}
Fix $\kappa \geq 3$. There are strictly increasing sequences of energies $\{ \F_n \}_{n=1} ^{\infty}$, $\{ \G_n \}_{n=1} ^{\infty}$, which depend on $\kappa$, such that $\{ \F_n \},  \{ \G_n \} \subset (\cos(\pi/ \kappa)+\cos(2\pi/ \kappa), 1+\cos(2\pi/ \kappa)) \cap \boldsymbol{\Theta}_{\kappa}(\Delta)$ and $\F_n, \G_n \nearrow \inf J_3 = 1+\cos(2\pi/ \kappa)$. $\F_{2n-1}$, $\F_{2n}$, $\G_{2n-1}$, $\G_{2n} \in \boldsymbol{\Theta}_{n, \kappa}(\Delta)$, $\forall n \geq 1$.
\end{theorem}

\begin{conjecture}
\label{conjecture_F_G}
Sequences $\{ \F_n \}_{n=1} ^{\infty}$, $\{ \G_n \}_{n=1} ^{\infty}$ of Theorem \ref{thm_decreasing energy general_k3} are distinct : $\{ \F_n \} \cap \{ \G_n \} = \emptyset$.
\end{conjecture}

Unfortunately we were not able to accurately numerically compute many solutions $\F_n$ and $\G_n$ and so we are not well positioned to conjecture on the rate of convergence of $\F_n$ and $\G_n$, but we speculate the rate is faster than the $O(1/n^2)$ rate of Conjecture \ref{conjecture12}. In section \ref{section_gen2} we state two Theorems and a Conjecture generalizing Theorem \ref{thm_decreasing energy general_k3} for $\{ \F_n \}$. A generalization for $\{ \G_n \}$ is very likely.

\begin{comment}
Based on some home cooking we also believe that the snippet $(1/2, e]$, $e := \frac{1}{2} (\frac{1}{2}( 5 - \sqrt{7}) )^{1/2} \simeq 0.542477$, which was observed to be a gap in \cite{GM2}, is an artifact (for some mysterious reason) :
\begin{conjecture}
\label{conjecture_k3_2d}
Fix $\kappa =3$. $(1/2, e] \subset \boldsymbol{\mu}_{\kappa}(\Delta)$. 
\end{conjecture}
\end{comment}

Hopefully it will become clear from our examples and constructions that there are \textit{many more} thresholds $\in  [0, \inf J_3 (\kappa)]$ for $\kappa \geq 3$ in addition to the sequences $\{ \F_n \}$ and $\{ \G_n \}$. Just how many more ? Here are a few open questions we find interesting :

$\bullet$ Is there a decreasing sequence $\{ \H_n \}_{n=1} ^{\infty} \subset \boldsymbol{\Theta}_{\kappa}(\Delta)$ with $\H_n \searrow 0$, for all $\kappa \geq 2$ ?

$\bullet$ Of the thresholds $\in \boldsymbol{\theta}_{0,\kappa}(\Delta)$, which ones are accumulation points, as a subset of $\boldsymbol{\Theta}_{\kappa}(\Delta)$ ?

$\bullet$ Are there accumulation points $ \in \boldsymbol{\Theta}_{\kappa}(\Delta) \setminus \boldsymbol{\theta}_{0,\kappa}(\Delta)$ ?

$\bullet$ What are the rates of convergence to the accumulation points $ \in \boldsymbol{\Theta}_{\kappa}(\Delta)$ ?

$\bullet$ Are there infinitely many accumulation points within $\boldsymbol{\Theta}_{\kappa}(\Delta)$ ?

$\bullet$ Is there an interval $I \subset \sigma(\Delta)$ for which $\boldsymbol{\Theta}_{\kappa}(\Delta)$ is dense in $I$ ?

%It is clear from our examples that a good number (but probably not all) of thresholds $\in \boldsymbol{\theta}_{0,\kappa}(\Delta)$ are accumulation points, as a subset of $\boldsymbol{\Theta}_{\kappa}(\Delta)$. It is an open problem for us to decide if there are accumulation points $\not \in \boldsymbol{\theta}_{0,\kappa}(\Delta)$.

%It is an open problem for us to decide if for larger values of $\kappa$, say $\kappa \geq 12$, $\boldsymbol{\Theta}_{\kappa}(\Delta)$ is an \textit{infinitely uncountable} set.

As $\kappa$ increases, the number of thresholds increases dramatically. The rate of increase is likely exponential. In section \ref{KAPPA10} we use $\kappa=10$ to construct a countable set of thresholds that is in one-to-one correspondence with the nodes of an infinite binary tree. The construction is merely to illustrate how easy it is to find an abundance of thresholds. However we do conjecture :

\begin{conjecture}
\label{conjecture6666}
Fix $\kappa \geq 1$. $\cup_{m \geq 0} \boldsymbol{\Theta}_{m,\kappa}(\Delta)$ and $\boldsymbol{\Theta}_{\kappa}(\Delta)$ are countable sets. %The threshold energies are all discrete, with the exception of a finite number of them being accumulation points.
\end{conjecture}

As far as the sets $\boldsymbol{\mu}_{\kappa}(\Delta)$ are concerned, we have had little success on $[0, \inf J_3(\kappa)]$. For instance, for $(0,1/2) \cap \boldsymbol{\mu}_{\kappa=3}(\Delta)$, we have only 1 piece of numerical evidence, namely $\left((9-\sqrt{33})/12, 2/7\right)$ $\simeq (0.2713, 0.2857) \subset \boldsymbol{\mu}_{\kappa=3}(\Delta)$, see Section \ref{k3_below_05}. We were not successful in finding other bands of a.c.\ spectrum on $(0,1/2)$ for $\kappa=3$. There are various explanations for this setback and these apply to all values of $\kappa$ and in general conceptually speaking. Either :
\begin{enumerate}
\item whenever we picked $\E_{i_1}$, $\E_{i_2} \in \boldsymbol{\Theta}_{\kappa}(\Delta)$ we were mistaken and there is in fact a threshold energy $\E \in \boldsymbol{\Theta}_{\kappa}(\Delta)$ lying between $\E_{i_1}$ and $\E_{i_2}$ that we are unaware of. 
\item simply we haven't tried enough conjugate operators $\mathbb{A}$. This is always a challenge because we never really know before going into a numerical computation if system \eqref{interpol_intro} should be exactly specified or underspecified and which multiples of $\kappa$ $\{j_1 \kappa, j_2 \kappa, j_3 \kappa, ...\}$ to choose.
\item our assumption that $\mathbb{A}$ is the \textit{same for all} energies $E \in (\E_{i_1}, \E_{i_2})$ is inadequate. Afterall there is no obvious reason why it should be the case. The $\rho_{j \kappa}$ may need to depend on $E$. Or, it may be that an infinite linear combination \eqref{LINEAR_combinationA} is required.
\end{enumerate}

We still don't understand the situation on $[0, \inf J_3(\kappa)]$ well, even for $\kappa =3$. Graphically we found a plethora of thresholds on this interval for $\kappa=3$, enough to put someone in a trance, see \cite[Section 9]{GM4}. In spite of this lack of understanding we dare conjecture boldly:

\begin{conjecture}
\label{conjecture33}
Fix $d=2$, $\kappa \geq 1$. Let $\E_{i_1}, \E_{i_2} \in \boldsymbol{\Theta}_{\kappa}(\Delta)$ be two consecutive thresholds -- meaning that there aren't any other thresholds in between $\E_{i_1}$ and $\E_{i_2}$. Then there is a (finite?) linear combination $\mathbb{A} = \sum_{j=1} ^{N} \rho_{j \kappa} A_{j \kappa}$ such that the Mourre estimate \eqref{mourreEstimate123} holds with $\mathbb{A}$ for every energy $E \in ( \E_{i_1}, \E_{i_2} )$. In particular, in light of Theorem \ref{lapy305}, $\sigma_p(\Delta+V)$ is locally finite on $( \E_{i_1}, \E_{i_2} )$, whereas the singular continuous spectrum of $\Delta+V$ is void.
\end{conjecture}

\textit{We are done discussing $d=2$}. In higher dimensions we have only 1 general result : thresholds in dimension $d$ generate thresholds in dimension $d+1$, via shifting. Recall notation \eqref{def:std}. 
%This in turn makes the problem of determining $\boldsymbol{\mu}_{\kappa} (\Delta)$ more complicated than originally expected.
\begin{Lemma}
\label{shift_threshold_d}
$\forall d \geq 1$, $\kappa \geq 1$, $m \in \N$, $\{ \cos(\frac{j \pi}{\kappa}) : 0 \leq j \leq \kappa \} + \boldsymbol{\Theta}_{m,\kappa}(\Delta[d]) \subset \boldsymbol{\Theta}_{m,\kappa}(\Delta[d+1])$.
\end{Lemma}
For $(d,\kappa) \in \N^* \times \{1\}$, the inclusion in Lemma \ref{shift_threshold_d} is in fact equality. But we conjecture : 
\begin{conjecture}
\label{conjecture13d33}
There are values of $\kappa \geq 2$ for which the inclusion in Lemma \ref{shift_threshold_d} is strict (see Example \ref{onlyexample}).
\end{conjecture}
Lemma \ref{shift_threshold_d} generalizes \cite[Lemma III.5]{GM2}. Our treatment of problem \textbf{(P)} in \textit{dimension 3} is brief. There is still considerable work to be done just to understand the case $\kappa=2$, especially on the interval $(0,1)$. \textit{Theorem \ref{thm3d_k2} and Conjecture \ref{conjecture11_3d} below are for the \underline{dimension $3$}}.
\begin{theorem}
\label{thm3d_k2} Fix $\kappa = 2$. We have :
\begin{itemize}
\item $(2,3) \subset \boldsymbol{\mu}_{\kappa}(\Delta)$ (proved in \cite{GM2}). 
\item $0,1,2,3 \in \boldsymbol{\theta}_{0,\kappa}(\Delta)$ (Lemma \ref{lemSUMcosINTRO}).
\item Let $\{ \E_n = \E_n (\kappa=2) \}$ be the sequence in Theorem \ref{thm_decreasing energy general}. Applying Lemmas \ref{shift_threshold_d} and \ref{Lemma_symmetryDelta} gives :
\item $\{\E_n+1\} \subset (1,2) \cap \boldsymbol{\Theta}_{\kappa}(\Delta)$, with $\E_n + 1 \searrow 1$, 
\item $\{\E_n \} \subset (0,1) \cap \boldsymbol{\Theta}_{\kappa}(\Delta)$, with $\E_n \searrow 0$, 
\item $\{-\E_n +1 \} \subset (0,1) \cap \boldsymbol{\Theta}_{\kappa}(\Delta)$, with $-\E_n + 1 \nearrow 1$. 
\end{itemize}
\end{theorem}
%From \cite{GM2}, we know that $(2,3) \subset \boldsymbol{\mu}_{\kappa}(\Delta)$. From Lemma \ref{lemSUMcosINTRO}, $0,1,2,3 \in \boldsymbol{\Theta}_{0,\kappa}(\Delta)$. Let $\{ \E_n \}$ be the sequence in Theorem \ref{thm_decreasing energy general}. Then by Lemma \ref{shift_threshold_d} we know that $\{\E_n+1\} \subset (1,2) \cap \boldsymbol{\Theta}_{\kappa}(\Delta)$, $\E_n + 1 \searrow 1$, $\{\E_n \} \subset (0,1) \cap \boldsymbol{\Theta}_{\kappa}(\Delta)$, $\E_n \searrow 0$, and using Lemma \ref{Lemma_symmetryDelta},  $\{-\E_n +1 \} \subset (0,1) \cap \boldsymbol{\Theta}_{\kappa}(\Delta)$, with $-\E_n + 1 \nearrow 1$. 
Our graphical evidence also suggests the following conjecture, although it is quite mysterious and surprising to us how and why it happens :
\begin{conjecture}
\label{conjecture11_3d}
Fix $\kappa =2$. Let $\{ \E_n = \E_n (\kappa=2) \}$ be the sequence in Theorem \ref{thm_decreasing energy general}. For each interval $( \mathcal{E}_n +1, \mathcal{E}_{n-1} +1 )$, $n \geq 1$,  the Mourre estimate \eqref{mourreEstimate123} holds with $\mathbb{A}(n) = \sum_{1 \leq q \leq N(n)} \rho_{j_q \kappa} (n) A_{j_q \kappa}$, $A_{j_q \kappa} = \sum_{1 \leq i \leq 3} A_i (j_q,\kappa)$, $\forall E \in ( \mathcal{E}_n +1, \mathcal{E}_{n-1} +1)$, where the coefficients $\rho_{j_q \kappa}(n)$ are \textit{exactly} those used in the 2-dimensional case, see Conjecture \ref{conjecture22}. In particular $\{\E_n +1 \}_{n=1}^{\infty} = (1,2) \cap \boldsymbol{\Theta}_{\kappa}(\Delta)$. 
\end{conjecture}

%\noindent{\textbf{Proposal of Result 4}.} \textit{Let $d=3$. Let $\kappa$ and $\E_n$ be as in Proposal 1. Then $\bigcup_{n=1} ^{\infty} ( \E_{n}+1, \E_{n-1}+1) \subset \boldsymbol{\mu}_{\kappa}(\Delta)$}.

Conjecture \ref{conjecture11_3d} may extend to $\kappa \geq 2$, but we have not looked into it. Other than the two sequences in $(0,1)$ in Theorem \ref{thm3d_k2}, we don't have any more knowledge about this interval.

%For $\kappa \geq 7$, Table \ref{tab:table1011} suggests that we have 3 gaps. We expect to handle the first gap as in Proposition \ref{new_Result}. For the second gap, we expect the situation to be similar to what we described for $\kappa =3$ and $4$. For the third gap, we have been able to cook up some linear combinations of conjugate operators, see Section \ref{Standard Laplacian, dimension 2, kappa 8, 2nd regime.}.

% because the focus of this article are the sets $\boldsymbol{\mu}_{\kappa}(\mathfrak{D})$ and $\boldsymbol{\Theta}_{\kappa}(\mathfrak{D})$.

We conclude the introduction with several comments. 

\begin{comment}
Coming back to our discussion right after Theorem \ref{thm_decreasing energy general_k3}, aside from 1 band we were not successful in numerically constructing $\mathbb{A}$s (by polynomial interpolation) to provide more bands of a.c.\ spectrum in $(0,1/2) \cap \boldsymbol{\mu}_{\kappa=3}(\Delta)$ in dimension 2. There are various explanations for this setback and these apply to all values of $\kappa$. Either :
\begin{enumerate}
\item whenever we picked $\E_{i_1}$, $\E_{i_2} \in \boldsymbol{\Theta}_{\kappa}(\Delta)$ we were mistaken and there is in fact a threshold energy $\E \in \boldsymbol{\Theta}_{\kappa}(\Delta)$ between $\E_{i_1}$ and $\E_{i_2}$ that we are unaware of. 
\item simply we haven't tried enough conjugate operators $\mathbb{A}$. This is always a challenge because we never really know before going into a numerical computation if system \eqref{interpol_intro} should be exactly specified or underspecified and which multiples of $\kappa$ $\{j_1 \kappa, j_2 \kappa, j_3 \kappa, ...\}$ to choose.
\item our assumption that $\mathbb{A}$ is the \textit{same for all} energies $E \in (\E_{i_1}, \E_{i_2})$ is inadequate. Afterall there is no obvious reason why it should be the case. The $\rho_{j \kappa}$ may need to depend on $E$.
\end{enumerate}
\end{comment}

In this article we construct \textit{finite} linear combinations of the form \eqref{LINEAR_combinationA}. It would be very interesting to know if there are energies $\in \boldsymbol{\mu}_{\kappa}(\Delta)$ for which an $\mathbb{A} =$ \textit{infinite sum} is required. Another related question : is there a $\boldsymbol{\Theta}_{\infty,\kappa}(\Delta)$ ? i.e.\ thresholds $\in \boldsymbol{\Theta}_{\kappa}(\Delta)$ with \eqref{special relationship_intro} $= g_{j \kappa} (\vec{x}_{\infty}) = \sum_{q=0} ^{\infty} \omega_q \cdot g_{j \kappa} (\vec{x}_q)$, $\forall j \in \N^*$ ? In the language of section \ref{geo_construction}, are there solutions $\in \mathfrak{T}_{\infty,\kappa}(\Delta)$ ?

%Another discretisation of the continuous Laplacian was proposed by Molchanov and Vainberg, see \cite{MV}, and will be the dealt with in an upcoming article.

To extend \eqref{LINEAR_combinationA} one may be tempted to consider an even larger class of conjugate operators of the form, say $\mathbb{A} = \sum_{1\leq i \leq d} p(S_i ^{\kappa}, S_i ^{-\kappa}) \cdot N_i + N_i \cdot p(S_i ^{\kappa}, S_i ^{-\kappa})$, where $p( \cdot, \cdot)$ is a polynomial in 2 variables, or perhaps even a continuous function of 2 variables, satisfying $(p(S_i ^{\kappa}, S_i ^{-\kappa}))^* = p(S_i ^{\kappa}, S_i ^{-\kappa})$. We believe such extension doesn't really add anything, as argued in section \ref{sine_series}.

There is the question of whether a LAP for $\Delta+V$ could hold in a neighborhood of some $E \in \boldsymbol{\Theta}_{\kappa}(\Delta)$, using a completely different idea or completely different tools. It is not clear at all to us if threshold energies $\in \boldsymbol{\Theta}_{\kappa}(\Delta)$ are artifacts of the mathematical tools we employ to analyze the a.c.\ spectrum, or if on the contrary they have a special physical significance, such as notable embedded eigenvalues or resonances under an appropriate (additional) perturbation. 
%In all the examples of discrete or continuous Schr\"odinger operators that we are aware of, eigenvalues embedded in the continuous spectrum also happen to be thresholds wrt.\ some conjugate operator. We would appreciate feedback if anyone is aware of the contrary.

A comment about the notion of \textit{threshold} for one particle discrete Schr\"odinger operators. These are traditionally defined as being the energies corresponding to the critical points of $\mathcal{F} \Delta \mathcal{F}^{-1} = d-2\sum_{1 \leq i \leq d }\sin^2(\xi_i/2)$ ($\mathcal{F} =$ \eqref{FourierTT}), $\xi_i \in [-\pi,\pi]$, see e.g.\ \cite{IJ}, \cite{NoTa}. This definition typically occurs in the context of the kernel of the resolvent of $\Delta$. $-d,d$ are elliptic thresholds ; $2-d, 4-d, ..., d-2$ are hyperbolic thresholds. In this article our notion of threshold is that of energies corresponding to roots of $\mathcal{F} [\Delta, \i \mathbb{A}]_{\circ} \mathcal{F}^{-1}$. It is not clear to us if it is a coincidence that the 2 sets of thresholds coincide when $\mathbb{A} = A_{\kappa=1}$, in which case $\mathcal{F} [\Delta, \i \mathbb{A}]_{\circ} \mathcal{F}^{-1} = \sum_{1 \leq i \leq d }\sin^2(\xi_i)$.

Recall that the Molchanov-Vainberg Laplacian $D$ is isomorphic to $\Delta$ in dimension 2, see \cite{GM2}. Thus, to search for bands $\subset \boldsymbol{\mu}_{\kappa}(\Delta)$ in dimension 2, it may be useful to use the results for $D$ as indication, see \cite{GM2}. Let us illustrate for $\kappa=2,3$. Fix $\kappa=2$. From \cite{GM2}, $(0,1/2) \subset \boldsymbol{\mu}_{2\kappa = 4}(D)$, which in turn implies a strict Mourre estimate for $\Delta$ on $(0,1)$, but it is with respect to a conjugate operator $\mathbb{B}_{\kappa=2}$ distinct from \eqref{LINEAR_combinationA} ($\mathbb{B}_{\kappa=2} = \pi A_{2\kappa=4} \pi ^{-1}$ in the notation of \cite{GM2}). In turn compactness and regularity of the commutator $[V, \i \mathbb{B}_{\kappa}]$ requires a condition different from $\eqref{generalLR condition}$, see \cite[condition (1.15)]{GM2}. Said condition is satisfied if for instance $(V-\tau_i^{\kappa}V)(n) = O(|n| ^{-1} \ln ^{-q} (|n|))$, $q>2$. As for $\kappa=3$, thanks to the fact that $(0,1/4) \subset \boldsymbol{\mu}_{2\kappa = 6}(D)$ (this is a numerical result, see \cite[Table 15]{GM2}), one obtains a strict Mourre estimate for $\Delta$ on $(0,1/2)$, but it is wrt.\ a conjugate operator $\mathbb{B}_{\kappa=3} = \pi A_{2\kappa=6} \pi ^{-1}$. These observations support Conjecture \ref{conjecture33}.

Finally, we wonder if there is a corresponding class of potentials to \eqref{generalLR condition} in the continuous operator case for which an analogous phenomena -- i.e.\ a plethora of thresholds embedded in the continuous spectrum and a LAP in between -- is observed.
%Perhaps our definition of threshold extends the more standard one.
%Finally, as mentioned previously, it is an open problem for us to decide if for larger values of $\kappa$, say $\kappa \geq 12$, and $d=2$, $\boldsymbol{\Theta}_{\kappa}(\Delta)$ could be an infinitely uncountable set, see Section XYZ.

\begin{comment}
kappa=12

\end{comment}

\noindent \textbf{Acknowledgements :} It is a pleasure to credit Laurent Beauregard, engineer at the European Space Agency in Darmstadt, for valuable contributions to the numerical implementations.

%\begin{remark} 
%For a fixed interval $I_n = (\E_{n},\E_{n-1})$ the linear combination $\sum_{j} \rho_{j\kappa} A_{j \kappa}$ may be chosen to be the same for all $E \in I_n$. This linear combination ought to be chosen precisely, but it is definitely not unique. The $\rho_{j \kappa}$ are determined by solving a system of $n$ linear equations.
%\end{remark}

\section{Basic properties and lemmas for the Chebyshev polynomials}

Let $T_{n}$ and $U_{n}$ be the Chebyshev polynomials of the first and second kind respectively of order $n$. They are defined by the formulas 
\begin{equation}
\label{def_T_U}
T_n(\cos(\theta)) = \cos(n \theta), \quad U_{n-1}( \cos(\theta)) = \sin(n \theta) / \sin(\theta), \quad \theta \in [-\pi,\pi], n \in \N^*.
\end{equation}

%Using $\sin(n \xi) = \sin(\xi) U_{n-1}(\cos(\xi))$, $n \in \N^*$, 
The first few $T_n$ Chebyshev polynomials are :
\begin{equation}
\label{cheby_23}
\begin{aligned}
T_2(x) &= 2x^2-1 \quad \text{and} \quad T_2(x) = T_2(y) \Leftrightarrow (x+y)(x-y)=0. \\
T_3(x) &= 4 x^3-3x \quad \text{and} \quad T_3(x) = T_3(y) \Leftrightarrow (x-y) [4(x^2+xy+y^2)-3]=0. 
%T_4(x) &= 8 x^4 - 8 x^2 + 1 \quad \text{and} \quad T_4(x) = T_4(y) \Leftrightarrow 8(x-y)(x+y) [x^2+y^2-1]=0.
\end{aligned}
\end{equation} 

The roots of $U_{n-1}$ are $\cos(l \pi / n)$, $1 \leq l \leq n-1$. We'll absolutely need a commutator $[ \cdot, \cdot]$ for \textit{functions}. For functions $f,g$ of real variables $x,y$, let 
\begin{equation}
\label{def_comm}
[f(x), g(y) ] := f(x) g(y) - f(y) g(x).
\end{equation}
\begin{remark}
The quantity $[f(x), g(y) ] / (x-y)$ is sometimes called \textit{Bezoutian} in the literature.
\end{remark}

\begin{Lemma}
\label{Tcos}
For $x,y \in [-1,1]$, $T_{\kappa}(x) = T_{\kappa}(y)$ if and only if $T_{\alpha \kappa}(x) = T_{\alpha \kappa}(y)$ for all $\alpha \in \N^*$. 
\end{Lemma}

\begin{Lemma}
\label{lemma_cos_alpha_beta}
Fix $\kappa \in \N^*$. If $\cos(\kappa \theta) = \cos(\kappa \phi)$ then 
$$ \sin(\kappa \phi ) \sin(2\kappa \theta) = \sin(\kappa \theta) \sin(2\kappa \phi) \Rightarrow \sin(\alpha \kappa \phi ) \sin(\beta \kappa \theta) = \sin(\alpha \kappa \theta) \sin(\beta \kappa \phi), \forall \alpha, \beta \in \N^*.$$
\end{Lemma}
To prove Lemma \ref{lemma_cos_alpha_beta} use induction, on $\alpha$ and $\beta$, apply angle formulas and Lemma \ref{Tcos}. 

Corollaries \ref{corollaryEquiv} and \ref{basic_lemma} are at the heart of our search for thresholds. 

\begin{corollary} 
\label{corollaryEquiv}
Let $\kappa \in \N^*$, $\kappa \geq 2$ be given. If $x,y \in \R$ are such that $U_{\kappa-1}(x)$, $U_{\kappa-1}(y) \neq 0$, then $T_{\kappa} (x) = T_{\kappa} (y) \Leftrightarrow [U_{\alpha \kappa-1}(x) , U_{\beta \kappa-1}(y) ] = 0, \forall \alpha, \beta \in \N^*$.
\end{corollary}
\begin{proof}
%First we have $\sin(\theta) = \pm \sqrt{1-x^2} \neq 0$ and $\sin (\phi) = \pm \sqrt{1-y^2} \neq 0$.
Let $x = \cos(\theta)$, $y = \cos(\phi)$. By assumption $\sin(\kappa \theta)$, $\sin(\kappa \phi) \neq 0$. So :
\begin{align*}
 [U_{\alpha \kappa-1}(x) , U_{\beta \kappa-1}(y) ] = 0, \forall \alpha, \beta \in \N^* 
& \Leftrightarrow \sin(\alpha \kappa \phi) \sin(\beta \kappa \theta) = \sin(\alpha \kappa \theta) \sin(\beta \kappa \phi), \forall \alpha, \beta \in \N^* \\
& \Leftrightarrow \sin(\kappa \phi) \sin(2\kappa \theta) = \sin(\kappa \theta) \sin(2\kappa \phi) \\
& \quad \quad \quad \quad  \underline{and} \quad \cos(\kappa \theta ) = \cos(\kappa \phi) \\
&\Leftrightarrow T_{\kappa} (x) = T_{\kappa} (y).
% \Leftrightarrow \cos(\kappa \theta ) = \cos(\kappa \phi)  
\end{align*}
\qed
\end{proof}

\begin{corollary}
\label{basic_lemma}
Let $\kappa \in \N^*$, $\kappa \geq 2$ be given. Let $x,y \in [-1,1]$. Then $[U_{\alpha \kappa-1}(x) , U_{\beta \kappa-1}(y) ]=0$ for all $\alpha, \beta \in \N^*$ if and only if $U_{\kappa-1}(x)=0$, or $U_{\kappa-1}(y)=0$, or $T_{\kappa}(x) = T_{\kappa}(y)$.
\end{corollary}
Another identity we'll exploit is 
\begin{equation}
\label{identity_derivative}
\frac{d}{dx} T_{\kappa}(x) = \kappa U_{\kappa-1}(x).
\end{equation}

Finally, we'll make use of the variations of $T_{\kappa}$ :

\begin{Lemma}
\label{variationsTk}
Fix $\kappa \geq 1$. $T_{\kappa}([-1,1]) = [-1,1]$. $T_{\kappa}(1) = 1$, $T_{\kappa}(-1) = (-1)^{\kappa}$. The local extrema of $T_{\kappa}$ in $[-1,1]$ are located at $\cos (j \pi / \kappa)$, $0 \leq j \leq \kappa$. On $(\cos(j \pi / \kappa), \cos((j-1) \pi / \kappa))$, $j \in \{ 0, ..., \kappa\}$, $T_{\kappa}$ is strictly increasing if $j$ is odd and strictly decreasing if $j$ is even. 
\end{Lemma}

\section{Functional representation of the strict Mourre estimate for $\Delta$ wrt.\ $\mathbb{A}$}
\label{stdLaplacianMourre}

\begin{comment}
We compute $[\Delta, \i A_{\kappa}]_{\circ}$ and study the sets $\boldsymbol{\mu}_{A_{\kappa}} (\Delta)$. First, we have the useful formulas 
\begin{equation}
\label{cheby}
\sin(n \xi) = 2^{n-1} \prod_{j=0} ^{n-1} \sin \left( \frac{\pi j}{n} + \xi \right) = \sin(\xi) U_{n-1}(\cos(\xi)), \quad n \in \N^*,
\end{equation}
where the $U_{n}$ are the Chebyshev polynomials of the second kind of order $n$. Thanks to \eqref{cheby} the commutator between the Standard Laplacian and $A_{\kappa}$, for general $d$ and $\kappa = (\kappa_j)_{j=1} ^d$, is
\end{comment}
Let $\mathcal{F} : \mathscr{H} \to L^2([-\pi,\pi]^d,d\xi)$ be the Fourier transform 
\begin{equation}
\normalsize
(\mathcal{F} u)(\xi) :=  (2\pi)^{-d/2} \sum \limits_{n \in \Z^d} u(n) e^{\i n \cdot \xi}, \quad  \xi=(\xi_1,\ldots,\xi_d).
\label{FourierTT}
\end{equation}
The commutator between $\Delta$ and $A_{j\kappa}$, computed against compactly supported sequences, is $[\Delta, \i A_{j\kappa}] =  \mathcal{F}^{-1} \left[  \sum_{i=1} ^d \sin(\xi_i) \sin(j\kappa \xi_i) \right] \mathcal{F} = \sum_{i=1} ^d (1-\Delta_i^2) U_{j\kappa-1} (\Delta_i)$, $\forall j\in\N^*$.
\begin{comment}
\begin{equation}
\label{MV:k333Ultra}
\begin{aligned}
[\Delta, \i A_{j\kappa}] &=  \mathcal{F}^{-1} \left[  \sum_{i=1} ^d \sin(\xi_i) \sin(j\kappa \xi_i) \right] \mathcal{F}  = \sum_{i=1} ^d (1-\Delta_i^2) U_{j\kappa-1} (\Delta_i), \quad j \in \N^*. 
\end{aligned}
\end{equation}
\end{comment}
So $[\Delta, \i A_{j\kappa}]$ extends to a bounded operator $[\Delta, \i A_{j\kappa}]_{\circ}$. Let 
\begin{equation}
\label{def_m}
m(x) := 1-x^2.
\end{equation}
Fix $E \in \sigma(\Delta)$ and consider the polynomial $g_{j\kappa}  ^E : [-1,1]^{d-1} \mapsto \R$,
\begin{equation}
\label{def:gE}
g_{j\kappa} ^E (x_1,...,x_{d-1}) :=  \sum_{i=1} ^{d-1} m(x_i) U_{j\kappa - 1} (x_i) + m\left(E- \sum_{i=1} ^{d-1} x_i \right) U_{j\kappa - 1} \left(E- \sum_{i=1} ^{d-1} x_i \right) .
\end{equation}

\begin{Lemma}
\label{lemma_rootsy}
The roots of $m(x) U_{j \kappa-1}(x)$ are $\{\cos(l \pi / (j\kappa)) : 0 \leq l \leq j \kappa \}$. The intersection over $j \in \N^*$ of the latter set is $\{ \cos(l \pi / \kappa) : 0 \leq l \leq \kappa \}$ and these are roots of $m(x) U_{j \kappa-1}(x)$, $\forall j \in \N^*$. 
\end{Lemma}

If the linear combination of conjugate operators is $\mathbb{A} = \sum_{j \geq 1} \rho_{j\kappa} \cdot A_{j\kappa}$ set $G_{\kappa} ^E : [-1,1]^{d-1} \mapsto \R$,
\begin{equation}
\label{def:GGE}
G_{\kappa} ^E (x_1,...,x_{d-1}) :=  \sum_{j \geq 1} \rho_{j\kappa} \cdot g_{j\kappa} ^E (x_1,...,x_{d-1}) .
\end{equation}
Of course, $G_{\kappa} ^E$ depends on the choice of the coefficients $\rho_{j\kappa}$, but it is not indicated explicitly in the notation. Recall $S_E$ defined by \eqref{constE_MV} (constant energy surface). Note that $S_E$ is symmetric in all variables. So $S'_E := \left. S_E \right|_{\R^{d-1}}$ is unambiguously defined. The point is that $\left. G_{\kappa}^E \right|_{S'_E}$ is a functional representation of $1_{\{ E \} } (\Delta) [\Delta, \i \mathbb{A}]_{\circ} 1_{\{ E \} } (\Delta)$. By functional calculus and continuity of the function $G_{\kappa} ^E$, $E \in \boldsymbol{\mu}_{\kappa} (\Delta)$ if and only if $\left. G_{\kappa}^E \right|_{S'_E} > 0$. Note also that \eqref{def:gE} and \eqref{def:GGE} are basically the same thing as \eqref{def:gE_intro} and \eqref{def:GGE_intro} but localized in energy $E$.

We highlight specially the 2 and 3-dimensional cases as this is our main focus. In dimension 2, we adopt the simpler notation :
\begin{equation}
\label{def:gE22}
g_{j\kappa} ^E (x) = m(x) U_{j\kappa - 1} (x) + m(E-x) U_{j\kappa - 1} (E-x),
\end{equation}
$x \in [\max(E-1,-1),\min(E+1,1)] \subset [-1,1]$. In dimension 3, we adopt the simpler notation :
\begin{equation}
\label{def:gE33}
g_{j\kappa} ^E (x,y) = m(x) U_{j\kappa - 1} (x) + m(y) U_{j\kappa - 1} (y)  + m(E-x-y) U_{j\kappa - 1} (E-x-y),
\end{equation}
$y \in [\max(E-2,-1),\min(E+2,1)]$ and $x \in [\max(E-y-1,-1),\min(E-y+1,1)]$. 
%To stress the dependence on the dimension, write $\Delta_d$ to mean $\Delta$ on $\ell^2(\Z^d)$.

%So 
%$$[\Delta, \i \sum _{j=1} ^N \rho_{j \kappa} A_{j\kappa}]_{\circ} = \sum _{j=1} ^N \rho_{j \kappa}  g_E ^{j\kappa} (\Delta_1, ..., \Delta_d).$$

\begin{comment}
\color{red}
these are corollaries

\begin{Lemma} For any $d \geq 1$, any $\kappa \in \N^*$, $\pm d \in \boldsymbol{\Theta}_{\kappa} (\Delta_d)$. 
\end{Lemma}
\begin{proof}
$E=d \Leftrightarrow E_1 = ... =E_d = 1  \Rightarrow g_E ^{j\kappa} (E_1,...,E_d) = 0$, $\forall j \in \N^*$. This implies the statement. Similarly for $E=-d$. 
\qed 
\end{proof}

\begin{Lemma} 
\label{lemma_belong_hard_threshold}
For any $d$ even, $\kappa \in \N^*$, $0 \in \boldsymbol{\Theta}_{\kappa} (\Delta)$. For any $d$ odd, $\kappa \in \N^*$, $\kappa$ even, $0 \in \boldsymbol{\Theta}_{\kappa} (\Delta)$.
\end{Lemma}
\begin{proof}
For the case where $d$ is even, take half of the $E_i$'s equal to $1$, the other half equal to $-1$. Then $E= \sum E_i = 0$ and $g_E ^{j \kappa} (E_1,...,E_d)=0$, $\forall j \in \N^*$. For the case where we assume $d$ odd and $\kappa$ even, let $E_i \equiv 0$. Then $E = \sum E_i = 0$ and $g_E ^{j\kappa} (E_1,...,E_d) = 0$, $\forall j \in \N^*$, since $U_{j\kappa-1}(0)=0$.
\qed
\end{proof}
\end{comment}
\color{black}

\color{black}
%Write $\Delta_d$ to mean $\Delta$ on $\ell^2(\Z^d)$ and $\Delta_{d-1}$ to mean $\Delta$ on $\ell^2(\Z^{d-1})$.

%For the next Lemma we require more notation to avoid confusion. Let $g_{E,[d]} ^{j\kappa}$ denote the function $g_E ^{j\kappa}$ from \eqref{def:gE} to specify it is a function of the $d$ variables $(E_1,...,E_d)$. 

\begin{comment}
\begin{Lemma} For any $d \geq 2$, any $\kappa \in \N^*$, one has
\[\boldsymbol{\Theta}_{\kappa} (\Delta_{d}) \cap [0,d] \supset \boldsymbol{\Theta}_{\kappa}(\Delta_{d-1}) \cap [0,d-1] +1.\]
\end{Lemma}
\begin{proof}
$g_{E,[d-1]} ^{j\kappa}(E_1,...,E_{d-1}) = g_{E, [d]} ^{j\kappa} (E_1,...,E_{d-1},1)$. This implies the statement.
\qed
\end{proof}
\end{comment}

\begin{Lemma}
\label{Lemma_symmetryDelta}
For any $d,\kappa \in \N^*$, $\boldsymbol{\mu}_{\kappa} (\Delta) = - \boldsymbol{\mu}_{\kappa} (\Delta)$. Taking complements, $\boldsymbol{\Theta}_{\kappa} (\Delta) = - \boldsymbol{\Theta}_{\kappa} (\Delta)$.
\end{Lemma}
\begin{proof}
The $U_n(\cdot)$ are even when $n$ is even, and odd when $n$ is odd. Also $S_{-E} = -S_E$. If $\kappa$ is even, and we have a linear combination $\sum_{j=1} ^N \rho_{j\kappa} A_{j\kappa}$ such that $G_{\kappa} ^E > 0$ on $S'_E$ so that $E \in \boldsymbol{\mu}_{\kappa} (\Delta)$, then taking $\sum_{j=1} ^N (-\rho_{j\kappa}) A_{j\kappa}$ gives $G_{\kappa} ^{-E} > 0$ on $S'_{-E}$ so that $-E \in \boldsymbol{\mu}_{\kappa} (\Delta)$. Thus $\boldsymbol{\mu}_{\kappa} (\Delta) \subset - \boldsymbol{\mu}_{\kappa} (\Delta)$. The argument is reversed for the reverse inclusion. On the other hand, if $\kappa$ is odd, and we have a linear combination $\sum_{j=1} ^N \rho_{j\kappa} A_{j\kappa}$ such that $G_{\kappa}^E > 0$ on $S'_E$ so that $E \in \boldsymbol{\mu}_{\kappa} (\Delta)$, then taking the linear combination $\sum_{j=1} ^N ((-1)^{j+1} \rho_{j\kappa}) A_{j\kappa}$ gives $G_{\kappa} ^{-E} > 0$ on $S'_{-E}$ so that $-E \in \boldsymbol{\mu}_{\kappa} (\Delta)$.
\qed
\end{proof}
Thanks to Lemma \ref{Lemma_symmetryDelta} we focus on positive energies only in this article. Also :

\begin{Lemma}
\label{Lemma_symmetryDelta_88}
For any $d \in \N^*$, for any $\kappa \in \N^*$ even (!), any $m \in \N$, $\boldsymbol{\Theta}_{m,\kappa} (\Delta) = - \boldsymbol{\Theta}_{m,\kappa} (\Delta)$.
\end{Lemma}

The proof of Lemma \ref{Lemma_symmetryDelta_88} follows directly from the definition of $\boldsymbol{\Theta}_{m,\kappa} (\Delta)$. It is an open problem for us to decide if Lemma \ref{Lemma_symmetryDelta_88} also holds for $\kappa$ odd. We also take the opportunity to prove Lemma \ref{shift_threshold_d}. 

\noindent \textit{Proof of Lemma \ref{shift_threshold_d}}.
Let $E \in  \boldsymbol{\Theta}_{m,\kappa} (\Delta[d])$, with $E = x_{q,1} + ... + x_{q,d}$ for $0 \leq q \leq m$. Set $\E = E+\cos(l\pi / \kappa)$, $0 \leq l \leq \kappa$. Then $\forall j \in \N^*$ (and using the notation \eqref{def:gE_intro} instead of \eqref{def:gE}) :
%$$g_{j \kappa} ^{\E} (\vec{x}_m, \cos(l\pi / \kappa)) = g_{j \kappa} ^{E} (\vec{x}_m) =  \sum _{q =0} ^{m-1} \omega_q \cdot g_{j \kappa} ^{E} (\vec{x}_q) = \sum _{q =0} ^{m-1} \omega_q \cdot g_{j \kappa} ^{\E} (\vec{x}_q, \cos(l\pi / \kappa)).$$
$$g_{j \kappa} (\vec{x}_m, \cos(l\pi / \kappa)) = g_{j \kappa} (\vec{x}_m) =  \sum _{q =0} ^{m-1} \omega_q \cdot g_{j \kappa} (\vec{x}_q) = \sum _{q =0} ^{m-1} \omega_q \cdot g_{j \kappa} (\vec{x}_q, \cos(l\pi / \kappa)).$$
This implies $\E \in \boldsymbol{\Theta}_{m,\kappa} (\Delta[d+1])$.
\qed

\begin{Lemma}
\label{symmetry_minima_E/2}
Let $d=2$. Each $g_{j\kappa} ^E$, and hence $G_{\kappa} ^E$, is symmetric about the axis $x = E/2$. In particular $\frac{d}{dx} g_{j\kappa} ^E(E/2) = 0$ for each $j$ and $\frac{d}{dx} G_{\kappa} ^E (E/2) = 0$ for any choice of coefficients $\rho_{j \kappa}$.
\end{Lemma}
\begin{proof}
Straightforwardly from \eqref{def:gE22}, $g_{j\kappa} ^E (E/2 - t) = g_{j\kappa} ^E (E/2 + t)$ for all $t\in\R$. 
\qed
\end{proof}

\section{Trying to identify thresholds by brute force : initial attempt}

In this section we have a first crack at trying to determine the thresholds $\in \boldsymbol{\Theta}_{m,\kappa}(\Delta)$. If we make simple assumptions, we are able to solve the equations (detailed below) for $m=0,1$ (in dimension 2) by brute force, but we don't know if these assumptions are satisfactory to produce \underline{\textit{all}} solutions for $m=0,1$. For $m \geq 2$ however, we have no idea how to solve the equations. In section \ref{geo_construction} a different approach is used to determine thresholds $\in \boldsymbol{\Theta}_{m,\kappa}(\Delta)$, in dimension 2.
%The reader should be warned that it is not clear to us if such strategy allows to determine \textit{all} thresholds. 

%If $E,Y$ are such a solution, then $G_E ^{\alpha \kappa} (Y) = \sum_{j \geq 1} \rho_{\alpha j \kappa}g_E ^{ \alpha j \kappa} (Y) = 0$ for any choice of coefficients $\rho_{\alpha j \kappa}$, any $\alpha \in \N^*$, and so $E \in \boldsymbol{\Theta}_{\alpha \kappa}(\Delta)$ for all $\alpha \in \N^*$. First let us signal that all the thresholds stated in Lemma \ref{lemSUMcos} satisfy this zeroth order condition. 

\subsection{Trying to determine $0^{th}$ order thresholds} In dimension 2, $E \in \boldsymbol{\Theta}_{0,\kappa} (\Delta) \cap [0,2]$ iff $\exists E \in [0,2]$ and $Y \in [E-1,1]$ such that $g_{j \kappa} ^E (Y) = 0$, $\forall j \in \N^*$. We start by proving Lemma \ref{lemSUMcosINTRO}.
\label{subsection0th}

\noindent \textit{Proof of Lemma \ref{lemSUMcosINTRO} }.
Let $E= \sum_{q=1}^d x_q$, $x_q = \cos( j_q \pi / \kappa)$. Then $g_{j\kappa} (x_1, ..., x_d) = 0$, $\forall j \in \N^*$.
\qed

We followed up Lemma \ref{lemSUMcosINTRO} with Conjecture \ref{conjecture0001}. Our evidence for Conjecture \ref{conjecture0001} is :
%As mentioned in the introduction, we actually believe $\boldsymbol{\theta}_{0,\kappa} (\Delta) = \boldsymbol{\Theta}_{0,\kappa} (\Delta)$ (Conjecture \ref{conjecture0001}). Our Conjecture  :

\begin{Lemma} 
\label{lem_2346}
For $\kappa = 2,3,4,6$, in dimension 2, $\boldsymbol{\theta}_{0,\kappa} (\Delta) = \boldsymbol{\Theta}_{0,\kappa} (\Delta)$.
\end{Lemma}

\begin{proof}
$g_{j \kappa} ^E (Y) = 0 \Leftrightarrow m(Y) U_{j\kappa-1}(Y) = -m(E-Y) U_{j\kappa-1}(E-Y)$. This happens iff :
\begin{itemize}
\item $m(Y) = 0$ and $m(E-Y)=0$, or
\item $m(Y) = 0$ and $U_{\kappa-1}(E-Y)=0$, or
\item $U_{\kappa-1}(Y)= 0$ and $m(E-Y)=0$, or
\item $U_{\kappa-1}(Y)= 0$ and $U_{\kappa-1}(E-Y)=0$, or 
\item $m(Y)$, $m(E-Y)$, $U_{\kappa-1}(Y)$ and $U_{\kappa-1}(E-Y)$ are all non-zero, and for all $j$,
\begin{equation}
\label{fracU}
\frac{U_{j \kappa-1}(Y)}{U_{j \kappa-1}(E-Y)} = - \frac{m(E-Y)}{m(Y)}.
\end{equation}
\noindent We don't know how to solve \eqref{fracU} directly, but instead we note that it implies 
\begin{equation} 
\label{crochetU}
[U_{j_1 \kappa-1}(Y) , U_{j_2 \kappa-1}(E-Y)]=0, \quad  \text{or equivalently}, \quad  \ T_{\kappa}(Y) = T_{\kappa}(E-Y),
\end{equation}
which is easy to solve (the equivalence holds due to Corollary \ref{corollaryEquiv}). So we solve \eqref{crochetU} and keep only those solutions that also satisfy \eqref{fracU}.
\end{itemize}
The solutions to the first 4 bullet points are included in Lemma \ref{lemSUMcosINTRO}. We briefly discuss the solutions to the 5th bullet point. For $\kappa = 2$, \eqref{crochetU} has solutions $E=0, 2Y$. As a general rule, for any $\kappa$, $E=2Y$ solves \eqref{crochetU} but not \eqref{fracU}. So $E=0$ is the only valid solution and it is included in Lemma \ref{lemSUMcosINTRO}. For $\kappa=3$, \eqref{crochetU} leads to $E-Y = (-Y - \sqrt{3} \sqrt{1-Y^2} )/2$ or $E-Y = (-Y + \sqrt{3} \sqrt{1-Y^2} )/2$. Plugging these into \eqref{fracU} for $j=1$, leads to the solutions $E= \pm 1/2, -1$ and $E = \pm 1/2, 1$, respectively, which are already included in Lemma \ref{lemSUMcosINTRO}. For $\kappa=4$, \eqref{crochetU} leads to $E=0$, or $E=2Y$ (which we reject), or $E-Y = \sqrt{1-Y^2}$ or $E-Y = - \sqrt{1-Y^2}$. Plugging the latter 2 equations into \eqref{fracU} for $j=1$, leads to the solutions $E= 0, \pm 1, \pm \sqrt{2}$, which are all included in Lemma \ref{lemSUMcosINTRO}. For $\kappa=6$, \eqref{crochetU} leads to $E=0$, or $E=2Y$ (which we reject), or $E-Y = (\pm Y \pm \sqrt{3} \sqrt{1-Y^2} )/2$ or $E-Y = ( \pm Y \mp \sqrt{3} \sqrt{1-Y^2} )/2$. Plugging these into \eqref{fracU} for $j=1$, leads to the solutions $E= \pm 1/2, \pm 1, 0, \pm \sqrt{3}/2$, which are all included in Lemma \ref{lemSUMcosINTRO}.
\qed
\end{proof}

%It would therefore appear, based on the cases of $\kappa=2,3,4,6$, that the Zeroth Order Strategy does not yield energies $E$ which are not already included in Lemma \ref{lemSUMcos}.

\subsection{Trying to determine $1^{st}$ order thresholds} 
\label{subsection1st}
We try to find thresholds $\in \boldsymbol{\Theta}_{1,\kappa} (\Delta) \cap [0,2]$ in dimension 2 : $E \in \boldsymbol{\Theta}_{1,\kappa} (\Delta)$ iff $\exists E \in [0,2]$, $Y_0,Y_1 \in [E-1,1]$, and $\omega_0 <0$ such that 
\begin{equation}
\label{OMEGAO}
\omega_0 = \frac{g_{j \kappa} ^E (Y_1)}{g_{j \kappa} ^E (Y_0)} = \frac{g_{l \kappa} ^E (Y_1)}{g_{l \kappa} ^E (Y_0)}, \quad \forall j,l \in \N^*
\end{equation}
(we assume $g_{j \kappa} ^E (Y_0) \neq 0$ otherwise we are back in the case of subsection \ref{subsection0th}), which leads to solving $[g_{j \kappa}^E (Y_0),g_{l \kappa} ^E (Y_1)]=0$. Expanding, this is :

\begin{equation}
\label{long_comm22}
\begin{aligned}
& m(Y_0) m(Y_1) [U_{j \kappa-1}(Y_0) , U_{l \kappa-1}(Y_1) ] \\
& \quad + m(E-Y_0) m(Y_1) [U_{j \kappa-1}(E-Y_0) , U_{l \kappa-1}(Y_1) ]  \\
& \quad + m(Y_0) m(E-Y_1) [U_{j \kappa-1}(Y_0) , U_{l \kappa-1}(E-Y_1) ] \\
& \quad + m(E-Y_0) m(E-Y_1) [U_{j \kappa-1}(E-Y_0) , U_{l \kappa-1}(E-Y_1) ] = 0.
\end{aligned}
\end{equation}
We don't know how to proceed in order to thoroughly solve this equation. Instead, we propose a handful of assumptions which simplify things (and basically amounts to setting each of the 4 terms in \eqref{long_comm22} equal to 0). The list of various Ansatz is :   
\begin{enumerate}
\item $E-Y_0 = 1$, $[U_{j \kappa-1}(E-1) , U_{l \kappa-1}(Y_1) ] = 0$ and $[U_{j \kappa-1}(E-1) , U_{l \kappa-1}(E-Y_1) ] = 0$,
\item$E-Y_0=-1$, $[U_{j \kappa-1}(E+1) , U_{l \kappa-1}(Y_1) ] =0$ and $[U_{j \kappa-1}(E+1) , U_{l \kappa-1}(E-Y_1) ] =0$,
\item $Y_0 = E/2$, $[U_{j \kappa-1}(E/2) , U_{l \kappa-1}(Y_1) ] =0$ and $[U_{j \kappa-1}(E/2) , U_{l \kappa-1}(E-Y_1) ] =0$,
\item $U_{\kappa-1}(Y_1)=0$, $[U_{j \kappa-1}(Y_0) , U_{l \kappa-1}(E-Y_1) ]=0$, $[U_{j \kappa-1}(E-Y_0) , U_{l \kappa-1}(E-Y_1) ]=0$,
\item $U_{\kappa-1}(E-Y_1)=0$, $[U_{j \kappa-1}(Y_0) , U_{l \kappa-1}(Y_1) ]=0$, $[U_{j \kappa-1}(E-Y_0) , U_{l \kappa-1}(Y_1) ]=0$,
\item $T_{\kappa}(Y_0)=T_{\kappa}(Y_1)$, $T_{\kappa}(E-Y_0)=T_{\kappa}(Y_1)$, $T_{\kappa}(Y_0)=T_{\kappa}(E-Y_1)$, and $T_{\kappa}(E-Y_0)=T_{\kappa}(E-Y_1)$.
\end{enumerate}
The correct statement is that $(i)$ implies \eqref{long_comm22} for $i=1,....,6$. Note also that for $i=1,..,5$ we have boiled down to 3 equations, 3 unknowns, whereas for $i=6$ we have created ourselves 4 equations, 3 unknowns. In solving $(i)$, we systematically use Corollary \ref{basic_lemma}. Once $(i)$ is solved, we compute $\omega_0$ as per \eqref{OMEGAO} and check if it is negative. If it is the case, we have found a valid solution to our problem \eqref{special relationship}. To speed up calculations, we always ignore solutions where $Y_0 = Y_1$ or $Y_0 = E-Y_1$ as these would lead to $\omega_0 =1$. 

\begin{Lemma} For $\kappa=2$, solutions to (1) are $E= 2/3$, $Y_1 = E/2$ and $E = 1/2$, $Y_1 = 0$. For $2 \leq i \leq 6$, either there are no solutions to (i) or they are the same as (1). For $\kappa =3$, solutions to (i) are in Table \ref{sol_kappa_2346}.

\begin{table}[H]
\footnotesize
  \begin{center}
    \begin{tabular}{c|c|c} % <-- Alignments: 1st column left, 2nd middle and 3rd right, with vertical lines in between
   (1) & (3) & (4)  \\ [0.5em]
      \hline
  $E=(5-3\sqrt{2})/7 \simeq 0.108$, $Y_1 = E/2 \simeq 0.054$  & $E=2/7 \simeq 0.285$, & $E = 1/ \sqrt{6} \simeq 0.408$, $Y_1 = 1/2$, $Y_0=-1/2$ \\ [0.5em]
  $E=(9-\sqrt{33})/12 \simeq 0.271$, $Y_1 = 1/2$  & \ \ $Y_1 = -1/2$ & $E = 1/4$, $Y_1 = 1/2$, $Y_0=(1+3\sqrt{5})/8 \simeq 0.963$ \\ [0.5em]
  $E=(9+\sqrt{33})/12 \simeq 1.228$, $Y_1 = 1/2$ & &  \\ [0.5em]
  $E=(5+3\sqrt{2})/7 \simeq 1.320$, $Y_1 = E/2 \simeq 0.660$ & & 
    \end{tabular}
  \end{center}
    \caption{Solutions to $(i)$ for $\kappa=3$. For (2), (5), (6) either there are no solutions or they are the same as in (1), (3), (4)}
        \label{sol_kappa_2346}
\end{table}
\normalsize

\begin{comment}
\begin{table}[H]
\small
  \begin{center}
    \begin{tabular}{c|c} % <-- Alignments: 1st column left, 2nd middle and 3rd right, with vertical lines in between
    $\kappa$ & $3$  \\ [0.5em]
      \hline
(1) &  $E=(5-3\sqrt{2})/7 \simeq 0.108$, $Y_1 = E/2 \simeq 0.054$   \\ [0.5em]
     & $E=(9-\sqrt{33})/12 \simeq 0.271$, $Y_1 = 1/2$   \\ [0.5em]
    &  $E=(9+\sqrt{33})/12 \simeq 1.228$, $Y_1 = 1/2$   \\ [0.5em]
    &  $E=(5+3\sqrt{2})/7 \simeq 1.320$, $Y_1 = E/2 \simeq 0.660$  \\ [0.5em]
(2) &    \\ [0.5em]
(3) & $E=2/7 \simeq 0.285$, $Y_1 = -1/2$    \\ [0.5em]
(4) & $Y_1 = 1/2$, $Y_0=-1/2$, $E = 1/ \sqrt{6} \simeq 0.408$  \\ [0.5em]
 & $Y_1 = 1/2$, $Y_0=(1+3\sqrt{5})/8 \simeq 0.963$, $E = 1/4$  \\ [0.5em]
(5) &   \\ [0.5em]
(6) &   \\ [0.5em]
    \end{tabular}
  \end{center}
    \caption{Solutions to $(i)$ for $\kappa=3$.}
        \label{sol_kappa_2346}
\end{table}
\normalsize
\end{comment}

\end{Lemma}

\subsection{Trying to determine $m^{th}$ order thresholds} For general $m \geq 1$, $E \in \boldsymbol{\Theta}_{m,\kappa} (\Delta) \cap [0,2]$ in dimension 2 iff $\exists E \in [0,2]$, $(Y_q)_{q=0}^m \subset [E-1, 1]$, and $(\omega_q)_{q=0}^{m-1} \subset \R$, $\omega_q \leq 0$, such that 
\begin{equation*}
\label{special relationship}
g_{j \kappa} ^E (Y_m) = \sum_{q=0} ^{m-1} \omega_q \cdot g_{j \kappa} ^E ( Y_q), \quad \forall j \in \N^* \quad (\omega_q \ independent \ of j).
\end{equation*} 
Now, if this linear relationship holds, it must be that for any choice of disctinct $j_1$, $j_2$, ..., $j_{m} \in \N^*$,
\begin{equation}
\begin{aligned}
\label{matrix_system_tosolve00}
\begin{pmatrix}
\omega_{m-1} \\
\omega_{m-2} \\
... \\
\omega_{0}
\end{pmatrix}
&= \begin{pmatrix} g_{\kappa} ^E (Y_{m-1}) & g_{\kappa} ^E (Y_{m-2}) & ... & g_{\kappa} ^E (Y_0) \\ g_{2\kappa} ^E (Y_{m-1}) & g_{2 \kappa} ^E (Y_{m-2}) & ... &  g_{2 \kappa} ^E (Y_0) \\
... & ... & ... & ... \\
g_{m \kappa} ^E (Y_{m-1}) & g_{m\kappa} ^E (Y_{m-2}) & ... &  g_{m \kappa} ^E (Y_0) \\
\end{pmatrix} ^{-1}
   \begin{pmatrix} g_{ \kappa} ^E (Y_{m}) \\ g_{2 \kappa} ^E (Y_{m}) \\
   ... \\
  g_{m  \kappa} ^E (Y_{m})  \end{pmatrix}  \\
  &=
\begin{pmatrix} g_{j_{1}  \kappa} ^E (Y_{m-1)} & g_{j_{1}  \kappa} ^E (Y_{m-2}) & ... & g_{j_{1}  \kappa} ^E (Y_0) \\ g_{j_{2}  \kappa} ^E (Y_{m-1}) & g_{j_{2}  \kappa} ^E (Y_{m-2}) & ... &  g_{j_{2}  \kappa} ^E (Y_0) \\
... & ... & ... & ... \\
g_{j_{m} \kappa} ^E (Y_{m-1}) & g_{j_{m} \kappa} ^E (Y_{m-2}) & ... &  g_{j_{m} \kappa} ^E (Y_0) \\
\end{pmatrix} ^{-1}
   \begin{pmatrix} g_{j_{1} \kappa} ^E (Y_m) \\ g_{j_{2} \kappa} ^E (Y_m) \\
   ... \\
  g_{j_{m}  \kappa} ^E (Y_m)  \end{pmatrix}.
\end{aligned}
\end{equation}
Note we are perhaps uncorrectly assuming that the above matrices are invertible, but at this point we're just trying to illustrate the concept. Unless there is a simplification we do not see, this system appears very complicated to solve directly, even for small $m$. In other words solving \eqref{matrix_system_tosolve00} apparently entails inverting Vandermonde matrices, whose entries consist of the polynomials $\{ g_{j_{q} \kappa} ^E (x)\}_{q=1}^{m}$ evaluated at $x = Y_0$, ..., $Y_{m-1}$ (which together with $Y_m$ and $E$ are unknowns). 

For $m=2$, \eqref{matrix_system_tosolve00} leads to solving (unknowns are $E,Y_0,Y_1,Y_2$) the system of 2 equations
\begin{equation}
\label{cases1}
\begin{cases}
\ \ [ g_{j_1 \kappa} ^E (Y_1) , g_{j_2 \kappa} ^E (Y_0)] \times [g_{2 \kappa} ^E (Y_0) , g_{\kappa} ^E (Y_2) ] =  [g_{\kappa} ^E (Y_1) , g_{2\kappa} ^E (Y_0) ] \times [ g_{j_2 \kappa} ^E (Y_0) , g_{j_1 \kappa} ^E (Y_2)], & \\
\ \ [ g_{j_1 \kappa} ^E (Y_1) , g_{j_2 \kappa} ^E (Y_0)] \times [g_{2 \kappa} ^E (Y_1) , g_{\kappa} ^E (Y_2) ] =  [g_{\kappa} ^E (Y_1) , g_{2\kappa} ^E (Y_0) ] \times [ g_{j_2 \kappa} ^E (Y_1) , g_{j_1 \kappa} ^E (Y_2)].
\end{cases}
\end{equation}
Of course, here one is interested in solutions where each of the commutator terms is non-zero, otherwise that leads us back to the case $m=1$. Frankly, we have no idea how to solve this system, and this is just $m=2$.

\begin{comment}
If such a solution occurs, then for any choice of coefficients $\rho_{ j\alpha \kappa}$, and any $\alpha \in \N^*$,
\begin{equation}
\label{special G relationship}
G_{E} ^{\alpha \kappa} (Y_m) = \sum _{j\geq 1} \rho_{j \alpha\kappa} g_E ^{j \alpha \kappa} ( Y_m) = \sum _{q=0} ^{m-1} \omega_q \sum _{j\geq 1} \rho_{j \alpha \kappa} g_E ^{j \alpha \kappa} ( Y_q) = \sum _{q=0} ^{m-1} \omega_q G_{E} ^{\alpha \kappa} (Y_q).
\end{equation}
If $E$ were to belong to $\boldsymbol{\mu}_{\alpha \kappa}(\Delta)$, the lhs.\ of \eqref{special G relationship} would be a strictly positive number, whereas the rhs.\ of \eqref{special G relationship} would be a strictly negative number : an absurdity. Hence the  special linear relationship \eqref{special relationship} implies $E \in \boldsymbol{\Theta}_{\alpha \kappa}(\Delta)$, for all $\alpha \in \N^*$. But it is important to clarify that we don't know how to solve \eqref{special relationship} generally -- we will develop (only ?) certain solutions -- and furthermore, it may very well be that solving \eqref{special relationship} doesn't even yield all thresholds. 
\end{comment}

\section{A geometric construction to find thresholds in dimension 2}
\label{geo_construction}

The idea below is our bread and butter to find thresholds $\in \boldsymbol{\Theta}_{m,\kappa} (\Delta)$ in dimension 2. In sections \ref{section J2}, \ref{section J3a} and \ref{section J3b} we apply the idea and prove the existence of valid solutions to \eqref{conjecture_system_conj_intro} and \eqref{conjecture_system_conj_intro_even}. 

Fix $\kappa \geq 2$, $n \in \N$. Consider real variables $\E_n, X_{0,n}, X_{1,n}, ..., X_{n,n}, X_{n+1,n}$. For $n \in \N^*$ \underline{odd} define $\mathfrak{T}_{n,\kappa}$ to be the set of $\E_n = \E_n (\kappa)$ such that the system of $n+3$ equations in $n+3$ unknowns 
\begin{equation}
\label{conjecture_system_conj_intro}
\begin{cases}
T_{\kappa}(X_{q,n}) = T_{\kappa}(X_{n-q,n}), \quad \forall \ q = 0,1,..., (n-1)/2&  \\
X_{n-q,n} = \E_n -X_{1+q,n}, \quad \forall \ q = -1, 0,..., (n-1)/2  & 
\end{cases}
\end{equation}
has a solution satisfying 
\begin{align}
\label{o1}
& X_{q,n}, X_{n-q,n} \in (-1,1), \quad \forall \ 0 \leq q \leq (n-1)/2, \\ 
\label{o2}
& X_{n+1,n} \in \{ \cos(j \pi / \kappa) : 0 \leq j \leq \kappa \}, \\
\label{o3}
& T'_{\kappa}(X_{q,n}) \cdot T'_{\kappa}(X_{n-q,n}) < 0, \quad \forall \ 0 \leq q \leq (n-1)/2.
\end{align}
When $n$ is fixed and no confusion can arise we will simply write $X_i$ instead of $X_{i,n}$. Note also that the $X_i$'s depend on $\kappa$.

\begin{comment}
Fix $\kappa \geq 2$, $n \in \N$. Consider real variables $\E_n, X_0, X_1, ..., X_n, X_{n+1}$. For $n \in \N^*$ \underline{odd} define $\mathfrak{T}_{n,\kappa}$ to be the set of $\E_n = \E_n (\kappa)$ such that the system of $n+3$ equations in $n+3$ unknowns 
\begin{equation}
\label{conjecture_system_conj_intro}
\begin{cases}
T_{\kappa}(X_q) = T_{\kappa}(X_{n-q}), \quad \forall \ q = 0,1,..., (n-1)/2&  \\
X_{n-q} = \E_n -X_{1+q}, \quad \forall \ q = -1, 0,..., (n-1)/2  & 
\end{cases}
\end{equation}
has a solution satisfying 
\begin{align}
\label{o1}
& X_q, X_{n-q} \in (-1,1), \quad \forall \ 0 \leq q \leq (n-1)/2, \\ 
\label{o2}
& X_{n+1} \in \{ \cos(j \pi / \kappa) : 0 \leq j \leq \kappa \}, \\
\label{o3}
& T'_{\kappa}(X_q) \cdot T'_{\kappa}(X_{n-q}) < 0, \quad \forall \ 0 \leq q \leq (n-1)/2.
\end{align}
\end{comment}
%\begin{remark} The solutions $X_0, X_1$,..., $X_{n+1}$ depend on $\kappa$ and $n$. Thus we could be writing $X_{0,n}(\kappa), X_{1,n}(\kappa)$,..., $X_{n+1,n}(\kappa)$, but choose not to to keep the notation simple. Later we use the notation $X_{0,n}, X_{1,n}$,..., $X_{n+1,n}$ to avoid confusion. The same remark applies to the $n$ even case discussed below.
%\end{remark}
\begin{proposition} (odd terms)
\label{prop1_intro}
Fix $\kappa \geq 2$, let $n \in \N^*$ odd be given. Let $(X_q)_{q=0}^{n+1}$, $\mathcal{E}_n (\kappa)$ be a solution to \eqref{conjecture_system_conj_intro} such that $\mathcal{E}_n (\kappa) \in \mathfrak{T}_{n,\kappa}$. Then for all $j \in \N^*$,
\begin{equation}
\label{problem1_tosolve_again_101_88}
g_{j \kappa} ^{\mathcal{E}_n} ( X_{(n+1)/2} ) = \sum _{q =0} ^{(n-1)/2} \omega_q \cdot g_{j \kappa} ^{\mathcal{E}_n} ( X_{q}), \quad \omega_q = 2 (-1)^{(n-1)/2-q} \frac{\prod _{p=(n+1)/2} ^{n-q} m(X_p) U_{\kappa-1} (X_p)}{\prod _{p'=q} ^{(n-1)/2} m(X_{p'}) U_{\kappa-1} (X_{p'})}.
\end{equation}
The components of $\omega$ are all strictly negative and independent of $j$. In particular, for all $n \in \N^*$, $\mathfrak{T}_{2n-1,\kappa} \subset \boldsymbol{\Theta}_{n,\kappa} (\Delta[d=2])$.
\end{proposition}

%there is $\omega := (\omega_{(n-1)/2} , \omega_{(n-1)/2-1}, ..., \omega_0 )^{T}$ (independent of $j$) such 

According to \eqref{o1} and \eqref{o3}, and using \eqref{identity_derivative}, note that the denominator of $\omega_q$ is well-defined. A similar comment applies to the $n$ even case, which we now describe. For $n \in \N^*$ \underline{even} define $\mathfrak{T}_{n,\kappa}$ to be the set of $\E_n$ such that the system of $n+3$ equations in $n+3$ unknowns 

\begin{equation}
\label{conjecture_system_conj_intro_even}
\begin{cases}
T_{\kappa}(X_{q,n}) = T_{\kappa}(X_{n-q,n}), \quad \forall \ q = 0,1,..., n/2-1 &  \\
X_{n-q,n} = \E_n -X_{1+q,n}, \quad \forall \ q = -1, 0,..., n/2 -1  & 
\end{cases}
\end{equation}
has a solution satisfying 
\begin{align}
\label{o11}
& X_{q,n}, X_{n-q,n} \in (-1,1), \quad \forall \ 0 \leq q \leq n/2-1 \\
\label{o22}
& X_{n/2,n}, X_{n+1,n} \in \{ \cos(j \pi / \kappa) : 0 \leq j \leq \kappa \}, \\
\label{o33}
& T'_{\kappa}(X_{q,n}) \cdot T'_{\kappa}(X_{n-q,n}) < 0, \quad \forall \ 0 \leq q \leq n/2-1.
\end{align}
Again, when $n$ is fixed and no confusion can arise we will simply write $X_i$ instead of $X_{i,n}$.

\begin{comment}
\begin{equation}
\label{conjecture_system_conj_intro_even}
\begin{cases}
T_{\kappa}(X_q) = T_{\kappa}(X_{n-q}), \quad \forall \ q = 0,1,..., n/2-1 &  \\
X_{n-q} = \E_n -X_{1+q}, \quad \forall \ q = -1, 0,..., n/2 -1  & 
\end{cases}
\end{equation}
has a solution satisfying 
\begin{align}
\label{o11}
& X_q, X_{n-q} \in (-1,1), \quad \forall \ 0 \leq q \leq n/2-1 \\
\label{o22}
& X_{n/2}, X_{n+1} \in \{ \cos(j \pi / \kappa) : 0 \leq j \leq \kappa \}, \\
\label{o33}
& T'_{\kappa}(X_q) \cdot T'_{\kappa}(X_{n-q}) < 0, \quad \forall \ 0 \leq q \leq n/2-1.
\end{align}
\end{comment}

\begin{proposition} (even terms)
\label{prop1_intro_even}
Fix $\kappa \geq 2$, let $n \in \N^*$ even be given. Let $(X_q)_{q=0}^{n+1}$, $\mathcal{E}_n (\kappa)$ be a solution to \eqref{conjecture_system_conj_intro_even} such that $\mathcal{E}_n (\kappa) \in \mathfrak{T}_{n,\kappa}$. Then for all $j \in \N^*$,
\begin{equation}
\label{problem1_tosolve_again_101even_88}
g_{j \kappa} ^{\mathcal{E}_n} ( X_{n/2} ) = \sum _{q =0} ^{n/2-1} \omega_q \cdot g_{j \kappa} ^{\mathcal{E}_n} ( X_{q}), \quad \omega_q =  (-1)^{n/2-1-q} \frac{\prod _{p=n/2+1} ^{n-q} m(X_p) U_{\kappa-1} (X_p)}{\prod _{p'=q} ^{n/2-1} m(X_{p'}) U_{\kappa-1} (X_{p'})}.
\end{equation}
The components of $\omega$ are all strictly negative and independent of $j$. In particular, for all $n \in \N^*$, $\mathfrak{T}_{2n,\kappa} \subset \boldsymbol{\Theta}_{n,\kappa} (\Delta[d=2])$.
\end{proposition}

\begin{remark} 
For $\kappa \geq 3$ and fixed $n$, systems \eqref{conjecture_system_conj_intro} and \eqref{conjecture_system_conj_intro_even} usually admit several solutions. This is because $T_{\kappa}$ is not an injective function on $[-1,1]$. So for given $x$, $T_{\kappa}(x) = T_{\kappa}(y)$ has several solutions (although we crucially require $T'_{\kappa} (x) \cdot T'_{\kappa} (y) < 0$, see \eqref{o3} and \eqref{o33}). The fact that the equation $T_{\kappa}(x) = T_{\kappa}(y)$ for given $x$ typically has several solutions means there is generally an abundance of solutions, see section \ref{KAPPA10} for an obvious graphical illustration.
\end{remark}

\begin{remark}
\label{obs_symmetry}
Equation $X_{n-q} = \E_n -X_{1+q}$ in \eqref{conjecture_system_conj_intro} and \eqref{conjecture_system_conj_intro_even} entails $\E_n  /2 - \min(X_{n-q}, X_{1+q}) = \max(X_{n-q}, X_{1+q}) - \E_n/2$. Graphically this equation says that $X_{n-q}$ is the symmetric of $X_{1+q}$ wrt.\ the vertical axis $x= \E_n /2$. This is very helpful to bear in mind to construct solutions. Furthermore, in the $n$ odd case, $X_{(n+1)/2} = \E_n /2$, which helps for a graphical construction.
\end{remark} 

Behind systems \eqref{conjecture_system_conj_intro} and \eqref{conjecture_system_conj_intro_even} is a simple graphical construction and interpretation. 

Figures \ref{fig:test_T3k3firsty}, \ref{fig:T3, increasing to 0.5} and \ref{fig:test_T3k3} illustrate solutions to Propositions \ref{prop1_intro} and \ref{prop1_intro_even} for $1 \leq n \leq 6$ and $\kappa=3$. The vertical dotted line is the axis of symmetry $x = \E_n / 2$. The key observation when looking at these graphs is that every point $(X_q, T_{\kappa}(X_q))$ always satisfies 2 crucial conditions :
\begin{enumerate}[label*=\arabic*.]
\item \textit{a symmetry condition} : each $X_q$ is the symmetric of another point $X_r$ wrt.\ the axis $x= \E_n /2$, namely $X_r = X_{n-q+1}$. This is Remark \ref{obs_symmetry}.
\item \textit{a level condition} : each $X_q$ satisfies at least one of the following 3 conditions :
\begin{enumerate}[label*=\arabic*.]
\item $m(X_q) =0$, or
\item $U_{\kappa-1}(X_q) = 0$ (equivalently $T'_{\kappa}(X_q)=0$), or 
\item $\exists X_r$ such that $T_{\kappa}(X_q) = T_{\kappa}(X_r)$ and $T'_{\kappa}(X_q) \cdot T'_{\kappa}(X_r) < 0$ (!), namely $X_r = X_{n-q}$. 
\end{enumerate}
\end{enumerate}

The symmetry and level conditions set the rules of the game to construct a valid threshold solution $\in \mathfrak{T}_{n,\kappa} (\Delta)$. Possible constructions are as follows :

\noindent \underline{\textbf{Algorithm for $n$ odd}} -- system \eqref{conjecture_system_conj_intro} along with conditions \eqref{o1}--\eqref{o3} : 
\begin{enumerate}
\item Fix $\kappa \geq 2$ and plot the Chebyshev polynomial $T_{\kappa}(x)$.
\item Initialize energy $E$ to a certain value such that $E/2 \in [-1,1] \setminus \{ \cos( j \pi / \kappa) : 0 \leq j \leq \kappa \}$, and draw the vertical axis of symmetry $x=E/2$.
\item Place the first point $(X_{q_0}, T_{\kappa}(X_{q_0}))$ such that $X_{q_0} = E/2$.
%\item Place $X_{q_2}$ s.t.\ $T_{\kappa}(X_{q_2}) = T_{\kappa}(X_{q_1})$ and $T'_{\kappa}(X_{q_2}) \cdot T'_{\kappa}(X_{q_1}) < 0$. [level condition 2.3].
%\item Place $X_{q_3}$ as the symmetric of $X_{q_2}$ wrt.\ the axis of symmetry. [symmetry condition].
\item Place $(X_{q_1}, T_{\kappa}(X_{q_1}))$, $(X_{q_2}, T_{\kappa}(X_{q_2}))$, ..., $(X_{q_{n+1}}, T_{\kappa}(X_{q_{n+1}}))$ by alternating between applying the level condition 2.3 wrt.\ the last point constructed, and then the symmetry condition wrt.\ to the last point constructed.
\item Finally calibrate $E$ in such a way that the $x$-coordinate of the last point constructed, $X_{q_{n+1}}$, also satisfies a level condition 2.1 or 2.2. Upon calibration, $E = \E_n$.
\end{enumerate}

\noindent \underline{\textbf{Algorithm for $n$ even}} -- system \eqref{conjecture_system_conj_intro_even} along with conditions \eqref{o11}--\eqref{o33} : 
\begin{enumerate}
\item Fix $\kappa \geq 2$ and plot the Chebyshev polynomial $T_{\kappa}(x)$.
\item Initialize energy $E$ to a certain value and draw the vertical axis of symmetry $x=E/2$.
\item Place a first point $(X_{q_0}, T_{\kappa}(X_{q_0}))$ such that $X_{q_0}$ satisfies a level condition 2.1 or 2.2, (and for $n$ even perhaps require $X_{q_0} \neq E/2$).
\item Place $X_{q_1}$ as the symmetric of $X_{q_0}$ wrt.\ the axis of symmetry.
%\item Place $X_{q_2}$ s.t.\ $T_{\kappa}(X_{q_2}) = T_{\kappa}(X_{q_1})$ and $T'_{\kappa}(X_{q_2}) \cdot T'_{\kappa}(X_{q_1}) < 0$. [level condition 2.3].
%\item Place $X_{q_3}$ as the symmetric of $X_{q_2}$ wrt.\ the axis of symmetry. [symmetry condition].
\item Place $(X_{q_2}, T_{\kappa}(X_{q_2}))$, $(X_{q_3}, T_{\kappa}(X_{q_3}))$, ..., $(X_{q_{n+1}}, T_{\kappa}(X_{q_{n+1}}))$ by alternating between applying the level condition 2.3 wrt.\ the last point constructed, and then the symmetry condition wrt.\ to the last point constructed.
\item Finally calibrate $E$ in such a way that the $x$-coordinate of the last point constructed, $X_{q_{n+1}}$, also satisfies a level condition 2.1 or 2.2. Upon calibration, $E = \E_n$.
\end{enumerate}

Note that the proposed Algorithms are dynamical constructions : the positioning of all the $X_{q_i}$'s depends on the value of $E$ (with the exception of $X_{q_0}$ in the $n$ even case). In the final step when $E$ is adjusted, all the $X_{q_i}$'s migrate (with the exception of $X_{q_0}$ that stays put in the $n$ even case). Adjusting $E$ preserves the symmetry and level conditions.

As we'll be using these systems as our only strategy to find thresholds for the rest of the article, it is useful to know that we can allow ourselves to focus only on positive energies :

\begin{Lemma}
\label{Lemma_symmetryTT}
$\mathfrak{T}_{n, \kappa} = - \mathfrak{T}_{n, \kappa}$, for all $n \in \N^*$, $\kappa \geq 2$.
\end{Lemma}

\begin{proof}
Use the fact that $T_{\kappa}$ and $U_{\kappa}$ are even polynomials if $\kappa$ is even and odd if $\kappa$ is odd. Also use \eqref{identity_derivative}. Finally, note that $\E_n, X_0, ..., X_{n+1}$ satisfy \eqref{conjecture_system_conj_intro} and \eqref{o1} -- \eqref{o3} (respectively \eqref{conjecture_system_conj_intro_even} and \eqref{o11} -- \eqref{o33}) iff $-\E_n, -X_0, ..., -X_{n+1}$ satisfy the same equations and conditions.
\qed
\end{proof}

Note the similarity of Lemma \ref{Lemma_symmetryTT} with Lemma \ref{Lemma_symmetryDelta}, but partial discrepency with Lemma \ref{Lemma_symmetryDelta_88} -- this is another issue we don't understand. It is an open problem for us to decide if $\mathfrak{T}_{2n-1, \kappa} \cup \mathfrak{T}_{2n, \kappa} = \boldsymbol{\Theta}_{n,\kappa} (\Delta[d=2])$ holds.

To close this section, we mention it is possible to express the thresholds $\E_n$ in systems \eqref{conjecture_system_conj_intro} and \eqref{conjecture_system_conj_intro_even} as solutions to a single equation (with nested square root terms). It's a matter of knowing which branch of $T_{\kappa} ^{-1} (x)$ to choose from. It is easier to explain with an example : using system \eqref{conjecture_system_conj_intro}, for $n=3$, let $X_{n+1} = X_4$ be given by \eqref{o2}. Then $\E_3 (\kappa)$ is the solution to the following equation :
\begin{align*}
\E_3 &= X_4 + X_0 = X_4+ T_{\kappa} ^{-1} T_{\kappa} (X_3) =  X_4 + T_{\kappa} ^{-1} T_{\kappa} (\E_3-X_1)   \\
&= X_4 + T_{\kappa} ^{-1} T_{\kappa} \left(\E_3- T_{\kappa} ^{-1} T_{\kappa}  (X_2) \right) =  X_4 + T_{\kappa} ^{-1} T_{\kappa} \left(\E_3- T_{\kappa} ^{-1} T_{\kappa}  \left(\E_3 / 2 \right) \right).
\end{align*}
\normalsize
The advantage of having 1 equation with 1 unknown over a system of equations with several unknowns is that it is easier (in our opinion) to solve numerically. Proposition \ref{lem1_sequence3344} is one (of many) applications of this idea.

\vspace{0.5cm}
\noindent \textit{Proof of Proposition \ref{prop1_intro}.}

From \eqref{conjecture_system_conj_intro}, $X_{(n+1)/2} = \mathcal{E}_n /2 \Rightarrow g_{j \kappa} ^{\mathcal{E}_n} ( X_{(n+1)/2} ) = 2 m(X_{(n+1)/2}) U_{j \kappa-1}( X_{(n+1)/2})$. For $q \geq 1$, let 
\begin{equation*}
\begin{aligned}
\tilde{\omega}_q &:= \frac{\omega_q}{2}  \prod _{r=0} ^{\frac{n-1}{2}} m(X_{r}) U_{\kappa-1} (X_{r}) =  (-1)^{\frac{n-1}{2}-q} \prod _{p=\frac{n+1}{2}} ^{n-q} m(X_p) U_{\kappa-1} (X_p) \prod _{p'=0} ^{q-1} m(X_{p'}) U_{\kappa-1} (X_{p'}).
\end{aligned}
\end{equation*} 
In order for what follows to apply to the cases $n=1,3$ as well, interpret $\prod _{\alpha} ^{\beta} =0$ and $\sum_{\alpha} ^{\beta} =0$ whenever $\beta < \alpha$. Multiplying \eqref{problem1_tosolve_again_101_88} throughout by $2^{-1}  \prod _{r=0} ^{(n-1)/2} m(X_{r}) U_{\kappa-1} (X_{r}) $ shows that \eqref{problem1_tosolve_again_101_88} is equivalent to
\begin{equation}
\label{problem1_tosolve_again_202}
\begin{aligned}
& m(X_{(n+1)/2}) U_{j \kappa-1}( X_{(n+1)/2}) \prod _{p'=0} ^{(n-1)/2} m(X_{p'}) U_{\kappa-1} (X_{p'}) \\
& \quad = (-1)^{(n-1)/2} m(X_0) U_{j\kappa-1} (X_0) \prod _{p=(n+1)/2} ^{n} m(X_p) U_{\kappa-1} (X_p)  \\
& \quad \quad +  \sum _{q =1} ^{(n-1)/2} \tilde{\omega}_q m(X_q) U_{j\kappa-1} (X_q) + \sum _{q =1} ^{(n-1)/2} \tilde{\omega}_q m(X_{n-q+1}) U_{j\kappa-1} (X_{n-q+1}) \\
& \quad = (-1)^{(n-1)/2} m(X_0) U_{j\kappa-1} (X_0) \prod _{p=(n+1)/2} ^{n} m(X_p) U_{\kappa-1} (X_p)  \\
& \quad \quad + m(X_{(n-1)/2}) U_{j\kappa-1} (X_{(n-1)/2}) m(X_{(n+1)/2}) U_{\kappa-1} (X_{(n+1)/2}) \prod _{p'=0} ^{(n-3)/2} m(X_{p'}) U_{\kappa-1} (X_{p'}) \\
& \quad \quad - (-1)^{(n-1)/2} m(X_0) U_{\kappa-1} (X_0) m(X_n) U_{j\kappa-1} (X_n) \prod_{p=(n+1)/2} ^{n-1} m(X_p) U_{\kappa-1} (X_p) \\
& \quad \quad +  \sum _{q =1} ^{(n-3)/2} \tilde{\omega}_q m(X_q) U_{j\kappa-1} (X_q) + \sum _{q =2} ^{(n-1)/2} \tilde{\omega}_q m(X_{n-q+1}) U_{j\kappa-1} (X_{n-q+1}).
\end{aligned}
\end{equation}
We apply Corollary \ref{corollaryEquiv}. The 2nd term on the rhs of \eqref{problem1_tosolve_again_202} equals the lone term on the lhs of \eqref{problem1_tosolve_again_202}. The 3rd term on the rhs of \eqref{problem1_tosolve_again_202} cancels the 1st term on the rhs of \eqref{problem1_tosolve_again_202}. Finally the 2 sums at the very end of the rhs of \eqref{problem1_tosolve_again_202} cancel each other ; specifically, the $q^{th}$ term in the first sum equals
$$(-1)^{(n-1)/2-q} m(X_q) U_{j\kappa-1} (X_q)  \prod _{p=(n+1)/2} ^{n-q} m(X_p) U_{\kappa-1} (X_p) \prod _{p'=0} ^{q-1} m(X_{p'}) U_{\kappa-1} (X_{p'})$$
and it cancels the $q+1^{th}$ term in the second sum which equals
$$-(-1)^{(n-1)/2-q} m(X_{n-q}) U_{j\kappa-1} (X_{n-q})  \prod _{p=(n+1)/2} ^{n-q-1} m(X_p) U_{\kappa-1} (X_p) \prod _{p'=0} ^{q} m(X_{p'}) U_{\kappa-1} (X_{p'}),$$
and this for $q=1,2,...,(n-3)/2$.

It remains to prove that $\omega_q <0$. Because we assume $X_q, X_{n-q} \in (-1,1)$, $\forall q=0,..., (n-1)/2$, $m(X_q), m(X_{n-q}) \in (0,1)$. Thus the sign of $\omega_q$ is that of 
$$(-1)^{(n-1)/2-q} \frac{U_{\kappa-1} (X_{n-q})}{U_{\kappa-1} (X_{q})} \frac{U_{\kappa-1} (X_{n-q-1})}{U_{\kappa-1} (X_{q+1})}  \frac{U_{\kappa-1} (X_{n-q-2})}{U_{\kappa-1} (X_{q+2})} \times... \times \frac{U_{\kappa-1} (X_{(n+1)/2})}{U_{\kappa-1} (X_{(n-1)/2})}.$$
By Lemma \eqref{o3}, $U_{\kappa-1} (X_{n-q-l}) / U_{\kappa-1} (X_{q+l}) < 0$ for $l=0,...,(n-1)/2-q$, and so the sign of $\omega_q$ is $(-1)^{(n-1)/2-q} \times (-1)^{(n-1)/2-q + 1} = -1$. 
\qed

\vspace{0.5cm}
\noindent \textit{Proof of Proposition \ref{prop1_intro_even}.} Very similar to that of Proposition \ref{prop1_intro} so we leave it to the reader. 
\qed

%To be clear, in dimension 2, $\mathfrak{T}_{2n,k} \cup \mathfrak{T}_{2n-1,k} \subset \boldsymbol{\Theta}_{n,\kappa} (\Delta) \subset \boldsymbol{\Theta}_{\kappa} (\Delta)$, and our strategy to find thresholds (step (2) above) is to prove that $\mathfrak{T}_{n,k} \neq \emptyset$. One point that we don't quite grasp is the fact that $\boldsymbol{\Theta}_{\kappa} (\Delta) = - \boldsymbol{\Theta}_{\kappa} (\Delta)$, for all $\kappa \geq 1$, see Lemma \ref{Lemma_symmetryDelta}, 

%and $\boldsymbol{\Theta}_{n,\kappa} (\Delta) = - \boldsymbol{\Theta}_{n,\kappa} (\Delta)$ for all $n, \kappa \geq 1$ \underline{and} $\kappa$ even, yet it is an open problem for us to decide if $\boldsymbol{\Theta}_{n,\kappa} (\Delta) = - \boldsymbol{\Theta}_{n,\kappa} (\Delta)$ holds for $\kappa \geq 3$ and odd. Also, to be clear, all thresholds we found either belong to $\boldsymbol{\theta}_{0,\kappa} (\Delta)$ or $\mathfrak{T}_{n,\kappa}$. 

\section{Graphical visualization of the threshold solutions \texorpdfstring{$\in \boldsymbol{\theta}_{0,\kappa}(\Delta)$}{TEXT} in dimension 2}
%\section{Graphical visualization of the threshold solutions in dimension 2}

\label{section_g222}
Recall the basic threshold solutions $\in \boldsymbol{\theta}_{0,\kappa}(\Delta)$ given in Lemma \ref{lemSUMcosINTRO}. In dimension 2, these can framed as solutions $\E_0$ to the following system :

\begin{equation}
\label{conjecture_system_conj_intro_0}
\begin{cases}
X_{1} = \E_0 - X_{0}, & \\
X_0, X_1 \in \{ \cos(j \pi / \kappa) : 0 \leq j \leq \kappa \}. &
\end{cases}
\end{equation}
Note that this is a simplification of system \eqref{conjecture_system_conj_intro_even} to $n=0$ ($q=-1$). A graphical interpretation of this system is given in Figure \ref{fig:test_T8Interpretation} for $\kappa=8$. The vertical dotted line is the equation $x = \E_0 /2$.

\begin{figure}[htb]
  \centering
 \includegraphics[scale=0.09]{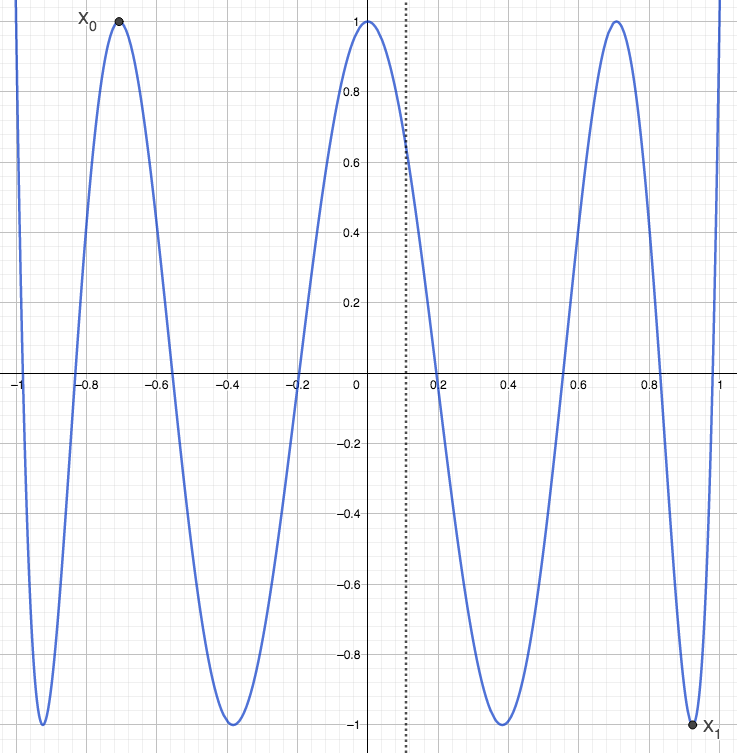}
 \includegraphics[scale=0.09]{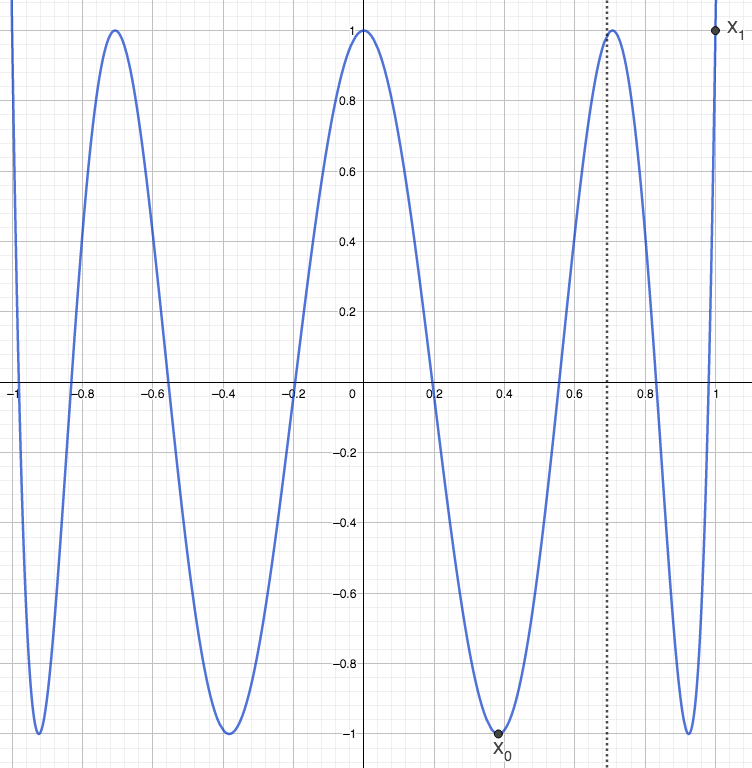}
   \includegraphics[scale=0.09]{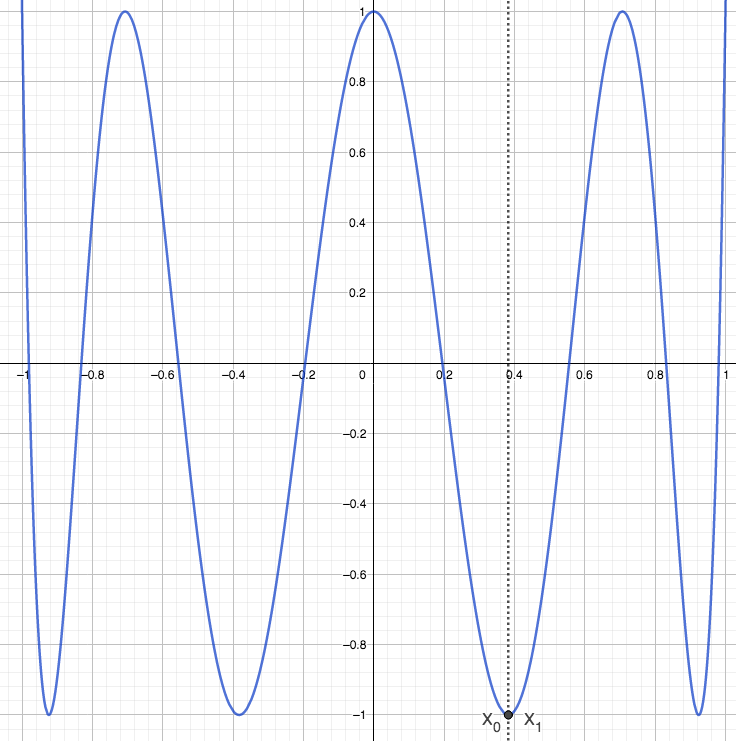}
  \includegraphics[scale=0.09]{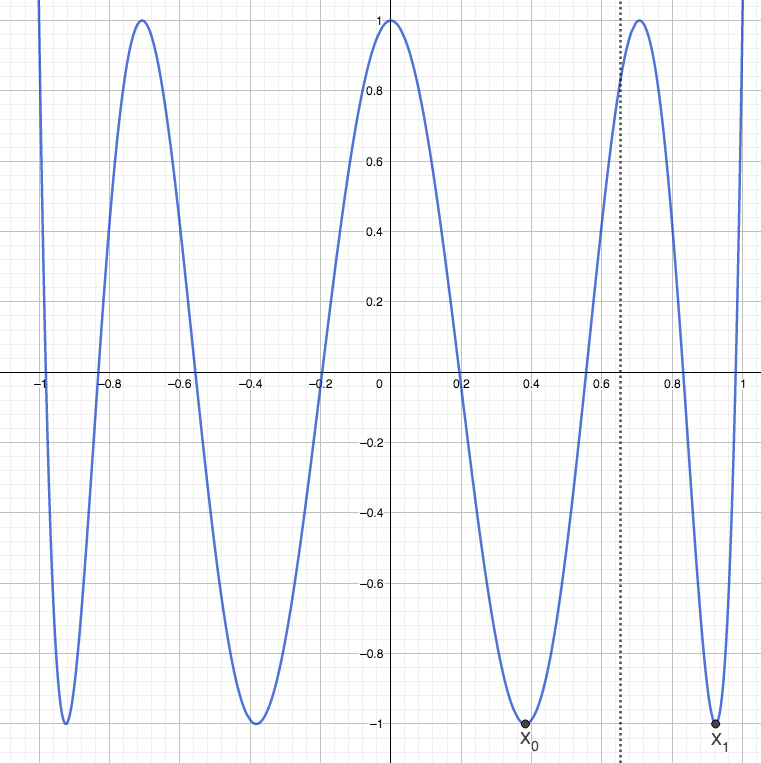}
   \includegraphics[scale=0.09]{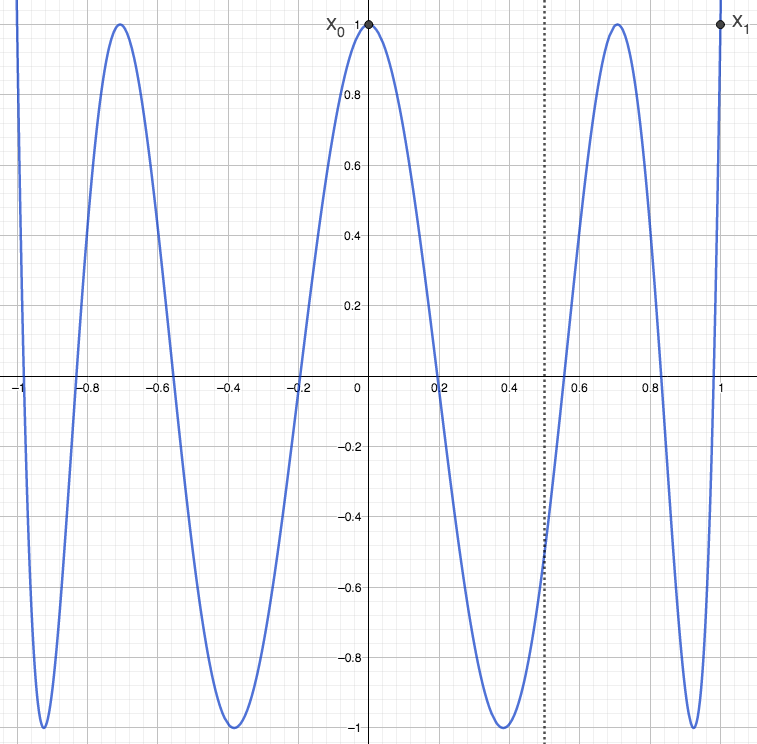}
  \includegraphics[scale=0.09]{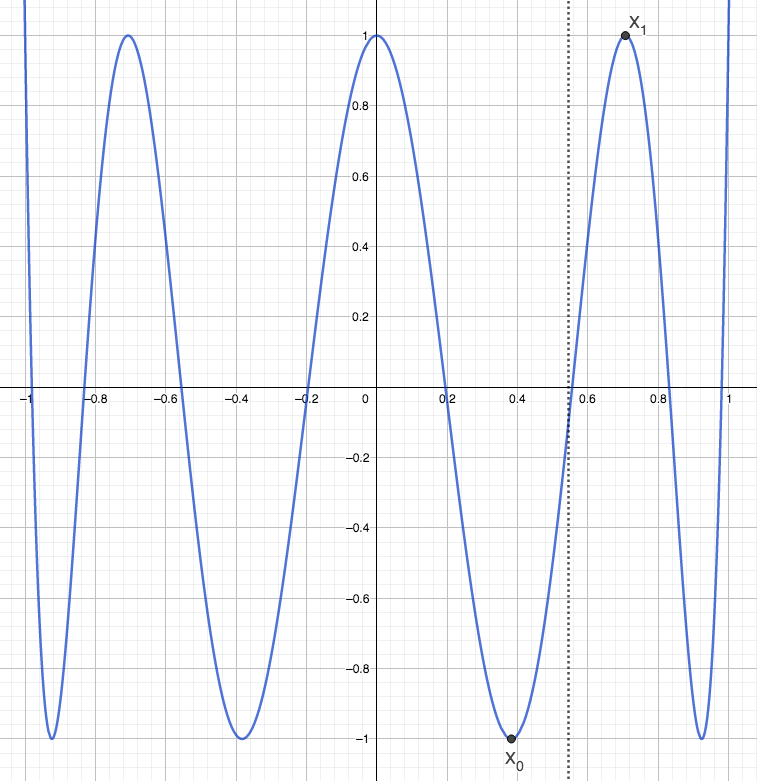}
\caption{$\kappa=8$. $T_{\kappa=8}(x)$. Illustration of solutions $\E_0$ in system \eqref{conjecture_system_conj_intro_0}. Left to right : $\E_0 = \cos(6\pi/ \kappa) + \cos(\pi / \kappa)$, $\E_0 = \cos(3\pi/ \kappa) + 1$, $\E_0 = \cos(3\pi/ \kappa) + \cos(3\pi / \kappa)$, $\E_0 = \cos(3\pi/ \kappa) + \cos(\pi/ \kappa)$, $\E_0 = \cos(4\pi/ \kappa) +1$, $\E_0 = \cos(2\pi/ \kappa) + \cos(3\pi/ \kappa)$.}
\label{fig:test_T8Interpretation}
\end{figure}

\section{A decreasing sequence of thresholds in \texorpdfstring{$J_2(\kappa) :=  (2\cos(\pi / \kappa ) , 1+ \cos(\pi / \kappa) )$}{TEXT}}
\label{section J2}

This entire section is in dimension 2. Using ideas of section \ref{geo_construction} we prove the existence of a sequence of threshold energies $\in J_2 :=  (2\cos(\pi / \kappa ) , 1+ \cos(\pi / \kappa) )$. Theorem \ref{thm_decreasing energy general} is a consequence of Propositions \ref{prop1}, \ref{prop2}, \ref{prop_interlace}, \ref{prop_convergeence}. We state these Propositions now, and prove them at the end of the section. We also prove Proposition \ref{propk2_2/n+2} and give justification for Conjecture \ref{conjecture12}.

\begin{proposition} (odd terms)
\label{prop1}
Fix $\kappa \geq 2$, and let $n \in \N^*$ odd be given. System \eqref{conjecture_system_conj_intro} admits a unique solution satisfying $\mathcal{E}_n \in J_2(\kappa)$ and
\begin{comment}
Consider real variables $\E_n, X_0, X_1, ..., X_n, X_{n+1}$. Then the system of $n+3$ equations in $n+3$ unknowns
\begin{equation}
\label{conjecture_system_conj}
\begin{cases}
T_{\kappa}(X_q) = T_{\kappa}(X_{n-q}), \quad \forall \ q = 0,1,..., \frac{n-1}{2} &  \\
X_{n-q} = \E_n -X_{1+q}, \quad \forall \ q = -1, 0,..., \frac{n-1}{2}  &  \\
X_{n+1} = 1& 
\end{cases}
\end{equation}
always admits a unique solution satisfying $\mathcal{E}_n \in J_2(\kappa)$ and
\end{comment}

\begin{equation}
\begin{aligned}
\label{chain_1}
& \mathcal{E}_n -1 = X_{0} < X_{1} < X_{2} < ... <  X_{(n-1)/2} \\
& \quad \quad \quad \quad \quad \quad \quad \quad \quad \quad \quad \quad < \cos(\pi / \kappa) < X_{(n+1)/2} < ... < X_{n} < X_{n+1} := 1.
\end{aligned}
\end{equation}

%\begin{equation}
%\label{J2_constraint}
%\mathcal{E}_n \in J_2(\kappa). %(2\cos(\pi / \kappa ) , 1+ \cos(\pi / \kappa) ).
%\end{equation}
Furthermore, this solution satisfies \eqref{o1}, \eqref{o2} and \eqref{o3}, and so $\E_n \in \mathfrak{T}_{n,\kappa}$.
\end{proposition} 

\begin{comment}
\begin{remark}
\label{sym_rem1}
By \eqref{conjecture_system_conj}, the $X_{n-q}$ and $X_{1+q}$ are located symmetrically with respect to the axis $x = \mathcal{E}_n /2$. Also, $X_{(n+1)/2} = \mathcal{E}_n /2$. The proof also reveals that $X_{(n-1)/2} < \cos(\pi / \kappa) < X_{(n+1)/2}$.
\end{remark}
\end{comment}

\begin{proposition}  (even terms)
\label{prop2}
Let $\kappa \geq 2$, and $n \in \N^*$ even be given. System \eqref{conjecture_system_conj_intro_even} admits a unique solution satisfying $\mathcal{E}_n \in J_2(\kappa)$ and 
\begin{comment}
Consider real variables $\E_n, X_0, X_1, ..., X_n, X_{n+1}$. Then the system of $n+3$ equations in $n+3$ unknowns
\begin{equation}
\label{conjecture_system_conj2}
\begin{cases}
T_{\kappa}(X_q) = T_{\kappa}(X_{n-q}), \quad \forall q = 0,1,..., \frac{n}{2}-1 & \\
X_{n-q} = \E_n -X_{1+q}, \quad  \forall q = -1, 0,..., \frac{n}{2}-1 & \\ 
X_{n/2} = \cos(\pi / \kappa), \quad X_{n+1} = 1 & 
\end{cases}
\end{equation}
always admits a unique solution satisfying $\mathcal{E}_n \in J_2(\kappa)$ and
\end{comment}
\begin{equation}
\label{chain_1even}
\mathcal{E}_n -1 = X_{0} < X_{1} < X_{2} < ... <  X_{n/2} = \cos(\pi / \kappa) < X_{n/2+1} < ... < X_{n} < X_{n+1} := 1.
\end{equation}
%\begin{equation}
%\label{J2_constraint_even}
%\mathcal{E}_n \in J_2(\kappa). %(2\cos(\pi / \kappa ) , 1+ \cos(\pi / \kappa) ).
%\end{equation}
Furthermore, this solution satisfies \eqref{o11}, \eqref{o22} and \eqref{o33}, and so $\E_n \in \mathfrak{T}_{n,\kappa}$.
\end{proposition} 

\begin{comment}
\begin{remark}
\label{sym_rem2}
By \eqref{conjecture_system_conj2}, the $X_{n-q}$ and $X_{1+q}$ are located symmetrically with respect to the axis $x = \mathcal{E}_n /2$. In particular $\cos(\pi / \kappa) = X_{n/2} < \mathcal{E}_n /2 < X_{n/2+1}$. 
\end{remark}
\end{comment}

\begin{comment}
\begin{remark}
Solutions to Propositions \ref{prop1} and \ref{prop2} are worked out for $(\kappa,n) \in \{3,4\} \times \{1,2,3,4,5,6\}$. The solutions suggest that the $\mathcal{E}_n$ may be algebraic or transcendental numbers.
\end{remark}

\begin{remark} 
%The system can be solved numerically with Python's \textit{fsolve} for instance, but the challenge is to have an adequate initial guess. 
It appears that $T_{\kappa}(X_q) = T_{\kappa}(E-X_{q+1})$ corresponds to the intersection of ellipses for $\kappa \geq 3$. Useful ?
\end{remark}
\end{comment}

Figure \ref{fig:test_T3k3firsty} illustrates the solutions in Propositions \ref{prop1} and \ref{prop2} for $1 \leq n \leq 6$ and $\kappa=3$. A few exact solutions are given in Table \ref{table with endpoints3and4}.

\begin{figure}[htb]
  \centering
 \includegraphics[scale=0.1]{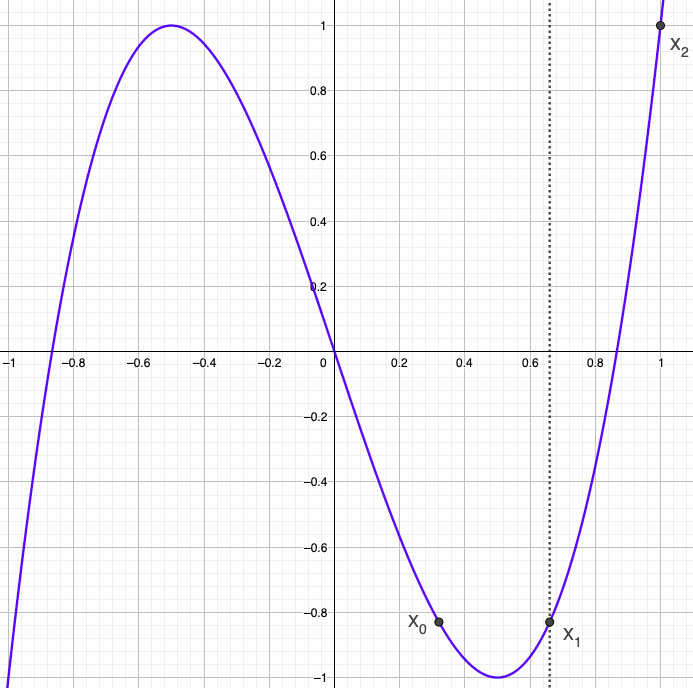}
  \includegraphics[scale=0.1]{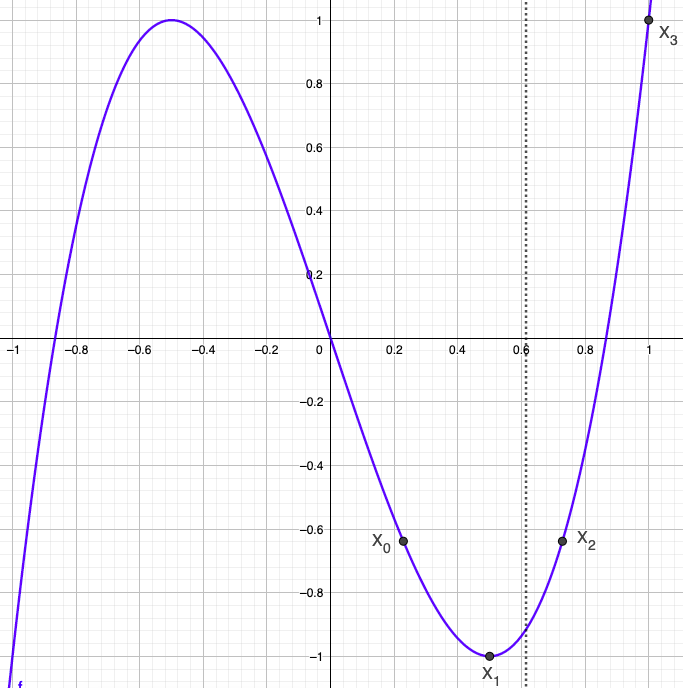}
  \includegraphics[scale=0.1]{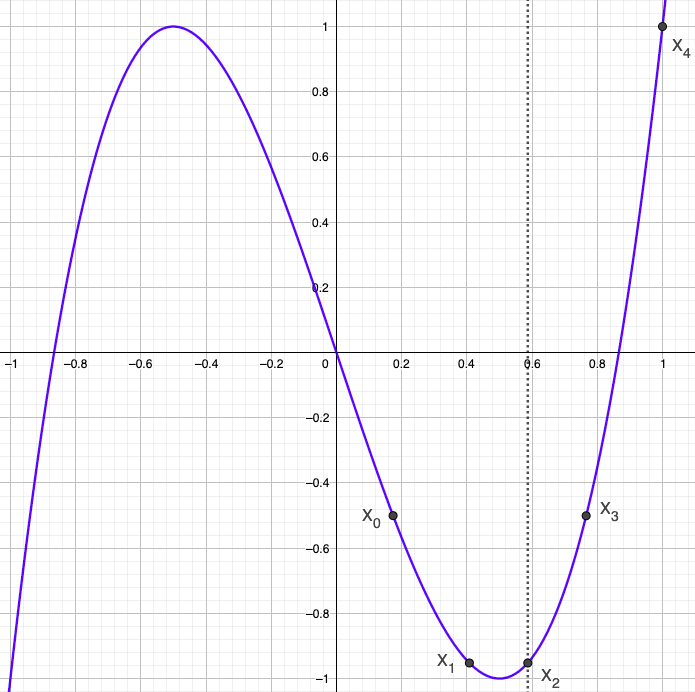}
    \includegraphics[scale=0.1]{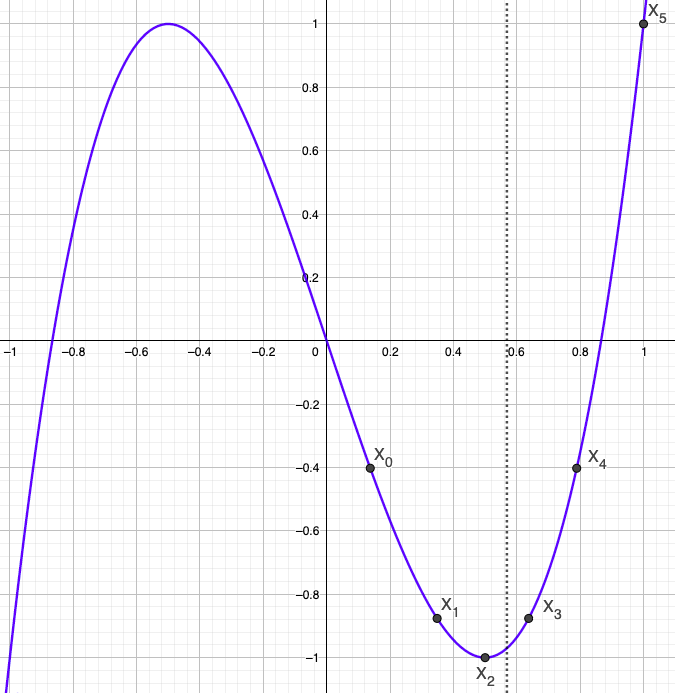}
      \includegraphics[scale=0.1]{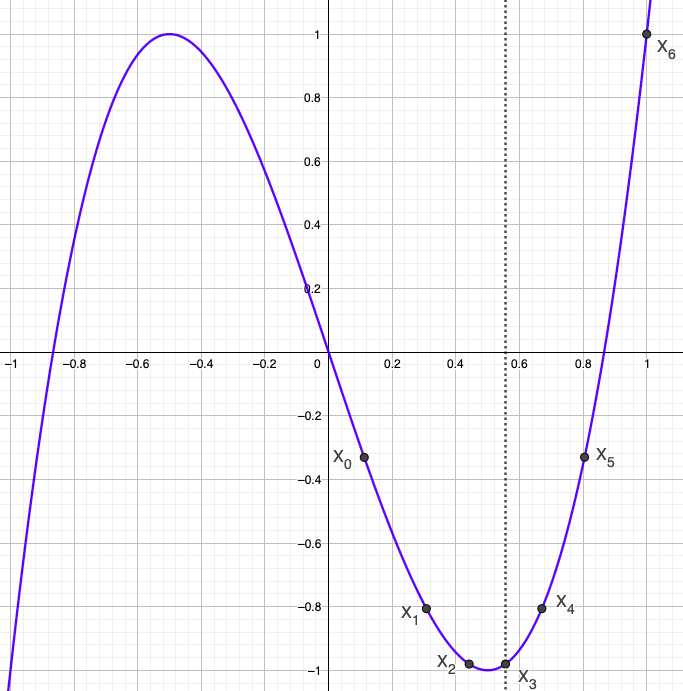}
      \includegraphics[scale=0.1]{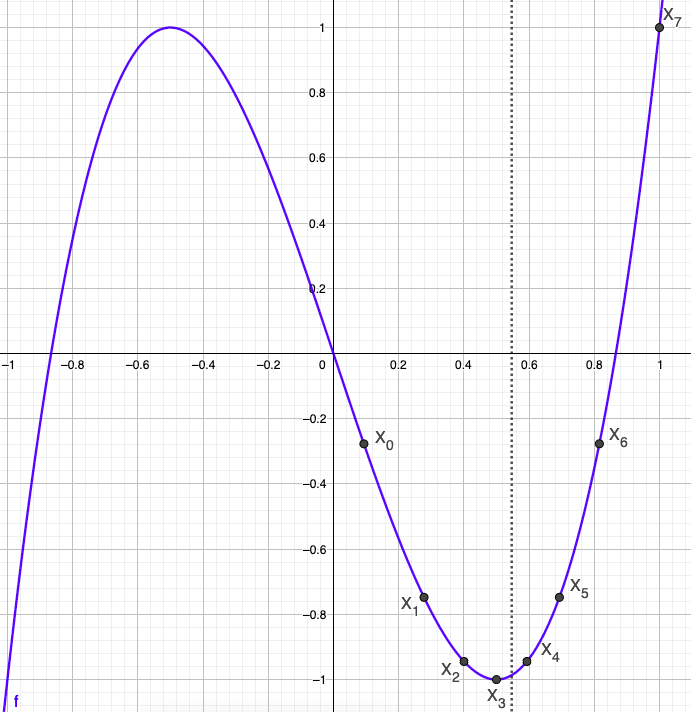}
\caption{$T_{\kappa=3}(x)$. Solutions $\E_n$ in Propositions \ref{prop1} and \ref{prop2}. Left to right : $\E_1 \simeq 1.3207$, $\E_2 \simeq 1.2287$, $\E_3 \simeq 1.1737$, $\E_4 \simeq 1.1375$, $\E_5 \simeq 1.1121$, $\E_6 \simeq 1.0934$.}
\label{fig:test_T3k3firsty}
\end{figure}

%\begin{proposition} 
%\label{k2_complete_sol}

\begin{proposition} 
\label{prop_interlace}
Fix $\kappa \geq 2$. The odd and even energy solutions $\mathcal{E}_n$ of Propositions \ref{prop1} and \ref{prop2} interlace and are a strictly decreasing sequence : $\mathcal{E}_{n+2} < \mathcal{E}_{n+1} < \mathcal{E}_{n}$, $\forall n\in \N^*$. 
\end{proposition} 

\begin{comment}
\begin{remark} To prove Proposition \ref{prop_interlace} we compare the $(X_q)_{q=0}^{n+1}$ of Propositions \ref{prop1} and \ref{prop2} for one value of $n$ to another (but for fixed $\kappa$). To remove ambiguity we occasionally use the notation 
\begin{equation}
\label{chain_111}
\mathcal{E}_n -1 =: X_{0,n} < X_{1,n} < X_{2,n} < ... < X_{n,n} < X_{n+1,n} := 1.
\end{equation}
\end{remark}
\end{comment}

\begin{proposition}
\label{prop_convergeence} 
Let $\E_n$ be the solutions in Propositions \ref{prop1} and \ref{prop2}. Then $\E_n \searrow 2\cos(\pi / \kappa)$. 
\end{proposition}

We now prove Proposition \ref{propk2_2/n+2}.

\noindent \textit{Proof of Proposition \ref{propk2_2/n+2}.}
Fix $\kappa=2$. We prove that the unique solutions of Prop.\ \ref{prop1} and \ref{prop2} are
$$\mathcal{E}_n = 2/(2+n), \quad X_q = \mathcal{E}_n - 1 + q \mathcal{E}_n, \quad q=0,...,n, \quad n \in \N^*.$$
In particular, $\E_n \searrow 0 = \inf J_2$ and Conjecture \ref{conjecture12} holds.
%\end{proposition}
%For $x \neq y$, $T_2(x) = T_2(y) \Leftrightarrow x=-y$. 

\noindent $\bullet$ Fix $n$ odd. $T_2(X_{q}) = T_2(X_{n-q}) \Leftrightarrow X_q =\pm X_{n-q} = \pm (\mathcal{E}_n-X_{q+1})$, $q=0,1,...,(n-1)/2$. Thanks to ordering \eqref{chain_1} and  $X_{(n+1)/2} = \mathcal{E}_n/2$ the appropriate sign is always $-$. Thus, recursively :
\begin{equation*}
\label{formula_En_odd}
X_0 = \mathcal{E}_n-1 = -\mathcal{E}_n + \left( - \mathcal{E}_n+ \left( ... \right) \right) = -(n+1)\mathcal{E}_n / 2 + \mathcal{E}_n/2 \Longrightarrow \mathcal{E}_n = 2/(n+2).
\end{equation*}
\noindent $\bullet$  Fix $n$ even. We also have $X_q = \pm (\mathcal{E}_n-X_{q+1})$, $q=0,1,...,n/2-1$. Thanks to the ordering \eqref{chain_1even} and the fact that $X_{n/2} = 0$ and we assume $\mathcal{E}_n \geq 0$ the appropriate sign is always $-$. Thus :
$$X_0 = \mathcal{E}_n-1 = -\mathcal{E}_n + \left( - \mathcal{E}_n + \left( ... \right) \right) = -n \mathcal{E}_n/2  \Longrightarrow \mathcal{E}_n = 2/(n+2).$$
Moreover, finite induction shows that the system also implies $X_q = \mathcal{E}_n -1 +q \mathcal{E}_n$, $q=1,2,...,n$. 
\qed

\begin{remark}
As a consequence of the proof of Proposition \ref{propk2_2/n+2}, the distance between adjacent $X_q$'s is the same for $\kappa=2$ : $|X_q - X_{q+1}| \equiv \E_n$, $0 \leq q \leq n$. This property is specific to $\kappa=2$.
\end{remark}

Let us express the solutions $\E_n$ as solutions to a single equation for $\kappa=3,4$. To do this we need to select the appropriate branches of $T_3 ^{-1}$ and $T_4 ^{-1}$. Let
\begin{equation*}
\begin{cases}
f_E : x \mapsto \frac{-(E-x) + \sqrt{3}\sqrt{1-(E-x)^2} }{2} & \text{if} \ \kappa = 3, \\
f_E : x \mapsto \sqrt{1-(E-x)^2} &  \text{if} \ \kappa = 4.
\end{cases}
\end{equation*}

\begin{comment}
\begin{cases}
\sigma = \frac{1}{2} & \text{if} \ \kappa = 3, \\
f_E : x \mapsto \frac{-(E-x) + \sqrt{3}\sqrt{1-(E-x)^2} }{2} & \text{if} \ \kappa = 3, 
\end{cases}
\quad
\begin{cases}
\sigma = \frac{1}{\sqrt{2}} &  \text{if} \ \kappa = 4, \\
f_E : x \mapsto \sqrt{1-(E-x)^2} &  \text{if} \ \kappa = 4.
\end{cases}
\end{comment}

Figure \ref{fig:test_T3k3_converge} illustrates solutions $\E_{2n}(\kappa)$ of Proposition \ref{lem1_sequence3344} for $\kappa = 3,4$. In these graphs the slope of the orange trend line is close to $-2$, and this is our rationale behind Conjecture \ref{conjecture12}. Numerically we found similar behavior for $\kappa = 5,6,8$, see \cite{GM4}. Based on these graphs the interested reader could conjecture on the constant $c(\kappa)$ in Conjecture \ref{conjecture12}.

\begin{figure}[htb]
  \centering
 \includegraphics[scale=0.3]{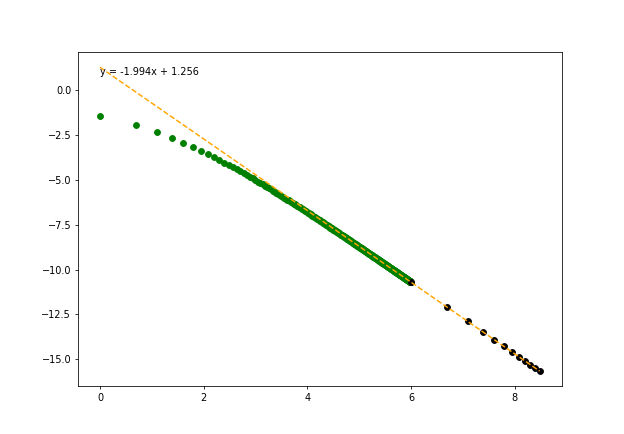}
  \includegraphics[scale=0.28]{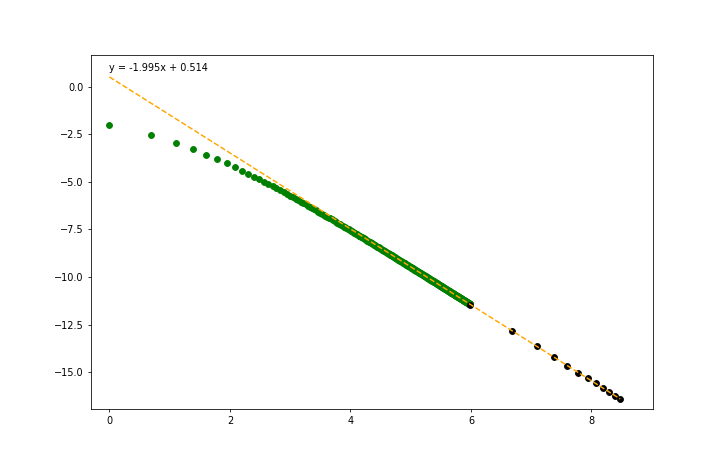}
\caption{Graphs with $\log(n)$ on $x$-axis and $\log(\E_{2n}(\kappa) - 2\cos(\pi / \kappa))$ on $y$-axis. Left : $\kappa=3$ ; Right : $\kappa=4$. Green dots are $1 \leq n \leq 400$ ; black dots are $n= 400, 800, 1200,...,4800$. Orange line is trend line based on linear regression of black dots.}
\label{fig:test_T3k3_converge}
\end{figure}

\begin{proposition}
\label{lem1_sequence3344}
Let $\kappa \in \{3,4\}$. Even terms: the $\E_{2n}$ in Theorem \ref{thm_decreasing energy general} are solution to the equation 
$\E_{2n} = 1+ f_{\E_{2n}}^{(n)}\left(\cos(\pi / \kappa) \right)$, $n \in \N^*$ ($f_{\E_{2n}}$ composed with itself $n$ times evaluated at $\cos(\pi / \kappa)$). Odd terms: the $\E_{2n-1}$ in Theorem \ref{thm_decreasing energy general} are solution to the equation $\E_{2n-1} = 1+f_{\E_{2n-1}}^{(n)}(\frac{\E_{2n-1}}{2})$. 
\end{proposition}

We now sequentially give the missing proofs of the aforementioned results in this section. We begin by a remark :

\begin{remark}
\label{r:ppcroissant}
Thanks to Lemma \ref{variationsTk}, 
\begin{enumerate}
\item Given $\cos(2\pi /\kappa)< a<b< \cos(\pi /\kappa)$, there exist unique
$\cos(\pi /\kappa) < b'< a'<1$ such that $T_\kappa(a)=T_\kappa(a')>T_\kappa(b')=T_\kappa(b)$. 
\item Given $\cos(\pi /\kappa) < b'< a'<1$, there exist unique
$\cos(2\pi /\kappa)< a<b< \cos(\pi /\kappa)$, such that $T_\kappa(a')=T_\kappa(a)>T_\kappa(b)=T_\kappa(b')$.
\end{enumerate}
Moreover, $(a,b)$ depends bi-continuously on $(a',b')$.
\end{remark}

\vspace{1cm}

\noindent \textit{Proof of Proposition \ref{prop1}.} 
We implement the dynamical algorithm for $n$ odd of section \ref{geo_construction}. Initialize energy $E$ to $E = E(\alpha) = 2\cos(\pi / \kappa)+\alpha$ with $\alpha \in \mathbf{A}_{max} := (0, 1-\cos(\pi/\kappa))$. First, by Remark \ref{obs_symmetry} we know $X_{(n+1)/2} = E/2 = \cos(\pi / \kappa) + \alpha/2 \in (\cos(\pi / \kappa), 1 )$. In particular when $\alpha \searrow 0_+$, note that $X_{(n+1)/2} \searrow \cos(\pi / \kappa)_+$. %In particular  for $\alpha$ small enough, it is closer to $\cos(\pi / \kappa)$ than $X_{0}$ is. 
Now, up to a smaller $\alpha$ still within $\mathbf{A}_{max}$, we construct inductively and continuously in $\alpha$ all of the remaining $X_q = X_q (\alpha)$, by checking all the constraints of \eqref{conjecture_system_conj_intro}, \eqref{o1} -- \eqref{o3}, but with the exception of \eqref{o2}, i.e.\ $X_{n+1} = 1 \Leftrightarrow X_0 = E-1$. %$T_{\kappa}(X_0) = T_{\kappa}(X_{n})$.

$\bullet$ $X_{(n-1)/2}$ is determined in $(\cos(2\pi / \kappa), \cos(\pi / \kappa) )$ so that $T_{\kappa}(X_{(n-1)/2}) = T_{\kappa}(X_{(n+1)/2})$. In particular, $X_{(n-1)/2}  < \cos(\pi / \kappa) < X_{(n+1)/2}$. Note that $X_{(n-1)/2} \nearrow \cos(\pi / \kappa)_-$, as $\alpha \searrow 0_+$.% \blue{here we use the second line of the system with $q=(n-1)/2$. C'est cela qui motive le fait de prendre $E = 2\cos(\pi / \kappa)_+$.}

$\bullet$ $X_{(n+3)/2}$ is the symmetric of $X_{(n-1)/2}$ wrt.\ $X_{(n+1)/2} = E/2$. So $X_{(n+1)/2} < X_{(n+3)/2}$. Up to a smaller $\alpha$ possibly, $X_{(n+3)/2} \in (\cos(\pi/\kappa),1)$. As $\alpha \searrow 0_+$, $X_{(n+3)/2} \searrow \cos(\pi / \kappa)_{+}$. 

$\bullet$ As per Remark \ref{r:ppcroissant}, $\exists ! X_{(n-3)/2} \in (\cos(2\pi / \kappa), \cos(\pi/\kappa))$ such that $ X_{(n-3)/2} <X_{(n-1)/2}$ and
$T_{\kappa}(X_{(n-1)/2}) = T_{\kappa}(X_{(n+1)/2}) < T_{\kappa}(X_{(n+3)/2})$. Again, $X_{(n-3)/2}\nearrow \cos(\pi / \kappa)_{-}$, as $\alpha \searrow 0_+$. 

$\bullet$ $X_{(n+5)/2}$ is the symmetric of $X_{(n-3)/2}$ wrt.\ $X_{(n+1)/2} = E/2$. Up to a smaller $\alpha$ possibly, $X_{(n+5)/2} \in (\cos(\pi / \kappa), 1)$. Since $X_{(n-3)/2} < X_{(n-1)/2}$, we infer $X_{(n+3)/2} < X_{(n+5)/2}$.  In particular, $T_{\kappa}(X_{(n-3)/2}) = T_{\kappa}(X_{(n+3)/2}) < T_{\kappa}(X_{(n+5)/2})$. Once more, $X_{(n+5)/2} \searrow \cos(\pi / \kappa)_{+}$, as $\alpha \searrow 0_+$. 

$\bullet$ We continue this ping pong game inductively till all of the $X_q =X_q(\alpha)$, $q=0,...,n+1$, have been defined. Note that the last step of the ping pong game was to place $X_{n+1}$ in such a way that it is the symmetric of $X_0$ wrt.\  $X_{(n+1)/2} = E/2$ (2nd line of \eqref{conjecture_system_conj_intro}). Now we consider the set $\mathbf{A}_n$ ($\mathbf{A}$ depends on $n$) of all the positive $\alpha$'s that allow a construction verifying :
\begin{equation}
\label{chain_reveal}
E -1 \leq X_{0} < X_{1} < X_{2} < ... <  X_{(n-1)/2} < \cos(\pi / \kappa) <X_{(n+1)/2} < ... < X_{n} < X_{n+1} \leq 1.
\end{equation}
$\mathbf{A}_n \subset \mathbf{A}_{max}$ since if $\alpha \geq 1-\cos(\pi / \kappa)$, $X_{n+1} = E - X_0 \geq 1 + \cos(\pi / \kappa) - X_0 \geq 1$. This observation will imply that $\E_n \in J_2 (\kappa)$ when the proof is over. As a side note, it is not hard to see that $\mathbf{A}_{n+2} \subset \mathbf{A}_{n}$ ; later in this section we prove $\cap_{n \in \N, \ n \ \text{odd}} \mathbf{A}_{n} = \emptyset$. It remains to argue that there is a unique $\alpha^* \in \mathbf{A}_n$ such that $X_{n+1}(\alpha^*) = 1 \Leftrightarrow X_0(\alpha^*) = E(\alpha^*)-1$. First, note that by construction, the chain of strict inequalities in \eqref{chain_reveal} remains valid as $\alpha$ increases in $\mathbf{A}_n$. Second, recall that $X_{n+1} \geq 1$ as $\alpha \geq 1 - \cos(\pi/ \kappa)$. Moreover, $X_{n+1}$ is strictly increasing for $\alpha \in \mathbf{A}_{n}$. Thirdly, and finally, note that by construction, as $\alpha$ increases in $\mathbf{A}_n$, $X_{n+1}$ must reach $1$ before $X_0$ reaches $\cos(2\pi / \kappa)$. This is because $T_{\kappa}(X_0) = T_{\kappa}(X_n) < T_{\kappa}(X_{n+1})$. Another way to see this is to argue by contradiction. If $X_0$ were to reach $\cos(2\pi / \kappa)$ before $X_{n+1}$ reaches $1$, then 
$$2\cos(\pi / \kappa) - \cos(2\pi / \kappa) \leq E- \cos(2\pi / \kappa) = E-X_0 = X_{n+1} < 1 \Rightarrow 2\cos(\pi / \kappa) - \cos(2\pi / \kappa) < 1,$$
which is a false statement. Thus, $\exists ! \alpha^*$ s.t.\ $X_{n+1}(\alpha^*) = 1$. The energy solution $\E_n$ is $E(\alpha^*)$. 
\qed

\vspace{1cm}

%\begin{proposition} (Ping-pong Lemma 2)
%\label{conj1holdsn=n even_1}
%Proposition \ref{prop2} holds for all $n$ even, and all $\kappa \geq 2$.
%\end{proposition}
%\begin{proof}

\noindent \textit{Proof of Proposition \ref{prop2}.} 

We implement the dynamical algorithm for $n$ even of section \ref{geo_construction}, and mimick the proof of Proposition \ref{prop1}. The main difference is that this time $X_{n/2} := \cos(\pi / \kappa)$. 
It implies that the values $X_0, X_1$, ..., $X_{n/2-1}$ will belong to $(\cos(2\pi /\kappa), \cos(\pi /\kappa))$, whereas the values $X_{n/2+1}, X_{n/2+2}$, ..., $X_{n}$ will belong to $(\cos(\pi /\kappa), 1)$. $X_{n+1}$ will be placed ultimately so that it equals 1.%, by definition.

%First $X_0: = E-1 = 2\cos(\pi / \kappa)+\alpha -1  \in  (\cos(2\pi / \kappa), \cos(\pi / \kappa) )$. 
Initialize energy $E$ to $E = E(\alpha) = 2\cos(\pi / \kappa)+\alpha$ with $\alpha \in \mathbf{A}_{max} := (0, 1-\cos(\pi/\kappa))$. First, $E/2 = \cos(\pi / \kappa) + \alpha/2 \in (\cos(\pi / \kappa), 1 )$. 
Now, up to a smaller $\alpha$ still within $\mathbf{A}_{max}$, we construct inductively and continuously in $\alpha$ all of the remaining $X_q = X_q (\alpha)$, by checking all the constraints of \eqref{conjecture_system_conj_intro_even}, \eqref{o11} -- \eqref{o33}, but with the exception of the $X_{n+1}$ condition in \eqref{o22}, i.e.\ $X_{n+1} = 1 \Leftrightarrow X_0 = E-1$.
%In particular when $\alpha\to 0^+$, note that $X_{(n+1)/2}$ tends to $\cos(\pi / \kappa)_+$. %In particular  for $\alpha$ small enough, it is closer to $\cos(\pi / \kappa)$ than $X_{0}$ is. 
%Now, up to a smaller $\alpha$, we construct inductively $X_n$ and continuously in $\alpha$, by checking all the constraints of \eqref{conjecture_system_conj_intro_even} but $T_{\kappa}(X_0) = T_{\kappa}(X_{n})$.

%Up to a smaller $\alpha$, $X_{n/2+1} \in (\cos(\pi/\kappa), 1)$.
$\bullet$ $X_{n/2+1}$ is the symmetric of $X_{n/2}$ wrt.\ $E/2$. So $X_{n/2} < E/2 < X_{n/2+1} = \cos(\pi / \kappa) + \alpha$. As per Remark \ref{r:ppcroissant}, $X_{n/2-1}$ is constructed in $(\cos(2\pi / \kappa), \cos(\pi / \kappa) )$ so that $T_\kappa( X_{n/2-1}) = T_\kappa (X_{n/2+1})$. We turn to $X_{n/2+2}$ which is is the symmetric of $X_{n/2-1}$ wrt.\ $E/2$. Up to a smaller $\alpha$ possibly, $X_{n/2+2} \in (\cos(\pi/\kappa), 1)$. 
As per Remark \ref{r:ppcroissant}, there is a unique $X_{n/2-2} \in (\cos(2\pi / \kappa), \cos(\pi/\kappa))$ such that $ X_{n/2-2} <X_{n/2-1}$ and
$T_{\kappa}(X_{n/2-2}) = T_{\kappa}(X_{n/2+2}) > T_{\kappa}(X_{n/2+1})$.

$\bullet$ We continue this ping pong game inductively till all of the $X_q =X_q(\alpha)$, $q=0,...,n+1$, have been defined. Note that the last step of the ping pong game was to place $X_{n+1}$ in such a way that it is the symmetric of $X_0$ wrt.\  $E/2$ (2nd line of \eqref{conjecture_system_conj_intro_even}). Now we consider the set $\mathbf{A}_n$ ($\mathbf{A}$ depends on $n$) of all the positive $\alpha$'s that allow a construction verifying :

\begin{equation}
\label{chain_1_reveal}
E -1 \leq X_{0} < X_{1} < X_{2} < ... <  X_{n/2} = \cos(\pi / \kappa) < X_{n/2+1} < ... < X_{n} < X_{n+1} \leq 1.
\end{equation} 
$\mathbf{A}_n \subset \mathbf{A}_{max}$ since if $\alpha \geq 1-\cos(\pi / \kappa)$, $X_{n+1} = E - X_0 \geq 1 + \cos(\pi / \kappa) - X_0 \geq 1$. As a side note, it is not hard to see that $\mathbf{A}_{n+2} \subset \mathbf{A}_{n}$ ; later in this section we prove $\cap_{n \in \N, \ n \ \text{even}} \mathbf{A}_{n} = \emptyset$. It remains to argue that there is a unique $\alpha^* \in \mathbf{A}_n$ such that $X_{n+1}(\alpha^*) = 1 \Leftrightarrow X_0(\alpha^*) = E(\alpha^*)-1$. First, note that by construction, the chain of strict inequalities in \eqref{chain_1_reveal} remains valid as $\alpha$ increases in $\mathbf{A}_n$. Second, recall that $X_{n+1} \geq 1$ as $\alpha \geq 1 - \cos(\pi/ \kappa)$. Moreover, $X_{n+1}$ is strictly increasing for $\alpha \in \mathbf{A}_{n}$. Thirdly, and finally, note that by construction, as $\alpha$ increases in $\mathbf{A}_n$, $X_{n+1}$ must reach $1$ before $X_0$ reaches $\cos(2\pi / \kappa)$ (see the previous proof for the argument). Thus, $\exists ! \alpha^*$ s.t.\ $X_{n+1}(\alpha^*) = 1$. The energy solution $\E_n$ is $E(\alpha^*)$.
%Another way to see this is to argue by contradiction. If $X_0$ were to reach $\cos(2\pi / \kappa)$ before $X_{n+1}$ reaches $1$, then 
%$$2\cos(\pi / \kappa) - \cos(2\pi / \kappa) \leq E- \cos(2\pi / \kappa) = E-X_0 = X_{n+1} < 1 \Rightarrow 2\cos(\pi / \kappa) - \cos(2\pi / \kappa) < 1,$$
%which is a false statement.  
\qed

\vspace{1cm}

%Next, we repeat the construction of the proof of Proposition \ref{prop2}, with $n+1$ instead of $n$ and with $E:=\mathcal{E}_n$ by checking all the constraints of \eqref{conjecture_system_conj_intro_even} but $T_{\kappa}(X_{0, n+1}) = T_{\kappa}(X_{n+1, n+1})$. As it is not the $(X_{i, {n+1}})$ obtained as result of Proposition \ref{prop2},
%we stress its dependency in $E$ by writing $(X_{i, {n+1}}(E))$. 
\noindent \textit{Proof of Proposition \ref{prop_interlace}.} 

Fix $n$ odd. So $n+1$ is even. Fix $E:= \min(\mathcal{E}_n, \E_{n+1})$ (we suppose at this point that we don't know which of the 2 energies is smaller) with $\E_n$ and $\E_{n+1}$ determined as in the proofs of Propositions \ref{prop1} and \ref{prop2} respectively. The construction gives $(X_{i,n} (E) )_{i=0} ^{n+1}$ satisfying \eqref{chain_reveal} and $(X_{i,n+1} (E) )_{i=0} ^{n+2}$ satisfying \eqref{chain_1_reveal}. By the choice of $E$ we either have $X_{n+1,n} (E) = 1$ or $X_{n+2,n+1} (E) = 1$. This is to be determined. Starting from the bottom of the well we see that :
\[X_{(n-1)/2,n} (E) <X_{(n+1)/2,n+1}(E) = \cos(\pi / \kappa) < E/2 = X_{(n+1)/2,n} (E) < X_{(n+1)/2+1,n+1}(E).\]
By the ping pong game that ensues, and using Remark \ref{r:ppcroissant}, we inductively infer 
\[X_{(n+1)/2+q,n} (E) < X_{(n+1)/2+q+1,n+1}(E), \quad  \mbox{ for } q=0,1,...,(n+1)/2.\]
So $X_{n+1,n} (E) < X_{n+2,n+1}(E)$. It must be therefore that $X_{n+2,n+1}(E) = 1$ and so $E = \E_{n+1} \leq \E_n$. Furthermore, $X_{n+1,n} (E) < X_{n+2,n+1}(E)$ implies $\mathcal{E}_{n+1} < \mathcal{E}_{n}$.
%So in particular,
%\[T_\kappa(X_{0,n+1}(E))= T_\kappa(E_n-1)= T_\kappa(X_{0,n})= T_{\kappa}( X_{n,n}) < T_{\kappa}( X_{n+1,n+1}(E)).\] 
%Recalling that $(X_{i, n+1})$ satisfies $T_{\kappa}(X_{0, n+1}) = T_{\kappa}(X_{n+1, n+1})$, by monotonicity of $T_\kappa$ with infer that $\mathcal{E}_{n+1} < \mathcal{E}_{n}$.

Fix $n$ even. So $n+1$ is odd. We proceed with the same setup as before. Fix $E:= \min(\mathcal{E}_n, \E_{n+1})$ with $\E_n$ and $\E_{n+1}$ determined as in the proofs of Propositions \ref{prop2} and \ref{prop1} respectively. The construction gives $(X_{i,n} (E) )_{i=0} ^{n+1}$ satisfying \eqref{chain_1_reveal} and $(X_{i,n+1} (E) )_{i=0} ^{n+2}$ satisfying \eqref{chain_reveal}. By the choice of $E$ we either have $X_{n+1,n} (E) = 1$ or $X_{n+2,n+1} (E) = 1$. This is to be determined. Starting from the bottom of the well we see that :
\[X_{n/2,n+1}(E) < X_{n/2,n} (E) = \cos(\pi / \kappa) < E/2 = X_{n/2+1,n+1}(E) < X_{n/2+1,n} (E) \]
By the ping pong game that ensues, and using Remark \ref{r:ppcroissant}, we inductively infer 
\[X_{n/2+q,n} (E) < X_{n/2+q+1,n+1}(E), \quad  \mbox{ for } q=0,1,...,n/2+1.\]
So $X_{n+1,n} (E) < X_{n+2,n+1}(E)$. It must be therefore that $X_{n+2,n+1}(E) = 1$ and so $E = \E_{n+1} \leq \E_n$. Furthermore, $X_{n+1,n} (E) < X_{n+2,n+1}(E)$ implies $\mathcal{E}_{n+1} < \mathcal{E}_{n}$.
\qed

Finally, to prove Proposition \ref{prop_convergeence}, we'll start with a Lemma which characterizes a geometric property of the graph of $T_{\kappa}$ :

\begin{Lemma}\label{l:croissdiff}
Let $\kappa\geq 2$. If $\cos(2\pi/\kappa) < a<\cos(\pi/\kappa)<b < 1$ are such that $T_\kappa(a) = T_\kappa(b)$, then 
\begin{equation}
\label{distanceLeftGreaterthanRight}
\cos(\pi/\kappa)-a>b-\cos(\pi/\kappa).
\end{equation}
\end{Lemma}
\begin{proof}
To prove this Lemma we analyze the function $S_\kappa(x):=T_\kappa(x)- T_\kappa(2\cos(\pi/\kappa)-x)$. If we can prove that $S_{\kappa} (x) >0$ on $(\cos(\pi/\kappa), 1)$, then $T_{\kappa}(a) = T_{\kappa}(b) >  T_{\kappa}(2\cos(\pi/\kappa) - b)$ for $b \in (\cos(\pi/\kappa), 1)$ and $a \in (\cos(2\pi / \kappa), \cos(\pi/\kappa))$. This implies $a < 2\cos(\pi/\kappa) - b$, i.e.\ \eqref{distanceLeftGreaterthanRight} holds. Clearly $S_\kappa ( \cos(\pi / \kappa)) =0$. Thus, to prove that  $S_{\kappa} (x) >0$ on $(\cos(\pi/\kappa), 1)$ it is enough to prove that $S_\kappa$ is strictly increasing on $(\cos(\pi/\kappa), 1)$. First note that 
\begin{align}\label{e:intervalle}
\cos(2\pi/\kappa)< 2 \cos(\pi/\kappa)-x< \cos(\pi/\kappa), \quad \mbox{ for } x\in (\cos(\pi/\kappa), 1).
\end{align}
We aim to show that $S_\kappa'(x)>0$ for $x \in (\cos(\pi/\kappa), 1)$. Fix $x\in (\cos(\pi/\kappa), 1)$.
\begin{align*}
S_\kappa'(x)&= \kappa(U_{\kappa-1}(x)+ U_{\kappa-1}(2 \cos(\pi/\kappa)-x))
\\
&= \kappa 2^{\kappa-1} \left( \prod_{j=1}^{\kappa-1}(x-\cos(j\pi/\kappa))+  \prod_{j=1}^{\kappa-1}(2\cos(\pi/\kappa)-x-\cos(j\pi/\kappa))\right)
\\
&= \kappa 2^{\kappa-1} (x-\cos(\pi/\kappa)) \left( \prod_{j=2}^{\kappa-1}(x-\cos(j\pi/\kappa))-  \prod_{j=2}^{\kappa-1}(2\cos(\pi/\kappa)-x-\cos(j\pi/\kappa))\right)
\end{align*}
Since cosine is decreasing on $(0, \pi)$ and since \eqref{e:intervalle} holds, note that $x-\cos(j\pi/\kappa) \geq x-\cos(\pi/\kappa)>0$ and $2\cos(\pi/\kappa)-x-\cos(j\pi/\kappa)>0$ for $j\in\{2, \ldots \kappa-1\}$. Next, for $j\in\{2, \ldots \kappa-1\}$, using again \eqref{e:intervalle}, note that
$$0<2\cos(\pi/\kappa)-x-\cos(j\pi/\kappa) < x-\cos(j\pi/\kappa).$$
Therefore, $S_\kappa'$ is positive on $(\cos(\pi/\kappa), 1)$. %Finally, recalling that $S_\kappa(\cos(\pi/\kappa))=0$, we obtain that $S_\kappa$ is positive on $(\cos(\pi/\kappa), 1)$. 
\end{proof}
\qed

\begin{remark}
$S_\kappa$ in Lemma \ref{l:croissdiff} is stricly convex. 
\end{remark}

%\begin{proposition}\label{p:EnConv}
%Given $\kappa\geq 2$, we have 
%\[\E_n \searrow \inf J_2 = 2 \cos(\pi / \kappa), \quad \mbox{ as } n\to\infty.\]
%\end{proposition}

\noindent \textit{Proof of Proposition \ref{prop_convergeence}.}  

By Proposition \ref{prop_interlace} and since $\E_n>2 \cos(\pi / \kappa)$, $\exists \ell$ such that $\E_n \searrow \ell \geq 2 \cos(\pi / \kappa)$. It is enough to show that
$\E_{2n+1} \to 2 \cos(\pi / \kappa)$, $n \in \N^*$. Therefore, we suppose that $n$ is odd. We proceed by contradiction. Suppose $2\epsilon:= \ell- 2 \cos(\pi / \kappa)>0$. Recall $X_{(n+1)/2,n} (\E_n) = \E_n /2$. Then for all $n \geq 1$ and odd, $X_{(n+1)/2,n} (\E_n) - \cos(\pi/\kappa) > \epsilon$. Choose $n$ odd large enough so that $n \epsilon > 1$. 

By Lemma \ref{l:croissdiff},  $\cos(\pi/\kappa) - X_{(n-1)/2,n} (\mathcal{E}_n) > X_{(n+1)/2,n} (\mathcal{E}_n) - \cos(\pi/\kappa) > \epsilon$. Next, 
since $X_{(n-1)/2,n} (\E_n)$ and $X_{(n+3)/2,n} (\E_n)$ are symmetric wrt.\ the axis $x=\mathcal{E}_n /2$, it must be that $X_{(n+3)/2,n} (\E_n) - \cos(\pi/\kappa) > 3 \epsilon$. Again, apply Lemma \ref{l:croissdiff} to get $\cos(\pi/\kappa) - X_{(n-3)/2,n} (\mathcal{E}_n) > 3 \epsilon$. Continuing in this way, we end up with 
$X_{(n+q)/2,n} (\E_n) - \cos(\pi/\kappa) > q \epsilon$ for $q=1,3,5,...,n$. But $X_{n,n} (\mathcal{E}_n) > n\epsilon > 1$ is absurd. We conclude that $\ell=2 \cos(\pi / \kappa)$. \qed

\section{An increasing sequence of thresholds below $J_3(\kappa) :=\left( 1+\cos(2 \pi / \kappa), 2\cos(\pi / \kappa) \right)$}
\label{section J3a}

This entire section is in dimension 2. We prove the existence of a sequence of threshold energies $\F_n = \F_n (\kappa) \nearrow \inf J_3 (\kappa)$. This section proves Theorem \ref{thm_decreasing energy general_k3} for $\{ \F_n \}$.

\begin{proposition} (odd terms)
\label{prop_k=3_odd_888}
Fix $\kappa \geq 3$, and $n \in \N^*$ odd. System \eqref{conjecture_system_conj_intro} has a unique solution (which we denote $\F_n$ instead of $\E_n$) such that $\F_n \in (\cos(\pi / \kappa) + \cos(2\pi / \kappa), 1+ \cos(2\pi / \kappa))$ and
\begin{equation}
\label{chain_1_k3_888}
\cos(2\pi / \kappa) =: X_{n+1} < X_{n}  < ... <  X_{(n+1)/2} = \F_n / 2  < X_{(n-1)/2} < ...  < X_0 < 1,
\end{equation}
%\begin{equation}
%\label{J2_constraint}
%\mathcal{E}_n \in J_2(\kappa). %(2\cos(\pi / \kappa ) , 1+ \cos(\pi / \kappa) ).
%\end{equation}
Furthermore, this solution satisfies \eqref{o1}, \eqref{o2} and \eqref{o3}, and so $\F_n \in \mathfrak{T}_{n,\kappa}$.
\end{proposition} 
\begin{comment}
Fix $\kappa \geq 3$, and let $n \in \N^*$, $n$ odd be given. Consider real variables $\F_n, X_0, X_1, ..., X_n, X_{n+1}$. Then the system of $n+3$ equations in $n+3$ unknowns
\begin{equation}
\label{conjecture_system_k3_888}
\begin{cases}
T_{\kappa}(X_q) = T_{\kappa}(X_{n-q}), & \forall \ q = 0,1,..., (n-1)/2, \\
X_{n-q} = \F_n -X_{1+q}, & \forall \ q = -1,0,1,..., (n-1)/2, \\
X_{n+1} = \cos(2\pi / \kappa)  &
\end{cases}
\end{equation}
admits a unique solution satisfying the constraints :
\begin{equation}
\label{chain_1_k3_888}
\cos(2\pi / \kappa) =: X_{n+1} < X_{n}  < ... <  X_{(n+1)/2} = \F_n / 2  < X_{(n-1)/2} < ...  < X_0 < 1,
\end{equation}
\begin{equation}
\label{J2_constraint_k3_888}
\mathcal{F}_n \in (\cos(\pi / \kappa) + \cos(2\pi / \kappa), 1+ \cos(2\pi / \kappa)).
\end{equation}
\end{comment}

\begin{proposition} (even terms)
\label{prop_k=3_even_888}
Fix $\kappa \geq 3$, and $n \in \N^*$ even. System \eqref{conjecture_system_conj_intro_even} has a unique solution (which we denote $\F_n$ instead of $\E_n$) such that $\F_n \in (\cos(\pi / \kappa) + \cos(2\pi / \kappa), 1+ \cos(2\pi / \kappa))$ and
\begin{equation}
\label{chain_1_k3_888even}
\cos(2\pi / \kappa) =: X_{n+1} < X_{n}  < ... < X_{n/2} = \cos(\pi / \kappa) < X_{n/2-1} < ... < X_0 < 1.
\end{equation}
%\begin{equation}
%\label{J2_constraint}
%\mathcal{E}_n \in J_2(\kappa). %(2\cos(\pi / \kappa ) , 1+ \cos(\pi / \kappa) ).
%\end{equation}
Furthermore, this solution satisfies \eqref{o11}, \eqref{o22} and \eqref{o33}, and so $\F_n \in \mathfrak{T}_{n,\kappa}$.
\end{proposition}

\begin{comment}
Fix $\kappa \geq 3$, and let $n \in \N^*$, $n$ even be given. Consider real variables $\F_n, X_0, X_1, ..., X_n, X_{n+1}$. Then the system of $n+3$ equations in $n+3$ unknowns
\begin{equation}
\label{conjecture_system_k3_888even}
\begin{cases}
T_{\kappa}(X_q) = T_{\kappa}(X_{n-q}), \quad \forall \ q = 0,1,..., n/2-1,  & \\
X_{n-q} = \F_n -X_{1+q}, \quad \forall \ q = -1,0,1,..., n/2-1, & \\
X_{n/2} = \cos(\pi / \kappa), \quad X_{n+1} = \cos(2\pi / \kappa) &
\end{cases}
\end{equation}
admits a unique solution satisfying the constraints :
\begin{equation}
\label{chain_1_k3_888even}
\cos(2\pi / \kappa) =: X_{n+1} < X_{n}  < ... < X_{n/2} = \cos(\pi / \kappa) < X_{n/2-1} < ... < X_0 < 1,
\end{equation}
\begin{equation}
\label{J2_constraint_k3_888even}
\mathcal{F}_n \in (\cos(\pi / \kappa) + \cos(2\pi / \kappa), 1+ \cos(2\pi / \kappa)).
\end{equation}
\end{comment}

Figure \ref{fig:T3, increasing to 0.5} illustrates the solutions in Propositions \ref{prop_k=3_odd_888} and \ref{prop_k=3_even_888} for $1 \leq n \leq 6$ and $\kappa=3$. Exact solutions for $\F_1$, $\F_2$ and $\F_3$ are :

$$\F_1 = 2/7, \quad \F_2 = 1/\sqrt{6}, \quad \F_3 = \frac{16}{247} + \frac{2\sqrt{2479}}{247}  \cos \left(\frac{1}{3} \arctan \left( \frac{741 \sqrt{2190}}{118457} \right) \right. $$

\begin{figure}[H]
  \centering
 \includegraphics[scale=0.118]{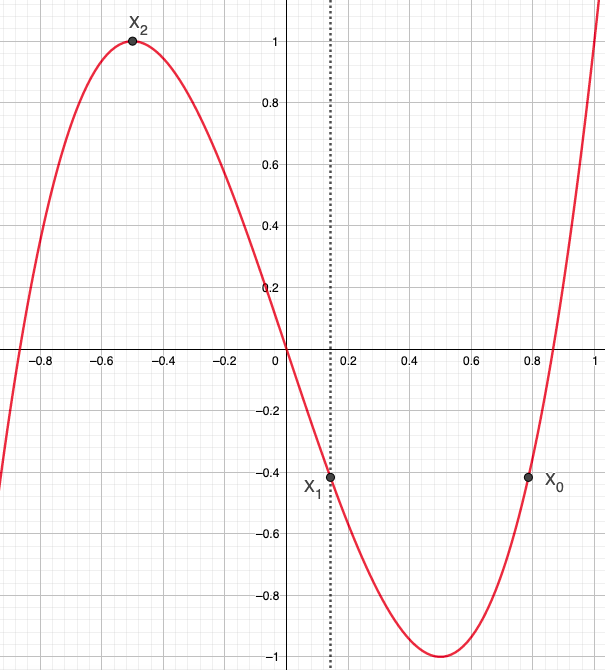}
 \includegraphics[scale=0.115]{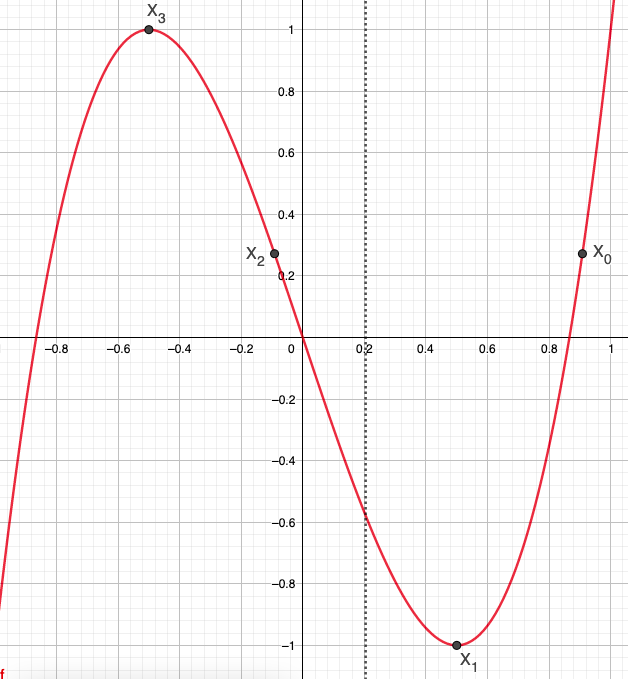}
  \includegraphics[scale=0.13]{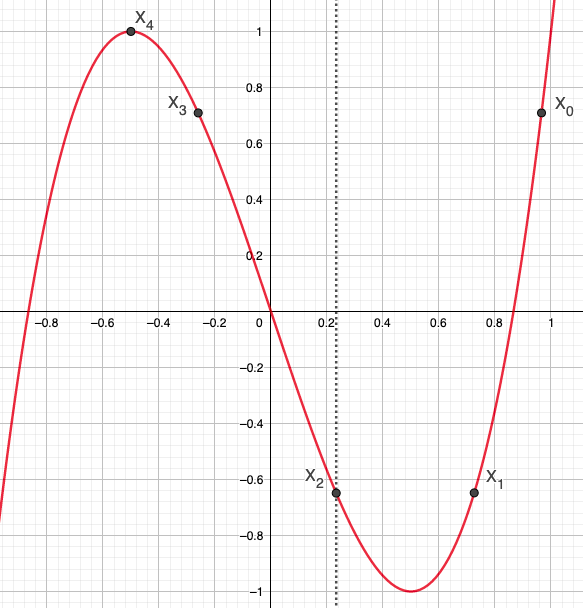}
  \includegraphics[scale=0.118]{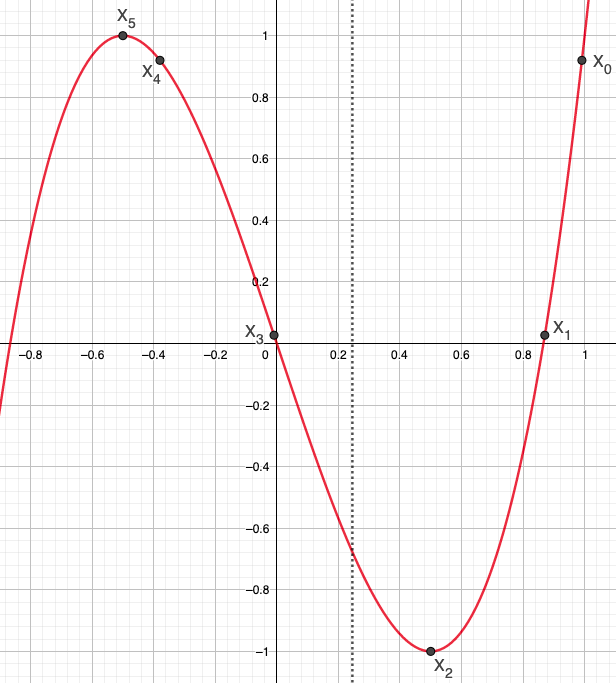}
  \includegraphics[scale=0.13]{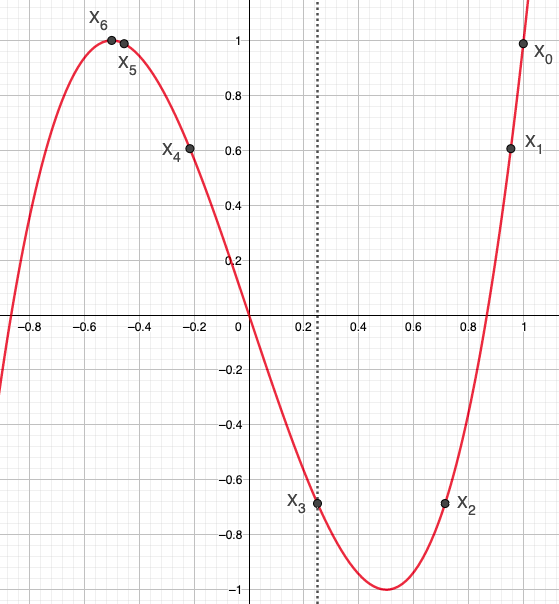}
  \includegraphics[scale=0.11]{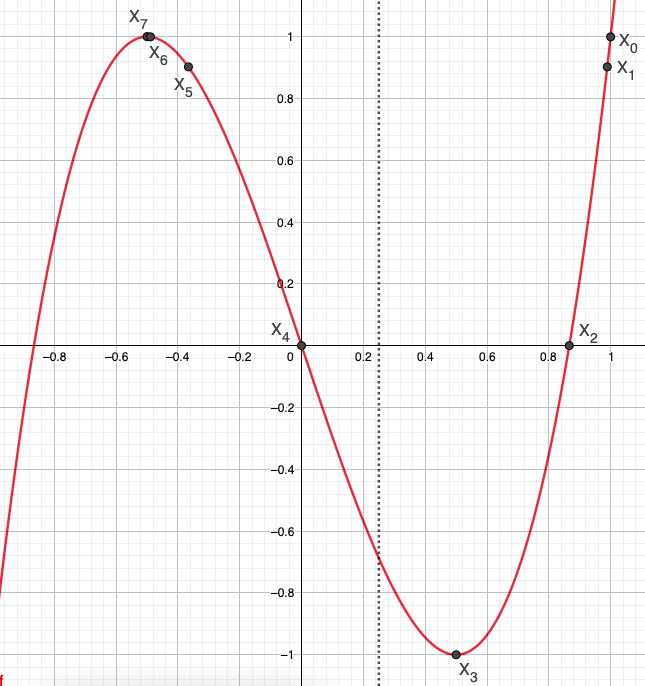}
\caption{$\kappa=3$. $T_{\kappa=3}(x)$. $\F_n$ of Propositions \ref{prop_k=3_odd_888} and \ref{prop_k=3_even_888}. Left to right : $\F_1 \simeq 0.2857$ ; $\F_2 \simeq 0.40824$ ; $\F_3 \simeq 0.46617$ ; $\F_4 \simeq 0.49099$ ; $\F_5 \simeq 0.49867$ ; $\F_6 \simeq 0.499920$.}
\label{fig:T3, increasing to 0.5}
\end{figure}

%I know Proposition \ref{prop_k=3_odd_888} holds for $\kappa = 3,4$ and I extrapolated to larger $\kappa$.

\begin{proposition} 
\label{prop_interlace_888}
Fix $\kappa \geq 3$. The odd and even energy solutions $\mathcal{F}_n$ of Propositions \ref{prop_k=3_odd_888} and \ref{prop_k=3_even_888} interlace and are a strictly increasing sequence : $\mathcal{F}_{n} < \mathcal{F}_{n+1} < \mathcal{F}_{n+2}$, $\forall n\in \N^*$. Also, $\F_n \nearrow \inf J_3 := 1+\cos(2\pi / \kappa)$. 
\end{proposition} 

\begin{remark}
Numerically we tried to compute several solutions $\F_n (\kappa=3)$ in order to conjecture the rate of convergence, but we struggled to get accurate numbers because of precision limitations. We found : $\F_8 \simeq 0.4999999956$, which suggests $\F_n$ converges rather quickly to $\inf J_3$.
\end{remark}

\section{Another increasing sequence of thresholds below $J_3(\kappa)$}
\label{section J3b}

%\noindent \underline{\textit{\textbf{Conjecture 0:}}} For fixed $\kappa \in \N^*$, $\kappa \geq 3$, $J_3 \simeq (1+\cos(2\pi / \kappa), 2\cos(\pi / \kappa) ) \subset \boldsymbol{\mu}_{\kappa}(\Delta)$. 

We revisit the sequence of the previous section but add a twist to it. The twist is that instead of placing $X_0$ on the branch of $T_{\kappa}$ to the right of $X_n$ we place it on the branch of $T_{\kappa}$ to the left of $X_n$; and instead of placing $X_{n+1}$ at $\cos(2\pi / \kappa)$, we place it at $1$. This gives a sequence which we believe is distinct. This section proves Theorem \ref{thm_decreasing energy general_k3} for $\{ \G_n \}$.

\begin{proposition} (odd terms)
\label{prop_k=3_odd}
Fix $\kappa \geq 3$, and $n \in \N^*$ odd. System \eqref{conjecture_system_conj_intro} has a unique solution (which we denote $\G_n$ instead of $\E_n$) such that $\G_n \in (\cos(\pi / \kappa) + \cos(2\pi / \kappa), 1+ \cos(2\pi / \kappa))$ and
\begin{equation}
\label{chain_1_k3}
\mathcal{G}_n -1 = X_{0} < X_{n} < X_{n-1} < ... <  X_{(n+1)/2} < X_{(n-1)/2} < ... < X_{1} < X_{n+1} := 1.
\end{equation}
%\begin{equation}
%\label{J2_constraint}
%\mathcal{E}_n \in J_2(\kappa). %(2\cos(\pi / \kappa ) , 1+ \cos(\pi / \kappa) ).
%\end{equation}
Furthermore, this solution satisfies \eqref{o1}, \eqref{o2} and \eqref{o3}, and so $\G_n \in \mathfrak{T}_{n,\kappa}$.
\end{proposition}

\begin{comment}
Fix $\kappa \geq 3$, and let $n \in \N^*$, $n$ odd be given. Consider real variables $\G_n, X_0, X_1, ..., X_n, X_{n+1}$. Then the system of $n+3$ equations in $n+3$ unknowns
\begin{equation}
\label{conjecture_system_k3}
\begin{cases}
T_{\kappa}(X_q) = T_{\kappa}(X_{n-q}), & \forall q = 0,1,..., (n-1)/2,  \\
X_{n-q} = \G_n -X_{1+q}, & \forall q = -1,0, ..., (n-1)/2,  \\
X_{n+1}=1 & 
\end{cases}
\end{equation}
admits a unique solution satisfying the constraints :
\begin{equation}
\label{chain_1_k3}
\mathcal{G}_n -1 = X_{0} < X_{n} < X_{n-1} < ... <  X_{(n+1)/2} < X_{(n-1)/2} < ... < X_{1} < X_{n+1} := 1,
\end{equation}
\begin{equation}
\label{J2_constraint_k3}
\mathcal{G}_n \in (\cos(\pi / \kappa) + \cos(2\pi / \kappa),1+\cos(2\pi / \kappa)).
\end{equation}
\end{comment}

\begin{remark}
\label{remark000}
The proof reveals that $X_0 \in(\cos(3\pi / \kappa), \cos(2\pi / \kappa))$, $X_n$, $X_{n-1}$, ..., $X_{(n+1)/2} \in (\cos(2\pi / \kappa), \cos(\pi / \kappa))$, and for $n \geq 3$, $X_{(n-1)/2}$,..., $X_1 \in (\cos(\pi / \kappa),1)$. This is important because it means that $U_{\kappa-1}(X_0) > 0$, $U_{\kappa-1}(X_n)$, ..., $U_{\kappa-1}(X_{(n+1)/2}) <0$, and for $n \geq 3$, $U_{\kappa-1}(X_{(n-1)/2})$, ..., $U_{\kappa-1}(X_1) >0$.
\end{remark}

\begin{proposition} (even terms)
\label{prop_k=3_even}
Fix $\kappa \geq 3$, and $n \in \N^*$ even. System \eqref{conjecture_system_conj_intro_even} has a unique solution (which we denote $\G_n$ instead of $\E_n$) such that $\G_n \in (\cos(\pi / \kappa) + \cos(2\pi / \kappa), 1+ \cos(2\pi / \kappa))$ and
\begin{equation}
\label{chain_1_k3even}
\mathcal{G}_n -1 = X_{0} < X_{n} < X_{n-1} < ... <  X_{n/2+1} < X_{n/2} := \cos(\pi / \kappa) < ... < X_{1} < X_{n+1}  := 1.
\end{equation}
%\begin{equation}
%\label{J2_constraint}
%\mathcal{E}_n \in J_2(\kappa). %(2\cos(\pi / \kappa ) , 1+ \cos(\pi / \kappa) ).
%\end{equation}
Furthermore, this solution satisfies \eqref{o11}, \eqref{o22} and \eqref{o33}, and so $\G_n \in \mathfrak{T}_{n,\kappa}$.
\end{proposition}

\begin{comment}
Fix $\kappa \geq 3$, and let $n \in \N^*$, $n$ even be given. Consider real variables $\G_n, X_0, X_1, ..., X_n, X_{n+1}$. Then the system of $n+3$ equations in $n+3$ unknowns
\begin{equation}
\label{conjecture_system_k3_even}
\begin{cases}
T_{\kappa}(X_q) = T_{\kappa}(X_{n-q}),  & \forall \ q = 0,1,..., \frac{n}{2}-1,   \\
X_{n-q} = \G_n -X_{1+q},  & \forall \ q = -1,0,..., \frac{n}{2}-1, \\ 
X_{\frac{n}{2}} = \cos(\pi / \kappa), \quad X_{n+1}=1 &  \\
\end{cases}
\end{equation}
admits a unique solution satisfying the constraints :
\begin{equation}
\label{chain_1_k3even}
\mathcal{G}_n -1 =: X_{0} < X_{n} < X_{n-1} < ... <  X_{n/2+1} < X_{n/2} < ... < X_{1} < X_{n+1}  := 1,
\end{equation}
\begin{equation}
\label{J2_constraint_k3even}
\mathcal{G}_n \in (\cos(\pi / \kappa) + \cos(2\pi / \kappa),1+\cos(2\pi / \kappa)).
\end{equation}
\end{comment}

\begin{remark}
\label{remark00}
The proof reveals that $X_0 \in (\cos(3\pi / \kappa), \cos(2\pi / \kappa))$, $X_n$, $X_{n-1}$, ..., $X_{n/2+1} \in (\cos(2\pi / \kappa), \cos(\pi / \kappa))$, and for $n \geq 4$, $X_{n/2-1}$, ..., $X_1 \in (\cos(\pi / \kappa), 1)$. This is important because it means that $U_{\kappa-1}(X_0) > 0$, $U_{\kappa-1}(X_n)$, ..., $U_{\kappa-1}(X_{n/2+1}) <0$, and for $n \geq 4$, $U_{\kappa-1}(X_{n/2-1})$, ..., $U_{\kappa-1}(X_1) >0$.
\end{remark}

Figure \ref{fig:test_T3k3} illustrates the solutions in Propositions \ref{prop_k=3_odd} and \ref{prop_k=3_even} for $1 \leq n \leq 6$ and $\kappa=3$. Exact solutions for $\mathcal{G}_1$, $\mathcal{G}_2$ and $\mathcal{G}_3$ are :
\begin{equation*}
\begin{aligned}
& \mathcal{G}_1 = (5-3\sqrt{2})/7 \simeq 0.10819, \quad \mathcal{G}_2 = (9-\sqrt{33})/12 \simeq 0.27129, \\
& \mathcal{G}_3 = \frac{1}{247} \left( \frac{199}{2} + \frac{1}{2} \sqrt{16838+a+b} - \frac{1}{2} \sqrt{32766 - a -b + \frac{4340412}{\sqrt{16383+a+b}}} \right) \simeq 0.382291, \\
& a := \frac{1}{3} (39558654304767 - 2109209004864 \sqrt{327})^{1/3}, \quad b := 741 (3601 + 192\sqrt{327})^{1/3}.
\end{aligned}
\end{equation*}

\begin{figure}[H]
  \centering
 \includegraphics[scale=0.11]{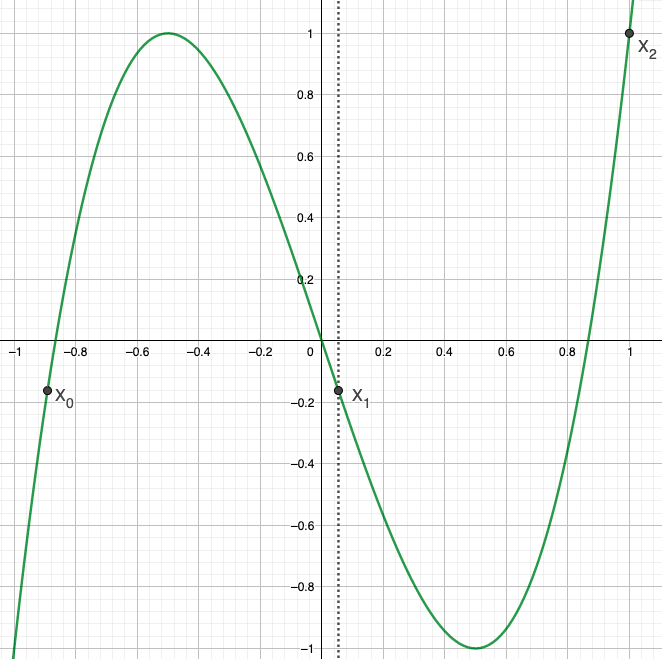}
  \includegraphics[scale=0.11]{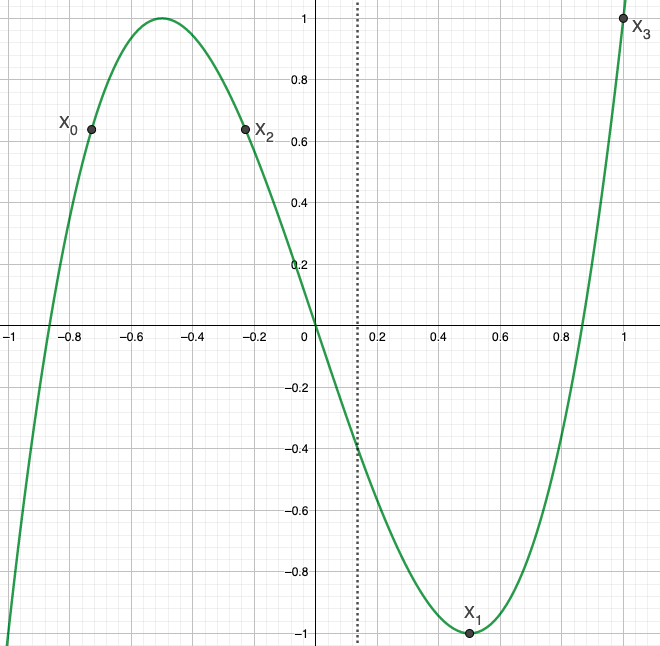}
  \includegraphics[scale=0.11]{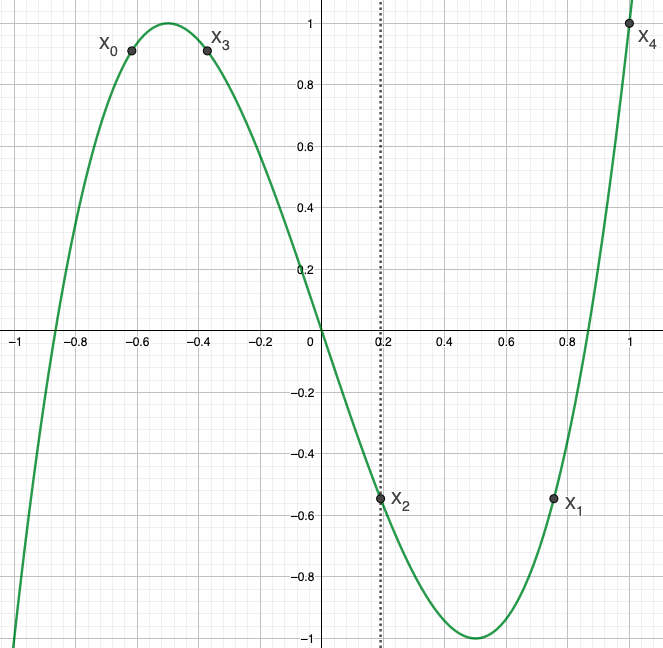}
    \includegraphics[scale=0.11]{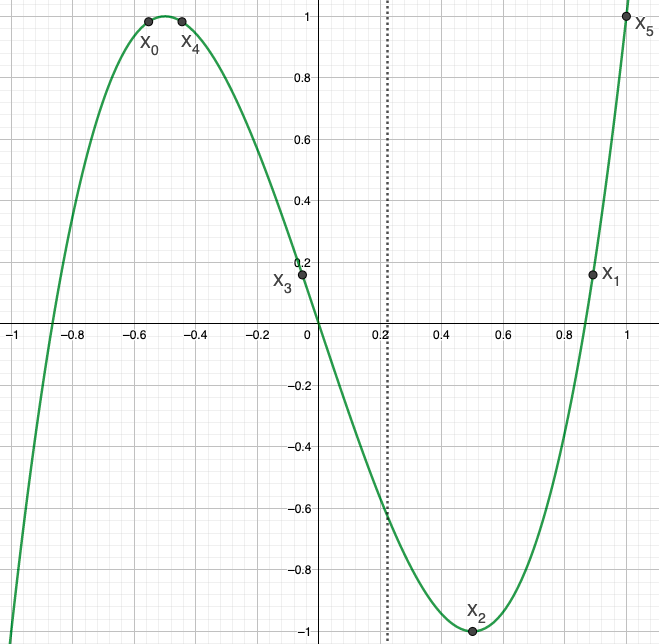}
      \includegraphics[scale=0.11]{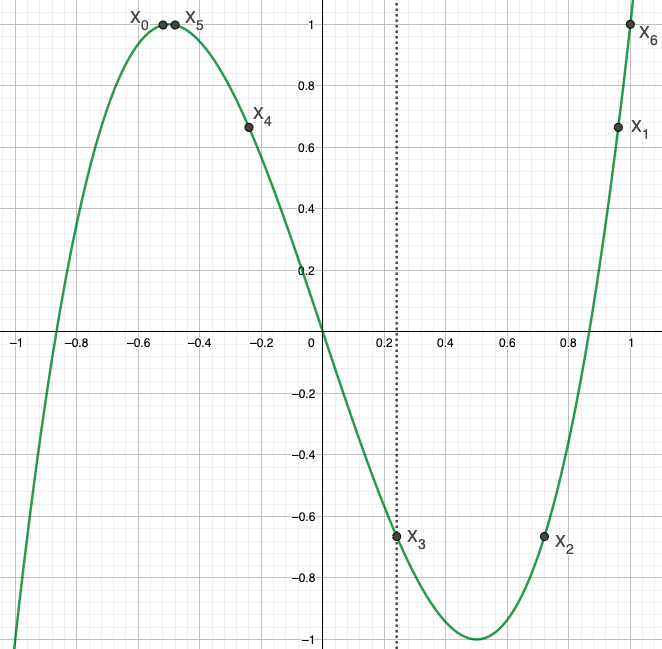}
       \includegraphics[scale=0.11]{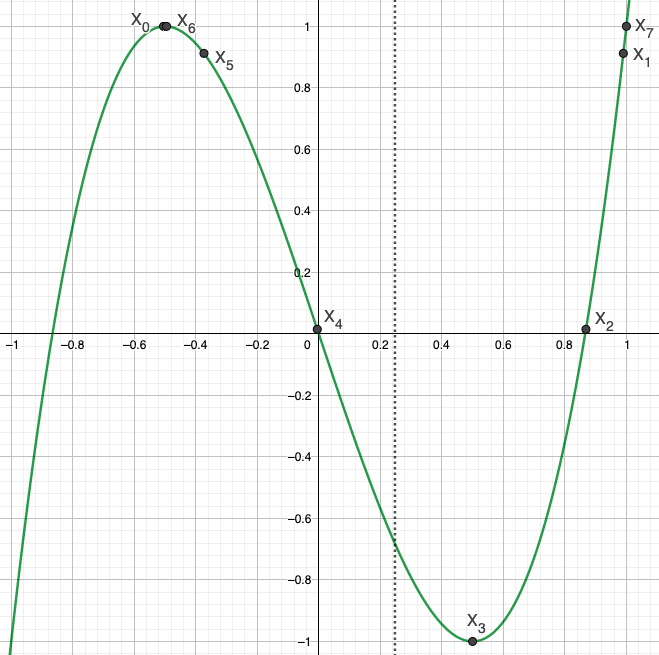}
\caption{$\kappa=3$. $T_{\kappa=3}(x)$. $\G_n$ of Propositions \ref{prop_k=3_odd} and \ref{prop_k=3_even}. Left to right : $\G_1 \simeq 0.10819$ ; $\G_2 \simeq 0.27129$ ; $\G_3 \simeq 0.382291$ ; $\G_4 \simeq 0.446704$ ; $\G_5 \simeq 0.480492$ ; $\G_6 \simeq 0.495054$.}
\label{fig:test_T3k3}
\end{figure}

To justify $\G_n \in (\cos(\pi / \kappa) + \cos(2\pi / \kappa), 1+ \cos(2\pi / \kappa))$ for $n$ even, we know that we want $X_{n/2+1} \in (\cos(2\pi / \kappa), \cos(\pi / \kappa))$, and since $X_{n/2}=\cos(\pi / \kappa)$ and $X_{n/2+1}$ is the symmetric of $X_{n/2}$ wrt.\ $\G_n /2$, it must be that $\G_n - X_{n/2} = \G_n - \cos(\pi / \kappa) \in (\cos(2\pi / \kappa), \cos(\pi / \kappa))$, i.e.\ $\cos(2\pi / \kappa) + \cos(\pi / \kappa) \leq \G_n \leq 2 \cos(\pi / \kappa)$. For the upper bound, we know that $X_0+X_{n+1} = X_0 + 1 = \G_n$ and $X_0 < \cos(2\pi / \kappa)$, so $\G_n < 1 +\cos(2\pi / \kappa)$. For $n$ odd we probably can get the range by using the interlacing property.

\begin{comment}
\begin{figure}[H]
  \centering
 \includegraphics[scale=0.16]{k3n=1}
  \includegraphics[scale=0.15]{k3n=2}
  \includegraphics[scale=0.16]{k3n=3}
    \includegraphics[scale=0.15]{k3n=4}
      \includegraphics[scale=0.15]{k3n=5}
\caption{$T_{\kappa=3}(x)$. $\G_n$ of Propositions \ref{prop_k=3_odd} and \ref{prop_k=3_even}. Left to right : $\G_1 \simeq 0.10819$ ; $\G_2 \simeq 0.27129$ ; $\G_3 \simeq 0.382291$ ; $\G_4 \simeq 0.446704$ ; $\G_5 \simeq 0.480492$.}
\label{fig:test_T3k3}
\end{figure}
\end{comment}

\begin{proposition} 
\label{prop_interlace_j4}
Fix $\kappa \geq 3$. The odd and even energy solutions $\mathcal{G}_n$ of Propositions \ref{prop_k=3_odd} and \ref{prop_k=3_even} interlace and are a strictly increasing sequence : $\mathcal{G}_{n} < \mathcal{G}_{n+1} < \mathcal{G}_{n+2}$, $\forall n\in \N^*$. Also, $\mathcal{G}_n \nearrow \inf J_3 := 1+\cos(2\pi / \kappa)$ as $n \to \infty$.
\end{proposition}

\section{A generalization of section \ref{section J2} : a sequence $\E_n \searrow 2\cos(j\pi/ \kappa)$}
\label{section_gen1}

The construction used to get a sequence in the right-most well of $T_{\kappa}(x)$ in section \ref{section J2} is not specific to the right-most well. One can build a similar sequence in other wells.

\subsection{Decreasing sequence in upright well, $j$ odd}
Figure \ref{fig:test_T8decreasingNormalWell} illustrates a decreasing sequence $\E_n \searrow 2 \cos(3\pi / \kappa)$, for $\kappa = 8$, $j=3$. Note that the dotted line $x=\E_n /2$ is to the right of the minimum $x= \cos(3\pi / \kappa)$ but converges to it.

\begin{figure}[htb]
  \centering
 \includegraphics[scale=0.09]{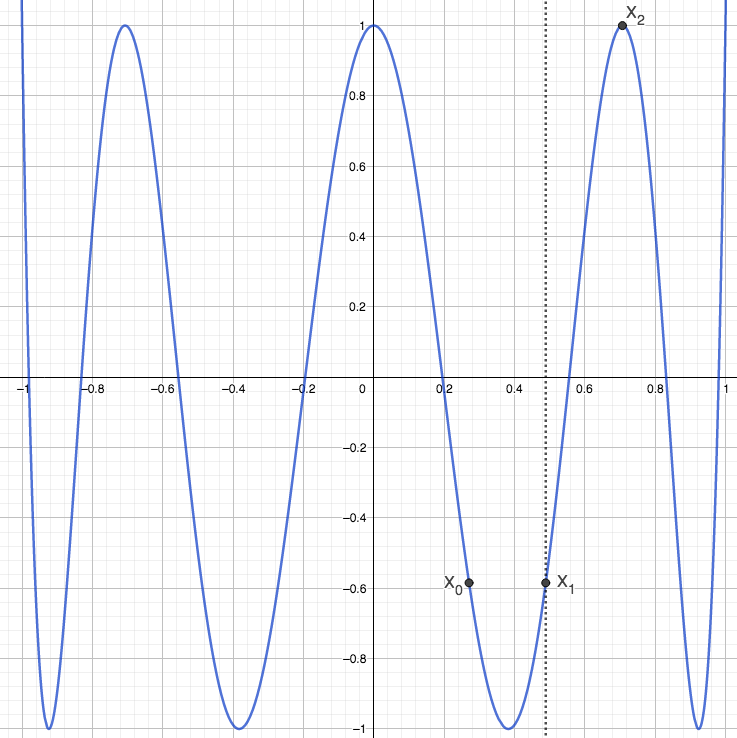}
 \includegraphics[scale=0.09]{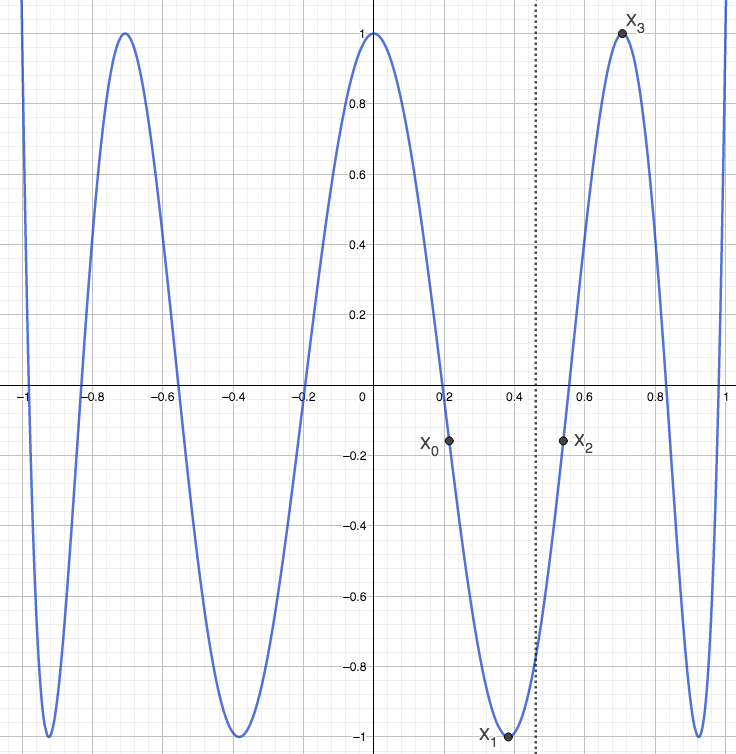}
   \includegraphics[scale=0.09]{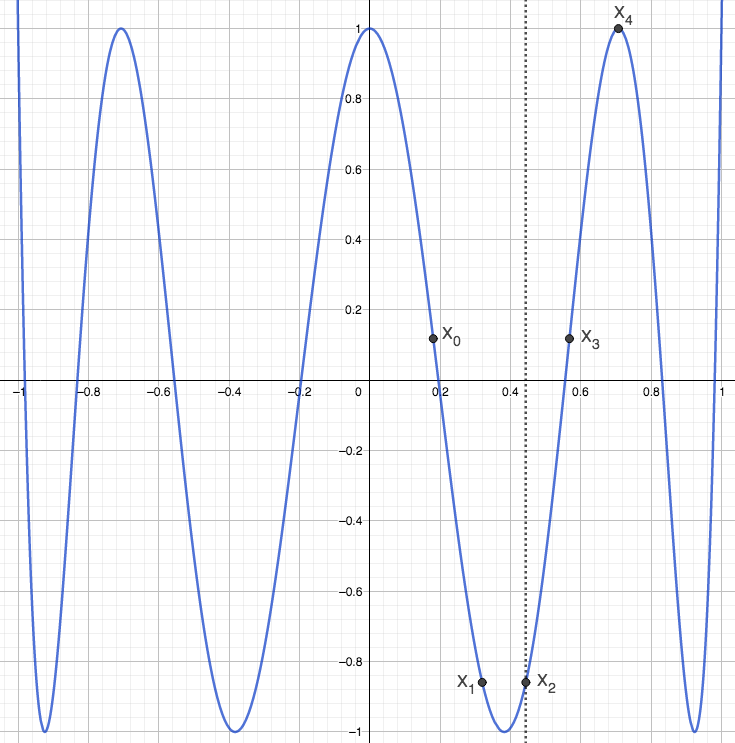}
  \includegraphics[scale=0.09]{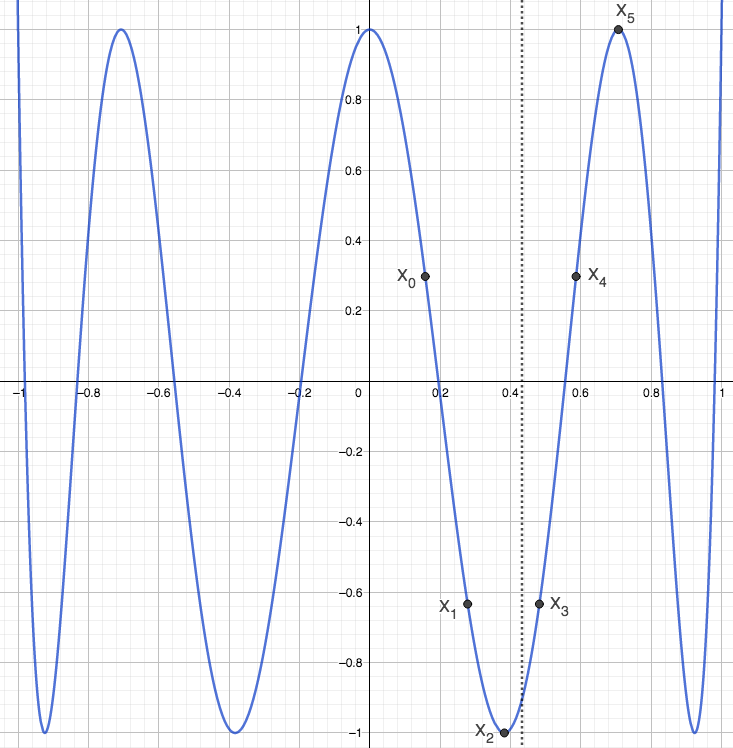}
   \includegraphics[scale=0.09]{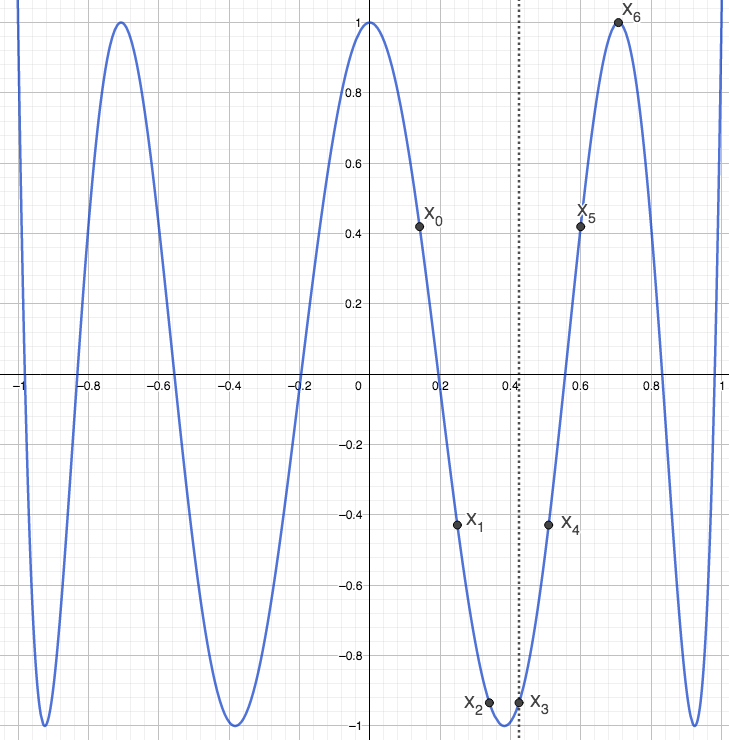}
  \includegraphics[scale=0.09]{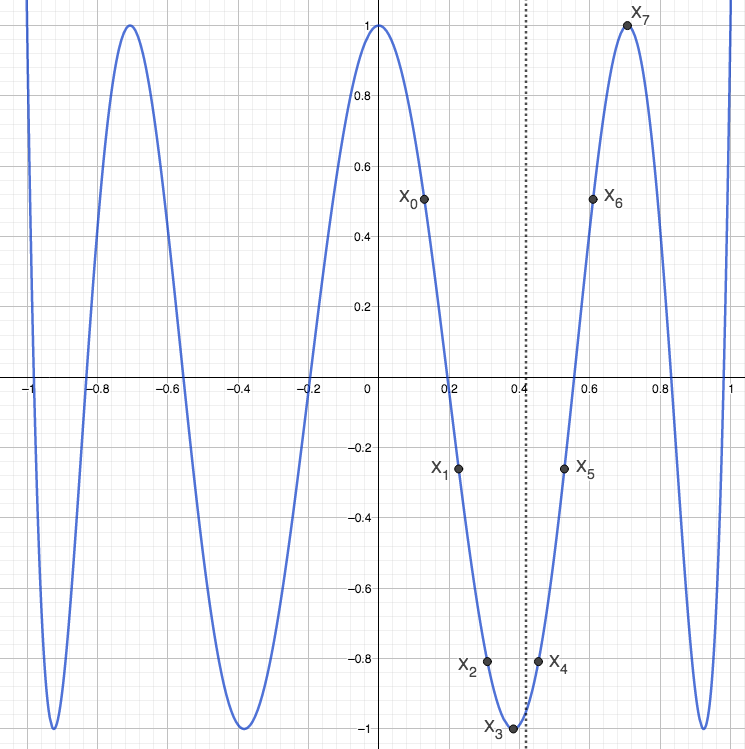}
\caption{$T_{\kappa=8}(x)$. Solutions $\E_n$. Left to right : $\E_1 \simeq 0.9781$, $\E_2 \simeq 0.9216$, $\E_3 \simeq 0.8877$, $\E_4 \simeq 0.8650$, $\E_5 \simeq 0.8488$, $\E_6 \simeq 0.8368$.}
\label{fig:test_T8decreasingNormalWell}
\end{figure}

Thus, we propose a generalization of Theorem \ref{thm_decreasing energy general} :
\begin{theorem} 
\label{thm_averagekk111}
%$\E_n \searrow 2\cos(j\pi/ \kappa)$
Fix $\kappa \geq 2$. Fix $1 \leq j \leq \floor*{\kappa/2}$, $j$ odd. There is a sequence $\{ \E_n \}_{n=1} ^{\infty}$, which depends on $\kappa$, s.t.\ $\{ \E_n \} \subset (2\cos(j\pi/ \kappa), \cos((j-1)\pi / \kappa) + \cos((j+1)\pi/ \kappa)) \cap \boldsymbol{\Theta}_{\kappa}(\Delta)$, and $\E_{n+2} < \E_{n+1} < \E_n$, $\forall n \in \N^*$. Also, $\E_{2n-1}$, $\E_{2n} \in \boldsymbol{\Theta}_{n, \kappa}(\Delta)$, $\forall n \geq 1$, and 
\begin{equation}
\label{chain_109}
\mathcal{E}_n -\cos((j-1)\pi / \kappa) = X_{0} < X_{1} < ... < X_{n} < X_{n+1} := \cos((j-1)\pi / \kappa).
\end{equation}
\end{theorem}

\subsection{Decreasing sequence in upside down well, $j$ even}

For $j$ even, the well is upside down. Figure \ref{fig:test_T8decreasingUpsideDownWell} illustrates a decreasing sequence $\E_n \searrow 2 \cos(2\pi / \kappa)$, for $\kappa = 8$, $j=2$. Note that the dotted line $x=\E_n /2$ is to the right of the maximum $x= \cos(2\pi / \kappa)$ but converges to it.

\begin{figure}[htb]
  \centering
 \includegraphics[scale=0.09]{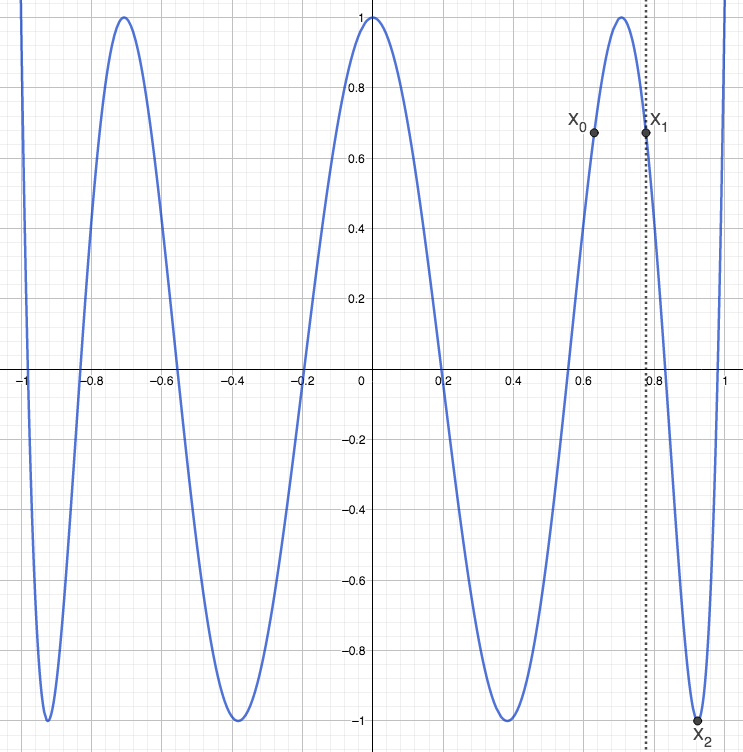}
 \includegraphics[scale=0.09]{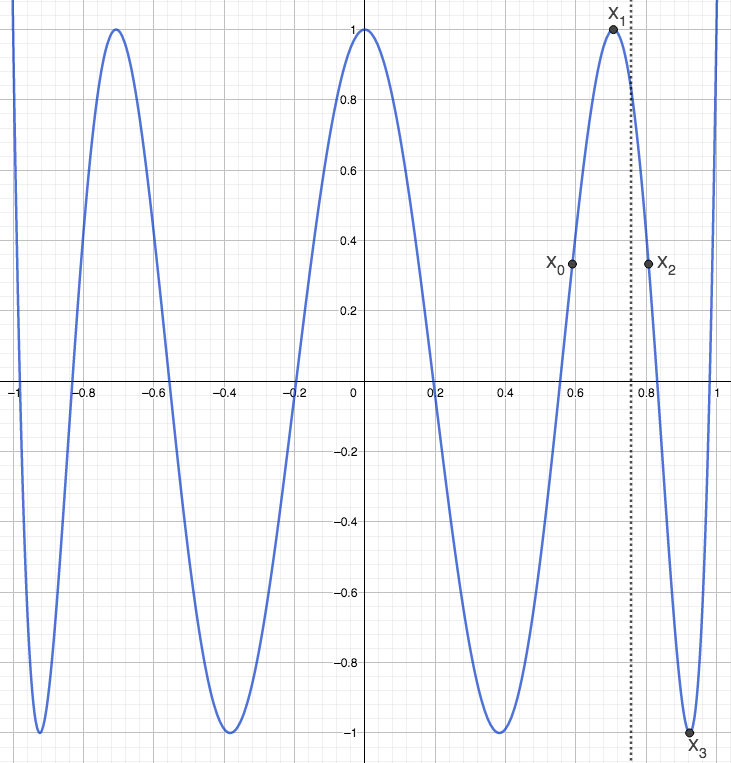}
  \includegraphics[scale=0.09]{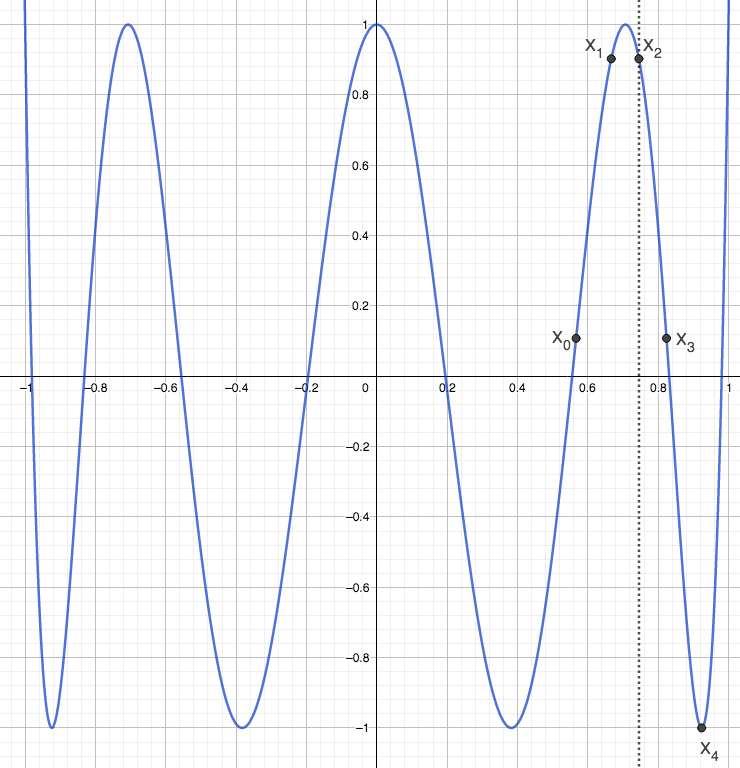}
  \includegraphics[scale=0.09]{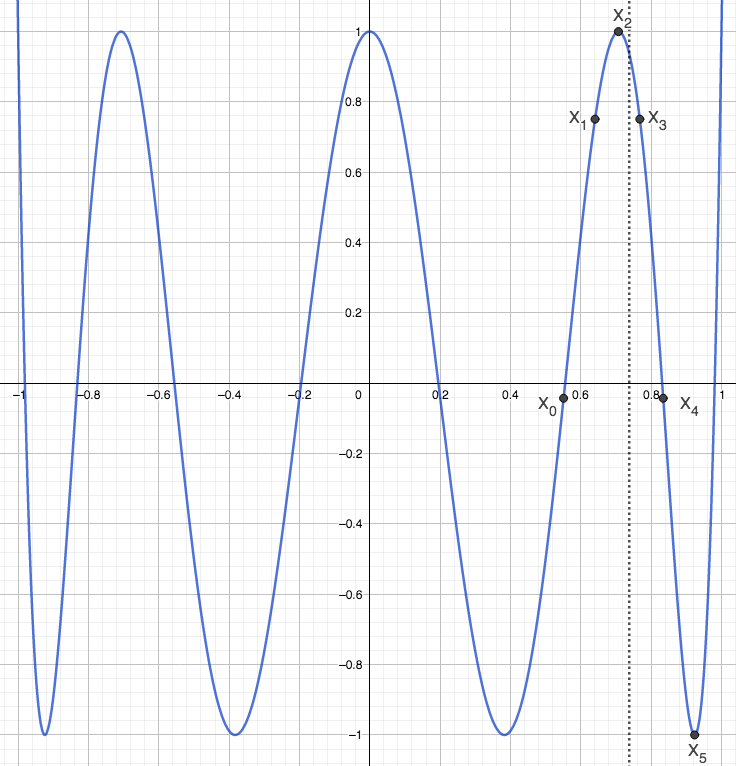}
  \includegraphics[scale=0.09]{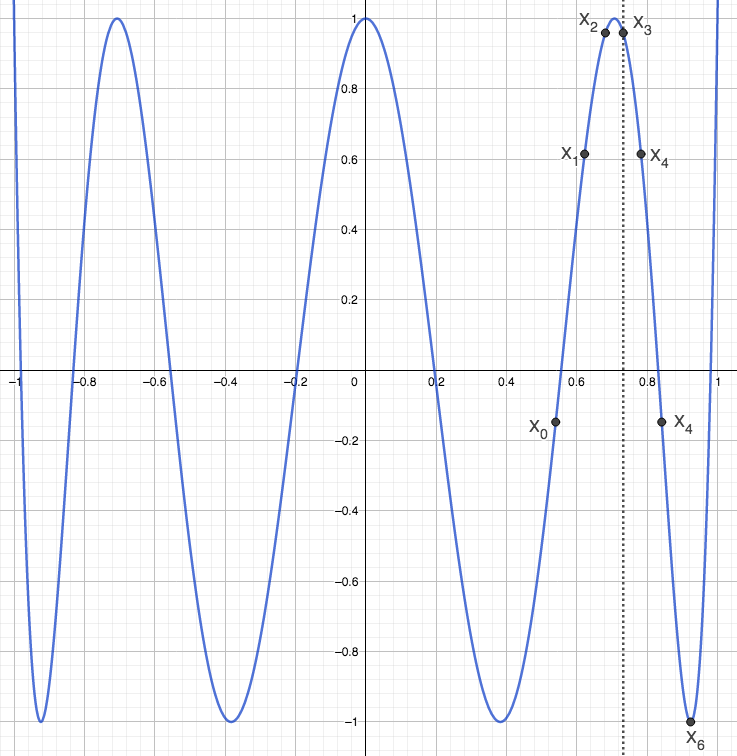}
  \includegraphics[scale=0.09]{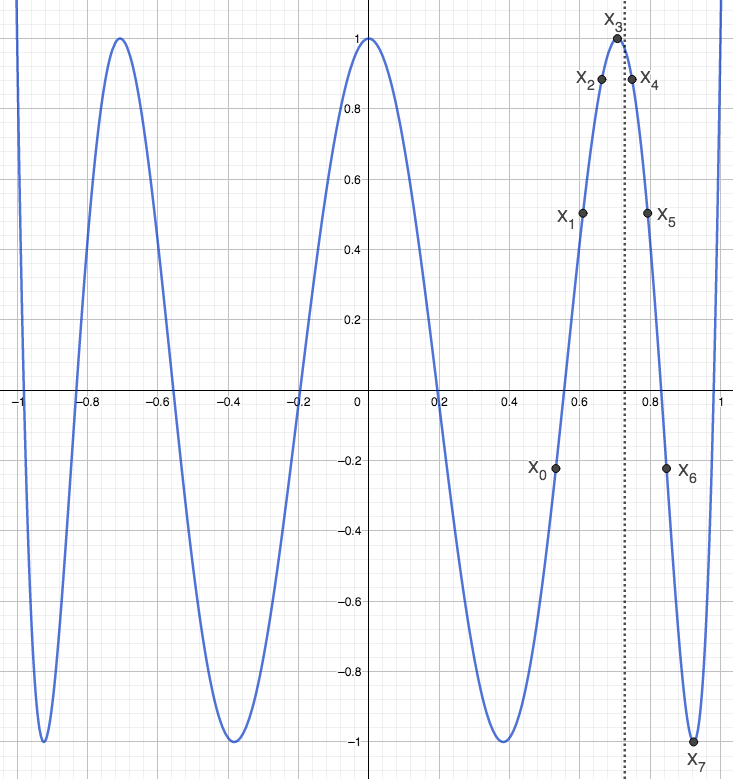}
\caption{$T_{\kappa=8}(x)$. Solutions $\E_n$. Left to right : $\E_1 \simeq 1.5536$, $\E_2 \simeq 1.5142$, $\E_3 \simeq 1.4906$, $\E_4 \simeq 1.4750$, $\E_5 \simeq 1.4640$, $\E_6 \simeq 1.4559$.}
\label{fig:test_T8decreasingUpsideDownWell}
\end{figure}

Thus, we propose a generalization of Theorem \ref{thm_decreasing energy general} :
\begin{theorem} 
\label{thm_averagekk112}
%$\E_n \searrow 2\cos(j\pi/ \kappa)$
Fix $\kappa \geq 2$. Fix $2 \leq j \leq \floor*{\kappa/2}$, $j$ even. There is a sequence $\{ \E_n \}_{n=1} ^{\infty}$, which depends on $\kappa$, s.t.\ $\{ \E_n \} \subset (2\cos(j\pi/ \kappa), \cos((j-1)\pi / \kappa) + \cos((j+1)\pi/ \kappa)) \cap \boldsymbol{\Theta}_{\kappa}(\Delta)$, and $\E_{n+2} < \E_{n+1} < \E_n$, $\forall n \in \N^*$. Also, $\E_{2n-1}$, $\E_{2n} \in \boldsymbol{\Theta}_{n, \kappa}(\Delta)$, $\forall n \geq 1$, and 
\begin{equation}
\label{chain_1099}
\mathcal{E}_n -\cos((j-1)\pi / \kappa) = X_{0} < X_{1} < ... < X_{n} < X_{n+1} := \cos((j-1)\pi / \kappa).
\end{equation}
\end{theorem}

\subsection{A comment on the proofs of these Theorems and a Conjecture on the limit}

To prove the Theorems \ref{thm_averagekk111} and \ref{thm_averagekk112} one needs to adapt the proofs of Propositions \ref{prop1}, \ref{prop2} and \ref{prop_interlace}. The adaptation of these Propositions is straightforward. As for the limit we conjecture :
\begin{conjecture}
\label{conj_JJJ}
Let $\{\E_n\}$ be the sequence in Theorems \ref{thm_averagekk111} and \ref{thm_averagekk112}. Then $\E_n \searrow 2\cos(j\pi/ \kappa)$.
\end{conjecture}

To prove this conjecture we tried adapting the proofs of Proposition \ref{prop_convergeence} and Lemma \ref{l:croissdiff} but to no avail. To adapt the Lemma however, we do conjecture :
\begin{conjecture}
\label{l:croissdiffJJ}
Let $\kappa\geq 2$. Fix $1 \leq j \leq \floor*{\kappa/2}$. If $\cos((j+1)\pi/\kappa)<a<\cos(j\pi/\kappa)<b<\cos((j-1)\pi/\kappa)$ are such that $T_\kappa(a) = T_\kappa(b)$, then 
\begin{equation}
\label{distanceLeftGreaterthanRightJJ}
\cos(j\pi/\kappa)-a>b-\cos(j\pi/\kappa).
\end{equation}
\end{conjecture}

We don't know how to adapt the proof of Lemma \ref{l:croissdiff} to prove Conjecture \ref{l:croissdiffJJ}. We would want to prove that the function $S_{\kappa} (x) := T_{\kappa} (x) - T_{\kappa} (2\cos(j \pi / \kappa) - x)$ is positive for $x \in \left( \cos( j \pi / \kappa), \cos( (j-1) \pi / \kappa) \right)$, for $j=1,2,.., \floor*{\kappa/2}$. We speculate this statement is true for $j=1,2,.., \floor*{\kappa/2}$. However the statement that $S_{\kappa} ' (x) \geq 0$ appears to be true only for $j=1$.

\section{A generalization of section \ref{section J3a} : a sequence $\F_n \nearrow \cos((j-1)\pi/ \kappa)+\cos((j+1)\pi/ \kappa)$}
\label{section_gen2}

Again, the construction used to get a sequence in the right-most well of $T_{\kappa}(x)$ in section \ref{section J3a} is not specific to the right-most well. One can build a similar sequence in other wells. In this section we get an increasing sequence $\F_n$. At first, it may be tempting to think that $\F_n \nearrow 2\cos(j\pi/ \kappa)$, but this is not the case because $\cos((j-1)\pi/ \kappa)+\cos((j+1)\pi/ \kappa) < 2\cos(j\pi/ \kappa)$. Instead, we have $\F_n \nearrow \cos((j-1)\pi/ \kappa)+\cos((j+1)\pi/ \kappa)$.

\subsection{Increasing sequence in upright well, $j$ odd}

Figure \ref{fig:test_T8increasingUpsideDownWell99} illustrates an increasing sequence $\F_n \nearrow \cos(4\pi / \kappa)+\cos(2\pi / \kappa)$, for $\kappa = 8$, $j=3$. Note that the dotted line $x=\F_n /2$ is to the left of the minimum $x= \cos(3\pi / \kappa)$, approaches it, but converges before.

\begin{figure}[htb]
  \centering
 \includegraphics[scale=0.085]{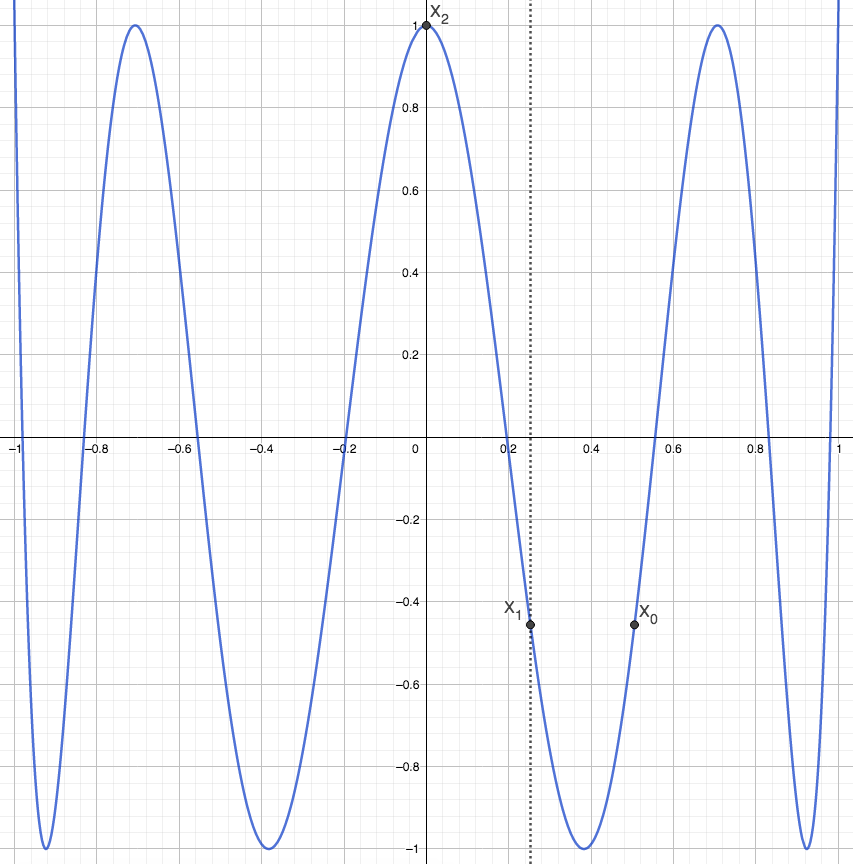}
 \includegraphics[scale=0.085]{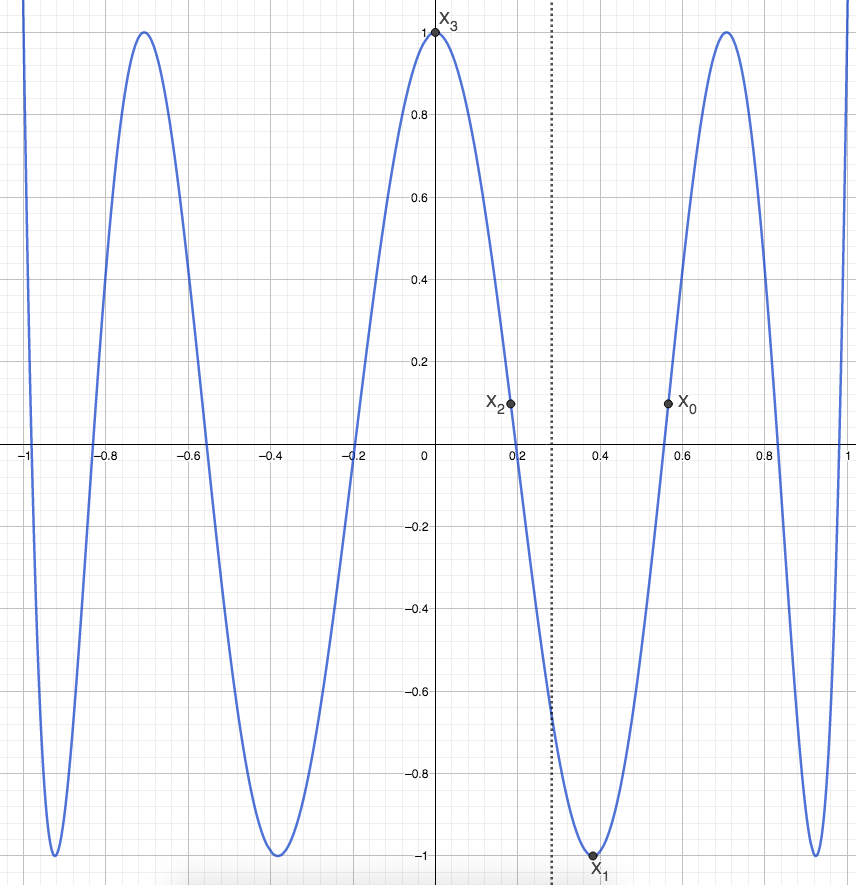}
   \includegraphics[scale=0.085]{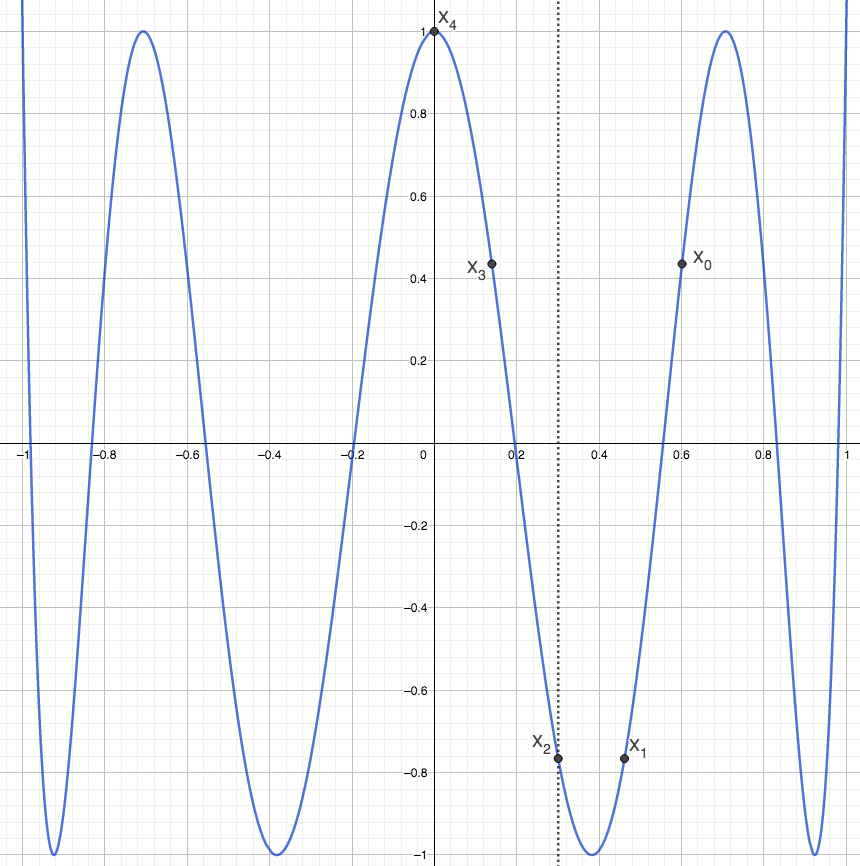}
  \includegraphics[scale=0.085]{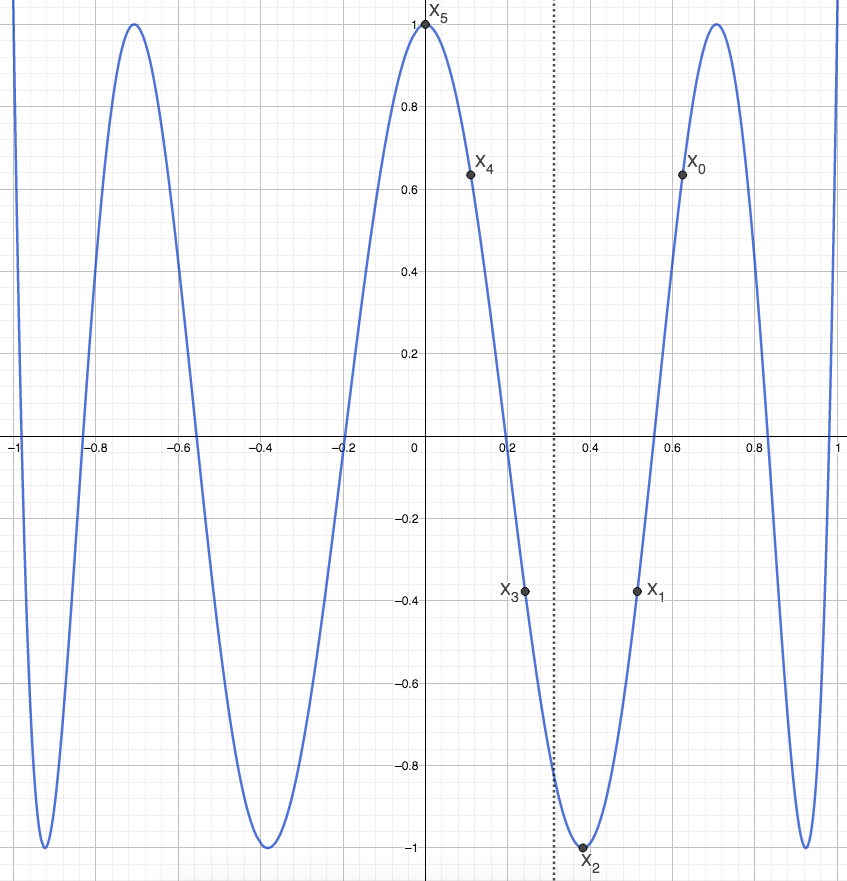}
  \includegraphics[scale=0.085]{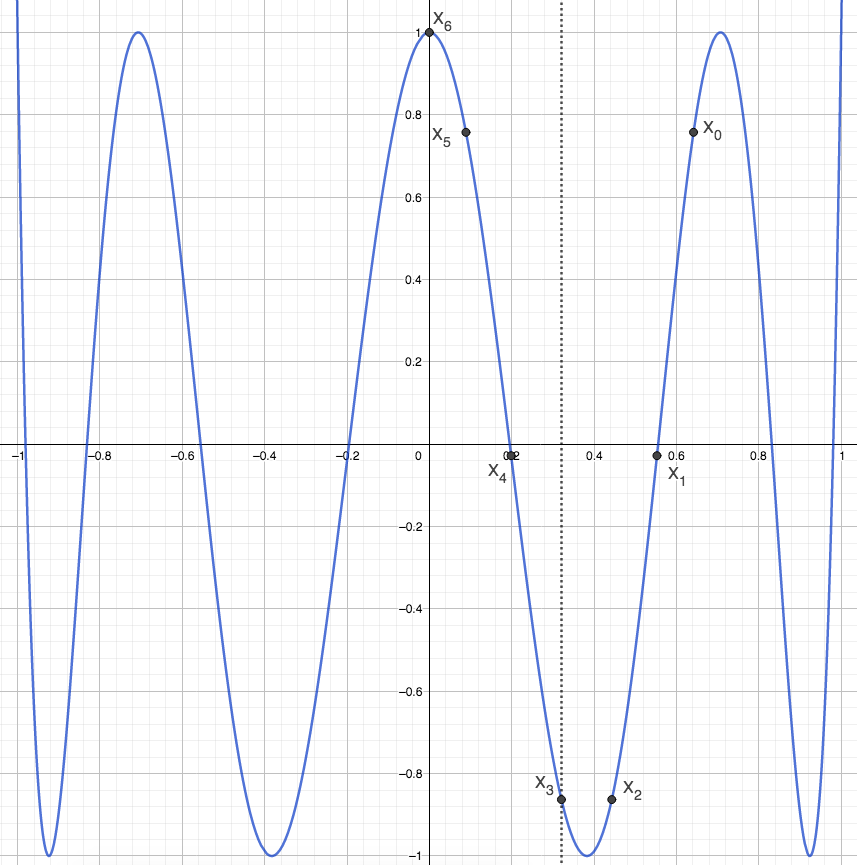}
  \includegraphics[scale=0.085]{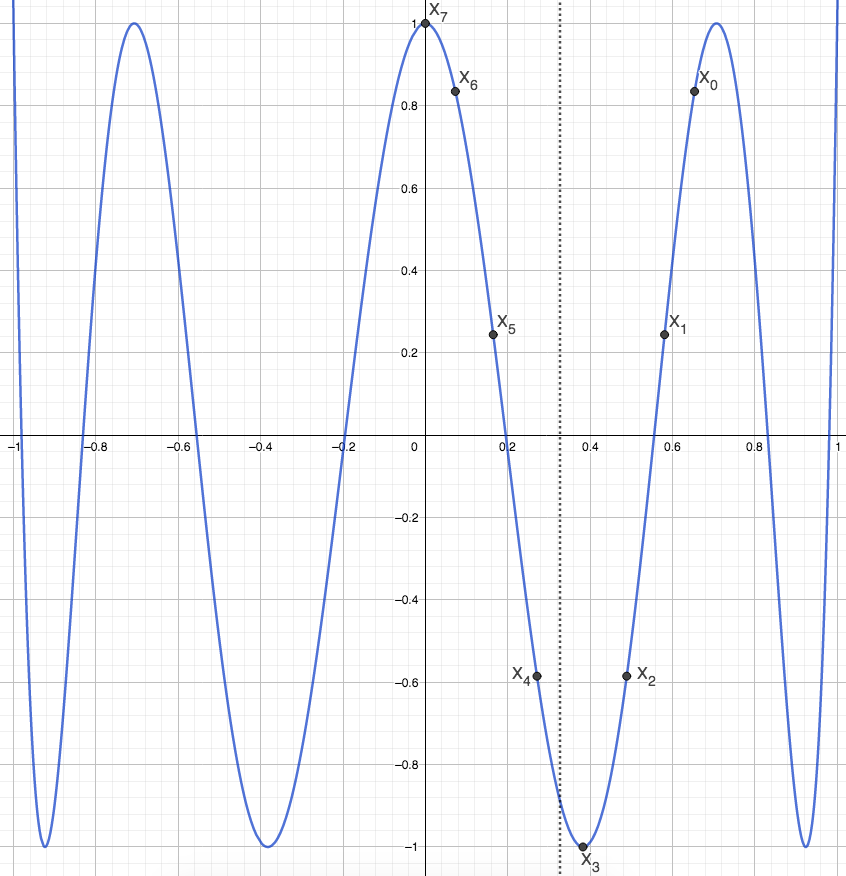}
\caption{$T_{\kappa=8}(x)$. Solutions $\F_n$. Left to right : $\F_1 \simeq 0.5054$, $\F_2 \simeq 0.5657$, $\F_3 \simeq 0.6014$, $\F_4 \simeq 0.6249$, $\F_5 \simeq 0.6415$, $\F_6 \simeq 0.6537$.}
\label{fig:test_T8increasingUpsideDownWell99}
\end{figure}

Thus, we propose a generalization :
\begin{theorem} 
\label{thm_averagekk113}
% and $\F_n \nearrow \cos((j-1)\pi / \kappa) + \cos((j+1)\pi/ \kappa)$
Fix $\kappa \geq 2$. Fix $1 \leq j \leq \floor*{\kappa/2}$, $j$ odd. There is a sequence $\{ \F_n \}_{n=1} ^{\infty}$, which depends on $\kappa$, s.t.\ $\{ \F_n \} \subset (2\cos((j+1)\pi/ \kappa), \cos((j-1)\pi / \kappa) + \cos((j+1)\pi/ \kappa)) \cap \boldsymbol{\Theta}_{\kappa}(\Delta)$, and $\F_n < \F_{n+1} < \F_{n+2}$, $\forall n \in \N^*$. Also, $\F_{2n-1}$, $\F_{2n} \in \boldsymbol{\Theta}_{n, \kappa}(\Delta)$, $\forall n \geq 1$, and 
\begin{equation}
\label{chain_1093}
\cos((j+1)\pi / \kappa) =: X_{n+1} < X_{n} < ... < X_{1} < X_{0} < \cos((j-1)\pi / \kappa).
%\mathcal{E}_n -\cos((j-1)\pi / \kappa) =: X_{0} < X_{1} < ... < X_{n} < X_{n+1} := \cos((j-1)\pi / \kappa).
\end{equation}
\end{theorem}

\subsection{Increasing sequence in upside down well, $j$ even}

For $j$ even, the well is upside down. Figure \ref{fig:test_T8increasingUpsideDownWell} illustrates an increasing sequence $\F_n \nearrow \cos(3\pi / \kappa)+\cos(\pi / \kappa)$, for $\kappa = 8$, $j=2$. Note that the dotted line $x=\F_n /2$ is to the left of the maximum $x= \cos(2\pi / \kappa)$, approaches it, but converges before.

\begin{figure}[htb]
  \centering
 \includegraphics[scale=0.09]{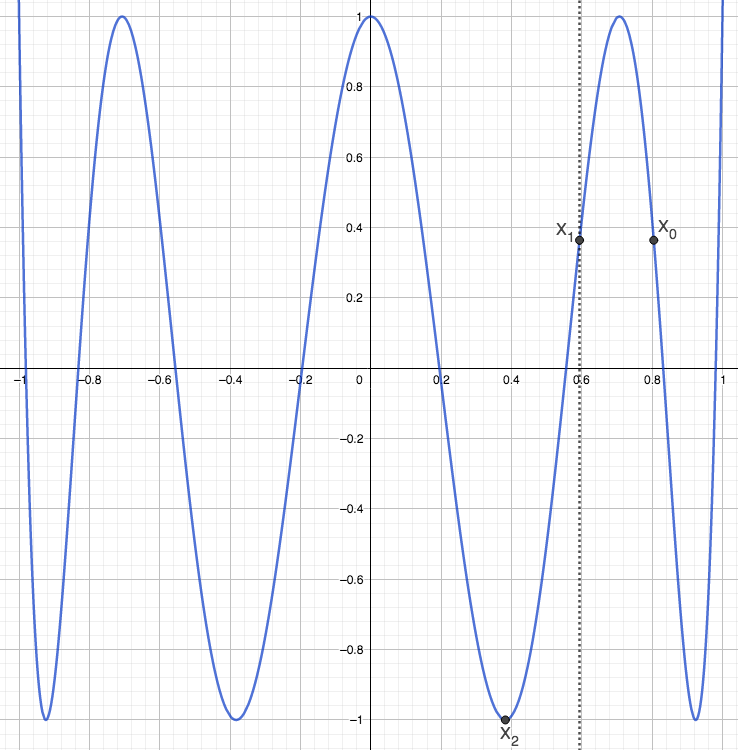}
 \includegraphics[scale=0.09]{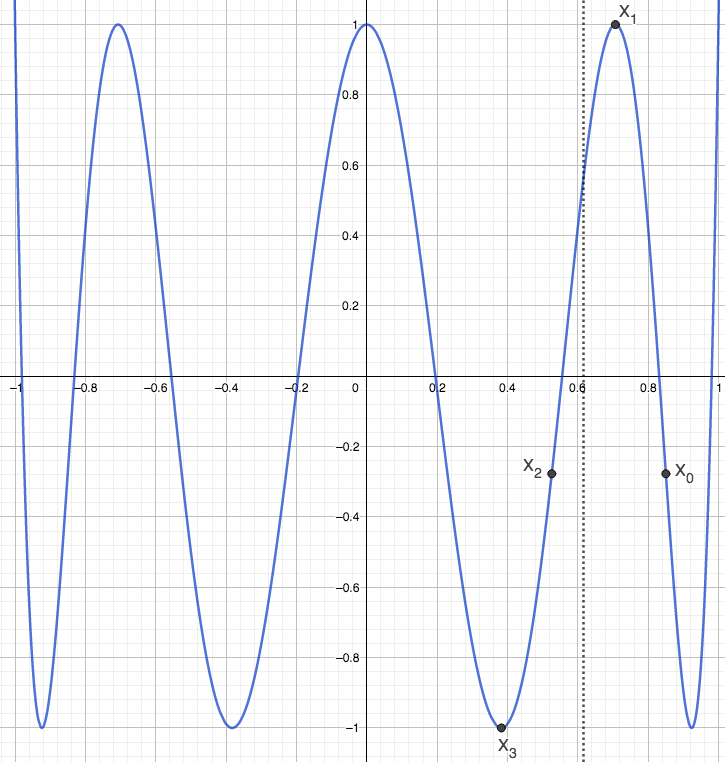}
   \includegraphics[scale=0.09]{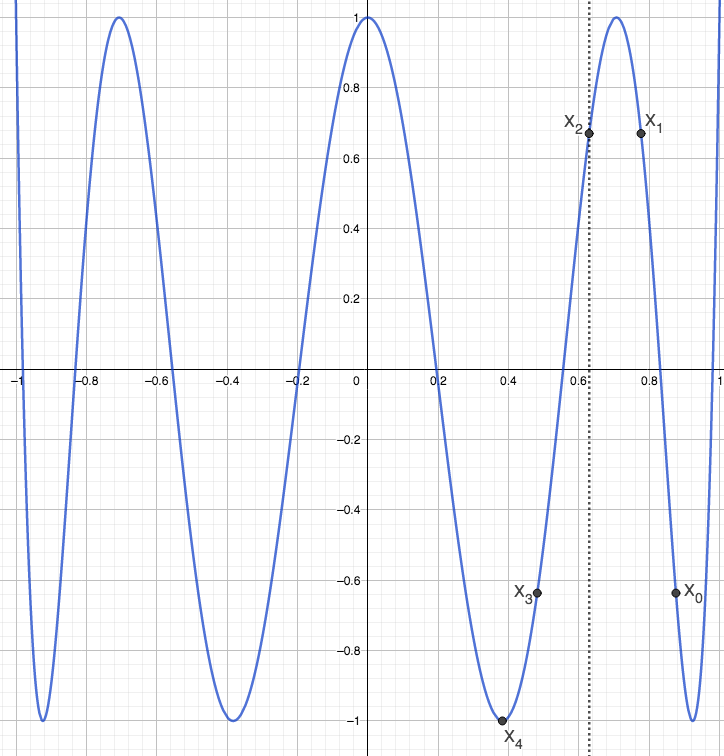}
  \includegraphics[scale=0.09]{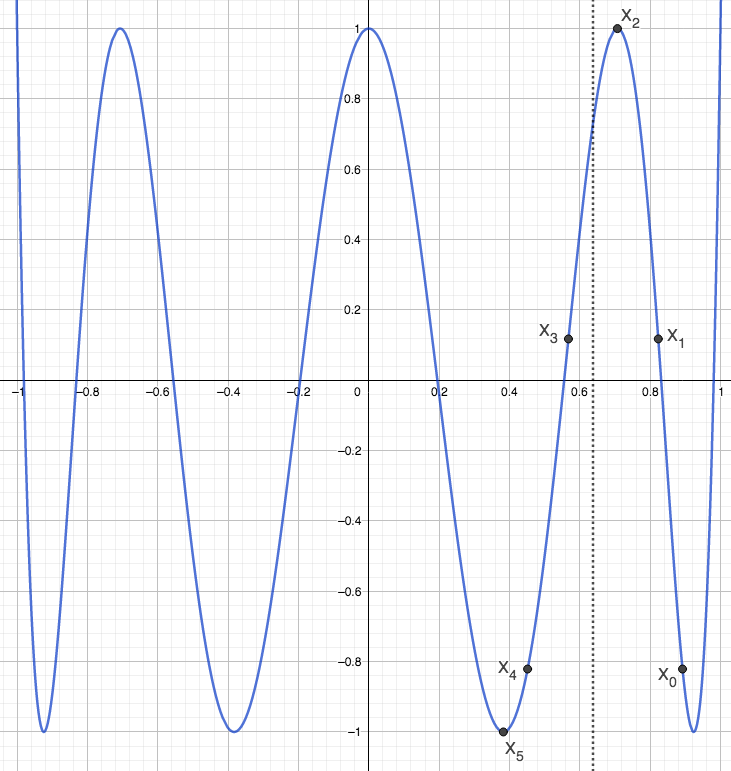}
  \includegraphics[scale=0.09]{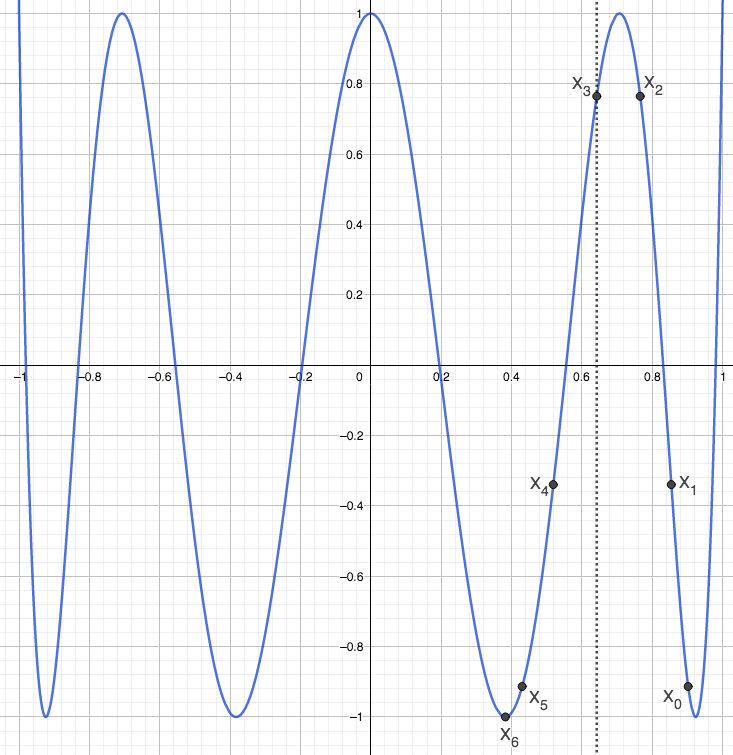}
  \includegraphics[scale=0.09]{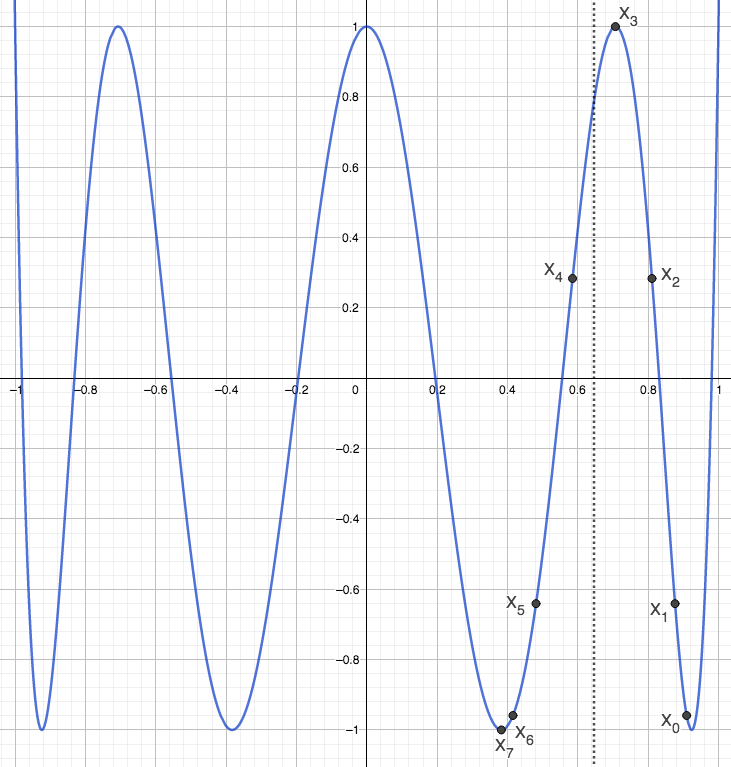}
\caption{$T_{\kappa=8}(x)$. Solutions $\F_n$. Left to right : $\F_1 \simeq 1.1874$, $\F_2 \simeq 1.2331$, $\F_3 \simeq 1.2589$, $\F_4 \simeq 1.2749$, $\F_5 \simeq 1.2853$, $\F_6 \simeq 1.2922$.}
\label{fig:test_T8increasingUpsideDownWell}
\end{figure}

Thus, we propose a generalization :
\begin{theorem} 
\label{thm_averagekk1145}
%$\F_n \nearrow \cos((j-1)\pi / \kappa) + \cos((j+1)\pi/ \kappa)$
Fix $\kappa \geq 2$. Fix $1 \leq j \leq \floor*{\kappa/2}$, $j$ even. There is a sequence $\{ \F_n \}_{n=1} ^{\infty}$, which depends on $\kappa$, s.t.\ $\{ \F_n \} \subset (2\cos((j+1)\pi/ \kappa), \cos((j-1)\pi / \kappa) + \cos((j+1)\pi/ \kappa)) \cap \boldsymbol{\Theta}_{\kappa}(\Delta)$, and $\F_n < \F_{n+1} < \F_{n+2}$, $\forall n \in \N^*$. Also, $\F_{2n-1}$, $\F_{2n} \in \boldsymbol{\Theta}_{n, \kappa}(\Delta)$, $\forall n \geq 1$, and 
\begin{equation}
\label{chain_10935}
\cos((j+1)\pi / \kappa) =: X_{n+1} < X_{n} < ... < X_{1} < X_{0} < \cos((j-1)\pi / \kappa).
%\mathcal{E}_n -\cos((j-1)\pi / \kappa) =: X_{0} < X_{1} < ... < X_{n} < X_{n+1} := \cos((j-1)\pi / \kappa).
\end{equation}
\end{theorem}

\subsection{Conjecture on the limit}

We conjecture :

\begin{conjecture} 
Let $\{ \F_n \}$ be the sequence in Theorems \ref{thm_averagekk113} and \ref{thm_averagekk1145}. Then 
$$\F_n \nearrow \cos((j-1)\pi / \kappa) + \cos((j+1)\pi/ \kappa).$$
\end{conjecture}

\section{Description of the polynomial interpolation in dimension 2}
\label{desc_scheme}

%In this section we will be illustrating using the case $\kappa=2$. Hopefully it will help to make things more understandable.

This entire section is in dimension 2. In this section we adapt the linear system \eqref{interpol_intro} to the interval $J_2(\kappa) := (2\cos(\frac{\pi}{\kappa}), 1+\cos(\frac{\pi}{\kappa}))$. This will setup our framework behind Conjecture \ref{conjecture22}. In sections \ref{appli_scheme_k2} and \ref{appli_scheme_k3} we numerically implement the equations of this section.

Fix $\kappa \geq 2$, $n \in \N^*$. First, let $\E_n, X_{0,n}, ..., X_{n+1,n}$ be the solutions of Propositions \ref{prop1} and \ref{prop2} (also Theorem \ref{thm_decreasing energy general}). Our aim is to find the coefficients $\rho_{j_q \kappa}$ of $\mathbb{A}(n) = \sum_{q=1} ^{N(n)} \rho_{j_{q} \kappa} (n) A_{j_q \kappa}$ so that a strict Mourre estimate holds on the interval $( \E_n , \E_{n-1} )$. For $n$ \underline{odd}, the linear system \eqref{interpol_intro} becomes (using notation \eqref{def:gE} and \eqref{def:GGE} instead) :

 %For a given energy $E >0$ and a $\mathbb{A} =$ \eqref{LINEAR_combinationA} we define \textit{critical values} values $E_1 \in [E-1,1]$ such that $G_E ^{\kappa} (E_1) = 0$. Wlog we always assume $E>0$, see Lemma \ref{Lemma_symmetryDelta}. 

%To determine the sequence $\{ \E_n\}_{n\geq 1}$ we start at $\E_0 := 1+\cos(\frac{\pi}{\kappa})$, the $1st$ band's right endpoint, and bootstrap our way down towards $2\cos(\frac{\pi}{\kappa})$, successively determining the band left endpoints.  

%\noindent \underline{\textit{\textbf{Conjecture 4:}}} It is possible to construct a conjugate operator $\mathbb{A}(n) = \sum_{j=1} ^{N(n)} \rho_{j \kappa} (n) A_{j \kappa}$ such that the strict Mourre estimate for $\Delta$ holds with respect to $\mathbb{A}(n)$ for every energy $E \in ( \mathcal{E}_n , \mathcal{E}_{n-1} )$. The coefficients $\rho_{j \kappa} (n)$ are determined by performing polynomial interpolation to $G_E ^{\kappa} (x) = \sum _{j=1} ^{N(n)} \rho_{j \kappa}(n) g_E ^{j \kappa} (x)$ based on the following constraints (for $n$ odd)
\begin{equation}
\label{system_interpol}
\begin{cases}
G_{\kappa} ^{\E_n} (X_{q,n}) = 0 & q = 0,...,(n-1)/2 \\
\frac{d}{dx} G_{\kappa} ^{\E_n} (X_{q,n}) = 0 & q = 1,...,(n-1)/2 \\
G_{\kappa} ^{\mathcal{E}_{n-1}} (X_{q,n-1}) = 0 & q = 0,...,(n-3)/2 \\
\frac{d}{dx} G_{\kappa} ^{\mathcal{E}_{n-1}} (X_{q,n-1}) = 0 & q = 1,...,(n-3)/2+1. \\
\end{cases}
\end{equation}
This system of $2n-1$ equations has at most rank $2n-1$, but part of our conjecture is that it always has rank $2n-1$ and so one solves for $2n-1$ unknown coefficients $\rho_{j\kappa}$. For $n$ \underline{even}, the linear system \eqref{interpol_intro} becomes (using notation \eqref{def:gE} and \eqref{def:GGE} instead) :

\begin{equation}
\label{system_interpol_even}
\begin{cases}
G_{\kappa} ^{\E_n} (X_{q,n}) = 0 & q = 0,...,n/2-1 \\
\frac{d}{dx} G_{\kappa} ^{\E_n} (X_{q,n}) = 0 & q = 1,...,n/2 \\
G_{\kappa} ^{\mathcal{E}_{n-1}} (X_{q,n-1}) = 0 & q = 0,...,n/2-1 \\
\frac{d}{dx} G_{\kappa} ^{\mathcal{E}_{n-1}} (X_{q,n-1}) = 0 & q = 1,...,n/2-1. \\
\end{cases}
\end{equation}
Again, this system of $2n-1$ equations has at most rank $2n-1$, but part of our conjecture is that it always has rank $2n-1$.

For the coefficients $\rho_{j_q \kappa}$ we will assume $\Sigma = \{ \rho_{j_1 \kappa}, \rho_{j_2 \kappa}, ..., \rho_{j_{2n} \kappa} \}$ and further always take the convention that $j_1 = 1$ and $\rho_{j_1 \kappa} = \rho_{\kappa} = 1$. Thus we have a system of $2n-1$ unknowns and $2n-1$ equations.

\begin{remark}
\label{remmmm}
Fix $n$ odd. By the second equation in \eqref{conjecture_system_conj_intro} $g_{j \kappa} ^{\mathcal{E}_n} (X_{1+q}) = g_{j \kappa} ^{\mathcal{E}_n} (X_{n-q})$ for $q=0,1,...,(n-1)/2$. So $G_{\kappa} ^{\mathcal{E}_n} (X_{1+q}) = G_{\kappa} ^{\mathcal{E}_n} (X_{n-q})$ for any choice of coefficients $\rho_{j \kappa}$. And by \eqref{problem1_tosolve_again_101_88}, $G_{\kappa} ^{\mathcal{E}_n} (X_{(n+1)/2}) = \sum_{q=0} ^{(n-1)/2} \omega_q G_{\kappa} ^{\mathcal{E}_n} (X_q)$. So to avoid obvious linear dependencies, we require the first line of system \eqref{system_interpol} only for $q=0,...,(n-1)/2$. Also a direct computation of $\frac{d}{dx} G_{\kappa} ^{\mathcal{E}_n} (x)$ shows that $\frac{d}{dx} G_{\kappa} ^{\mathcal{E}_n}(X_{1+q}) = \frac{d}{dx} G_{\kappa} ^{\mathcal{E}_n}(X_{n-q})$ for $q=0,1,...,(n-1)/2$ and any choice of coefficients $\rho_{j \kappa}$. And by Lemma \ref{symmetry_minima_E/2}, $\frac{d}{dx} G_{\kappa} ^{\mathcal{E}_n} (X_{(n+1)/2}) =0$ always holds. So to avoid obvious linear dependencies, we require the second line of system \eqref{system_interpol} only for $q=1,...,(n-1)/2$ (we don't include $q=0$ either but that is for a separate reason based only on numerical and graphical evidence).
\end{remark}

\begin{remark}
Analogous remark to \ref{remmmm} holds for $n$ even.
\end{remark}

\section{Application of Polynomial Interpolation to the case $\kappa = 2$ in dimension 2}
\label{appli_scheme_k2}

%In this Section we implement the scheme of the previous Section for the case $\kappa = 2$. We take for granted the following postulate :

%\noindent \underline{\textbf{Postulate}} : The values $\E_n$ (the band endpoints) are given by the formula $\E_n = \frac{2}{n+2}$ and the other critical values $X_1, ...,X_n$ are given by $X_q = \E_n - 1 + q \E_n$, $q = 1,...,n$. 

%\begin{remark}
%As stated in the introduction, we only have numerical evidence to justify this. 
%\end{remark}

In this section we prove Theorem \ref{OnlyMourreband} and give some graphical illustrations to justify our Conjecture \ref{conjecture22} for $n=1,2,3,4$ and $\kappa=2$. See the plots of $G_{\kappa=2} ^E(x)$ in Figure \ref{fig:test_k2_3333388aftesty}. In this Figure, note that in the left and right-most columns, $G_{\kappa=2} ^E \geq 0$, whereas in the middle column, $G_{\kappa=2} ^E (x) > 0$, $\forall x \in [E-1,1]$. Throughout this section the sequence $\E_n$ referred to is that of Theorem \ref{thm_decreasing energy general} and Propositions \ref{prop1} and \ref{prop2}. In the middle column of Figure \ref{fig:test_k2_3333388aftesty} the values of $E$ used to illustrate $G_{\kappa=2} ^E$ are, from to top to bottom, $E = 0.8, 0.6, 0.45, 0.36$. Subsections \ref{sub1_k2}, \ref{sub2_k2}, \ref{sub3_k2} and \ref{sub4_k2} detail the calculations underlying Figure \ref{fig:test_k2_3333388aftesty}.

\noindent \textit{Proof of Theorem \ref{OnlyMourreband}}. 

Fix $d=\kappa=2$. It is enough to prove that $G_{\kappa} ^E (x) = g_{2} ^E (x) + 9/14 g_4 ^E (x)$ is such that $G_{\kappa} ^E (x) >0$ for $E \in (2/3,1)$ and $x \in [E-1,1]$. Thanks to the computer, 
$G_{\kappa} ^E (x) = -(180 E / 7 ) (x^2 - E x + r_- r_+) (x^2 -E x + s_- s_+)$ where 
\footnotesize
$$r_{\pm} = \frac{1}{2} \left(E \pm \frac{\sqrt{20 - 15 E^2 - 2\sqrt{5} \sqrt{16 - 20 E^2 + 9 E^4}}}{\sqrt{15}} \right), \ \  s_{\pm} = \frac{1}{2} \left(E \pm \frac{\sqrt{20 - 15 E^2 + 2\sqrt{5} \sqrt{16 - 20 E^2 + 9 E^4}}}{\sqrt{15}} \right),$$
\normalsize
which satisfy $r_- + r_+ = s_- + s_+ = E$ and
\footnotesize
$$r_- r_+ = \left( 15 E^2 - \sqrt{5} \sqrt{16 - 20 E^2 + 9 E^4} - 10 \right)/30, \quad s_- s_+ = \left( 15 E^2 + \sqrt{5} \sqrt{16 - 20 E^2 + 9 E^4} - 10 \right)/30.$$
\normalsize
The parabola $x^2 - E x + r_- r_+$ has real roots iff $E \in [-2/3,2/3]$. Thus for $E \in (2/3,1)$, this parabola is strictly positive. On the other hand, the parabola $x^2 -E x + s_- s_+$ is strictly negative if and only if $x \in (s_- , s_+)$ (one checks that $s_-,s_+$ are real numbers if and only if $E \in [-2,2]$). For a fixed value of $E$, we want $G_{\kappa} ^E (x) >0$ for all $x \in [E-1,1]$. Thus we are led to solve $s_- \leq E-1 \leq 1 \leq s_+$ which has solutions $E \in [2/3,1] \cup [4/3,2]$. This implies the result. Note also $G_{\kappa} ^E (x) = 0$ for $E=2/3$, $x=-1/3,1/3,1$ and $E=1$, $x=0,1$, as expected.
\normalsize
\qed

\subsection{A specificity of the polynomial interpolation problem for $\kappa=2$}

We have a remark that is specific to the case $\kappa = 2$. Recall the definitions of $g_{j \kappa}$ and $G_{\kappa}$ given by \eqref{def:gE_intro} and \eqref{def:GGE_intro} respectively. For $\kappa = 2$, note that the linear span of $\{ U_{j \kappa-1} (x) : j \geq 1 \}$ equals the linear span of $\{ x^{2j-1} : j \geq 1 \}$. So we can interpret the polynomial interpolation problem of finding coefficients $\rho_{j \kappa}$ such that $G_{\kappa} (\vec{x}) > 0$, instead as, we are looking for an odd function $f_-$ such that $G_{\kappa=2} (\vec{x}) = \sum_{1 \leq i \leq d} (1-x_i^2) f_-(x_i) > 0$ for all $\vec{x} \in S_E$. We don't know if this remark is helpful but we found it still interesting enough to mention.

Table \ref{sol_kappa_2_____} gives the inputs we need to feed linear systems \eqref{system_interpol} and \eqref{system_interpol} into the computer. 

\begin{table}[H]
\footnotesize
  \begin{center}
    \begin{tabular}{c|c|c|c} % <-- Alignments: 1st column left, 2nd middle and 3rd right, with vertical lines in between
    $n$ & Left endpoint & Right endpoint & $\Sigma = $     \\ [0.2em]
      \hline
1 & $\E_1 = \frac{2}{3}$, $X_1 = \E_1 /2 = \frac{1}{3}$ & $\E_0 = 1$ & \{2,4\} \\[0.1em]
2 & $\E_2 = \frac{1}{2}$, $X_1 = 0$, $X_2 = \frac{1}{2}$ & $\E_1 = \frac{2}{3}$, $X_1 = \frac{1}{3}$ & $\{2,4,6,10\}$  \\[0.1em]
3 & $\E_3 = \frac{2}{5}$, $X_1 = -\frac{1}{5}$, $X_2 = \frac{1}{5}$, $X_3 = \frac{3}{5}$ & $\E_2 = \frac{1}{2}$, $X_1 = 0$, $X_2 =\frac{1}{2}$ & $\{2,4,6,8,10,12\}$ \\[0.1em]
4 & $\E_4 = \frac{1}{3}$, $X_1 = -\frac{1}{3}$, $X_2 = 0$, $X_3 = \frac{1}{3}$, $X_4 = \frac{2}{3}$ & $\E_3 = \frac{2}{5}$, $X_1 = -\frac{1}{5}$, $X_2 = \frac{1}{5}$, $X_3 = \frac{3}{5}$ & $\{2,4,6,8,10,12,14,20\}$
    \end{tabular}
  \end{center}
    \caption{Data to setup polynomial interpolation on $(\E_n, \E_{n-1})$. $\kappa=2$. $d=2$.}
        \label{sol_kappa_2_____}
\end{table}
\normalsize

\subsection{$1st$ band ($n=1$)} 
\label{sub1_k2}
%Left endpoint : $\E_1 = \frac{2}{3}$, $X_1 = \E_1 /2 = \frac{1}{3}$. Right endpoint : $\E_0 = 1$. 
By Lemma \ref{lemSUMcosINTRO}, $\E_0 \in \boldsymbol{\Theta}_{0,\kappa}(\Delta)$ and $G_{\kappa} ^{\E_0} (\E_0-1) = 0$ is always true. By Lemma \ref{symmetry_minima_E/2}, $\frac{d}{dx} G_{\kappa} ^{\E_1} (X_1) = 0$ is always true. So these are fake constraints. Then we compute 
%and there are no $\tilde{X}_n$'s between $\E_0 - 1$ and $1$. 
%We suppose $\Sigma = \{2,4\}$. 

\begin{equation*}
M = 
\begin{bmatrix}
g_{2} ^{\E_1} (\E_1-1) & g_{4} ^{\E_1} (\E_1-1) 
\end{bmatrix}
=
\begin{bmatrix}
-\frac{16}{27} & \frac{224}{243} 
\end{bmatrix}, 
\quad 
\rho = 
\begin{bmatrix}
1, & \rho_{4}
\end{bmatrix} ^T.
\quad 
M \rho = 0 \Rightarrow \rho_{4} = 9/14.
\end{equation*}

\color{black}

%\begin{proposition}
%\label{proofMourrek2band1}
%The linear combination constructed above gives strict positivity on $(2/3,1)$, i.e.\ $G_{\kappa=2} ^E (x) = g_{2} ^E (x) + 9/14 g_4 ^E (x)$ is such that $G_{\kappa=2} ^E (x) >0$ for $E \in (2/3,1)$ and $x \in [E-1,1]$, and $G_{\kappa=2} ^E (x) = 0$ for $E=2/3$, $x=-1/3,1/3,1$ and $E=1$, $x=0,1$.
%\end{proposition}
%\begin{proof} 
%\end{proof}

\subsection{$2nd$ band ($n=2$)} 
\label{sub2_k2}
%Left endpoint $\E_2 = \frac{1}{2}$, $X_1 = 0$, $X_2 = \frac{1}{2}$. On the right we have $\E_1 = \frac{2}{3}$, $X_1 = \frac{1}{3}$. We suppose $\Sigma = \{2,4,6,10\}$. 
Matrix $M$ equals 

\begin{equation*}
\begin{bmatrix}
g_{2} ^{\E_2} (\E_2-1) & g_{4} ^{\E_2} (\E_2-1) & g_{6} ^{\E_2} (\E_2-1) & g_{10} ^{\E_2} (\E_2-1) \\[0.5em] 
\frac{d}{dx} g_{2} ^{\E_2} (X_1)  & \frac{d}{dx} g_{4} ^{\E_2} (X_1) & \frac{d}{dx} g_{6} ^{\E_2} (X_1) & \frac{d}{dx} g_{10} ^{\E_2} (X_1) \\[0.5em] 
g_{2} ^{\E_1} (\E_1-1) & g_{4} ^{\E_1} (\E_1-1) & g_{6} ^{\E_1} (\E_1-1) & g_{10} ^{\E_1} (\E_1-1)  \\[0.5em]  
\end{bmatrix}
=
\begin{bmatrix}
- 3/4 & 3/4 & 0 & 3/4  \\[0.5em] 
3/2 & -13/2 & 12 & 9/2   \\[0.5em] 
-\frac{16}{27} & \frac{224}{243} & -\frac{1840}{2187} & \frac{42416}{177147} \\[0.5em] 
\end{bmatrix}.
\end{equation*}

\begin{comment}
\begin{equation*}
\begin{bmatrix}
g_{E_L} ^{2} (E_L-1) & g_{E_L} ^{4} (E_L-1) & g_{E_L} ^{6} (E_L-1) & g_{E_L} ^{10} (E_L-1)  \\[1em] 
g_{E_L} ^{2} (X_1) & g_{E_L} ^{4} (X_1) & g_{E_L} ^{6} (X_1) & g_{E_L} ^{10} (X_1)  \\[1em] 
\frac{d}{dE_1} g_{E_L} ^{2} (X_1) & \frac{d}{dE_1} g_{E_L} ^{4} (X_1)  & \frac{d}{dE_1} g_{E_L} ^{6} (X_1) & \frac{d}{dE_1} g_{E_L} ^{10} (X_1)  \\[1em] 
g_{E_R} ^{2} (E_R-1) & g_{E_R} ^{4} (E_R-1) & g_{E_R} ^{6} (E_R-1) & g_{E_R} ^{10} (E_R-1)  \\[1em] 
g_{E_R} ^{2} (\tilde{X}_1) & g_{E_R} ^{4} (\tilde{X}_1) & g_{E_R} ^{6} (\tilde{X}_1) & g_{E_R} ^{10} (\tilde{X}_1)  \\[1em] 
\frac{d}{dE_1} g_{E_R} ^{2} (\tilde{X}_1) & \frac{d}{dE_1} g_{E_R} ^{4} (\tilde{X}_1) & \frac{d}{dE_1} g_{E_R} ^{6} (\tilde{X}_1) & \frac{d}{dE_1} g_{E_R} ^{10} (\tilde{X}_1)  
\end{bmatrix}
=
\begin{bmatrix}
- 3/4 & 3/4 & 0 & 3/4  \\[1em] 
3/4 & -3/4 & 0 & -3/4  \\[1em] 
3/2 & -13/2 & 12 & 9/2   \\[1em] 
-\frac{16}{27} & \frac{224}{243} & -\frac{1840}{2187} & \frac{42416}{177147} \\[1em] 
\frac{32}{27} & -\frac{448}{243} & \frac{3680}{2187} & -\frac{84832}{177147} \\[1em] 
0 & 0 & 0 & 0
\end{bmatrix}.
\end{equation*}
\end{comment}
%To the system $M \rho =0$ one finds the solution 
%$$\rho = [1, \rho_{4}, \rho_{6}, \rho_{10} ] ^{T} = \bigg[1, \frac{598}{787}, \frac{464}{2361}, \frac{189}{787}\bigg]^{T} \simeq [1, 0.7598, 0.1965, 0.2401] ^{T}.$$ 
$$M \rho =0 \Rightarrow \rho = [1, \rho_{4}, \rho_{6}, \rho_{10} ] ^{T} = [1, 598/787, 464/2361, 189/787 ]^{T} \simeq [1, 0.7598, 0.1965, 0.2401] ^{T}.$$

\color{black}

\subsection{$3rd$ band ($n=3$)}
\label{sub3_k2}
%$\E_3 = \frac{2}{5}$, $X_1 = -\frac{1}{5}$, $X_2 = \frac{1}{5}$, $X_3 = \frac{3}{5}$. On the right we have $\E_2 = \frac{1}{2}$, $X_1 = 0$, $X_2 =\frac{1}{2}$.  We suppose $\Sigma = \{2,4,6,8,10,12\}$. 

Matrix $M$ equals

\begin{equation*}
\begin{bmatrix}
g_{2} ^{\E_3} (\E_3-1) & ... & g_{12} ^{\E_3} (\E_3-1)  \\[0.5em] 
g_{2} ^{\E_3} (X_1) & ... & g_{12} ^{\E_3} (X_1) \\[0.5em] 
\frac{d}{dx} g_{2} ^{\E_3} (X_1) & ... & \frac{d}{dx} g_{12} ^{\E_3} (X_1) \\[0.5em] 
g_{2} ^{\E_2} (\E_2-1) & ... & g_{12} ^{\E_2} (\E_2-1)  \\[0.5em] 
\frac{d}{dx}g_{2} ^{\E_2} (X_1) & ... & \frac{d}{dx}g_{12} ^{\E_2} (X_1)
\end{bmatrix}
= 
\begin{bmatrix}
-\frac{96}{125} & \frac{1344}{3125} & \frac{41184}{78125} & -\frac{1416576}{1953125} & -\frac{5907936}{48828125} & \frac{968071104}{1220703125} \\[0.5em] 
\frac{48}{125} & \frac{864}{3125} & -\frac{112752}{78125} & \frac{3328704}{1953125} & - \frac{37319952}{48828125} & -\frac{174668256}{1220703125} \\[0.5em]
% \frac{48}{125} & \frac{864}{3125} & -\frac{112752}{78125} & \frac{3328704}{1953125} & -\frac{37319952}{48828125} & -\frac{174668256}{1220703125}  \\[1em]
\frac{48}{25} & -\frac{32}{5} & \frac{93072}{15625} & \frac{357824}{78125} & -\frac{139067664}{9765625} &  \frac{485868192}{48828125}  \\[0.5em]
- 3/4 & 3/4 & 0 & -3/4 & 3/4 & 0 \\[0.5em]
3/2 & -13/2 & 12 & -23/2 & 9/2 & 0
\end{bmatrix}.
\end{equation*}
The system $M \rho =0$ has the solution $\rho = [ \rho_2, \rho_4, \rho_6, \rho_8, \rho_{10}, \rho_{12}]^{T}$ equals
\begin{equation*}
\bigg[ 1, \frac{879159}{627154}, \frac{368515}{313577}, \frac{419505}{627154}, \frac{83750}{313577}, \frac{1146875}{14424542} \bigg]^{T}  \simeq [1, 1.4018, 1.1751, 0.6689, 0.2670, 0.0795 ]^{T}.
\end{equation*}
%$G_E ^{\kappa} (E_1)$ is plotted in Figure \ref{fig:test_k2_3333388af} for some values $E \in [\E_3, \E_2]$. Denote $R_k$ the $kth$ row of $M$. The first 3 rows of $M$ are linearly dependent : $R_3 = -2 R_1 -2R_2$. $R_5$ has zeros due to Lemma \ref{symmetry_minima_E/2}. $R_6$ and $R_7$ are linearly dependent (as it was the case in the previous iteration). With the computer one checks that the first 5 rows of $M$ have rank 3, while the bottom 3 rows have rank 2. Overall 
$M$ has rank 5, which is the number of unknown parameters in $\rho$.

\subsection{$4th$ band $(n=4)$} 
\label{sub4_k2}
%$\E_4 = \frac{1}{3}$, $X_1 = -\frac{1}{3}$, $X_2 = 0$, $X_3 = \frac{1}{3}$, $X_4 = \frac{2}{3}$ ; $\E_3 = \frac{2}{5}$, $X_1 = -\frac{1}{5}$, $X_2 = \frac{1}{5}$, $X_3 = \frac{3}{5}$.  We suppose $\Sigma = \{2,4,6,8,10,12,14,20\}$. 

Matrix $M$ is too large to write out. $M \rho =0 \Rightarrow [\rho_2, \rho_4, \rho_6, \rho_8, \rho_{10}, \rho_{12}, \rho_{14},\rho_{20}]^{T} = $

\footnotesize
\begin{equation*}
\begin{aligned}
& \Big[1, \frac{91312852394687883497633}{60342582799484620292280}, \frac{2214634471921172327027}{1508564569987115507307}, \frac{7972149322756114102612}{7542822849935577536535}, \frac{1480399911203653752057}{2514274283311859178845}, \\
&\quad \frac{305875819949113732527}{1149382529513992767472}, \frac{6549161566940548875}{71836408094624547967}, - \frac{16094965193651953125}{8045677706597949372304} \Big]^{T} \\
&\simeq [1, 1.5132, 1.4680, 1.0569, 0.5887,  0.2661, 0.0911, -0.0020]^{T}.
\end{aligned}
\end{equation*}
\normalsize
%The function $e^{-k(\frac{E}{2}-E_1)^2}G_E ^{\kappa} (E_1)$ is plotted in Figure \ref{fig:test_k2_333338899} for some values $E \in [\E_4, \E_3]$, and some $k>0$. The factor $e^{-k(\frac{E}{2}-E_1)^2}$ is added to help visualisation. $M$ consists of 10 rows. According to Python, the first 3 rows of the matrix $M$ are linearly dependent : $R_3 = - R_1 - R_2$. The first 5 rows of $M$, which correspond to the constraints linked to the band left endpoint, $E_L = \frac{1}{3}$, have rank 4, while the last 5 rows of $M$, which correspond to the constraints linked to the band right endpoint, $E_R = \frac{2}{5}$, have rank 3. Overall 

$M$ has rank 7, which corresponds to the number of unknown parameters in $\rho$.

\begin{comment}
\begin{figure}[H]
  \centering
 \includegraphics[scale=0.168]{k2.left.band3}
  \includegraphics[scale=0.169]{k2.middle.band3}
  \includegraphics[scale=0.169]{k2.right.band3}
\caption{$G_E^{\kappa=2}(E_1)$ for $3rd$ band, $E_1 \in [E-1,1]$. From left to right: $E=\frac{2}{5}$, $E=0.45$, $E=\frac{1}{2}$. $G_E^{\kappa=2}(E_1) > 0$ in the middle picture, but not in the other two. }
\label{fig:test_k2_3333388af}
\end{figure}
\end{comment}

\begin{figure}[H]
  \centering
   \includegraphics[scale=0.198]{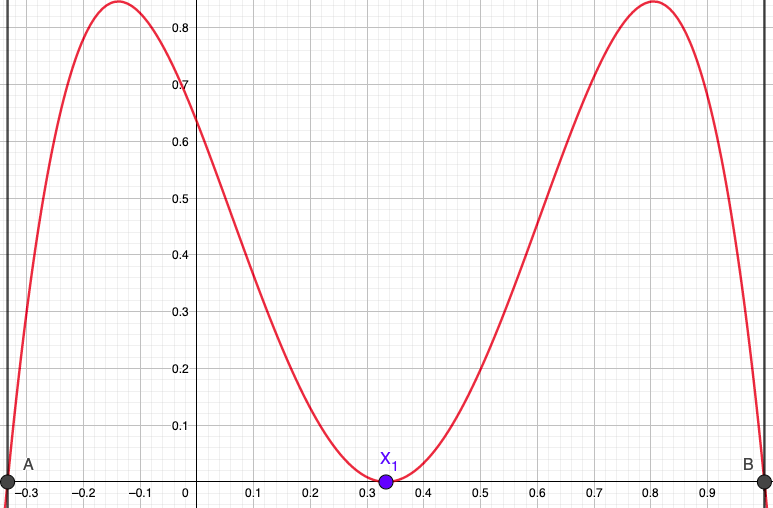}
  \includegraphics[scale=0.185]{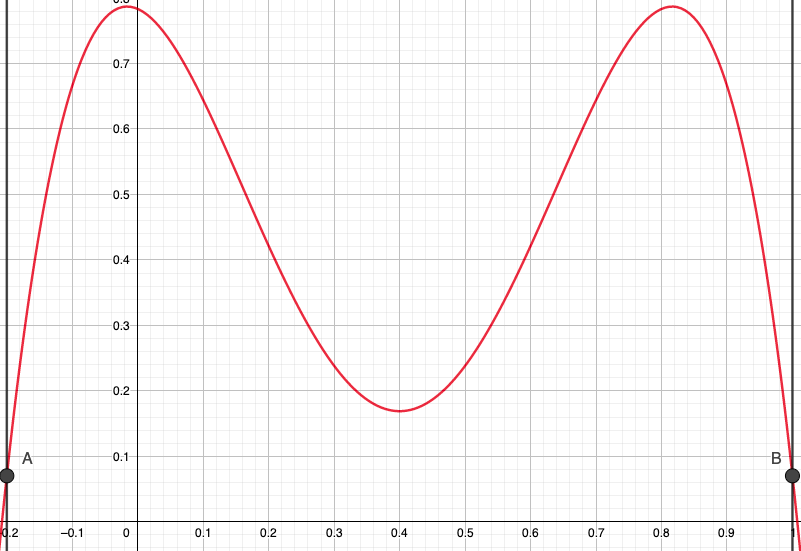}
  \includegraphics[scale=0.155]{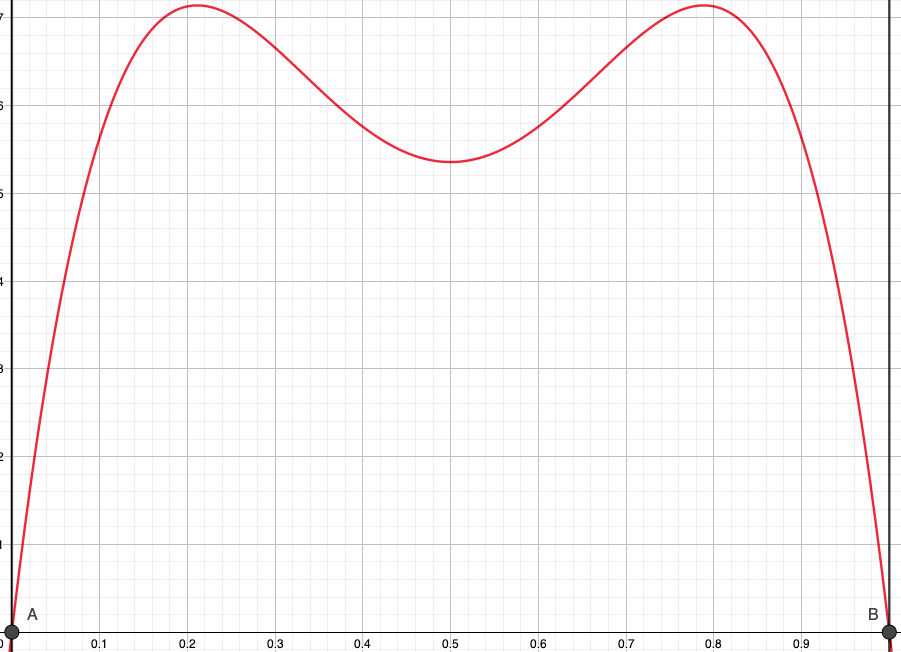}
% \caption{$G_E^{\kappa=2}(E_1)$ for $1st$ band, $E_1 \in [E-1,1]$. From left to right: $E=\frac{2}{3}$, $E=0.8$, $E=1$.  $G_E^{\kappa=2}(E_1) > 0$ in the middle picture, but not in the other two. }
\\
   \includegraphics[scale=0.164]{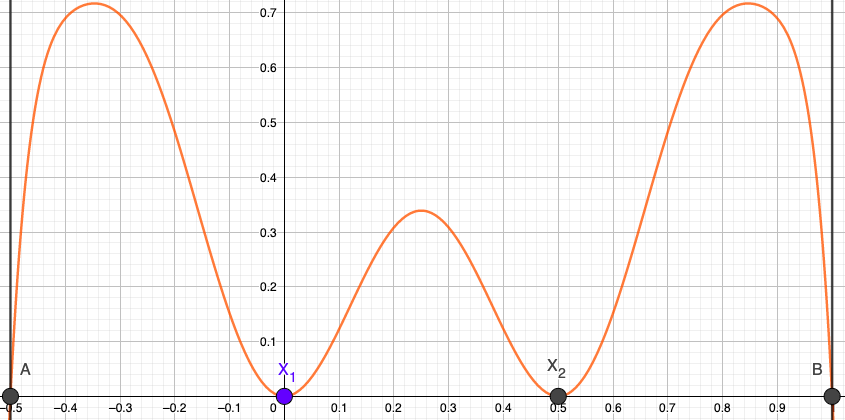}
  \includegraphics[scale=0.205]{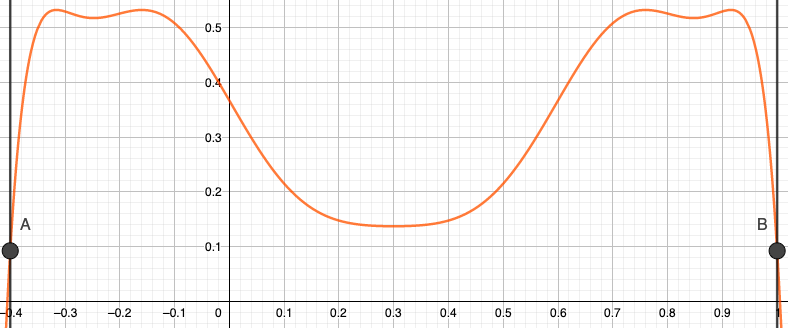}
  \includegraphics[scale=0.155]{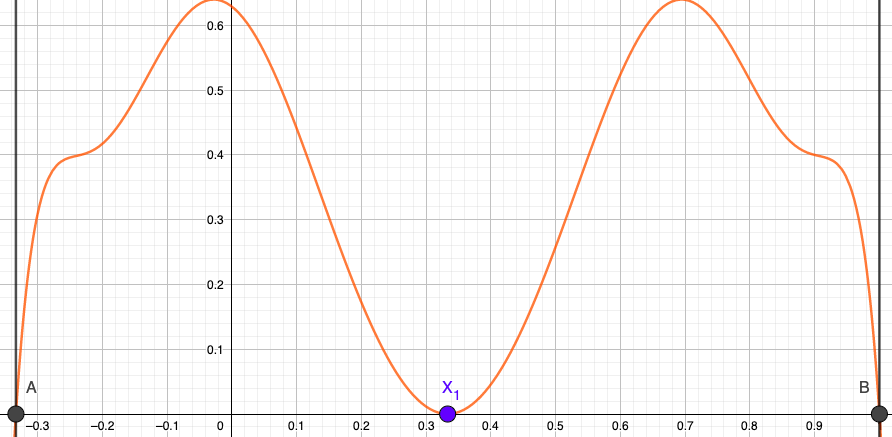}
  \\
%\caption{$G_E^{\kappa=2}(E_1)$ for $2nd$ band, $E_1 \in [E-1,1]$. From left to right: $E=\frac{1}{2}$, $E=0.6$, $E=\frac{2}{3}$. $G_E^{\kappa=2}(E_1) >0$ in the middle picture, but not in the other two. }
 \includegraphics[scale=0.168]{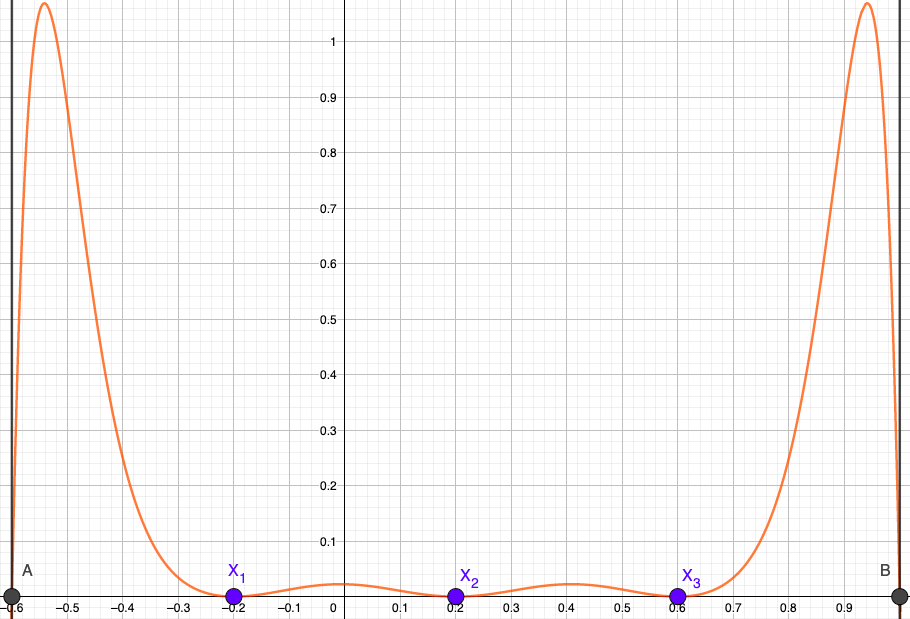}
  \includegraphics[scale=0.169]{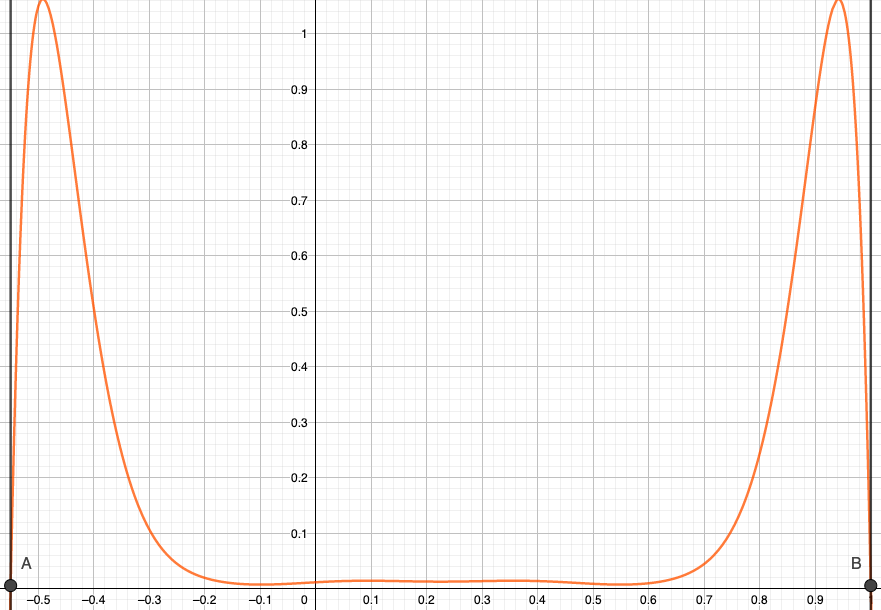}
  \includegraphics[scale=0.169]{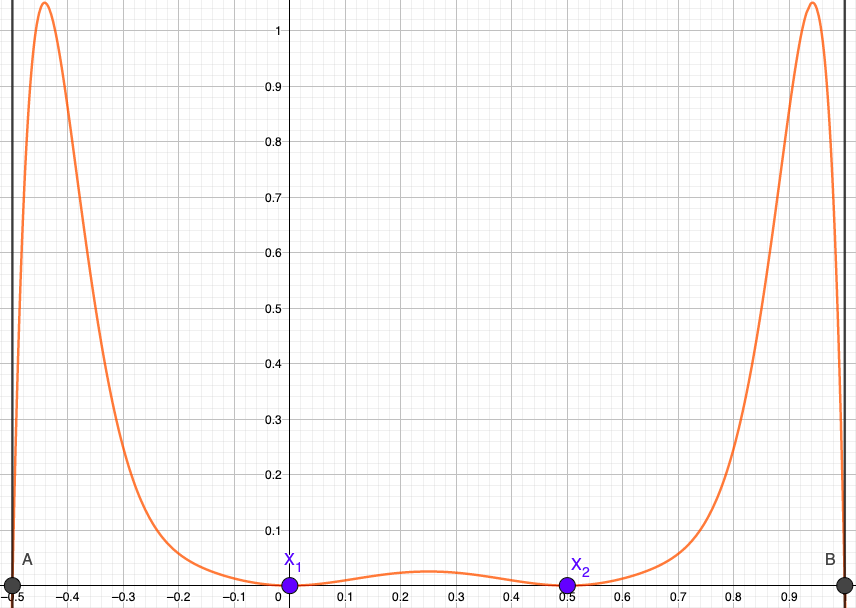}
% \caption{$G_E^{\kappa=2}(E_1)$ for $3rd$ band, $E_1 \in [E-1,1]$. From left to right: $E=\frac{2}{5}$, $E=0.45$, $E=\frac{1}{2}$. $G_E^{\kappa=2}(E_1) > 0$ in the middle picture, but not in the other two. }
\\
 \includegraphics[scale=0.2]{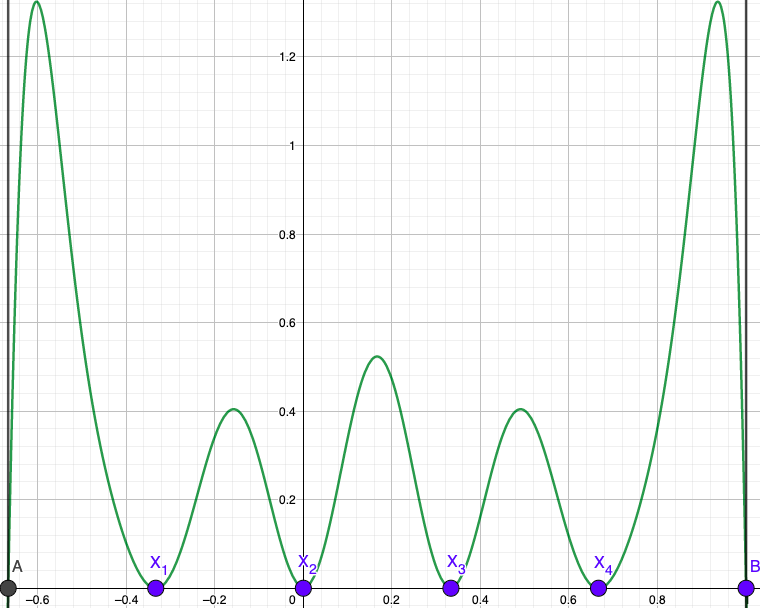}
  \includegraphics[scale=0.2]{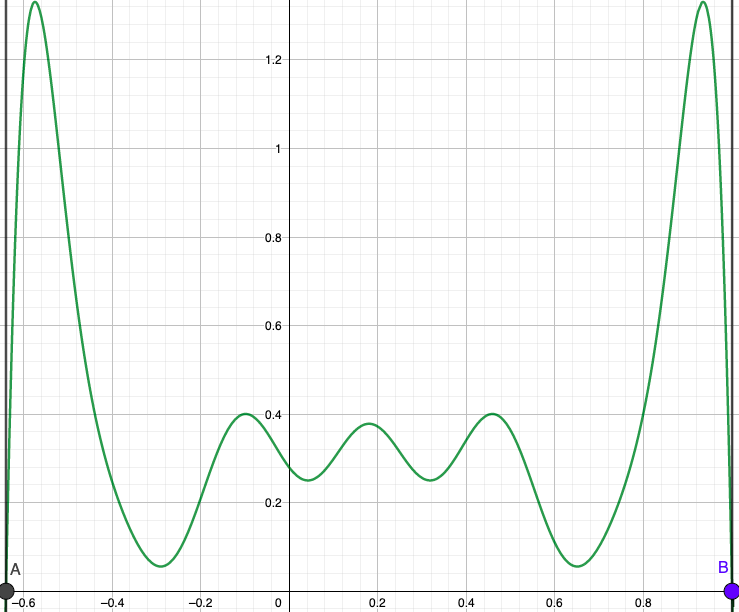}
  \includegraphics[scale=0.205]{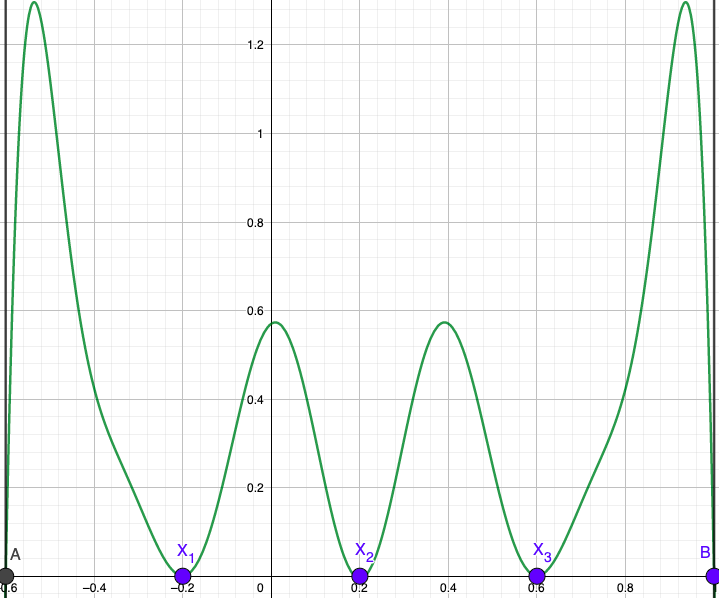}
  % $E_1 \in [E-1,1]$. From left to right: $E=\frac{1}{3}$, $0.36$, $\frac{2}{5}$. $G_E^{\kappa=2}(E_1) >0$ in middle picture, but not in the other two.
\caption{Plots of $G_{\kappa=2}^E (x)$, $x \in [E-1,1]$. Rows correspond to $n=1,2,3,4$. Left column : $E = \E_n$. Middle : some $E \in (\E_n, \E_{n-1})$. Right column : $E = \E_{n-1}$.}
\label{fig:test_k2_3333388aftesty}
\end{figure}

\color{black}

%These graphs are the only evidence we possess to justify our claim that the function $E_1 \mapsto G_E^{\kappa=2}(E_1) >0$ for $E \in (\frac{1}{3},\frac{2}{5})$.

\begin{comment}
\begin{figure}[H]
  \centering
 \includegraphics[scale=0.2]{k2,left,band44}
  \includegraphics[scale=0.2]{k2,middle,band44}
  \includegraphics[scale=0.205]{k2,right,band44}
\caption{$e^{-k(\frac{E}{2}-E_1)^2} G_E^{\kappa=2}(E_1)$ for $4th$ band. $E_1 \in [E-1,1]$. From left to right: $E=\frac{1}{3}$, $0.36$, $\frac{2}{5}$. $G_E^{\kappa=2}(E_1) >0$ in middle picture, but not in the other two. }
\label{fig:test_k2_333338899}
\end{figure}
\end{comment}

\color{black}

\begin{remark}
%When matrix $M$ is filled with exact entries, we can check the linear dependencies between its rows using Python's \textit{linsolve()} which allows for symbolic (and thus exact) computations. 

%When $M$ is filled with floats (because the exact expressions are either too complicated or unknown), 
We check the linear dependencies between the rows of $M$ using linear regression (Python's \textit{statsmodels} for example) and assess based on the $R^2$ among various statistics. 
\end{remark}

\section{Application of Polynomial Interpolation for $\kappa = 3$ in dimension 2}
\label{appli_scheme_k3}

%In this Section we apply the scheme for $\Delta$ in dimension 2, $\kappa=3$. We note 
%\begin{equation}
%\label{T3}
%T_3(x) = 4 x^3-3x \quad \text{and} \quad T_3(x) = T_3(y) \Leftrightarrow (x-y) [4(x^2+xy+y^2)-3]=0.
%\end{equation} 

%In this Section we shall refer to the constraints $\E_n \in J_2 = [2\cos(\pi/3), 1+ \cos(\pi/3)] = (1, 1.5)$ and $\E_n-1 < X_1 < X_2 < ... <X_n < 1$ as the \textit{context}.

In this section we give graphical illustrations to justify our Conjecture \ref{conjecture22} for $n=1,2,3,4$ and $\kappa=3$. See the plots of $G_{\kappa=3} ^E(x)$ in Figure \ref{fig:test_k2_33333si}. In this Figure, note that in the left and right-most columns, $G_{\kappa=3} ^E \geq 0$, whereas in the middle column, $G_{\kappa=3} ^E (x) > 0$, $\forall x \in [E-1,1]$. Throughout this section the sequence $\E_n$ referred to is that of Theorem \ref{thm_decreasing energy general} and Propositions \ref{prop1} and \ref{prop2}. In the middle column of Figure \ref{fig:test_k2_33333si} the values of $E$ used to illustrate $G_{\kappa=3} ^E$ are, from to top to bottom, $E = 1.4, 1.26, 1.2, 1.15$. Subsections \ref{sub1_k2}, \ref{sub2_k2}, \ref{sub3_k2} and \ref{sub4_k2} detail the calculations underlying Figure  \ref{fig:test_k2_33333si}.

Table \ref{sol_kappa_3_____} gives the inputs we need to feed linear systems \eqref{system_interpol} and \eqref{system_interpol} into the computer. 

\begin{table}[H]
\footnotesize
  \begin{center}
    \begin{tabular}{c|c|c|c} % <-- Alignments: 1st column left, 2nd middle and 3rd right, with vertical lines in between
    $n$ & Left endpoint & Right endpoint & $\Sigma = $     \\ [0.2em]
      \hline
1 & $\E_1 = \frac{5+3\sqrt{2}}{7} \simeq 1.3203$, $X_1 = \frac{5+3\sqrt{2}}{14}$ & $\E_0 = 1+\cos(\frac{\pi}{3}) = \frac{3}{2}$ & $\{3,6\}$ \\[0.1em]
2 & $\E_2 = (9 + \sqrt{33})/12 \simeq 1.2287$, $X_1 = 1/2$, $X_2 = (3+\sqrt{33})/12$ & $=$ left endpt data of $n=1$  & $\{3,6,9,18\}$ \\[0.1em]
3 & $\E_3 \simeq 1.1737$, $X_1 \simeq 0.4077$, $X_2 = \E_3/2$, $X_3 = \E_3 - X_1$. & $=$ left endpt data of $n=2$ &$\{3, 6, 9, 12, 15, 18 \}$  \\[0.1em]
4 & $\E_4 \simeq 1.1375$, $X_1 \simeq 0.3484$, $X_2 = 1/2$, $X_3 = \E_4 - X_2$, $X_4 = \E_4 - X_1$ & $=$ left endpt data of $n=3$ & $\{3, 6, 9, 12, 15, 18, 21, 36\}$  \\[0.1em]
    \end{tabular}
  \end{center}
    \caption{Data to setup polynomial interpolation on $(\E_n, \E_{n-1})$. $\kappa=3$. $d=2$.}
        \label{sol_kappa_3_____}
\end{table}
\normalsize

\subsection{$1st$ band $(n=1)$} 
\label{sub1_k3}
%$\E_1 = \frac{5+3\sqrt{2}}{7} \simeq 1.3203$, $X_1 = \frac{5+3\sqrt{2}}{14}$. On the right, $\E_0 = 1+\cos(\frac{\pi}{3}) = \frac{3}{2}$. We choose $\Sigma = \{3,6\}$ and get 
$M \rho =0 \Rightarrow \rho = [\rho_{3}, \rho_6 ]^{T} = [1, \frac{170 - 81 \sqrt{2}}{92}]^{T} \simeq [1,0.6027]^{T}$. 

%The function $G_E ^{\kappa} (E_1)$ is plotted in Figure \ref{fig:test_k2_33333si} for some values $E \in [\E_1, \E_0]$. 

%% For this same linear combination of conjugate operators one may wonder if there is strict positivity on other intervals besides $(\frac{5+3\sqrt{2}}{7}, \frac{3}{2})$. The answer is yes. Reading off the graph positivity occurs on
%% $$(-2, -1.981) \cup (-1.961, -1.5) \cup (-1,-0.749) \cup (0.5172, 0.640) \cup ((5+3\sqrt{2})/7, 1.5) \cup (1.661, 2).$$

%Highlighted in red are intervals that cannot be obtained with just $\Sigma = [3]$ (see Table \ref{tab:table1011}).
%\rho_6 = 0.6027032766062967505
% E_L = 1.32037724101704

\begin{comment}
\begin{figure}[H]
  \centering
 \includegraphics[scale=0.207]{k3,left,band1}
  \includegraphics[scale=0.22]{k3,middle,band1}
  \includegraphics[scale=0.206]{k3,right,band1}
\caption{$G_E^{\kappa=3}(E_1)$. $E_1 \in [E-1,1]$. From left to right: $E= \E_1 = \frac{5+3\sqrt{2}}{7}$, $E=1.4$, $E= \E_0 = \frac{3}{2}$. $G_E^{\kappa=3}(E_1) > 0$ in the middle picture, but not in the other two. }

\end{figure}
\end{comment}

\begin{figure}[H]
  \centering
 \includegraphics[scale=0.207]{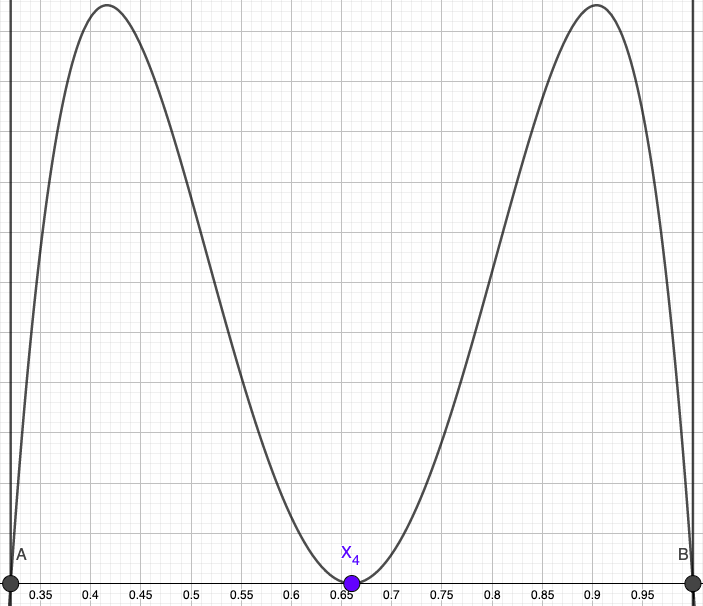}
  \includegraphics[scale=0.22]{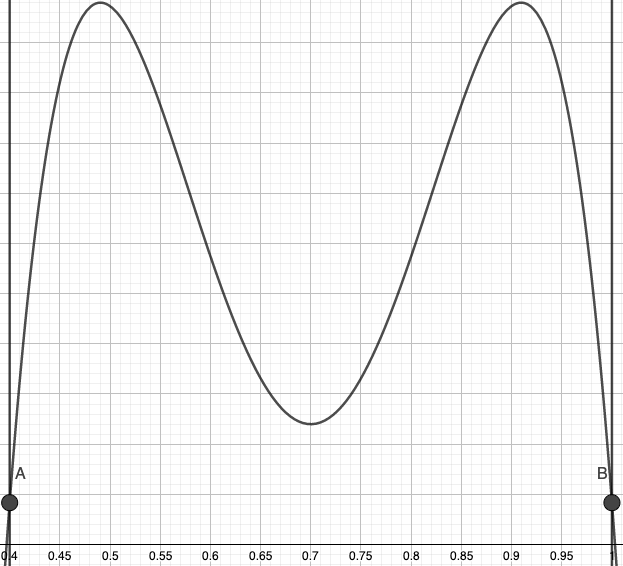}
  \includegraphics[scale=0.206]{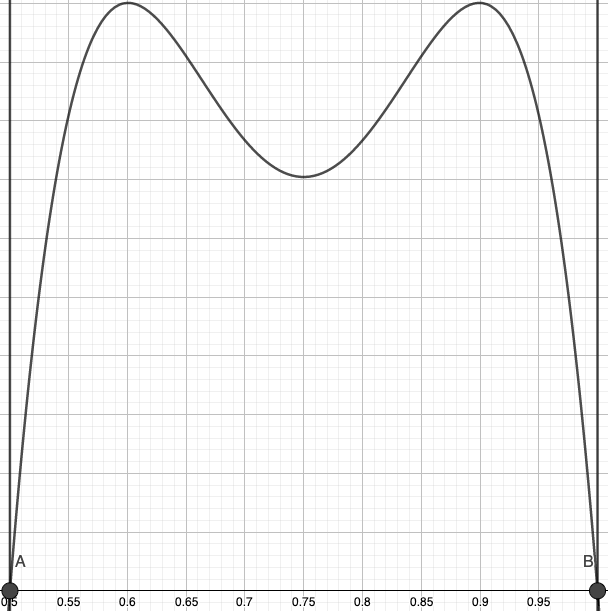}
%\caption{$G_E^{\kappa=3}(E_1)$. $E_1 \in [E-1,1]$. From left to right: $E= \E_1 = \frac{5+3\sqrt{2}}{7}$, $E=1.4$, $E= \E_0 = \frac{3}{2}$. $G_E^{\kappa=3}(E_1) > 0$ in the middle picture, but not in the other two. }
\\
 \includegraphics[scale=0.18]{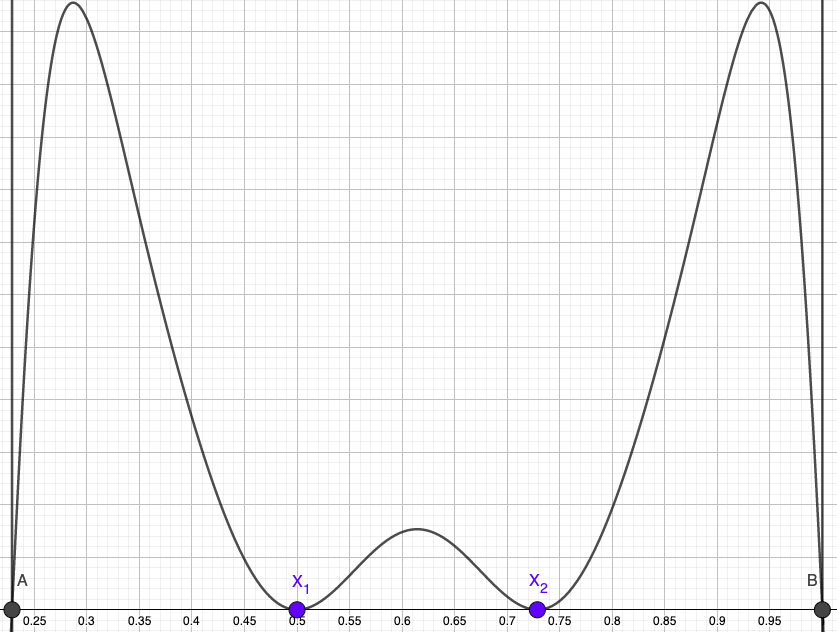}
  \includegraphics[scale=0.185]{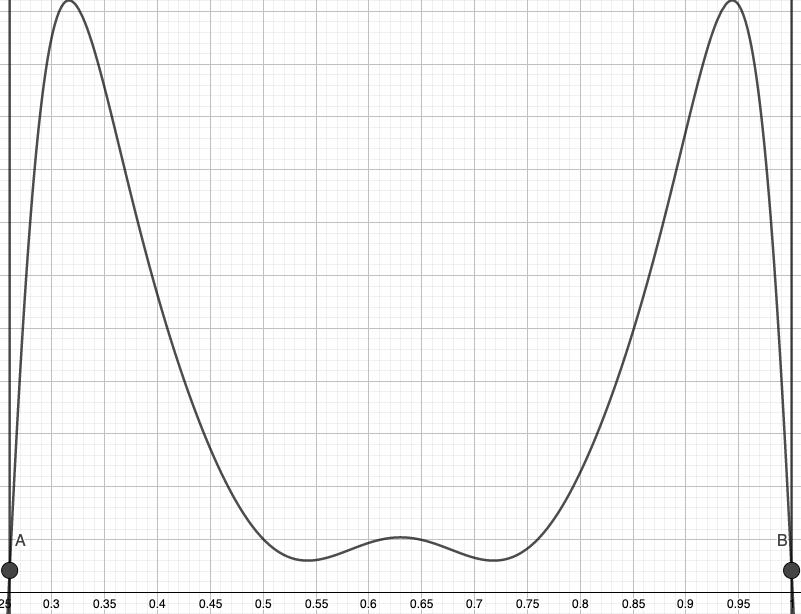}
  \includegraphics[scale=0.20]{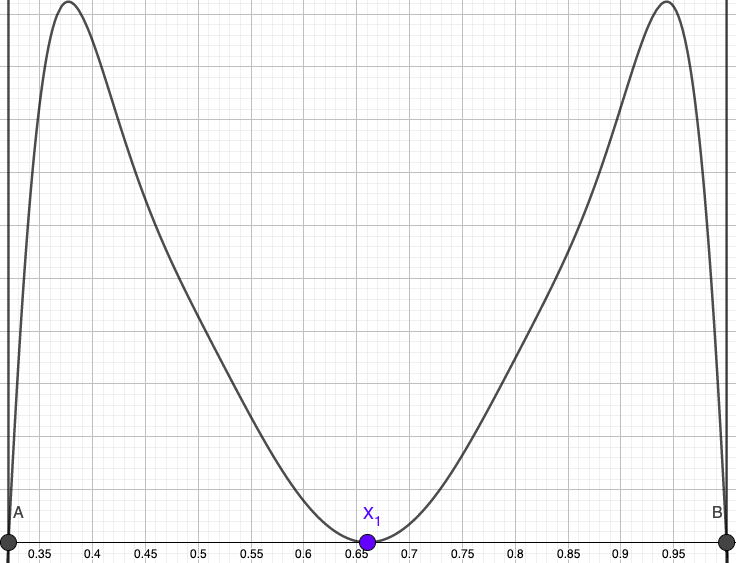}
%\caption{$G_E^{\kappa=3}(E_1)$. $E_1 \in [E-1,1]$. From left to right: $E=\E_2 = \frac{9+\sqrt{33}}{12}$, $E=1.26$, $E=\E_1 = \frac{5+3\sqrt{2}}{7}$. $G_E^{\kappa=3}(E_1) > 0$ in the middle picture, but not in the other two. }
\\
 \includegraphics[scale=0.205]{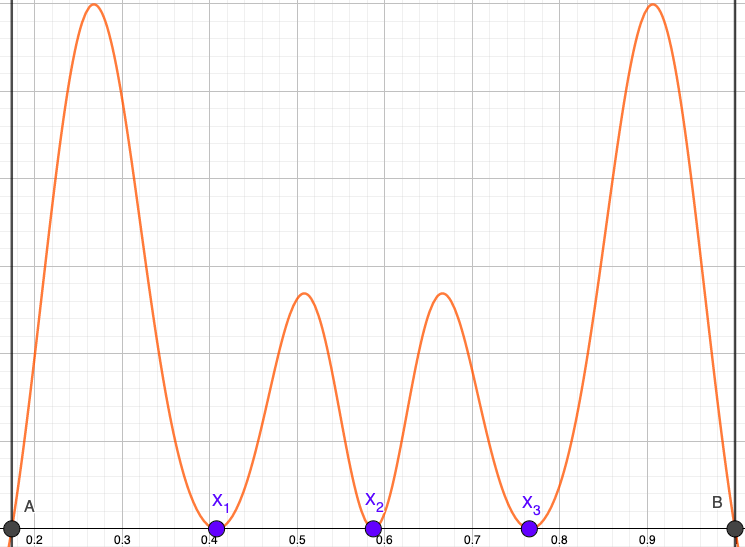}
  \includegraphics[scale=0.175]{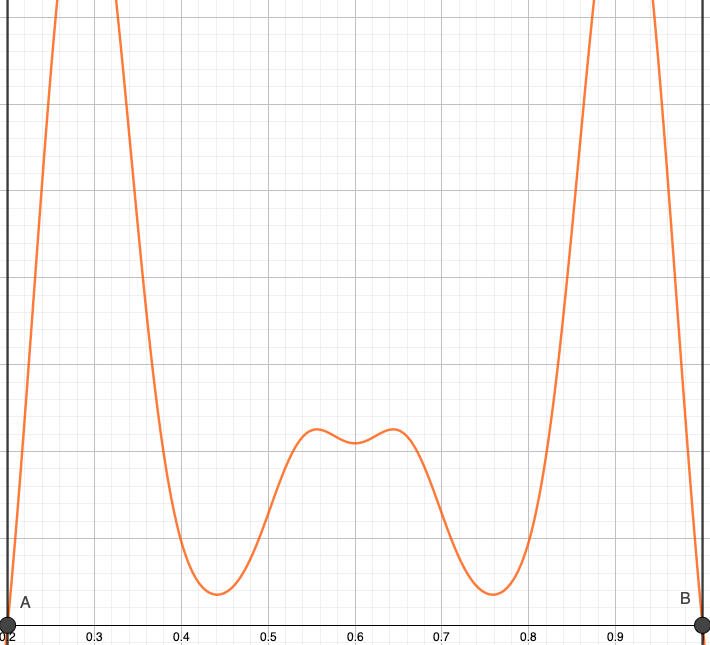}
  \includegraphics[scale=0.19]{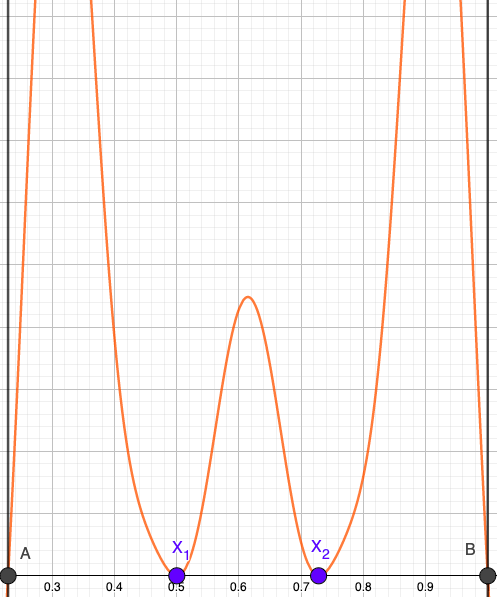}
%\caption{$e^{-k(\frac{E}{2}-E_1)^2}G_E^{\kappa=3}(E_1)$, $k >0$. $E_1 \in [E-1,1]$. Left to right: $E=\E_3$, $E=1.2$, $E= \E_2 =\frac{9+\sqrt{33}}{12}$. $G_E^{\kappa=3}(E_1) > 0$ in middle picture, but not in the other two.}
\\
 \includegraphics[scale=0.22]{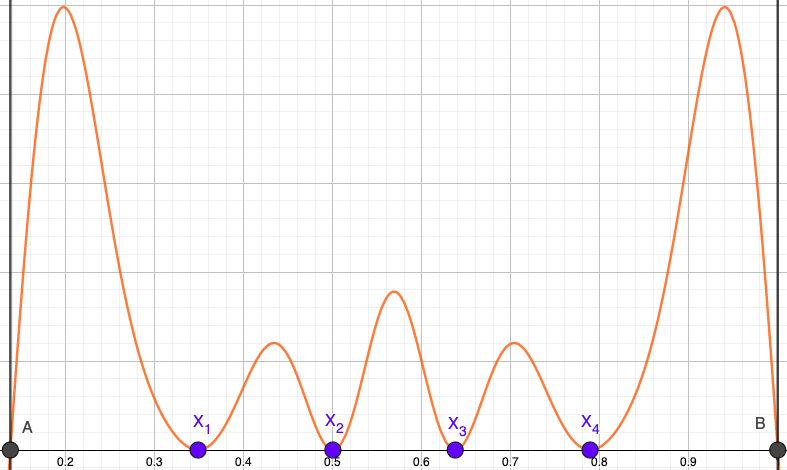}
  \includegraphics[scale=0.21]{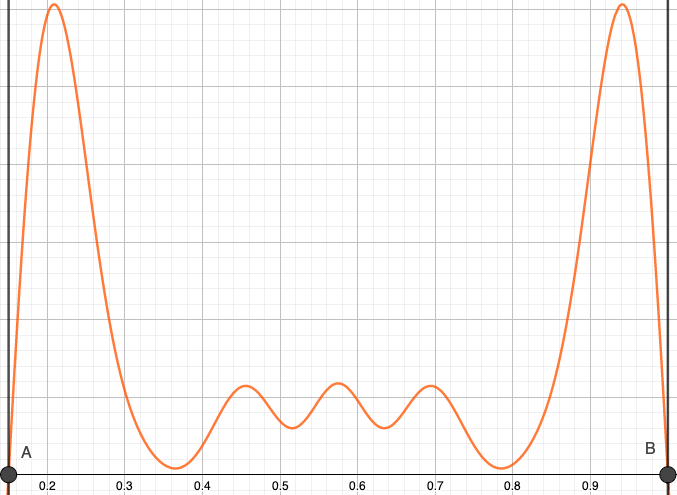}
  \includegraphics[scale=0.15]{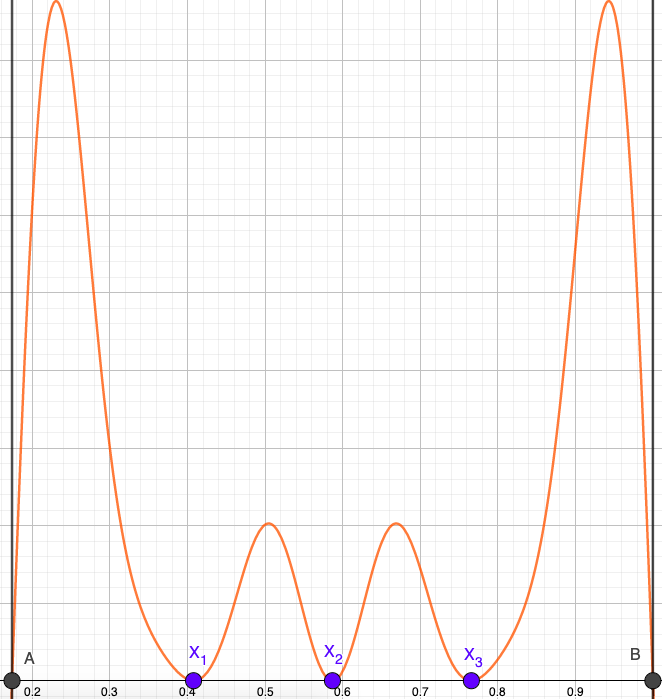}
  %G_E^{\kappa=3}(E_1)$, $k >0$. $E_1 \in [E-1,1]$. Left to right: $E=\E_4$, $E=1.15$, $E=\E_3 \simeq 1.1737$. $G_E^{\kappa=3}(E_1) > 0$ in middle picture.
\caption{Plots of $G_{\kappa=3}^E (x)$, $x \in [E-1,1]$. Rows correspond to $n=1,2,3,4$. Left column : $E = \E_n$. Middle : some $E \in (\E_n, \E_{n-1})$. Right column : $E = \E_{n-1}$.}
\label{fig:test_k2_33333si}
\end{figure}

\subsection{$2nd$ band ($n=2$)} 
\label{sub2_k3}
%$X_1 = 1/2$. Then we solve $T_3(\E_2-1) = T_3(\E_2 -1/2)$ which has solutions $\E_2 = (9 \pm \sqrt{33})/12$. Using the context, it must be that $\E_2 = (9 + \sqrt{33})/12 \simeq 1.2287$. It follows that $X_2 = \E_2 - X_1 = (3+\sqrt{33})/12$. We suppose $\Sigma = [3,6,9,18]$. 

$M \rho =0 \Rightarrow \rho = [\rho_3, \rho_6, \rho_9, \rho_{18}]^T \simeq [1, 0.8854, 0.2861, -0.0452]^{T}.$

\subsection{$3rd$ band $(n=3)$} 
\label{sub3_k3}
$X_2 = \E_3/2$ and $T_3(\E_3-1) = T_3(\E_3-X_1)$, $T_3(X_1) = T_3(\E_3/2)$. Applying \eqref{cheby_23} leads to
\begin{equation*}
\begin{cases}
4\left( (\E_3-1)^2 + (\E_3-1)(\E_3-X_1) + (\E_3-X_1)^2 \right) = 3 \\
4\left( X_1^2 + X_1 \E_3 /2 + \E_3^2 /4 \right) = 3.
\end{cases}
\end{equation*}
Subtract the equations to get $X_1 = (3 \E_3 -11 \E_3^2/4-1)(1-7 \E_3 /2)^{-1}$. This leads to 
\begin{equation}
\label{E_L_k3_3}
\E_3 = \frac{199}{494} + \frac{\sqrt{49149 + 247 t + 2223 s}}{494 \sqrt{3}} +\frac{1}{2} \sqrt{\frac{32766}{61009} - \frac{t}{741} - \frac{3s}{247} + \frac{4340412 \sqrt{3}}{61009 \sqrt{49149 + 247 t + 2223s}}},
\end{equation}
where $t = \sqrt[3]{2625129-139968 \sqrt{327}}$ and $s = \sqrt[3]{3601+192\sqrt{327}}$. So $\E_3 \simeq 1.1737$. On the other hand, $\E_3 = (X_1-11X_1 ^2 +37/4) (3+9X_1)^{-1}$. So 
\footnotesize
$$X_1 = -\frac{73}{494} + \frac{\sqrt{2129040 + 247 q + 17784 r}}{1976 \sqrt{3}} - \frac{1}{2} \sqrt{\frac{88710}{61009} - \frac{q}{11856} - \frac{3r}{494} + \frac{5651424}{61009} \frac{\sqrt{3}}{\sqrt{2129040 + 247 q + 17784 r}}}.$$
\normalsize
$q = \sqrt[3]{569651470848-1240676352 \sqrt{327}}$, and $r = \sqrt[3]{1526201 + 3324 \sqrt{327}}$. So $X_1 \simeq 0.4077$.
%The minimal polynomials of $\E_3$ and $X_1$ are respectively
The minimal polynomials of $\E_3$ is
$$mp(E) = 247 E^4 - 398 E^3 + 141 E^2 - 20 E +4,$$
%$$mp(X) = 247 X^4 + 146 X^3 - 237 X^2 - 88 X +937/16.$$
%For computational purposes we want to know the minimal polynomial of $\E_3 - X_1 = X_3$. One way to find it is to use $T_3(\E_3-1) = T_3(X_3) \Rightarrow X_3 = T_3 ^{-1} T_3 (\E_3-1)$. Taking the appropriate branch of the inverse function, we get 
Also, $X_3 = -(\E_3-1)/2 + \sqrt{3} \sqrt{1 - (\E_3-1)^2}/2 \simeq 0.76598$.
%and its minimal polynomial is
%$$mp(X) = 3952 X^4 - 8704 X^3 + 5136 X^2 + 680 X - 983.$$
Next we suppose $\Sigma = [3,6,9,12,15,18]$. Python says the solution to $M \rho =0$ is \\
$\rho = [\rho_3, \rho_6, \rho_9,\rho_{12}, \rho_{15}, \rho_{18}]^T \simeq [1, 1.38266, 1.09831, 0.56967, 0.18700, 0.03160]^{T}.$

%This time we fill the matrix $M$ with floats as opposed to exact entries because Python struggles otherwise. 

%We are also keen to confirm that the rows of $M$ satisfy the same linear dependencies as in the case $\kappa =2$ (Subsection \ref{sub3_k2}). 
%Using  standard linear regression methods we infer it is the case, judging on an excellent $R^2$. 

%The function $e^{-k(\frac{E}{2}-E_1)^2}G_E ^{\kappa} (E_1)$ is plotted in Figure \ref{fig:test_k2_33333si245} for some values $E \in [\E_3, \E_2]$. The factor $e^{-k(\frac{E}{2}-E_1)^2}$ was added to help visualise the graph.
\begin{comment}
\begin{figure}[H]
  \centering
 \includegraphics[scale=0.205]{k3,left,band3}
  \includegraphics[scale=0.175]{k3,middle,band3}
  \includegraphics[scale=0.19]{k3,right,band3}
\caption{$e^{-k(\frac{E}{2}-E_1)^2}G_E^{\kappa=3}(E_1)$, $k >0$. $E_1 \in [E-1,1]$. Left to right: $E=\E_3$, $E=1.2$, $E= \E_2 =\frac{9+\sqrt{33}}{12}$. $G_E^{\kappa=3}(E_1) > 0$ in middle picture, but not in the other two.}
\label{fig:test_k2_33333si245}
\end{figure}
\end{comment}

\subsection{$4th$ band ($n=4$)} 
\label{sub4_k3}
$X_2 = 1/2$, and $T_3(\E_4-1) = T_3(\E_4-X_1)$ and $T_3(X_1) = T_3(\E_4-1/2)$. Thus :
\begin{equation*}
\label{sys1000001}
\begin{cases}
 (\E_4-1)^2 + (\E_4-1)(\E_4-X_1) + (\E_4-X_1)^2  = \frac{3}{4} & \\
 X_1^2 + X_1(\E_4-\frac{1}{2})  +(\E_4-\frac{1}{2}) ^2  = \frac{3}{4} & 
\end{cases}
\Leftrightarrow 
\begin{cases}
X_1^2 + X_1(\E_4-\frac{1}{2}) + (\E_4-\frac{1}{2})^2 = \frac{3}{4} & \\
X_1 = \frac{8 \E_4^2-8 \E_4+3}{16 \E_4-6}. &
\end{cases}
\end{equation*}
Technically speaking there are 4 solutions for $\E_4 \simeq 0, 0.23, 0.44, 1.13$. Given the context we infer
\begin{equation}
\label{solution_k3_5th_band}
% % \E_4 = \frac{1}{28}(17+\sqrt{58}(w+w^{-1})) 
\E_4 =   \frac{17}{28} + \frac{\sqrt{58}}{14} \cos \left( \frac{1}{3} \arctan \left( \frac{21 \sqrt{687}}{691} \right) \right) \simeq 1.1375.
\end{equation}
%% where $w = z ^{1/3} = e^{\i \theta/3}$, $z = \frac{691 + \i 21 \sqrt{687}}{116 \sqrt{58}} = e^{\i \theta}$, $\theta \in [0,\pi]$.
The minimal polynomial of $\E_4$ is $mp(E) = 224 E^3 - 408 E^2 + 198E -27.$
To solve for $X_1$, eliminate the $\E_4^2$ in the system. One gets $\E_4 = (-8X_1^2+10X_1+7)/(24X_1)$.
Technically there are 4 possibilities for $X_1 \simeq -0.70, -0.5, 0.34, 0.89$. Given the context we infer
\begin{equation}
\begin{aligned}
\label{solution_k3_5th_band_X}
% % X_1 = \frac{5}{28} - \frac{\sqrt{43}}{28} \left( (v+v^{-1}) + \i \sqrt{3} (v-v^{-1}) \right)      
% $v = z^{1/3} = e^{\i \theta /3}$, $z = \frac{-2347+\i 147 \sqrt{687}}{688 \sqrt{43}} = e^{\i \theta}$, $\theta \in [0, \pi]$. 
X_1 & =  \frac{5}{28} - \frac{\sqrt{43}}{14 } \left( \sqrt{3} \sin (\beta) -\cos (\beta) \right) \simeq 0.3484.
\end{aligned}
\end{equation}
where $\beta = \frac{1}{3} \arctan\left( \frac{147 \sqrt{687}}{2347} \right)$. The minimal polynomial of $X_1$ is $ mp(X) = 224 X^3 -120 X ^2 -126 X +49$.
The minimal polynomial of $X_3 = \E_4 - X_2 \simeq 0.6375$ is $mp(X) = 112 X^3 - 36 X^2 - 21 X -1$. Determining the minimal polynomial of $X_4 = \E_4-X_1 \simeq 0.7890$ didn't seem worth the effort. Next we suppose $\Sigma = [3,6,9,12,15,18,21,36]$. We fill the matrix $M$ with floats (exact numbers are getting complicated). Python says the solution to $M \rho =0$ is $[\rho_3, \rho_6, \rho_9,\rho_{12}, \rho_{15}, \rho_{18}, \rho_{21}, \rho_{36}]^T \simeq [1, 1.48690, 1.34705, 0.86282, 0.39635, 0.12308, 0.02043, -0.00012]^{T}.$
%Function $e^{-k(\frac{E}{2}-E_1)^2}G_E ^{\kappa} (E_1)$ is plotted in Figure \ref{fig:test_k2_33333si24569} for some values $E \in [E_L, E_R]$. We have positivity and so $\Sigma$ is valid. 
On the other hand it seems $\Sigma =  [3,6,9,12,15,18,21,3j]$ is not valid for $j=8,9,10,11$.

\section{How to choose the correct indices $\Sigma =  \{ j_1 \kappa , j_2 \kappa , ..., j_{2n}\kappa \}$ ?}
\label{howchooseindices}

This section is in dimension 2. In this section our message is : certainly not any $2n$-linear combination of the form \eqref{LINEAR_combinationA} works on the interval $J_2(\kappa) :=  (2\cos(\pi / \kappa ) , 1+ \cos(\pi / \kappa) )$, but many do work. %To this end we repeat the computations performed in Subsections \ref{sub2_k2} -- \ref{sub4_k2}, for $\kappa=2$.

$\bullet$ For $\kappa=2$, $n=2$ : let us assume instead $\Sigma = \{2, 4, 8, 16\}$. Then we find coefficients 
$$\rho = [\rho_2, \rho_4, \rho_8, \rho_{16}]^{T} = \bigg[1, \frac{3975779}{5332320}, -\frac{2837741}{10664640},  -\frac{13851}{1184960}\bigg]^{T} \simeq [1, 0.7456, -0.2660, -0.0116]^{T}.$$
%Using these coefficients we then plot the function $E_1 \mapsto G_E^{\kappa=2}(E_1)$ for $E \in [\frac{1}{2},\frac{2}{3}]$ and come to the realisation that it too is strictly positive for $E \in (\frac{1}{2},\frac{2}{3})$. So this $\Sigma$ is also a valid combination.
This gives a valid combination. We have additionally checked that the combinations $\Sigma = \{2,4,8, 2j\}$ are also valid, for $j=5,6,8$, whereas $\Sigma = \{2, 4, 6, 8\}$, $\{4, 6, 8,10\}$ and $\{2, 6, 8, 10\}$ are not valid. Figure \ref{fig:test_other examples} has 2 examples of valid solutions and 2 examples of invalid solutions, on the interval $(\E_2, \E_1) = (1/2,2/3)$.

\begin{figure}[H]
  \centering
 \includegraphics[scale=0.16]{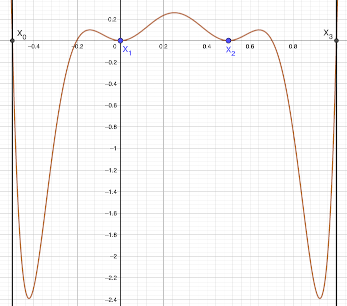}
  \includegraphics[scale=0.18]{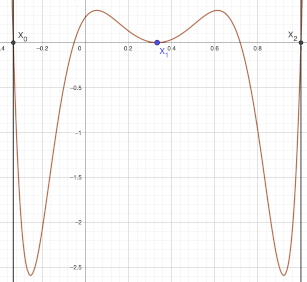}
 \includegraphics[scale=0.20]{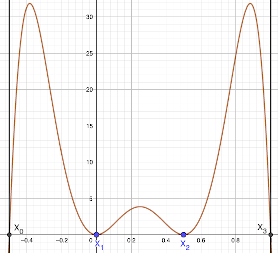}
  \includegraphics[scale=0.2]{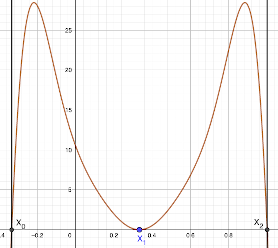}
\\
 \includegraphics[scale=0.18]{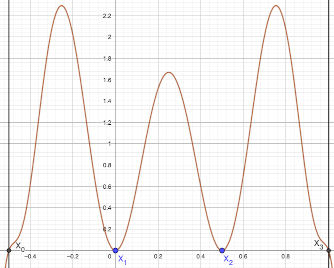}
  \includegraphics[scale=0.18]{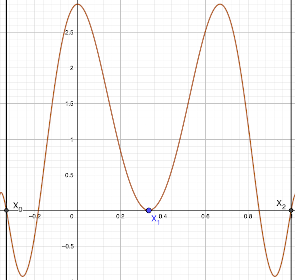}
 \includegraphics[scale=0.16]{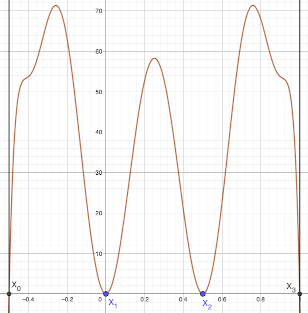}
  \includegraphics[scale=0.16]{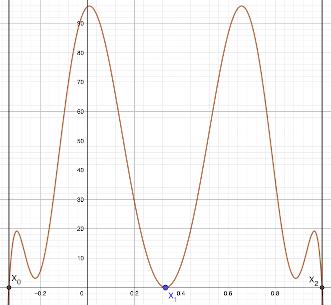}
  
\caption{Plot of $G_{\kappa=2}^E (x)$, $x \in [E-1,1]$. Top row : $\Sigma = \{2,4,6,8\}$ for first 2 graphs and $\Sigma = \{2,4,6,12\}$ for last 2 graphs. Bottom row : $\Sigma = \{2,6,8,10\}$ for first 2 graphs and $\Sigma = \{2,4,10,14\}$ for last 2 graphs. In all cases, Left : $E = \E_2(\kappa=2) = 1/2$. Right : $E = \E_1(\kappa=2) = 2/3$.}
\label{fig:test_other examples}
\end{figure}

$\bullet$ For $\kappa=3$, $n=2$ : other combinations of $\Sigma$'s are valid. For instance $\Sigma = \{3,6,9,21\}$ is also valid, yielding positivity on the same band, but $\Sigma = \{3,6,9,12\}$ and $\{3,6,9,15\}$ are not valid.

\begin{comment}
\begin{figure}[H]
  \centering
 \includegraphics[scale=0.22]{k3,left,band4}
  \includegraphics[scale=0.21]{k3,middle,band4}
  \includegraphics[scale=0.15]{k3,right,band4}
\caption{$e^{-k(\frac{E}{2}-E_1)^2}G_E^{\kappa=3}(E_1)$, $k >0$. $E_1 \in [E-1,1]$. Left to right: $E=\E_4$, $E=1.15$, $E=\E_3 \simeq 1.1737$. $G_E^{\kappa=3}(E_1) > 0$ in middle picture.}
\label{fig:test_k2_33333si24569}
\end{figure}
\end{comment}

\begin{comment}
\begin{equation}
\begin{cases}
T_3(E-1) = T_3(E-X_1) & \\
T_3(X_1) = T_3(E-1/2) & \\
\end{cases}
\end{equation}

Expand the 2 equations and subtract to get rid of the $X^2$ term. Then isolate $X$ in terms of $E$. We get 

\begin{equation}
\begin{cases}
4( X^2 + X(E-1/2) + (E-1/2)^2) = 3 & \\
X = \frac{8E^2-8E+3}{16E-6}
\end{cases}
\end{equation}
\end{comment}

\section{Conjecture for the interval $J_3(\kappa) := (1+\cos(2\pi / \kappa), 2\cos(\pi / \kappa))$}
\label{sec_Conjecture J3}

In this section we give evidence for Conjecture \ref{conjecture_k344_2d}. We only do $\kappa=3,4$.

\begin{table}[H]
  \begin{center}
    \begin{tabular}{c|c|c|c} % <-- Alignments: 1st column left, 2nd middle and 3rd right, with vertical lines in between
      $\kappa$  & $\beta,\beta' $ & proved in \cite{GM2} & Conjecture based on Table \ref{tab:tableDelta2dexploreKAPPA=3_againnn} (improvement) \\ [0.1em]
      \hline
         3 & $\beta = \frac{1}{2} (\frac{1}{2}( 5 - \sqrt{7}) )^{1/2} \simeq 0.5425$ & $(\beta,1) \subset \boldsymbol{\mu}_{\kappa}(\Delta)$  & $J_3 = (1/2,1) \subset \boldsymbol{\mu}_{\kappa}(\Delta)$ \\[0.1em]   
         4 & $\beta' = (\frac{3}{2}-\frac{1}{\sqrt{5}} )^{1/2} \simeq 1.0261$ & $(\beta',\sqrt{2}) \subset \boldsymbol{\mu}_{\kappa}(\Delta)$  & $J_3 = (1,\sqrt{2}) \subset \boldsymbol{\mu}_{\kappa}(\Delta)$ \\[0.1em]   
    \end{tabular}
  \end{center}
    \caption{Sets $\subset \boldsymbol{\mu}_{\kappa}(\Delta)$. $d=2$. Proved in \cite{GM2} vs.\ conjectured}
        \label{tab:tableDelta222222}
\end{table}

\begin{table}[H]
  \begin{center}
    \begin{tabular}{c|c} % <-- Alignments: 1st column left, 2nd middle and 3rd right, with vertical lines in between
      Intervals $\subset \boldsymbol{\mu} _{\kappa} (\Delta)$. $\kappa=3$  &  Intervals $\subset \boldsymbol{\mu} _{\kappa} (\Delta)$. $\kappa=4$ \\ [0.1em]
      \hline
         $(0.518, 0.64) : \rho = (1, 0.6)$   & $(1.00209, 1.049) : \rho = (1, 0, 0, 0.8331)$   \\[0.1em]   
         $(0.5082 ,0.54) : \rho = (1,1,0.8)$  &  $(1.0011, 1.096) : \rho = (1, 0.3, 0, -0.1, 0, 0.465)$  \\[0.1em]   
         %$(0.5079 ,0.518) : \rho = (1,1,0.9)$ \\[0.5em]   
         %$(0.5076, 0.5082) : \rho = (1, 0.7, 0.9)$  \\[0.5em] 
         $(0.5037, 0.5087) : \rho = (1, 0.926, 0.5, 2.5)$ & \\[0.1em]   
     \hline
      Union of intervals : $(0.5037, 0.64)$ & Union of intervals : $(1.0011, 1.096)$
    \end{tabular}
  \end{center}
    \caption{Sets $\subset \boldsymbol{\mu}_{\kappa}(\Delta)$. $d=2$. $\rho$ of the form : $\rho = (\rho_{\kappa}, \rho_{2\kappa}, \rho_{3\kappa}, ....)$}
        \label{tab:tableDelta2dexploreKAPPA=3_againnn}
\end{table}

The coefficients $\rho = (\rho_{j\kappa})$ in Table \ref{tab:tableDelta2dexploreKAPPA=3_againnn} were found just by fiddling around with the coefficients and the graphs (we don't know how to cast this problem into polynomial interpolation).

\section{$\left(9-\sqrt{33})/12, 2/7\right)$ is a band of a.c.\ spectrum for $\kappa=3$ in dimension 2 : evidence}
\label{k3_below_05}

\begin{figure}[H]
  \centering
 \includegraphics[scale=0.2]{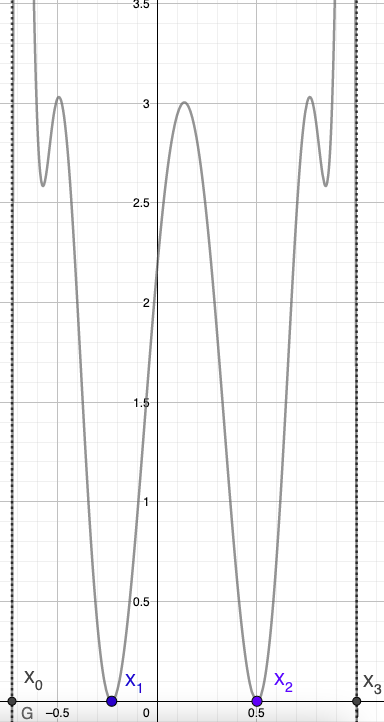}
  \includegraphics[scale=0.2]{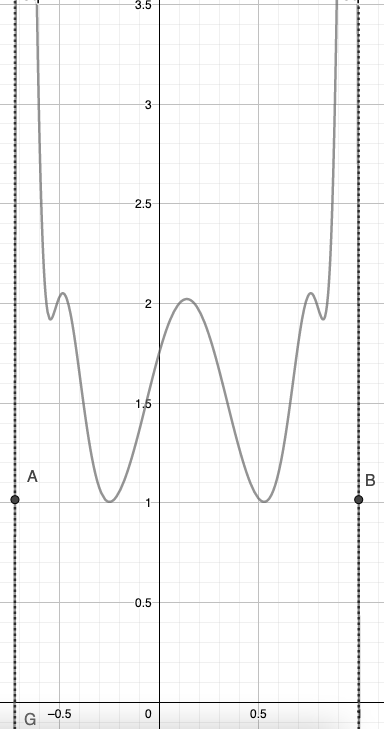}
  \includegraphics[scale=0.2]{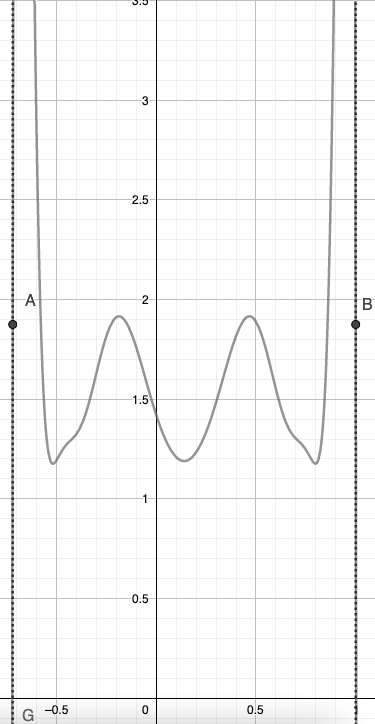}
    \includegraphics[scale=0.2]{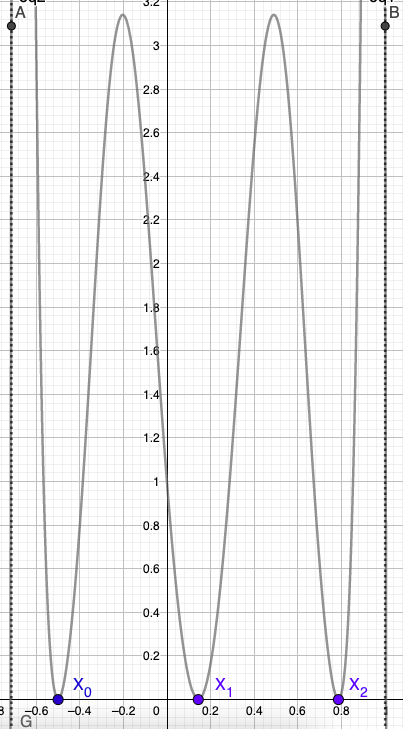}
\caption{$G_{\kappa=3} ^E (x)$, $x \in [E-1,1]$. From left to right: $E=(9-\sqrt{33})/12 \simeq 0.2713$, $E=0.276$, $E=0.28$, $E=2/7 \simeq 0.2857$.  $G_{\kappa=3} ^E (x) > 0$ in the 2 middle pictures, but not in the other two. }
\label{fig:kappa3_lower_energies}
\end{figure}

The 2 energies $(9-\sqrt{33})/12$ and $2/7$ belong to $\boldsymbol{\Theta}_{1,\kappa=3}(\Delta)$ as per Table \ref{sol_kappa_2346}. Python says the solution to $M \rho =0$ is $[\rho_3, \rho_6, \rho_9,\rho_{12}, \rho_{15}]^T \simeq [-2.1648 \cdot t, -7.2577 \cdot t, 22.5984 \cdot t, 3.3111 \cdot t, t]^{T}, \ t \in \R$. $G_{\kappa=3} ^E (x)$ is plotted in Figure \ref{fig:kappa3_lower_energies} for $t=1$.

\section{Numerical evidence for a band of a.c.\ spectrum for $\kappa=8$ in dimension 2}
\label{KAPPA8more}
Fix $\kappa = 8$, $d=2$. We see from Table \ref{tab:table1011} that $(0, 1+ \cos(\frac{3\pi}{8}))$ was identified as a gap in our prior work. Can we find a linear combination of conjugate operators that gives positivity on an interval to reduce this gap ?
We look for a band that is adjacent and to the left of $\H_0:= 1+ \cos(\frac{3\pi}{8})$. Our numerical calculations suggest that $\exists \H_1$ such that $(\H_1, \H_0) \subset  \boldsymbol{\mu}_{\kappa=8} (\Delta)$, but the point is that the system of linear equations \eqref{interpol_intro} is underspecified.

To determine the nearest threshold to the right of $\H_0$, let's call it $\H_1$. We assume $\H_1 \in \boldsymbol{\Theta}_{1,\kappa}(\Delta)$. So $\omega_0 g_{j\kappa}^{\H_1} (X_1) = g_{j\kappa} ^{\H_1} (X_0)$, $\forall j \in \N^*$. To solve this equation, we follow the path of Ansatz (1) in the context of \eqref{long_comm22} : we make the assumption that $X_0=\H_1-1$ and $T_8(X_0) = T_8 ( X_1) = T_8(\H_1-X_1)$. Thus we have a system of 3 equations, 3 unknowns ($\H_1$, $X_0$, $X_1$). There are many solutions to this system; we will focus on the one where
$$\H_1 = \frac{1}{2} \left(\sqrt{2} + 2 \sqrt{\sqrt{2}-1}\right) \simeq 1.3507, \quad X_1 = \frac{2}{5} + \frac{1}{5} \H_1 + \frac{2}{5} \H_1^2 - \frac{2}{5} \H_1^3 \simeq 0.4142.$$
We then compute $\omega_0 = g_{\kappa} ^{\H_1} (X_0) / g_{\kappa} ^{\H_1} (X_1) \simeq - 0.7427 < 0$, which means that $\H_1 \in \boldsymbol{\Theta}_{1,8}(\Delta)$. Next, we assume a linear combination of the form $\mathbb{A} = A_{8} + \rho _{16} A_{16} + \rho_{24} A_{24}$. So $G_{8} ^E (x) = g_{8} ^E  (x) +  \rho_{16} g_{16} ^E (x) + \rho_{24} g_{24} ^E (x)$.  To determine $\rho_{16}$ and $\rho_{24}$ we perform polynomial interpolation with the following \underline{\textit{only}} constraint : $G_{8} ^{\H_1} (X_0) = 0$. Note that by construction this constraint is equivalent to $G_{8} ^{\H_1} (X_1) = 0$ and $G_{8} ^{\H_1} (\H_1-X_1) = 0$. Also, by Lemma \ref{lemSUMcosINTRO}, $G_{8} ^{\H_0} (\H_0-1) = 0$ always holds. The interesting difference here is that we also have $[g_{8j_1} ^{\H_1} (X_0), \frac{d}{dx} g_{8j_2} ^{\H_1} (X_1) ] =0$ for all $j_1,j_2 \in \N^*$ (expand and use the fact $T_8(X_0) = T_8(X_1) =T_8(\H_1-X_1)$ and $(x^2-1)\frac{d}{dx} U_{j \kappa-1} (x) = j\kappa T_{j\kappa} (x) - x U_{j\kappa-1}(x)$). It is not clear to us if this means that $\exists \lambda$ such that $g_{8j} ^{\H_1} (X_0) = \lambda \frac{d}{dx} g_{8j} ^{\H_1} (X_1)$ for all $j \in \N^*$. It seems that the constraint $\frac{d}{dx} G_{8} ^{\H_1} (X_1) = 0$ is redundant. One also checks that $\frac{d}{dx} G_{8} ^{\H_1} (X_2) = 0$. Coming back to our interpolation problem, $G_{8} ^{\H_1} (X_0) = 0$ leads to 
$$\rho_{16} = -\left( g_{8} ^{\H_1} (X_0) + \rho_{24} g_{24} ^{\H_1} (X_0) \right) / g_{16} ^{\H_1} (X_0) = -\left( U_7(\H_1- 1) + \rho_{24} U_{23}(\H_1-1) \right) / U_{15}(\H_1-1).$$
The point is that here we have 1 degree of freedom, namely $\rho_{24}$. Numerically, $\rho_{16} \simeq 0.51952 +1.40530 \rho_{24}$. Graphically, it appears that $G_{8} ^E (x) >0$ for $E \in (\H_1,\H_0)$ and $\rho_{24}$ roughly in $(-0.36,-0.51)$. 

\begin{figure}[H]
  \centering
 \includegraphics[scale=0.208]{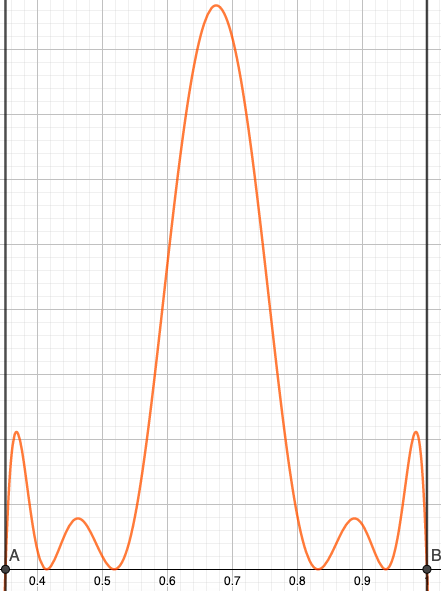}
  \includegraphics[scale=0.26]{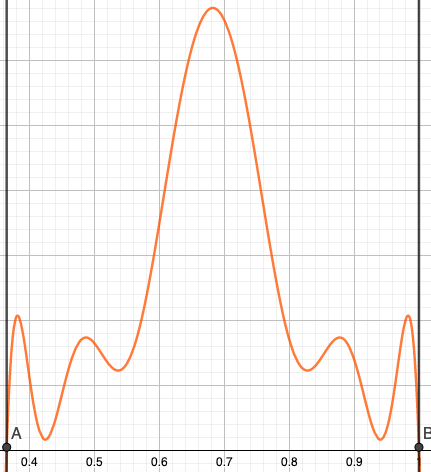}
  \includegraphics[scale=0.29]{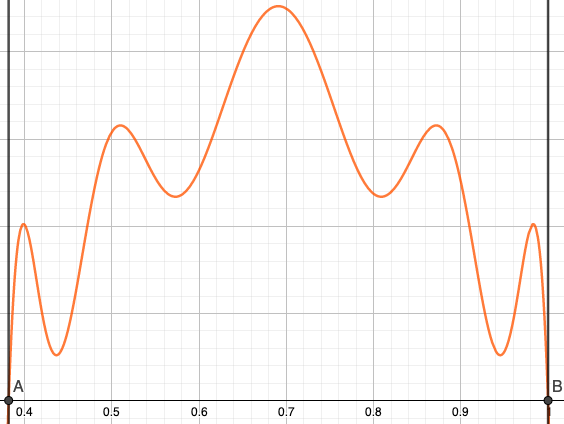} 
  \\
   \includegraphics[scale=0.203]{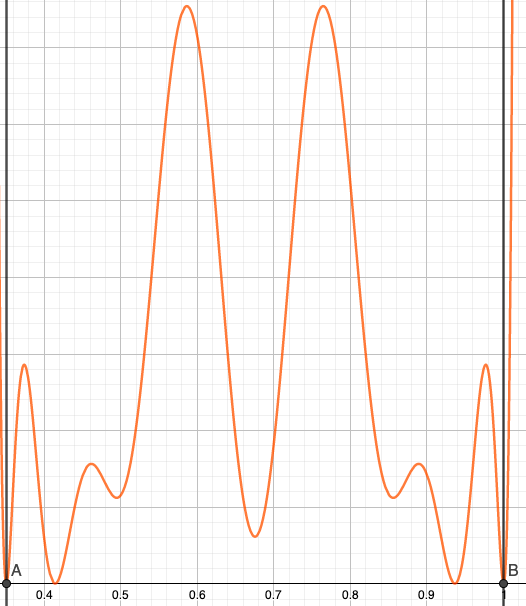}
  \includegraphics[scale=0.19]{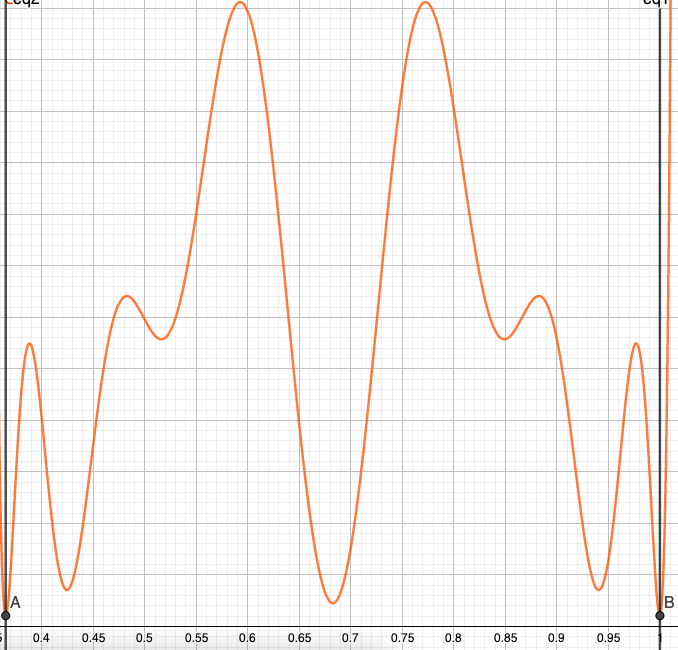}
  \includegraphics[scale=0.211]{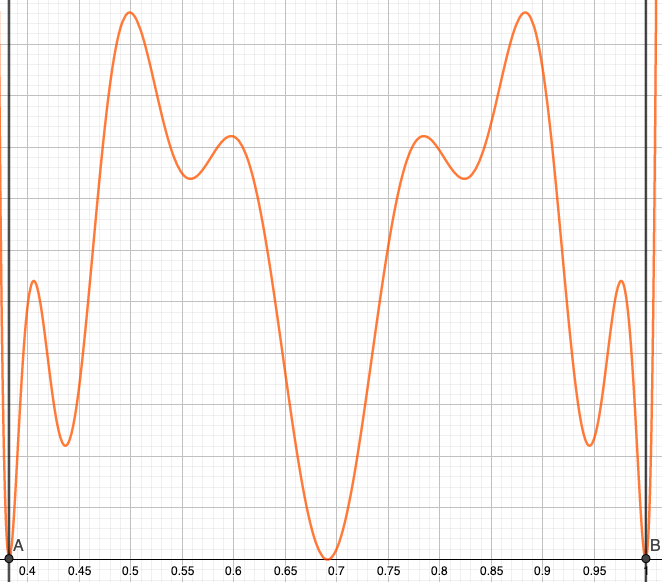} 
\caption{$G_{\kappa=8}^E (x)$, $x \in [E-1,1]$. Top : $\rho_{24}=-0.05$, bottom : $\rho_{24}=-0.36$. From left to right: $E=\H_1$, $E=1.365$, $E=\H_0$. $G_{\kappa=8} ^E (x) > 0$ in middle picture, but not in the other two. }
\label{fig:test_k2_3333388000}
\end{figure}

\section{The case of $\kappa=2$ in dimension 3}

We illustrate the situation for $\kappa=2$, and the $4^{th}$ band, namely $(1+\E_4, 1+\E_3) = (4/3, 7/5)$. We use the linear combination $\sum_j \rho_{j\kappa} A_{j\kappa}$ where the coefficients $\rho_{j\kappa}$ are the same as in dimension 2, i.e.\ the ones found in Subsection \ref{sub4_k2}.

Figure \ref{fig:dim3_k2_band4,n} shows the function $G_{\kappa=2} ^E (x,y)$ at $E = 1+\E_4 = 4/3$ and certain values of $y$. We observe a curious phenomenon. While the pattern observed is the same as the one occurring in dimension 2, the novelty is that it is now occurring simultaneously all for the same energy. In other words the pattern observed in dimension 3 for the $n^{th}$ band is the collection of the patterns observed for bands $1,2,...n$ in dimension 2. The graph of $G_{\kappa=2} ^E (x,y)$ at $y = -2/3$ is not plotted because it is basically non-existant ($x=z=1$ is forced and so the graph is only defined at $x=1$, and $G_{\kappa=2} ^{E=4/3} (1,-2/3) = 0$). For $y \in [E-2,1] \setminus \{-2/3, -1/3,0,1/3,2/3,1\}$ $G_{\kappa=2} ^E (x,y)$ appears to be strictly positive but that is irrelevant.

\begin{figure}[H]
  \centering
 \includegraphics[scale=0.17]{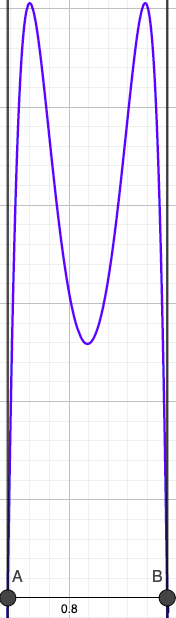}
 \includegraphics[scale=0.153]{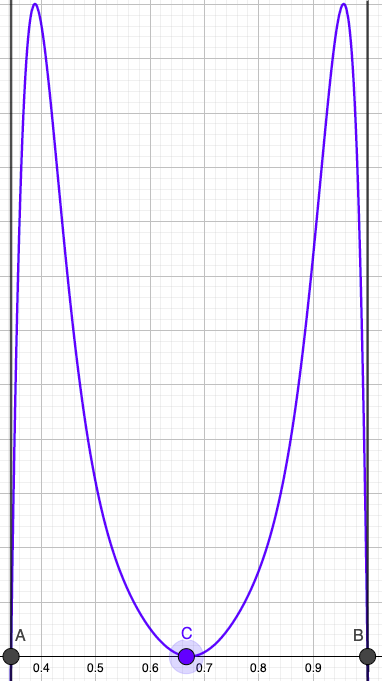}
 \includegraphics[scale=0.152]{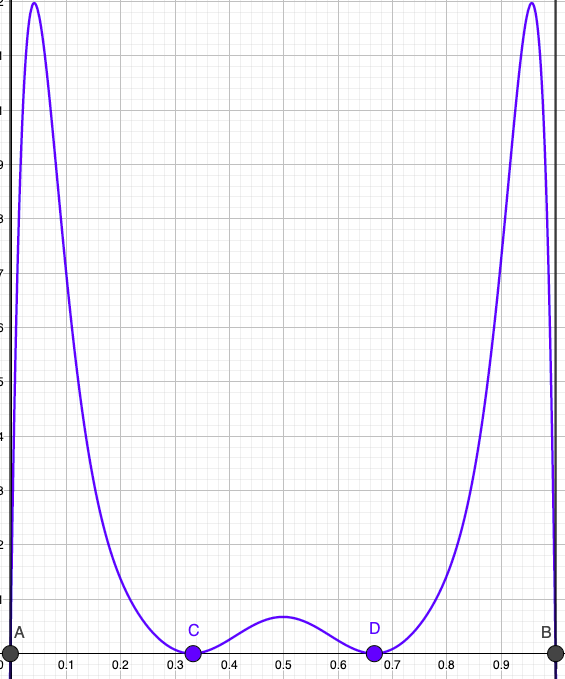}
 \includegraphics[scale=0.154]{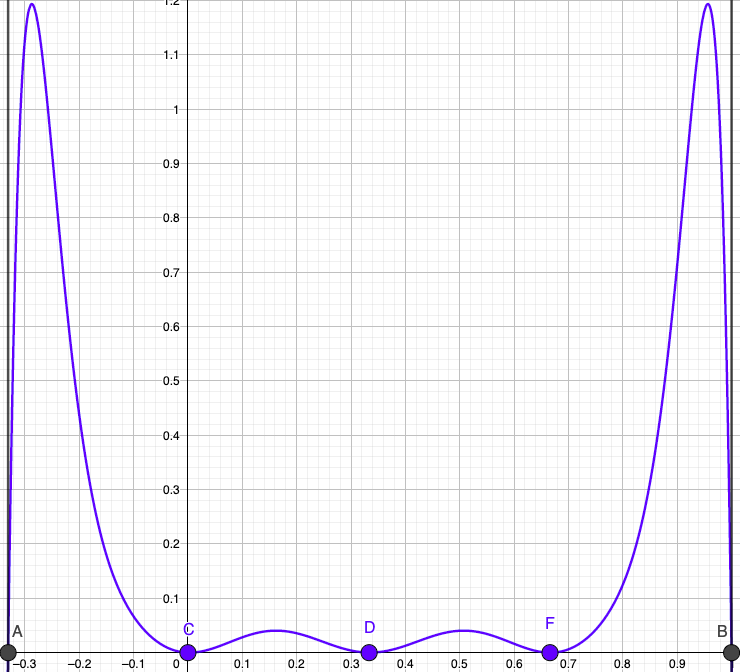}
 \includegraphics[scale=0.18]{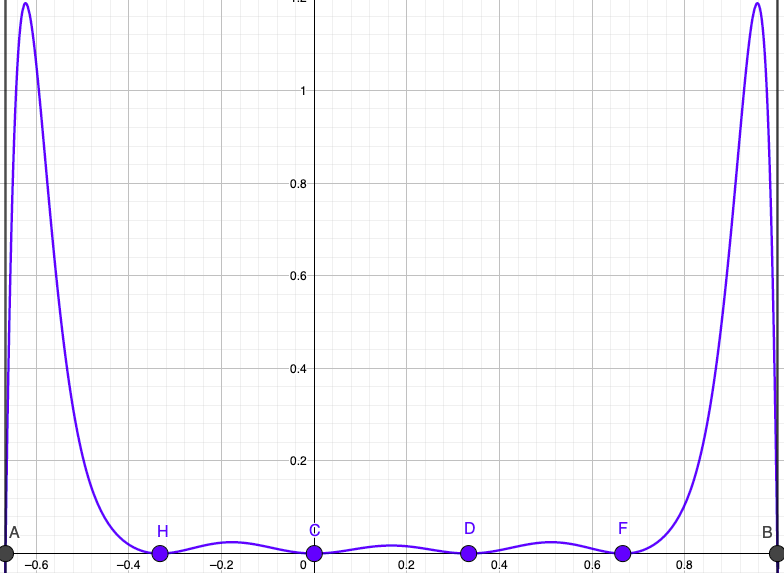}
\caption{$G_{\kappa=2} ^E (x,y)$. $E = 1+\E_4 = 4/3$, $x \in [E-y-1,1]$. From left to right: $y=-1/3, 0,1/3,2/3,1$.}
\label{fig:dim3_k2_band4,n}
\end{figure}

Figure \ref{fig:dim3_k2_band4,m} shows the function $G_{\kappa=2} ^E (x,y)$ at $E = 1+\E_3 = 7/5$ and certain values of $y$.

\begin{figure}[H]
  \centering
 \includegraphics[scale=0.165]{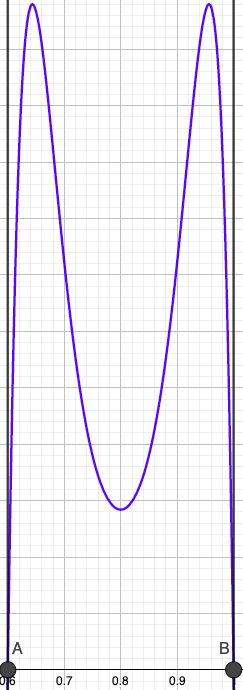}
 \includegraphics[scale=0.165]{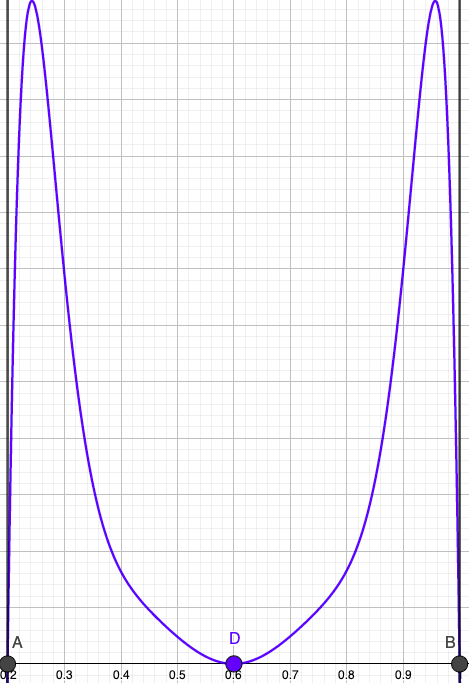}
 \includegraphics[scale=0.165]{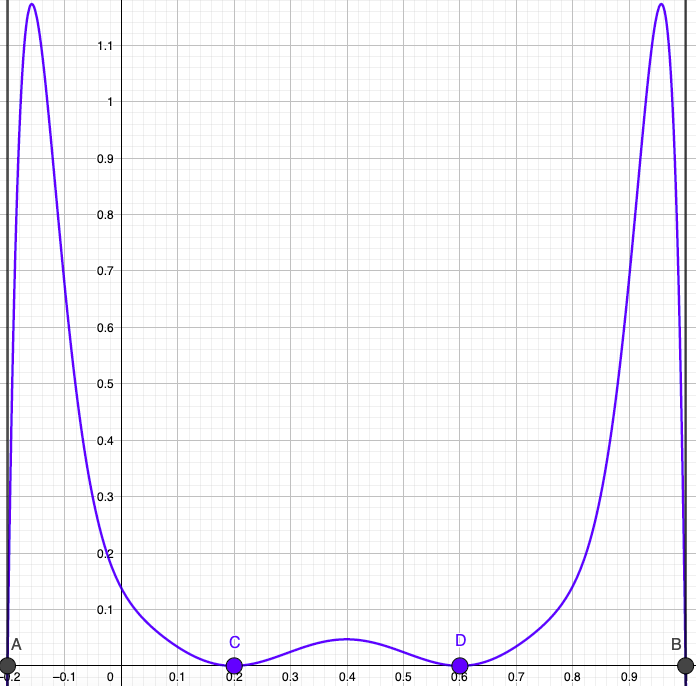}
 \includegraphics[scale=0.205]{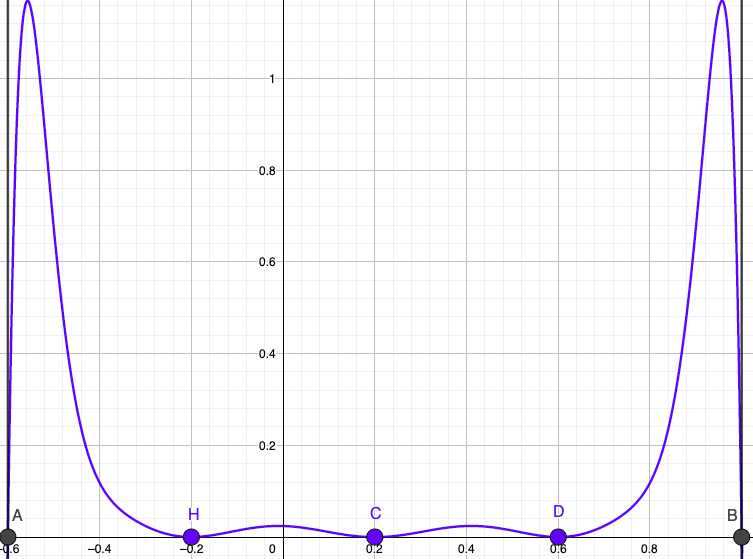}
\caption{$G_{\kappa=2} ^E (x,y)$. $E = 1+\E_3 = 7/5$, $x \in [E-y-1,1]$. From left to right: $y=-1/5, 1/5,3/5,1$.}
\label{fig:dim3_k2_band4,m}
\end{figure}

\section{A threshold in dimension 3 : An example to support Conjecture \ref{conjecture13d33}}

We find a threshold for $\kappa = 3$ in dimension 3. The example justifies our Conjecture \ref{conjecture13d33} :

\begin{example} 
\label{onlyexample}
First we prove $E = \frac{3}{2\sqrt{7}} \in \boldsymbol{\Theta}_{1,\kappa=3} (\Delta[d=3])$. Note $\frac{3}{2\sqrt{7}} \simeq 0.56695$. To do this we follow the procedure in section \ref{subsection1st}, adpated to $d=3$. $E \in \boldsymbol{\Theta}_{1,\kappa} (\Delta[d=3])$ iff $\exists E \in [0,3]$, $Y_i \in [\max(E-2,-1),\min(E+2,1)]$, and $X_i \in [\max(E-Y_i-1,-1),\min(E-Y_i+1,1)]$, $i=0,1$, and $\omega_0 <0$ such that 
\begin{equation}
\label{OMEGAOd3}
\omega_0 = \frac{g_{j \kappa} ^E (X_0, Y_0)}{g_{j \kappa} ^E (X_1,Y_1)} = \frac{g_{l\kappa} ^E (X_0, Y_0)}{g_{l\kappa} ^E (X_1,Y_1)}, \quad \forall j,l \in \N^*.
\end{equation}
Note $g_{j\kappa} ^E (\cdot, \cdot)$ is given by \eqref{def:gE33}. Thus we want to solve $[g_{j \kappa} ^E (X_0, Y_0) , g_{l \kappa} ^E (X_1, Y_1) ] =0$. We make the simplifying assumption that $X_0 = Y_0 = E/3 \neq \pm 1$. Thus, the latter commutator equation reduces to :
\begin{equation}
\label{long_comm22d3}
\begin{aligned}
& m(X_1)  [U_{j \kappa-1}(E/3) , U_{l \kappa-1}(X_1) ] + m(Y_1) [U_{j \kappa-1}(E/3) , U_{l \kappa-1}(Y_1) ]  \\
& \quad + m(E-X_1-Y_1) [U_{j \kappa-1}(E/3) , U_{l \kappa-1}(E-X_1-Y_1) ] = 0.
\end{aligned}
\end{equation}
By virtue of Corollary \ref{basic_lemma}, this equation is satisfied if we further assume :
$$ T_{\kappa=3} (E/3) = T_{\kappa=3} (X_1) = T_{\kappa=3} (Y_1) = T_{\kappa=3} (E- X_1 - Y_1).$$
This is 3 equations, 3 unknowns ($E,X_1, Y_1$). According to Python, a solution is $E = \frac{3}{2\sqrt{7}}$, $X_1 = Y_1 = 2/ \sqrt{7}$. We then check numerically that $\omega_0$ given by \eqref{OMEGAOd3} is $<0$, and it is because $\omega_0 = - 1.8$. We therefore have $E \in \boldsymbol{\Theta}_{1,\kappa=3} (\Delta[d=3])$.

From our discussion in section \ref{subsection1st}, and notably the results in Table \ref{sol_kappa_2346}, we are inclined to believe that $\frac{3}{2\sqrt{7}} - \{ \cos(l \pi / \kappa) : 0 \leq l \leq \kappa \} \not \in \boldsymbol{\Theta}_{1,\kappa=3} (\Delta[d=2])$. Of course, our example is not a hard proof because we don't know exactly what all the energies in $\boldsymbol{\Theta}_{1,\kappa=3} (\Delta[d=2])$ are. In any case, we also conjecture that $\frac{3}{2\sqrt{7}} - \{\cos(l \pi / \kappa) : 0 \leq l \leq \kappa \}  \not \in \boldsymbol{\Theta}_{\kappa=3} (\Delta[d=2])$. 
%\red{Also compute sequence in Table \ref{table with endpoints3and4}.}
\end{example}

\section{Appendix : Recap of prior results (\cite{GM2})} 
\label{appendix_std}

Table \ref{tab:table1011} are the bands $\subset \boldsymbol{\mu}_{\kappa}(\Delta)$ identified in \cite{GM2}. These were obtained using the linear combination of conjugate operators \eqref{LINEAR_combinationA} with $\rho_{j \kappa}=1$ if $j=1$, 0 otherwise.

\begin{table}[H]
  \begin{center}
    \begin{tabular}{c|c|c} 
      $\kappa$ & Intervals $\subset \boldsymbol{\mu} _{\kappa} (\Delta)$. $d=2$. &  Intervals $\subset \boldsymbol{\mu} _{\kappa} (\Delta)$. $d=3$.    \\ [0.1em]
      \hline
      $1$  & $(0,2)$ & $[0,1) \cup (1,3)$ \\ [0.1em]
      $2$  & $(1,2)$ & $(2,3)$  \\ [0.1em]
      $3$  &  $(0.542477,1.000000) \cup (1.500000,2)$  &  $(2.5,3)$  \\[0.1em]
      $4$ &  $( 1.026054 , 1.414214 ) \cup ( 1.707107 , 2 ) $ &  $(2.072,2.121) \cup (2.707,3)$  \\[0.1em]
      $5$ &  $( 1.326098 , 1.618034 ) \cup ( 1.809017 , 2 )$ &  $(2.353,2.427) \cup (2.809,3)$ \\[0.1em]
      $6$ &  $( 1.511990 , 1.732051 ) \cup ( 1.866026 , 2 )$ &  $(2.529,2.598) \cup (2.866,3)$ \\[0.1em]
      $7$ &  $( 1.222521 , 1.246980 ) \cup ( 1.632351 , 1.801938 ) \cup ( 1.900969 , 2 )$ &  $(2.645,2.703) \cup (2.901,3)$  \\[0.1em]
      $8$ &  $( 1.382684 , 1.414214 ) \cup ( 1.713916 , 1.847760) \cup ( 1.923880 , 2 )$ &  $(2.724,2.771) \cup (2.924,3)$ \\[0.1em]
    \end{tabular}
  \end{center}
   \caption{Sets $\subset \boldsymbol{\mu}_{\kappa} (\Delta) \cap [0,d]$ found using the trivial linear combination (\cite{GM2})}
           \label{tab:table1011}
\end{table}

In \cite{GM2} we conjectured exact formulas for the most of the band endpoints. In dimension 2 : 
\begin{equation*}
\begin{cases}
\text{Interval} \ = \left( \ \_ \_ \ ,2\cos(\pi/ \kappa) \right) \cup \left(1+\cos(\pi/ \kappa) ,2 \right),  & \kappa = 3-6,  \\
\text{Interval} \ = \left( 1+\cos(3\pi/ \kappa) , 2\cos(2\pi/ \kappa) \right) \cup \left( \ \_ \_ \ , 2\cos(\pi/ \kappa) \right) \cup \left( 1+\cos(\pi/ \kappa) , 2 \right),  & \kappa = 7,8.
\end{cases}
\end{equation*}
And in dimension 3 :
\begin{equation*}
\text{Interval} \ = \left(  \ \_ \_  \ , 3\cos(\pi/ \kappa) \right) \cup \left(2+\cos(\pi/ \kappa),3 \right), \quad \kappa = 4-8.
\end{equation*}

\section{Appendix : Algorithm details}
The following simple algorithms were used to visually assess the positivity of $G_{\kappa} ^E$.

When $\Delta$ in dimension $2$, we used the simple algorithm :
\begin{itemize}
\item For all $E \in [-2,2]$ : 
\begin{itemize}
\item let $y = E-x$
\item check if the function $x \mapsto G_{\kappa} ^E (x)$ has same sign on the interval $x \in [\max(E-1,-1),\min(E+1,1)]$.
\end{itemize}
\end{itemize}

When $\Delta$ in dimension $3$, we used the simple algorithm :
\begin{itemize}
\item For all $E \in [0,3]$ : 
\item For all $y \in [\max(E-2,-1),\min(E+2,1)]$ :
\begin{itemize}
\item let $z = E-x-y$
\item check if the function $x \mapsto G_{\kappa} ^E (x,y)$ has same sign on the interval $x \in [\max(E-y-1,-1),\min(E-y+1,1)]$.
\end{itemize}
\end{itemize}

\section{Appendix : The conjugate operator as a Fourier sine series}
\label{sine_series}

Let $A_{j\kappa}$ be given in \eqref{LINEAR_combinationA}, $\mathcal{F}=$ \eqref{FourierTT}. $\mathcal{F} A_{j\kappa} \mathcal{F}^{-1}  =  (2\i)^{-1} \sum _{i=1} ^d [ \sin(j \kappa \xi_i) \frac{\partial}{ \partial \xi_i} + \frac{\partial }{\partial \xi_i} \sin(j \kappa \xi_i) ]$.
Let $f : [-\pi, \pi] \to \R$ be a $C^1$ function. Let $f = f_{+} + f_-$ be the decomposition of $f$ into an even $f_+$ and odd $f_-$ function. Consider conjugate operators of the form
\begin{equation}
\label{gen_a}
\mathbf{a} := \frac{\i}{2} \sum _{i=1} ^d f(\xi_i) \frac{\partial}{\partial \xi_i} + \frac{\partial}{\partial \xi_i} f(\xi_i), \quad \mathbf{a}_{\pm} := \frac{\i}{2} \sum _{i=1} ^d f_{\pm}(\xi_i) \frac{\partial}{\partial \xi_i} + \frac{\partial}{\partial \xi_i} f_{\pm}(\xi_i).
\end{equation}
One has 
$$\mathcal{F} [ \Delta, \i \mathbf{a} ]_{\circ} \mathcal{F}^{-1} = \sum _{i=1}^d \sin(\xi_i) f(\xi_i), \quad \mathcal{F} [ \Delta, \i \mathbf{a}_{\pm} ]_{\circ} \mathcal{F}^{-1} = \sum _{i=1}^d \sin(\xi_i) f_{\pm}(\xi_i).$$

\begin{Lemma}
If there is a continuous function $f$ such that the strict Mourre estimate \eqref{mourreEstimate123} holds for $\Delta$ wrt.\ $\mathbf{a}$ in a neighborhood of $E$ then it holds for $\Delta$ wrt.\ $\mathbf{a}_-$ in a neighborhood of $E$.
\end{Lemma}
\begin{proof}
Note that $(\xi_1,...,\xi_d) \in S_E \Leftrightarrow (-\xi_1,...,-\xi_d) \in S_E$. If the strict Mourre estimate holds for $\Delta$ wrt.\ $\mathbf{a}$ in a neighborhood of $E$, then $\exists t >0$ such that

\begin{equation}
\begin{aligned}
\chi_{\{E\}} \left(\sum_i \cos(\xi_i) \right) \sum _{i=1}^d \sin(\xi_i) (f_+(\xi_i)+f_-(\xi_i)) & \geq t \cdot \chi_{\{E\}} \left(\sum_i \cos(\xi_i) \right), \quad and \\
\chi_{\{E\}} \left(\sum_i \cos(-\xi_i) \right) \sum _{i=1}^d \sin(-\xi_i) (f_+(-\xi_i)+f_-(-\xi_i)) & \geq t \cdot \chi_{\{E\}} \left(\sum_i \cos(-\xi_i) \right).
\end{aligned} 
\end{equation}
Adding the 2 inequalities gives $\chi_{\{E\}} \left(\sum \cos(\xi_i) \right) \sum _{i=1}^d \sin(\xi_i) f_-(\xi_i)  \geq  t \chi_{\{E\}} \left(\sum \cos(\xi_i) \right)$, which means that the strict Mourre estimate holds for $\Delta$ wrt.\ $\mathbf{a}_-$ in a neighborhood of $E$.
\qed
\end{proof}
This Lemma means that if we assume a conjugate operator of the form \eqref{gen_a}, it is enough to restrict our attention to odd functions. Furthermore, odd $C^1$ functions can be expressed via a Fourier sine series. Thus a conjugate operator of the form \eqref{LINEAR_combinationA} is in line with this observation, but also takes into account the $\kappa$--periodic specificity of the potential $V$.

\section{$\kappa=10$ : thresholds in 1-1 correspondence with the nodes of a binary tree}
\label{KAPPA10}

Figure \ref{fig:test_T10k10} illustrates a sequence of thresholds $\searrow 0$. To construct this sequence, the first point we place is $(0, T_{\kappa}(0)) = (0,-1)$. Then we place pairs of points satisfying first the symmetry condition and then the level condition 2.3. Finally the last point is placed such that it satisfies the symmetry condition wrt.\ the last point constructed, as well as equals $(\cos(2\pi / \kappa), T_{\kappa}(\cos(2\pi / \kappa))) = ( \cos(2\pi / \kappa), 1)$ or $(\cos(4\pi / \kappa), T_{\kappa}(\cos(4\pi / \kappa))) = ( \cos(4\pi / \kappa), 1)$. This sequence is in 1-to-1 correspondence with the nodes of an infinite binary tree, because at every time we want to fulfill the level condition 2.3, i.e.\  place $(X_r, T_{\kappa}(X_r))$ such that $T_{\kappa}(X_r) = T_{\kappa}(X_{q})$ and $T'_{\kappa}(X_r) \cdot T'_{\kappa}(X_{q}) < 0$, we give ourselves the option of choosing from 2 different branches of $T_{\kappa}(\cdot)$ (we could choose many more branches !).

\begin{sidewaysfigure}
  \centering
  \caption{$T_{\kappa=10}(x)$. A sequence of threshold solutions $\searrow 0$.}
 \includegraphics[trim=1cm 120cm 1cm 0cm, width=1.0\textheight]{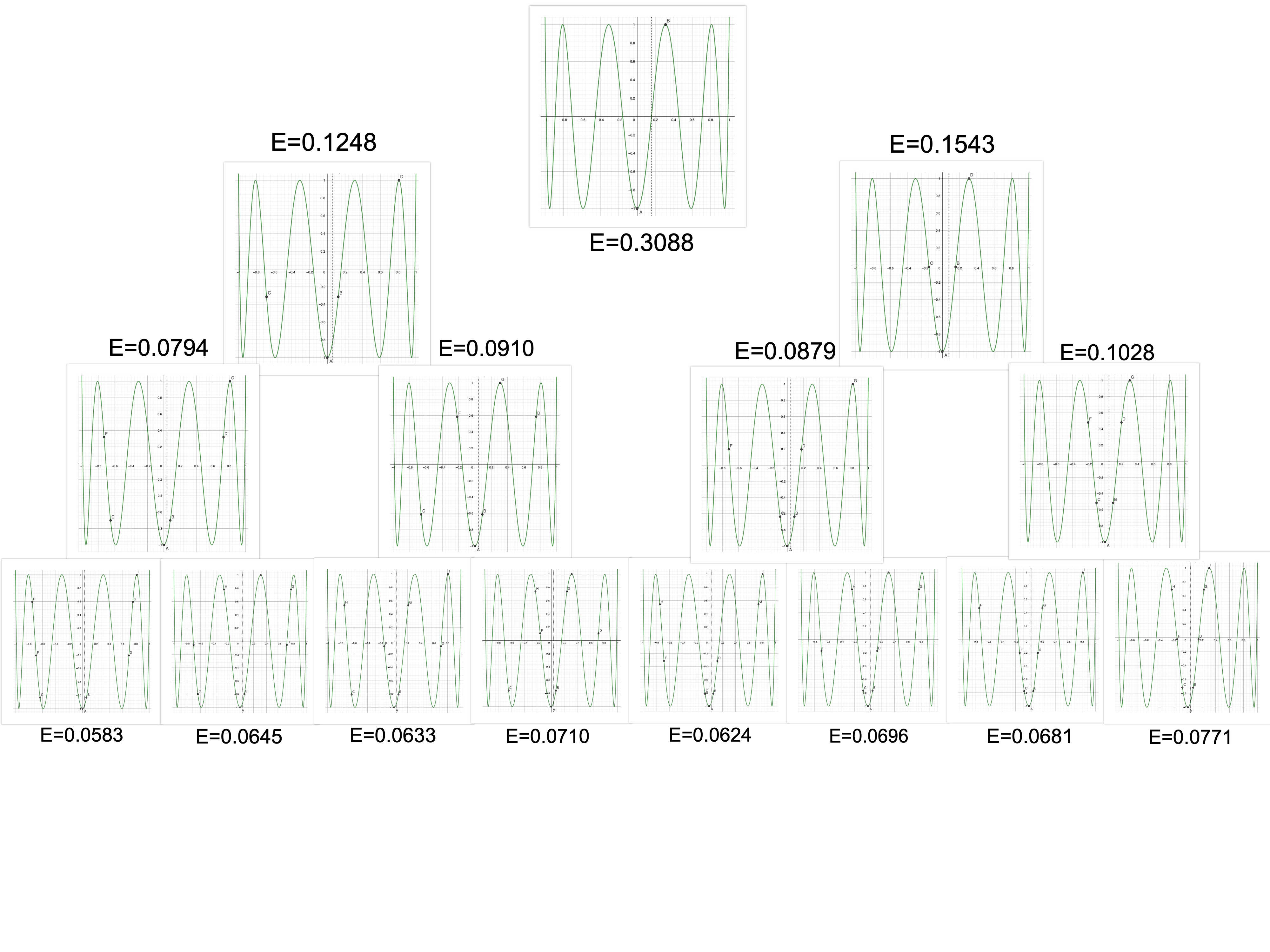}
\label{fig:test_T10k10}
\end{sidewaysfigure}


\begin{thebibliography}{XXXX}
\bibitem[ABG]{ABG} W.O.\ Amrein, A.\ Boutet de Monvel, and V.\ Georgescu: \emph{$C_0$-groups, commutator methods and spectral theory of $N$-body hamiltonians}, Birkh{\"a}user, (1996).
\bibitem[BSa]{BSa} A.\ Boutet de Monvel, J.\ Sahbani: \emph{On the spectral properties of discrete Schr{\"o}dinger operators: the multi-dimensional case}, Rev.\ in Math.\ Phys.\ \textbf{11}, No.\ 9, p.\ 1061--1078, (1999).
\bibitem[FH]{FH} R.\ Froese, I.\ Herbst: \emph{Exponential bounds and absence of positive eigenvalues for N-body Schr\"odinger operators.}, Comm.\ Math.\ Phys.\, \textbf{87}, no.\ 3, p.\ 429--447, (1982/83).
%\bibitem[G]{G} C.\ G{\'e}rard: \emph{A proof of the abstract limiting absorption principle by energy estimates}, J.\ Funct.\ Anal.\ \textbf{254}, No.\ 11, p.\ 2707--2724, (2008).
%\bibitem[GGo]{GGo} V.\ Georgescu, S.\ Gol{\'e}nia: \emph{Isometries, Fock spaces and spectral analysis of Schr{\"o}dinger operators on trees}, J.\ Funct.\ Anal.\ \textbf{227}, p.\ 389--429, (2005). 
%\bibitem[GJ2]{GJ2} S.\ Gol{\'e}nia, T.\ Jecko : \emph{Weighted Mourre's commutator theory, application to Schr{\"o}dinger operators with oscillating potential}, J.\ Oper.\ Theory, No.\ 1, p.\ 109--144, (2013).
%\bibitem[GM]{GM} S.\ Gol{\'e}nia, M.\ Mandich : \emph{Propagation estimates for one commutator regularity}, Integr.\ Equ.\ Oper.\ Theory, 90:47, (2018). 
\bibitem[GM1]{GM1} S.\ Gol{\'e}nia, M.\ Mandich : \emph{Limiting absorption principle for discrete Schr\"odinger operators with a Wigner-von Neumann potential and a slowly decaying potential}, Ann.\ Henri Poincar\'e \textbf{22}, 83--120, (2021). %\url{https://arxiv.org/abs/2002.04909}.
\bibitem[GM2]{GM2} S.\ Gol{\'e}nia, M.\ Mandich : \emph{Bands of absolutely continuous spectrum for lattice Schr{\"o}dinger operators with a more general long range condition}, J.\ Math.\ Phys., Vol.\ 62, Issue 9, (2021). 
\bibitem[GM4]{GM4} S.\ Gol{\'e}nia, M.\ Mandich : \emph{Additional numerical and graphical evidence to support some Conjectures on discrete Schr{\"o}dinger operators with a more general long range condition}, preprint on arxiv (not intended for publication). 
\bibitem[IJ]{IJ} K.\ Ito, A.\ Jensen : \emph{Branching form of the resolvent at thresholds for multi-dimensional discrete Laplacians}, J.\ Funct.\ Anal., 277, No.\ 4, 965--993, (2019).
%\bibitem[IJ2]{IJ2} K.\ Ito, A.\ Jensen : \emph{Hypergeometric expression for the resolvent of the discrete Laplacian in low dimensions}, Integr.\ Equ.\ Oper.\ Theory \textbf{93}, 32, (2021). %\url{https://arxiv.org/pdf/2004.05866.pdf}, (2020).
%\bibitem[IK]{IK} H.\ Isozaki, E.\ Korotyaev : \emph{Inverse problems, trace formulae for discrete Schr\"odinger operators}, Ann.\ Henri Poincar\'e, \textbf{13}, 751--788, (2012).
\bibitem[Ki]{Ki} A.\ Kiselev: \emph{Imbedded singular continuous spectrum for Schr\"odinger operators}, J.\ of the AMS, Vol.\ 18, Num.\ 3, (2005), 571--603.
\bibitem[Li1]{Li1} W.\ Liu : \emph{Absence of singular continuous spectrum for perturbed discrete Schr{\"o}dinger operators}, J.\ of Math.\ Anal.\ and Appl., Vol.\ 472, Issue 2, 1420--1429, (2019).
\bibitem[Li2]{Li2} W.\ Liu : \emph{Criteria for embedded eigenvalues for discrete Schr\"odinger Operators}, International Mathematics Research Notices, rnz262, \url{https://doi.org/10.1093/imrn/rnz262}, (2019).
\bibitem[Mo1]{Mo1} E.\ Mourre: \emph{Absence of singular continuous spectrum for certain self-adjoint operators}, Comm.\ Math.\ Phys., \textbf{78}, p.\ 391--408, (1981).
\bibitem[Mo2]{Mo2} E.\ Mourre: \emph{Op\'erateurs conjugu\'es et propri\'et\'es de propagation}, Comm.\ Math.\ Phys., \textbf{91}, p.\ 279--300, (1983). 
%\bibitem[Ma1]{Ma1} M.\ Mandich : \emph{The limiting absorption principle for the discrete Wigner-von Neumann operator}, J.\ Funct.\ Anal., Vol.\ 272, Issue 6, (2017), p.\ 2235--2272.
\bibitem[Ma]{Ma} M.\ Mandich : \emph{Sub-exponential decay of eigenfunctions for some discrete Schr{\"o}dinger operators}, J.\ Spectr.\ Theory 9, 21--77, (2019).
%\bibitem[MV]{MV} S.\ Molchanov, B.\ Vainberg : \emph{Scattering on the system of the sparse bumps : multidimensional case}. Appl.\ Anal.\ 71, 167--185 (1998). 
%\bibitem[NT]{NT} S.\ Nakamura, Y.\ Tadano : \emph{On a continuum limit of discrete Schr\"odinger operators on square lattice}, J.\ Spectral. Theory, (2019). 
\bibitem[NoTa]{NoTa} Y.\ Nomura, K.\ Taira : \emph{Some properties of threshold eigenstates and resonant states of discrete Schr\"{o}dinger operators}, Ann.\ Henri Poincar\'{e} \textbf{21}, 2009--2030, (2020).
%\bibitem[R]{R} C.\ Remling : \emph{Discrete and embedded eigenvalues for one-dimensional Schr\"odinger operators}, Comm.\ Math.\ Phys., \textbf{271}, 275 -- 287, (2007).
\end{thebibliography}
\end{document}